\documentclass{cpamart1}     
\received{Month 200X}       
\volume{000}
\startingpage{1}                      

\authorheadline{C.Kim and D.Lee}
\titleheadline{theBoltzmann equation with the specular BC}

\usepackage{amssymb}
\usepackage{color}

\usepackage{mathptmx}

\newtheorem{theorem}{Theorem}[section]
\newtheorem{lemma}[theorem]{Lemma}
\newtheorem{corollary}[theorem]{Corollary}
\newtheorem{proposition}[theorem]{Proposition}
\theoremstyle{definition}
\newtheorem{definition}[theorem]{Definition}
\theoremstyle{remark}

\let\e=\varepsilon

\let\p=\partial

\let\O=\Omega

\let\o=w

\let\L=\Lambda

\let\b=\beta

\newcommand{\R}{\mathbb{R}}

\newcommand{\be}{\begin{equation}}
\newcommand{\bm}{\begin{multline}}
\newcommand{\ee}{\end{equation}}
\newcommand{\dd}{\mathrm{d}}

\newcommand{\xb}{x_{\mathbf{b}}}
\newcommand{\tb}{t_{\mathbf{b}}}
\newcommand{\vb}{v_{\mathbf{b}}}

\newcommand{\xf}{x_{\mathbf{f}}}
\newcommand{\tf}{t_{\mathbf{f}}}
\newcommand{\vf}{v_{\mathbf{f}}}
\newcommand{\X}{\mathbf{x}}
\newcommand{\V}{\mathbf{v}}

\newcommand{\dist}{\mathrm{dist}}
\numberwithin{equation}{section}
\numberwithin{theorem}{section}

\def\p{\partial}

\def\O{\Omega}
\def\R{\mathbb{R}}

\def\B{\begin{equation}}
\def\E{\end{equation}}
\def\BN{\begin{eqnarray*}}
\def\EN{\end{eqnarray*}}

\begin{document}                        


\title{The Boltzmann equation with specular boundary condition in convex domains}

\author{Chanwoo Kim} {The University of Wisconsin-Madison}

\author{Donghyun Lee}{The University of Wisconsin-Madison}





\begin{abstract}
We establish the global-wellposedness and stability of the Boltzmann equation with the specular reflection boundary condition in general smooth convex domains when an initial datum is close to the Maxwellian with or without a small external potential. In particular, we have completely solved the long standing open problem after an announcement of \cite{SA} in 1977.
\end{abstract}

\maketitle   



\tableofcontents



\section{Introduction}
Kinetic theory describes the dynamics of any system made up of a large number of particles (e.g. gas, plasma) by a distribution function in the phase space. The Boltzmann equation describes the interaction among particles and it is one of the foundations of the kinetic theory for dilute collections of gas particles undergoing elastic binary collisions. In the presence of \textit{an external potential}, a density of dilute charged gas particles is governed by \textit{the Boltzmann equation}
\begin{equation}\label{Boltzmann_phi}
\partial_t F + v \cdot\nabla_x F -\nabla_x (  \phi(t,x) + \Phi(x))\cdot\nabla_v F = Q(F,F), \ \ \ \ \ F(0,x,v) = F_0 (x,v),
\end{equation}
%
where $F(t,x,v)$ is a distribution function of the gas particles at a time $t\geq 0,$ a position $x\in  \Omega \subset \mathbb{R}^{3} $, and a velocity $v\in \mathbb{R}^{3}$. Here, the collision operator $Q$ takes the form of 
\begin{eqnarray}
Q(F_{1},F_{2})
&:=&\int_{\mathbb{R}^{3}}\int_{\mathbb{S}^{2}} \Big[ B(v-u,\omega)
F_{1}(u^{\prime })F_{2}(v^{\prime }) - B(v-u,\omega) F_{1}(u)F_{2}(v) \Big] \dd \omega \dd u  \notag\\
&:=&Q_{+}(F_{1},F_{2})-Q_{-}(F_{1},F_{2})  ,
\notag
\end{eqnarray}%
where $u^{\prime }=u+[(v-u)\cdot \omega ]\omega $, $v^{\prime
}=v-[(v-u)\cdot \omega ]\omega $, and
$B(v-u,\omega) =|(v-u) \cdot \omega|$ (hard sphere). It is well-known (\cite{K2}) that 
the following \textit{local Maxwellian} is an equilibrium solution to (\ref{Boltzmann_phi})
\begin{equation}
\mu_E(x,v)   \ = \ \mu(v) e^{-\Phi  (x)} , \label{E_Max}
\end{equation}
where $\mu(v)=e^{- {|v|^{2}} /2} 
$ is the standard \textit{global Maxwellian}.

In many physical applications, e.g. dilute gases passing objects and a plasma inside tokamak devices, particles are interacting not only with each other but also with the boundary. Various important phenomena occur when gas particles interact with the boundary, such as the formation and propagation of singularities (\cite{Kim11,GKTT1, GKTT2}). In the presence of the boundary, a kinetic equation has to be supplemented with boundary conditions modeling the interaction between the particles and the boundary.
%
Among other boundary conditions (See \cite{CIP,Guo10}), in this paper, we focus on one of the most basic conditions, a so-called \textit{specular reflection boundary condition}. This mathematical model takes into account a case that if a gas particle hits the boundary, then it bounces back with the opposite normal velocity and the same tangential velocity, as a billiard: 
\begin{equation}\label{specular_F}
F(t,x,v) = F(t,x,R_{x}v) \ \ \ \text{for} \ \ x \in \p\Omega,
\end{equation}
where $R_{x}v= v- 2 (n(x) \cdot v)n(x)$. We note that the local Maxwellian (\ref{E_Max}) satisfies the boundary condition (\ref{specular_F}).

Despite extensive developments in the study of the Boltzmann theory, many basic boundary problems, especially regarding the specular reflection BC with general domains, have remained open. In a landmark paper of 1974, Ukai constructs the first global-in-time solutions near Maxwellians to the Boltzmann equation with non-trivial spatial dependence in a periodic box (no boundary).  Not long after, in 1977, Shizuta and Asano announced the construction of global solutions to the Boltzmann equation (\ref{Boltzmann_phi}) with no external potential ($\phi\equiv0 \equiv \Phi$) near Maxwellians in \textit{smooth convex domains} with specular reflection boundary condition (\cite{SA}), but without mathematical proofs. It took more than 30 years to encounter the first mathematical resolution: Guo, in \cite{Guo10}, developed a novel $L^{2}-L^{\infty}$ argument to construct a unique solution to the Boltzmann equation (\ref{Boltzmann_phi}) with no external potential for the specular reflection boundary condition. An asymptotic stability of the global Maxwellian $\mu$ is proven when an initial datum is close to $\mu$. However, such results in \cite{Guo10} are established under an \textit{extra strong condition}, namely, the boundary is a level set of a \textit{real analytic} function. Indeed, \textit{the analyticity condition} is crucially used to verify the key part of the proof in \cite{Guo10}. Finally, in this paper, we are able to establish the global wellposedness and stability of the Boltzmann equation for the specular reflection BC \textit{without the analyticity} and thereby we completely settle this classical long-standing (nearly 40 years) open question in the Boltzmann theory in the affirmative! In fact, main result even goes beyond the original open question in \cite{SA}: a nontrivial external potentials $\phi(t,x)$  and $\Phi(x)$ can be allowed. We discuss more on the external potential problem in another paragraphs.

Here let us only briefly mention some relevant works. In \cite{BG, Kuo}, the well-posedness and asymptotic stability of the global Maxwellian are studied when the boundary condition is any convex combination of the specular reflection BC and the diffusive BC \textit{except} the pure specular reflection boundary condition. For large amplitude solutions, an asymptotic stability of the global Maxwellian is established in \cite{DV} with or without the boundary, provided certain a-priori strong Sobolev estimates can be verified. Recently boundary regularity and singularity of solutions are extensively studied in \cite{Kim11, GKTT1, GKTT2}. We refer \cite{M} among others for the weak solution contents. 


Mathematical problems on the Boltzmann equation with \textit{an external potential} also have drawn lots of attention. In \cite{K2}, the stability of the Maxwellian $\mu_{E}$ in (\ref{E_Max}) is established with a fixed external potential $\Phi(x)$, which can be large, in a periodic box. The Vlasov-Poisson-Boltzmann system (VPB), which takes account of self-consistent electric fields by charged particles, is studied in \cite{Guo_P} when solutions and fields are small perturbations in a periodic box. However, in many important physical applications (e.g. semiconductor, tokamak), the charged dilute gas interacts with the boundary. 
One of the major difficulty is that trajectories are curved and behave in a very complicated way when they hit the boundary. As the first step toward studying models of dilute charged gases interacting with a self-consistent field and boundary, in this paper we establish the global well-posedness of the Boltzmann equation coupled with small \textit{external potentials} and the specular reflection BC.

An external potential and a boundary condition play an important role in the evolution of macroscopic quantities such as the total mass,  total momentum, and  total energy.  Let $F$ be a solution to (\ref{Boltzmann_phi}) satisfying the specular reflection boundary condition (\ref{specular_F}). We have the total mass conservation and the evolution of the total energy as
\begin{equation}
   \iint_{\Omega\times\R^{3}} F(t )   =  \iint_{\Omega\times\R^{3}} F_{0}   ,\label{conserv_F_mass}
\end{equation}
\begin{equation}
\begin{split}
  \iint_{\Omega\times\R^{3}} \Big( \frac{|v|^{2}}{2} + \Phi  \Big) F(t ) +
\int^{t}_{0}\iint_{\O\times\R^{3} }
F(s ) v\cdot \nabla_{x}   \phi (s) \\
= 	 \iint_{\Omega\times\R^{3}} \Big( \frac{|v|^{2}}{2} + \Phi   \Big) F_{0} 
.\label{evol_energy}
\end{split}
\end{equation}
%
%

%

\noindent  By normalization, without loss of generality, we assume that 
\begin{equation}\begin{split}\label{normalize_M}
& \iint_{\O\times \R^{3}} F_{0} (x,v)   =  \iint_{\O\times \R^{3}} \mu_{E} (x,v) , \\
&  \iint_{\O\times \R^{3}} \Big(\frac{|v|^{2}}{2}  + \Phi   (x)\Big) F_{0} (x,v)   =  \iint_{\O\times \R^{3}} \Big(\frac{|v|^{2}}{2}  + \Phi  (x)\Big)\mu_{E} (x,v)  .
\end{split}
\end{equation}

We consider a momentum for a special case: a domain $\Omega$ is \textit{axis-symmetric} if there are vectors $x_{0}$ and $\varpi$ such that 
\begin{equation}\label{axis-symmetric}
\{  (x-x_{0}) \times \varpi\} \cdot n(x) =0 \ \ \ \ \  \text{for all} \ x \in \p\Omega.
\end{equation}
In the case of an axis-symmetric domain, we assume a degenerate condition for the external fields as  
\begin{equation}\label{degenerate}
\{  (x-x_{0}) \times \varpi\} \cdot \nabla_{x} ( \phi(t,x) +  \Phi (x)) =0 \ \  \ \ \ \text{for all}  \  t\geq 0 \ \text{and}  \ x\in \O.
\end{equation}
Then, assuming both (\ref{axis-symmetric}) and (\ref{degenerate}), we have an evolution of an angular momentum as
\begin{equation}\label{conserv_F_angular}
\iint_{\O\times \R^{3}} \{  (x-x_{0}) \times \varpi\} \cdot v  F(t )
=   \iint_{\O\times \R^{3}} \{  (x-x_{0}) \times \varpi\} \cdot v  F_{0}
.
\end{equation}
In this case, we set
\begin{equation}
\label{normalize_M_angular}
\iint_{\O\times \R^{3}} \{  (x-x_{0}) \times \varpi\} \cdot v  F_{0}(x,v)
= 0.
\end{equation}

Furthermore the entropy $$\mathcal{H}(F) : =
\iint_{\O\times \R^{3}} F  \ln F  $$ satisfies the following inequality (\textit{H-theorem})
\begin{equation}\label{entropy_bound}
\mathcal{H}(F(t)) - \mathcal{H} (\mu_{E}) \leq \mathcal{H}(F_{0}) - \mathcal{H} (\mu_{E}) .
\end{equation}

\bigskip

Now we are ready to state our main theorems.

\begin{theorem}\label{theorem_time}
	Let $w = (1+|v|)^{\beta}$ for $\beta>5/2$.  Assume that the domain $\O \subset \R^{3}$ is $C^{3}$ and convex in (\ref{convexity_eta}). Assume that $\phi(t,x) \in C_{t,x}^{2,\gamma}$ and $\Phi (x)\in C_{x}^{2,\gamma}$ for some $0<\gamma$, $\|\phi+ \Phi\|_{ C^2}\ll 1$, 
	and \begin{equation}
	\label{exp_phi}
	\sup_{t \geq 0} e^{\lambda_{\phi} t}	\| \phi(t) \|_{C^{1}}
	< \delta_{\phi} <+\infty.
	\end{equation}

	\noindent	Assume (\ref{normalize_M}). 
	If $F_{0} = \mu_{E} + \sqrt{\mu_{E}} f_{0} \geq 0$ and $\| w f_{0} \|_{\infty} + |\mathcal{H}(F_{0})-\mathcal{H} (\mu_{E})| + \delta_{\phi} + \delta_{\phi}/ \lambda_{\phi} \ll 1$, then %
	%
	there exists a unique global-in-time solution 
	\begin{equation}\label{def_f}
	F(t)= \mu_{E} + \sqrt{\mu_{E}}f(t) \geq 0,
	\end{equation}
	to (\ref{Boltzmann_phi}) satisfying the specular reflection boundary condition (\ref{specular_F}).  Moreover,
	\begin{equation}\label{final_est_time}
	\sup_{t\geq 0}\| w f (t) \|_{\infty}\lesssim  \| w f_{0 } \|_{\infty} + 
	|\mathcal{H}(F_{0}) - \mathcal{H}(\mu_{E})| + \delta_{\phi}+ \delta_{\phi}/ \lambda_{\phi}.
	\end{equation}
	
	\noindent 	Furthermore, (\ref{conserv_F_mass}), (\ref{evol_energy}), and (\ref{entropy_bound}) hold for all $t\geq 0$.
\end{theorem}
 
 Here a $C^{3}$ domain means that for any boundary point ${p} \in \partial{\Omega}$, locally there exists an one-to-one and onto $C^{3}$-function $\eta_{\mathbf{p}}$ such that 
%
$\eta_{ { {p}}}(\X_{ {{p}},1},\X_{ {{p}},2},\X_{ {{p}},3}) \in \p \Omega$ if and only if $\X_{ {{p}},3}=0$ (See (\ref{eta})). 
The convexity is defined as follows, for $C_{\O}>0$,
\begin{equation}\label{convexity_eta}
\sum_{i,j=1}^{2} \xi_{i} \xi_{j}\p_{i} \p_{j} \eta _{{p}}   ( \X_{{p},1},  \X_{{p},2},0)\cdot  
\p_{3} \eta_{{p}}(\X_{{p},1}, \X_{{p},2},0)
\ \leq \  - C_{\Omega} |\xi|^{2} \ \ \ \text{for}  \ \ \xi \in \mathbb{R}^{2}.
\end{equation}
Here, $C^{\alpha, \gamma}$ stands the standard H\"older space. 


\bigskip

In the presence of a time-independent external potential ($\phi\equiv 0$), the asymptotical stability of the local Maxwellian $\mu_{E}$ is studied.

\begin{theorem}\label{theorem_decay}
	Assume the same conditions in Theorem \ref{theorem_time} before (\ref{exp_phi}). Let 
	\begin{equation}
	\phi\equiv0.\label{phi=0}
	\end{equation}
	Assume (\ref{normalize_M}). If both (\ref{axis-symmetric}) and (\ref{degenerate}) hold, then we assume (\ref{normalize_M_angular}). If $\|wf_0\|_\infty \ll 1$, then there exists a unique global-in-time solution $F=\mu_E + \sqrt{\mu_E}f \geq 0$ to (\ref{Boltzmann_phi}) with (\ref{specular_F}). Moreover, for some $\lambda= \lambda (\Omega, \Phi) > 0$ we have
	\begin{equation}\label{f_decay}
	\sup_{  t\geq 0} e^{\lambda t} \|w f(t)\|_\infty \lesssim \|wf_0\|_\infty .
	\end{equation}  
	
	\noindent Furthermore, the total mass and energy are conserved (\ref{conserv_F_mass}), (\ref{evol_energy}) with $\phi\equiv 0$, and the total angular momentum is so (\ref{conserv_F_angular}) if both (\ref{axis-symmetric}) and (\ref{degenerate}) hold.

	%
	%
	%
\end{theorem}

Remark that we do not have a quantitative bound of $\lambda$ in (\ref{f_decay}). The main reason is that we use a non-constructive method to prove $L^{2}$ coercivity in Proposition \ref{prop_coercivity}.

We remark that in the both theorems we only need that the domain $\O$ is smooth and convex but \textit{not real analytic}. We also note that in \cite{K2} we need a stronger $C^{3}$ assumption for the time-independent external potential to establish the well-posedness.

\vspace{4pt}

To illustrate the main ideas of the paper, it is convenient to play with the perturbation $f$. The function $f$ in (\ref{def_f}) solves
\begin{equation}\label{E_eqtn}
\p_{t}f + v\cdot \nabla_{x} f  - \nabla_{x}(  \phi + \Phi ) \cdot \nabla_{v} f+e^{- \Phi } L f 
=  - (\frac{1}{2}f + \sqrt{\mu_E})v\cdot \nabla_{x}  \phi  + e^{-  \frac{\Phi  }{2}}\Gamma(f,f)  ,
\end{equation} 
and satisfies
\begin{equation}\label{specular_f}
f(t,x,v) = f(t,x,R_{x}v) \  \ \ \ \text{for} \  \ x \in \p\O.
\end{equation}	
We recall the definition of the linearized collision operator (see \cite{CIP}): 
\begin{equation}  \label{def_L}
Lf=-\frac{1}{\sqrt{\mu}}[Q(\mu,\sqrt{\mu} f)+ Q(\sqrt{\mu} f ,\mu)],
\end{equation}
and the nonlinear collision operator: $$
\Gamma(f,g)=\frac{1}{2\sqrt{\mu} }[Q(\sqrt{\mu} f,\sqrt{\mu} g)+Q(\sqrt{\mu}
g,\sqrt{\mu} f)].$$ 

It is well-known that
\begin{equation}
L f= \nu f -Kf,  \notag
\end{equation}
where {the collision frequency is defined as} 
\begin{equation}
\nu(v)
:=\int_{\mathbb{R}^3} \int_{%
	\mathbb{S}^2} |(v-u)\cdot \o | \sqrt{\mu}(u)\mathrm{d}\o \mathrm{d} u. 
\notag
\end{equation}
{For the hard sphere case, there are positive
	numbers $C_0$ and $C_1$ such that, for $\langle v\rangle:= \sqrt{1+ |v|^{2}}$%
	, 
	\begin{equation}
	C_0 \langle v\rangle \le \nu(v)\le C_1 \langle v\rangle.  \label{nu0}
	\end{equation}
	Moreover, the compact operator on 
	$L^2(\mathbb{R}_v^3)$, %
	$K$ is defined as 
	\begin{equation}
	K f= \frac{1}{\sqrt{\mu}}[Q_+(\mu,\sqrt{\mu} f)+Q_+(\sqrt{\mu} f,\mu)
	-Q_-(\mu,\sqrt{\mu} f)]=\int_{\mathbb{R}^3}  \mathbf{k} (v,u) f(u) \mathrm{d} u.  \notag
	\end{equation}

	\subsection{$L^{p}-L^{\infty}$-bootstrap argument via the triple iterations} In order to handle the quadratic nonlinearity of $\Gamma(f,f)$, it is important to derive an $L^{\infty}$-control of the solutions of (\ref{E_eqtn}). To illustrate the main idea, we consider a simplified linear problem
	\begin{equation}\label{simple_eqtn}
	\p_{t} f + v\cdot \nabla_{x} f  - \nabla_{x}\Phi(t,x) \cdot \nabla_{v}f+     f = \int_{|u| \leq N}  f(u) \dd u.
	\end{equation}
	Here, $\Phi(t,x)$ is a time-dependent potential and we can regard $\phi(t,x) + \Phi (x)$ in (\ref{Boltzmann_phi}) as it.

	We note that due to the boundary condition (\ref{specular_f}), the trajectory \\
	$(X(s;t,x,v), V(s;t,x,v))$ is defined as the backward billiard trajectory which is curved by the external field (or force) $-\nabla\Phi$. Let $t^{1}, x^{1}$ be the first backward bouncing time and the position of the trajectory sitting on a position $x$ with a velocity $v$ at time $t$. Then we define $v^{1} = R_{x^{1}} v$ where $R_{x^{1}} v$ is defined in (\ref{specular_F}). Inductively we can define the cycles $(t^{\ell}, x^{\ell}, v^{\ell})$ and $X_{\mathbf{cl}}(s;t,x,v) = X(s;t^{\ell}, x^{\ell}, v^{\ell})$ and $V_{\mathbf{cl}}(s;t,x,v) = V(s;t^{\ell}, x^{\ell}, v^{\ell})$ for $s \in [t^{\ell+1},t^{\ell}]$. The Duhamel formula of (\ref{simple_eqtn}) along this trajectory is given by 
	\begin{equation}\label{DH1}
	\begin{split}
	f(t,x,v ) &= e^{- t} f_{0}(X_{\mathbf{cl}}(0;t,x,v), V_{\mathbf{cl}}(0;t,x,v)) \\
	&\quad + \int^{t}_{0} e^{- (t-s)} 
	\int_{|u| \leq N} f(s, X_{\mathbf{cl}}(s;t,x,v), u) \dd u
	\dd s.
	\end{split}
	\end{equation}
	Plugging the Duhamel formula into the integrand $f(s, X(s;t,x,v), u)$, we get
	\begin{equation}\label{DH2}
	\begin{split}
	&f(t,x,v)\\
	=& \int^{t}_{0} e^{- (t-s)} \int^{s-\e}_{0} e^{- (s-s^{\prime})}
	\iint_{|u|\leq N, |u^{\prime}| \leq N}
	f(s^{\prime}, X_{\mathbf{cl}}(s^{\prime}; s, X_{\mathbf{cl}}(s;t,x,v) , u), u^{\prime}) \dd u^{\prime}\dd u \dd s^{\prime}\dd s\\
	&+ \textit{initial datum's contributions} + O(\e).
	\end{split}
	\end{equation} Throughout this paper, we use $O_{a}(A)$ for some function which depends on $a$ and is size of $A$. 
	
	In the absence of the boundary and the external potential, the trajectory $X(s;t,x,v)$ is a straight line and we can compute the Jacobian of $u \mapsto X(s^{\prime}; s, X(s;t,x,v) , u)$ explicitly which has a positive lower bound away from a small set of $s$. Therefore we obtain, via a change of variables,
	\begin{equation}\label{p_infty}
	\| f  \|_{L^{\infty}} \lesssim \| f \|_{L^{p} } + \textit{data}
	+ \textit{small terms}.
	\end{equation}
	
	\noindent Unfortunately, trajectories are very complicated when the specular reflection BC is imposed. In fact, in the case of the specular reflection BC, such a lower bound of Jacobian is only known when the domain is convex and \textit{real analytic} in the absence of an external potential \cite{Guo10}. 
	
	\vspace{4pt}
	
	The main contribution of this paper is to establish $L^{p}-L^{\infty}$ bootstrap estimate as (\ref{p_infty}), when the domain is smooth and convex and the external potential is $C^{2,\gamma}$ and small in $C^{2}$. For readers' convenience, we write a rough version of this result:

	\textit{\textbf{A rough version of Theorem \ref{prop_full_rank}.}  Applying the Duhamel formula once again to (\ref{DH2}) (triple iterations), we have }
	\begin{equation}\label{DH3}
	\begin{split}
	& f(t,x,v) \\
	=& \int^{t}_{0} e^{- (t-s)} \int^{s }_{0} e^{- (s-s^{\prime})}
	\int^{s^{\prime}-\e}_{0} e^{-(s^{\prime} - s^{\prime\prime})}\iiint_{|u|\leq N, |u_{1}| \leq N, |u^{\prime\prime}| \leq N}
	\\
	& \ \ \ \ \  \times 
	f( s^{\prime\prime}, X_{\mathbf{cl}}(s^{\prime\prime}; s^{\prime}, X_{\mathbf{cl}}(s^{\prime}; s, X_{\mathbf{cl}}(s;t,x,v) , u), u^{\prime}), u^{\prime\prime}
	) \dd u^{\prime\prime}  \dd u^{\prime} \dd u \dd s^{\prime\prime} \dd s^{\prime} \dd s\\
	& + \textit{initial datum's contributions} + O(\e).
	\end{split}
	\end{equation}

	\vspace{4pt}
	
	\noindent \textit{Let $(\hat{u}_{1}, \hat{u}_{2})$ and $(\hat{u}^{\prime}_{1}, \hat{u}^{\prime}_{2})$ be the spherical coordinate of $\hat{u}= \frac{u}{|u|} \in \mathbb{S}^{2}$ and $\hat{u}^{\prime}= \frac{u^{\prime}}{|u^{\prime}|}\in \mathbb{S}^{2}$, respectively. Then, if $s^{\prime}$ and $s^{\prime\prime}$ are away from some local $C^{0,\gamma}$-functions, then locally we can choose two distinct variables $\{\zeta_{1}, \zeta_{2}\}$ among $\{|u|, \hat{u}_{1},  \hat{u}^{\prime}_{1}, \hat{u}^{\prime}_{2}\}$ such that }
	
	\vspace{4pt}
	\textit{
		\begin{equation}\label{quant lower bound}
		\begin{split}
		\Big|\det \left(\frac{\p X_{\mathbf{cl}}(s^{\prime\prime}; s^{\prime}, X_{\mathbf{cl}}(s^{\prime}; s, X_{\mathbf{cl}}(s;t,x,v) , u), u^{\prime})}{\p (|u^{\prime}|, \zeta_{1}, \zeta_{2})}\right)\Big| \ \ \text{has a positive lower bound}.
		\\
		\end{split}
		\end{equation} }
	
	\vspace{2pt}
	
	\noindent \textit{As a consequence we achieve (\ref{p_infty}).
	}
	
	\vspace{2pt}
	
	We remark that the regularity of such $C^{0,\gamma}$-functions is determined and restricted crucially by the regularity of the external potential $\Phi \in C^{2,\gamma}$. Moreover, this $C^{0,\gamma}$-regularity is a (minimal) condition to guarantee that we can construct \textit{small} $\e$-neighborhooods of the graph of them.
	
	There are several key ingredients in the proof of Theorem \ref{prop_full_rank}: 

	\vspace{2pt}

	\textit{Specular Basis and Geometric Decomposition. } Assume that $t^{\ell+1}<s^{\prime}< t^{\ell}$ and hence $X_{\mathbf{cl}}(s^{\prime}; s, X_{\mathbf{cl}}(s;t,x,v) , u)$ is in between $\ell-$bounce and $(\ell+1)-$bounce. Then we know that
	\begin{equation}\label{X_|u|}
	{\p_{|u|} X_{\mathbf{cl}}(s^{\prime}; s, X_{\mathbf{cl}}(s;t,x,v) , u)}
	=  {v^{\ell}}/{|v^{\ell}|} +O(\| \Phi \|_{C^{2}}).
	\end{equation}  
	
	\noindent On the other hand, for $\hat{u}= (\hat{u}_{1} , \hat{u}_{2}) \in\mathbb{S}^{2}$, we have
	\begin{equation}\label{X_perp}
	\nabla_{\hat{u}}X_{\mathbf{cl}}(s ^{\prime}; s, X_{\mathbf{cl}}(s;t,x,v) , u)= \nabla_{\hat{u}}x^{\ell} - (t^{\ell} -s^{\prime}) \nabla_{\hat{u}}v^{\ell} - \nabla_{\hat{u}} t^{\ell} v^{\ell}+O(\| \Phi \|_{C^{2}}).
	\end{equation}
	Among other terms, $\p_{\hat{u}}t^{\ell}$ is the most delicate term to control since $t^{\ell}$ depends on all the cycles $(x^{l}, v^{l})$ for $l=1,2,\cdots, \ell-1$. Fortunately, this harmful term appears only in the direction of $\frac{v^{\ell}}{|v^{\ell}|}$! Inspired by this observation we define the \textit{specular basis} $\{ \mathbf{e}^{\ell}_{0}, \mathbf{e}^{\ell}_{\perp,1},\mathbf{e}^{\ell}_{\perp,2} \}$ which are an orthonormal basis with $\mathbf{e}^{\ell}_{0}=  {v^{\ell}}/{|v^{\ell}|}$ and $\mathbf{e}^{\ell}_{\perp,i}$ are perpendicular to $\mathbf{e}^{\ell}_{0}$. See (\ref{orthonormal_basis}).

	Now we \textit{decompose} $\nabla_{|u|, \hat{u}_{1}, \hat{u}_{2}} X_{\mathbf{cl}}(s^{\prime}; s, X_{\mathbf{cl}}(s;t,x,v) , u)$ into

	\begin{equation}\label{decom_intro}
	\nabla X_{\mathbf{cl}} = \big(\nabla X_{\mathbf{cl}}\big)_{\parallel} + \big( \nabla X_{\mathbf{cl}}\big)_{\perp}
	:=\big(\nabla X_{\mathbf{cl}}  \cdot \mathbf{e}^{\ell}_{0}\big)\mathbf{e}^{\ell}_{0} + \nabla X_{\mathbf{cl}}  - \big(\nabla X_{\mathbf{cl}} \big)_{\parallel}
	.
	\end{equation}
	Then we have the following similarity relations, from (\ref{X_|u|}) and (\ref{X_perp}),
	\begin{equation}\label{dec_mat}\begin{split}
	\frac{\p X_{\mathbf{cl}}
	}{\p (|u|, \hat{u} ) } &\sim  \left(  \begin{array}{c}
	\big(\frac{\p X_{\mathbf{cl}}}{\p (|u|, \hat{u} ) }\big)_{\parallel}   \\
	\big(\frac{\p X_{\mathbf{cl}}}{\p (|u|, \hat{u} ) }  \big)_{\perp}
	\end{array}  \right) \\
	&
	\sim \left(\begin{array}{c|cc}
	-(t-s) & * - |v^{\ell}| \nabla_{\hat{u}_{1}, \hat{u}_{2}} t^{\ell}    \\ \hline
	\mathbf{0}_{2 \times 1} &
	\begin{array}{c}
	\big(\nabla_{ \hat{u}_{1}, \hat{u}_{2} } x^{\ell } - (t^{\ell}-s^{\prime}) \nabla_{  \hat{u}_{1}, \hat{u}_{2} } v^{\ell}\big) \cdot \mathbf{e}^{\ell}_{\perp,1}\\
	\big(\nabla_{  \hat{u}_{1}, \hat{u}_{2} } x^{\ell } - (t^{\ell}-s^{\prime}) \nabla_{  \hat{u}_{1}, \hat{u}_{2} }v^{\ell}\big) \cdot \mathbf{e}^{\ell}_{\perp,2}
	\end{array}
	\end{array} \right) + O_{  \Omega  }(\| \Phi \|_{C^{2}}).\end{split}
	\end{equation}
	See (\ref{sub_R}) for the precise form. 
	
	Due to this \textit{geometric decomposition}, we are able to relate $\frac{\p X_{\mathbf{cl}}
	}{\p (|u|, \hat{u} ) } $ to 
	the mapping 
	\begin{equation}\label{u to ell}
	(|u|, \hat{u}_{1}, \hat{u}_{2})  \ \mapsto \  (x^{\ell}, v^{\ell}).
	\end{equation}
	Note that the map (\ref{u to ell}) is closely related to the billiard map \cite{CM} which turns out to be ``controllable'' than $\frac{\p X_{\mathbf{cl}}
	}{\p (|u|, \hat{u} ) } $. Moreover, the form of the first column of (\ref{dec_mat}) clearly guarantees that this Jacobian matrix is at least rank 1 for a small $\|\Phi \|_{C^{2}}$. 
	
	\vspace{2pt}

	\textit{Diffeomorphism and Specular Matrix. } By the chain rule, we can view (\ref{u to ell}) as the compositions of 
	\begin{equation}\label{chain billiard}
	(|u|, \hat{u}_{1}, \hat{u}_{2})
	\ \mapsto \  (x^{1}, v^{1}) \ \mapsto \  (x^{2}, v^{2}) \ \mapsto \   \cdots
	\ \mapsto \  (x^{\ell}, v^{\ell})  .
	\end{equation}
	In the absence of external potentials, the map $(x^{l},v^{l}) \mapsto (x^{l+1},v^{l+1})$ is the billiard table and it is well-known that this map is diffeomorphic \cite{CM}. 
	
	The quantitative study of such a map, especially in 3D domains, is performed recently in the work \cite{GKTT1} of the first author with his collaborators when the trajectories are very close to the boundary with almost tangential velocities to the boundary (grazing trajectories) in the absence of external potentials. However,
	these estimates cannot be sufficient for our purpose since it only can provide the information for the grazing trajectories. Moreover, the proof of \cite{GKTT1} heavily relies on the fact the ODE of the trajectory is autonomous. In the presence of a time-dependent external potential, however, the ODE of $(X_{\mathbf{cl}},V_{\mathbf{cl}})$ becomes non-autonomous which obstructs generalizing the result of \cite{GKTT1} to the time-dependent external potential case.
	
	We are able to overcome this difficulty by an advance of direct computations. In this paper, we succeed to perform the (almost) explicit computations of the Jacobian matrix of (\ref{chain billiard}) in the presence of a small time-dependent external potential. This also allows us to understand the role of the regularity of the external potential in (\ref{quant lower bound}). We expect that this technical improvement will allow us some generalization of the work of \cite{GKTT1}.
	%

	Equipped with this quantitative estimate, we study the lower right $2 \times 2$ submatrix of (\ref{dec_mat}). In order to use the diffeomorphim property of (\ref{chain billiard}), we employ the \textit{Specular matrix} $\mathcal{R}$ which is a $4 \times 4$ full rank matrix and essentially the Jacobian matrix of (\ref{chain billiard}) expressed with the \textit{specular basis}. The precise form can be computed as in (\ref{specular_matrix}) and the entries are $C^{0,\gamma}$ if the external potential is $C^{2,\gamma}$. Indeed, the lower right $2 \times 2$ submatrix of (\ref{dec_mat}) can be written as in (\ref{sub_R}),
	\begin{equation}
	\text{right upper } 2\times2 \text{ submatrix of } \mathcal{R}  - (t^{\ell}-s^{\prime}) \times \text{right lower } 2\times2 \text{ submatrix of } \mathcal{R}.
	\end{equation}
	Since at least one entry of right $4\times 2$ submatrix of $\mathcal{R}$ should not be zero as a polynomial of $s$, we are able to show that $(|u| , \hat{u}_{1}) \mapsto X$ is at least \textit{rank 2} if $s$ is away from some $C^{0,\gamma}$-function of $(t,x,v)$ in Lemma \ref{lemma rank 2}.
	
	\vspace{2pt}
	
	\textit{Triple Iterations. } Unfortunately, this \textit{rank 2} is still not sufficient for our purpose. The key idea to overcome this difficulty is the \textit{triple iterations} in (\ref{DH3}), applying the Duhamel formula (\ref{DH1}) once again to (\ref{DH2}). 
	One more iteration makes the game more feasible since now we have more free parameters to play with: $\{|u|, \hat{u}_{1}, |u^{\prime}|,\hat{u}^{\prime}_{1},\hat{u}^{\prime}_{2}\}\in\R^{5}.$ 
	Due to the observation (\ref{X_|u|}), we need to choose $|u^{\prime}|$ and two other free parameters $\{\zeta_{1}, \zeta_{2}\}$ so that the following map is rank 3,
	\begin{equation}\label{mappp}
	(|u^\prime|, \zeta_{1}, \zeta_{2}) \mapsto
	X(s^{\prime\prime}; s^{\prime}, X(s^{\prime}; s, X(s;t,x,v), u), u^{\prime}).
	\end{equation}
	We use the full structure of the \textit{specular matrix} and carefully study the quadratic polynomial (Lemma \ref{zero_poly}) to achieve a positive lower bound of the Jacobian of (\ref{mappp}) in Lemma \ref{lemma rank 3}. The convexity of the domain (\ref{convexity_eta}) is used crucially to control the number of bounces in Lemma \ref{OV}.

	\vspace{4pt}
	
	\subsection{$L^{p}$-bounds} Now we illustrate the $L^{p}$ control of the Boltzmann solution. Due to the $L^{p}-L^{\infty}$ bootstrap estimate (\ref{p_infty}), such $L^{p}$ estimates would provide $L^{\infty}$ control. 
	
	\vspace{4pt}
	
	$L^{1}$\textit{-bound in the case of a time-dependent potential}. In order to show the stability of $\mu_{E}$ in the presence of time-dependent potential $\phi$, we utilize the following bound of \cite{bounded_boltzmann, Kim11}. 
	\begin{lemma}\label{entropy_bound_intro} 
		\begin{equation}\label{est_via_entropy}
		\begin{split}
		& |F- \mu_{E}|  \mathbf{1}_{|F- \mu_{E}|\geq \bar{\delta} \mu_{E}}\\
		&\leq
		\frac{4}{\bar{\delta}} \big\{ 
		(F  \ln F  - \mu_{E} \ln \mu_{E} ) -(F - \mu_{E}) + \Big(  \frac{|v|^{2}}{2} + \Phi(x) \Big) (F - \mu_{E}) 
		\} .
		\end{split}
		\end{equation}
		%
		%
	\end{lemma}
	Applying $L^{p}-L^{\infty}$ bootstrap argument via the triple iteration, the $L^{\infty}$-norm of the solution is mainly bounded by $L^{1}$-norm of $  |F- \mu_{E}|  \mathbf{1}_{|F- \mu_{E}|\geq \bar{\delta} \mu_{E}}$. By Lemma \ref{entropy_bound_intro}, we further it by the differences in the entropy, total mass, total energy of the solution and $\mu_{E}$. A new difficulty in the presence of the time-dependent potential $\phi(t,x)$ is that the total energy is not preserved anymore (\ref{evol_energy}). Via the Gronwall's inequality, we are able to prove that $\|w f(t) \|_{\infty}$ can grow in time at most as $e^{C (\| \phi \|_{\infty} + \| w f \|_{\infty})t}$. With strongly decaying potential and small $f$, we can prove that the total energy is close to the initial total energy for all time. This weighted $L^{\infty}$-bound is sufficient to prove the existence, uniqueness, and the stability of $\mu_{E}$ in Theorem \ref{theorem_time}.
	
	\vspace{4pt}
	
	$L^{2}$\textit{-decay in the case of a time-independent potential}. It is well-known \cite{CIP} that the linear operator $L$ is only semi-positive 
	\begin{equation}\label{semi-positive}
	\int_{\R^{3}} L f f \dd v \geq  {\delta}_{L} \| \sqrt{\nu} (\mathbf{I} - \mathbf{P} ) f\|_{L^{2}(\R^{3})}^{2},
	\end{equation}
	where $\|  \cdot \|_{\nu} = \| \nu^{1/2} \cdot  \|_{L^{2}}$. The null space of $L$ is a five-dimensional subspace of $L^{2} (\R^{3})$ spanned by $\big\{ \sqrt{\mu_{E}}, v\sqrt{\mu_{E}},  |v|^{2} \sqrt{\mu_{E}}\big\}$ and the projection of $f$ onto such null space is denoted by  
	\begin{equation}\label{E_Pf}
	\mathbf{P}f (t,x,v) \ := \ \Big\{ a(t,x) + v\cdot  b(t,x) + |v|^{2}   c(t,x)\Big\} \sqrt{\mu_{E} }.
	\end{equation} 
	Due to this missing term the Boltzmann equation is degenerated dissipative. In order to prove $L^{2}$-decay, we need a coercivity estimate. Following the argument of \cite{Guo10} we first consider
	\begin{equation} \label{linearized_eqtn}
	\p_{t}f + v\cdot \nabla_{x} f  - \nabla_{x} \Phi(x) \cdot \nabla_{v} f  + e^{- \Phi (x) } L f 
	= 0.
	\end{equation}
	
	\begin{proposition}\label{prop_coercivity}
		Let $\Phi(x) \in C^{1}$. Assume that $f$ solves (\ref{linearized_eqtn}) and satisfies the specular reflection BC and (\ref{conserv_F_mass}), (\ref{evol_energy}) with $\phi\equiv0$ for $F= \mu_{E} + \sqrt{\mu_{E}}f$. Furthermore, for an axis-symmetric domain (\ref{axis-symmetric}) with a degenerate potential (\ref{degenerate}), we assume (\ref{conserv_F_angular}). Then there exists 
		$C
		> 0$ such that, for all $N \in \mathbb{N}
		$, 
		\begin{equation}\label{coercive}
		\int^{N+1}_{N} \| \mathbf{P} f (t) \|_{2}^{2} \dd t
		\leq C
		\int^{N+1}_{N} \| (\mathbf{I} - \mathbf{P}) f (t) \|_{\nu}^{2} \dd t,
		\end{equation}
		where $\mathbf{P}f$ is defined in (\ref{E_Pf}). 
	\end{proposition} 
	We remark that we do not need any smallness of $\Phi$ in this linear theorem. A direct consequence of (\ref{coercive}) is an exponential decay-in-time of $\| f(t) \|_{L^{2}(\O\times\R^{3})}$. Then following the argument of \cite{Guo10}, we are able to show an exponential decay-in-time of $\|w f(t) \|_{L^{\infty}(\O\times\R^{3})}$.

	The proof of this proposition is based on \textit{the contradiction argument} of \cite{Guo10, K2}. As a consequence, we do not have any quantitative estimate of $C$ and the decay rate. By negating the coercivity of (\ref{coercive}) and some normalization of (\ref{normal Zm}), we obtain a weakly convergence sequence $Z^{m}$ whose component orthogonal to the null space of $L$ is vanishing as $m\rightarrow \infty$. The weak limit $Z$ satisfies the conservation laws as $(\ref{mass_linear})-(\ref{angular_linear})$ and the specular reflection BC (\textit{Step 7} in the proof of Proposition \ref{prop_coercivity}) and 
	\begin{equation} \label{specular_Z}
	b(t,x)\cdot {n}(x) = 0 \ \ \ \text{for almost all }    x\in\p\O.  
	\end{equation} 
	Moreover, $Z$ is remaining in the null space of $L$ and solving the transport equation (\ref{eqtn_Z}) without $e^{-\Phi}LZ$. As a consequence, the components $a,b,$ and $c$ of (\ref{E_Pf}) solve the systems (\cite{K2}) of 
	\begin{equation} \label{macro_eq}
	\begin{split}
	\p_i c &= 0 ,    \\
	\p_t {c} + \p_i  {b}_i  &= 0 ,  \\
	\p_i {b}_{j} + \p_j {b}_{i} &= 0 ,\quad i\neq j ,   \\
	\p_t {b}_{i} + \p_i {a} - 2c \p_i \Phi  &= 0,  \\
	\p_t {a} - \nabla_x\Phi \cdot {b} &= 0 .
	\end{split}
	\end{equation}
	Unlike the case of $\Phi\equiv 0$ in \cite{Guo10}, an explicit forms of $a,b,$ and $c$ cannot be obtained. We use the boundary condition (\ref{specular_Z}) and the conservation laws carefully and conclude that 
	\begin{equation}\label{Z=0}
	Z(t,x,v)=0 \ \  \ \text{almost all } t,x,v.
	\end{equation}
	
	On the other hand, due to the normalization (\ref{normal Zm}), the $L^{2}$-norm of $\mathbf{P}Z^{m}$ is always $1$ identically. Away from the boundary $\p\O$, the weak convergence is actually strong convergence due to the Velocity average lemma. For the shell-like subset of $\O$, using the Duhamel form along the trajectory, we are able to bound the integration over this shell-like subset by the interior integration (Lemma \ref{boundary_interior}). Therefore, $Z^{m}\rightarrow Z$ strongly and the $L^{2}$-norm of $Z$ equals $1$, which is a contradiction to (\ref{Z=0}).

	\section{Specular trajectories with a small time-dependent potential}
	
	In (\ref{Boltzmann_phi}), a time-dependent potential is given by $\Phi(x)+\phi(t,x)$. In this section, we write this potential as $\Phi(t,x)$ for convenience. The corresponding characteristic equation is
	\begin{equation}\label{E_Ham} 
	\frac{d}{ds}X(s;t,x,v)   \ =  \ V(s;t,x,v),\ \ \  \frac{d}{ds} V(s;t,x,v)  \ = \  - \nabla_{x} \Phi(s,X(s;t,x,v)). 
	\end{equation} 
	
	\begin{definition} 
		\noindent We recall the standard notations from \cite{GKTT1}. We define 
		\begin{equation}\begin{split}\label{backward_exit}
		\tb(t,x,v)  &:=  \sup \big\{  s \geq 0 :  X(\tau;t,x,v)  \in \Omega  \ \ \text{for all}  \ \tau \in ( t-s ,t)  \big\}, \\
		\xb(t,x,v)  &:=  X ( t- \tb(t,x,v);t,x,v),\ \
		\vb(t,x,v)  =  V ( t- \tb(t,x,v);t,x,v) ,  \end{split}
		\end{equation}
		and similarly,
		\begin{equation}\begin{split}\label{forward_exit}
		\tf(t,x,v)  &:=  \sup \big\{  s \geq 0 :  X(\tau;t,x,v)  \in \Omega  \ \ \text{for all}  \ \tau \in ( t,t+s)  \big\}, \\
		\xf(t,x,v)  &:=  X ( t+ \tf(t,x,v);t,x,v), \ \
		\vf(t,x,v)  =  V ( t+ \tf(t,x,v);t,x,v). 
		\end{split}\end{equation}
		Here, $\tb$ and $\tf$ are called the backward exit time and the forward exit time, respectively. We also define the specular cycle as in \cite{GKTT1}. We set $(t^{0}, x^{0}, v^{0}) = (t,x,v)$. Inductively, we define
		\begin{equation}\begin{split}\label{specular_cycles}
		t^{k} & = t^{k-1} - \tb (t^{k-1}, x^{k-1}, v^{k-1}),\\
		x^{k} &= X(t^{k}; t^{k-1}, x^{k-1}, v^{k-1}), \\
		v^{k}   &=  
		R_{x^{k}} V(t^{k}; t^{k-1}, x^{k-1}, v^{k-1}),
		\end{split}	
		\end{equation}
		where
		\begin{eqnarray*}
			R_{x^{k}} V(t^{k}; t^{k-1}, x^{k-1}, v^{k-1}) &=& V(t^{k}; t^{k-1}, x^{k-1}, v^{k-1}) 	\\
			&& - 2 \big( n(x^{k})\cdot V(t^{k}; t^{k-1}, x^{k-1}, v^{k-1}) \big) n(x^{k}).
		\end{eqnarray*}
		
		\noindent We define the specular characteristics as
		\begin{equation}\label{X_cl}
		\begin{split}
		X_{\mathbf{cl}} (s;t,x,v) &= \sum_{k} \mathbf{1}_{s \in ( t^{k+1}, t^{k}]}
		X(s;t^{k}, x^{k}, v^{k}),\\
		V_{\mathbf{cl}} (s;t,x,v) &= \sum_{k} \mathbf{1}_{s \in ( t^{k+1}, t^{k}]}
		V(s;t^{k}, x^{k}, v^{k}).
		\end{split}
		\end{equation}
		For the sake of simplicity we abuse the notation of (\ref{X_cl}) by dropping the subscription $\mathbf{cl}$ in this section.
	\end{definition} 
	
	From the assumptions of Theorem \ref{theorem_time} and Theorem \ref{theorem_decay}, for any ${p} \in \partial{\Omega}$, there exists sufficiently small $\delta_{1}>0, \delta_{2}>0$, and an one-to-one and onto $C^{3}$-function 
	\begin{equation}\label{eta}
	\begin{split}
	\eta_{{p}}:  \{\X_{{p}} \in \mathbb{R}^{3}:  \X_{{p},3}<0  \} \cap B(0; \delta_{1}) \ &\rightarrow  \ \Omega \cap B({p}; \delta_{2}),\\
	\X_{{p}}=(\X_{{p},1},\X_{{p},2},\X_{{p},3})	 \ &\mapsto \  (x_{1},x_{2},x_{3}) = \eta_{{p}}  (\X_{{p},1},\X_{{p},2},\X_{{p},3}),
	\end{split}
	\end{equation} 
	and $\eta_{ { {p}}}(\X_{ {{p}},1},\X_{ {{p}},2},\X_{ {{p}},3}) \in \p \Omega$ if and only if $\X_{ {{p}},3}=0$. We define the transformed velocity field at $\eta_{{p}} (\X_{{p}})$ as  
	\begin{equation}\label{v_p}
	\mathbf{v}_{ i}(\X_{{p}}) : = \frac{\p_{i} \eta_{{p}} (\X_{{p}})}{\sqrt{g_{{p},ii} (\X_{{p}})}} \cdot v.
	\end{equation}	
	
	{For any two dimensional smooth manifold $\mathcal{S}$, we can find a local orthogonal parametrization from $\mathbb{R}^{2}$ to $\p\mathcal{S}$. (See Corollary 2, page 183, \cite{dC}, for example.) Therefore, we assume     }  
	\begin{equation}\label{orthonormal_eta}
	\Big\{ \frac{\p_{1} \eta_{{p}}}{ \sqrt{g_{{p},11}}}, \frac{\p_{2} \eta_{{p}}}{ \sqrt{g_{{p},22}}},\frac{\p_{3} \eta_{{p}}}{ \sqrt{g_{{p},33}}}\Big\} \text{ is orthonormal at } \X_{{p},3} =0,
	\end{equation}
	where $g_{p,ij} := \langle \p_{i} \eta_{p}, \p_{j} \eta_{p} \rangle$. \\
	
	\noindent And, for second derivative $\p_i\p_j \eta_{p}$, we define Christoffel symbol $\Gamma_{p,ij}^{k}$ by
	\begin{equation} \label{Gamma}
	\p_{ij}\eta_p = \sum_{k} \Gamma_{p,ij}^{k} \p_{k}\eta_p.
	\end{equation}
}  
\noindent Moreover, by reparametrization, we may assume that $g_{{p}, 33} (\X_{{p},1},\X_{{p},2},\X_{{p},3})=1$ whenever it is defined. Without loss of generality, the outward normal at the boundary is, for $x =\eta_{{p}} (\X_{{p},1}, \X_{{p},2},0)  \in \p\Omega$,
	\begin{equation}\label{normal_eta}
	\begin{split}
	n(x) &= n(\eta_{{p}} (\X_{{p},1}, \X_{{p},2},0)) = 
	\p_{3} \eta_{{p}} (\X_{{p},1},\X_{{p},2},0)\\
	&= \frac{\p_{2} \eta_{{p}}}{ \sqrt{g_{{p},22}}}\times \frac{\p_{3} \eta_{{p}}}{ \sqrt{g_{{p},33}}}\Big|_{ (\X_{{p},1}, \X_{{p},2},0))}.
	%
	\end{split}
	\end{equation}

\noindent For each $k=0,1,2,3, \cdots,$ we assume that ${p}^{k} \in \p\Omega$ is chosen to be close to $x^{k}$ as in (\ref{eta}). Then, we define 
\begin{equation}\label{x^k}
\begin{split}
\X^{k}_{{p}^{k}}&:=(\X^{k}_{{p}^{k},1}, \X^{k}_{{p}^{k},2} , 0 )  \  \text{ 
	such that } \    x^{k}= \eta_{{p}^{k}} (\X^{k}_{{p}^{k}}), \\   \V^{k}_{p^{k},i} &:=\V_{i}^{k} (\X^{k}_{{p}^{k}}) = \frac{\p_{i} \eta_{{p}^{k}} (\X^{k}_{{p}^{k}})}{\sqrt{g_{{p}^{k},ii} (\X^{k}_{{p}^{k}})}} \cdot v^{k}.\end{split}
\end{equation}
Note that, due to (\ref{orthonormal_eta}), at the boundary, 
\begin{equation} \label{v_v}
v^{k}_{i} = \sum_{\ell=1}^{3} \V^{k}_{{p}^{k}, \ell}
\frac{\p_{\ell} \eta_{{p}^{k},i}  }{\sqrt{g_{{p}^{k}, \ell\ell} }}\Big|_{x^{k}}.
\end{equation}

%
%

\begin{lemma}\label{Jac_billiard} Assume that $\Omega$ and $\Phi$ are $C^{2}$. Consider $(t^{k+1}, \X^{k+1}_{p^{k+1}}, \V^{k+1}_{p^{k+1}})$ as a functin of $(t^{k+1}, \X^{k}_{p^{k}}, \V^{k}_{p^{k}})$. Then for $i,j=1,2,$   \\ 	
	\begin{equation} \label{Est-1}
	\begin{split} 
	\frac{\p (t^k - t^{k+1})}{\p \X^{k}_{{p}^k , j}} 
	&=  \frac{-1}{   \V^{k+1}_{ {{p}^{k+1}} ,3 }   }
	\frac{  \p_{{3}} \eta_{{p}^{k+1}}(x^{k+1})}{\sqrt{g_{{p}^{k+1},33}(x^{k+1})}} \cdot 
	\bigg[  \p_{j} \eta_{{p}^{k}} ( \X^{k}_{{p}^{k},1}, \X^{k}_{{p}^{k},2},0 ) - (t^{k}-t^{k+1})
	\frac{\p v^{k} }{\p \X^{k}_{{p}^{k},j}}
	\bigg] 
	\\ 
	&\quad + O_{\Omega}( \|\Phi\|_{C^{2}})  \frac{(t^k - t^{k+1})^2}{|\mathbf{v}_{p^{k+1},3}^{k+1}|}  \big(1 + (t^k - t^{k+1}) |\mathbf{v}^{k}_{p^{k}}| \big) e^{\| \Phi \|_{C^{2}}(t^{k} - t^{k+1})^{2} }
	,    \\
	\end{split}
	\end{equation}

	\begin{equation} \label{Est-2}
	\begin{split}
	\frac{\p \X_{{p}^{k+1},i}^{k+1}}{\p { \X_{{p}^{k},j}^{k}}}   
	&=   
	\frac{1}{\sqrt{g_{p^{k}, ii} (x^{k+1})}} \Big[\frac{\p_{i} \eta_{p^{k+1}} (x^{k+1})}{ \sqrt{g_{p^{k}, ii} (x^{k+1})}} + 
	\frac{\V^{k+1}_{p^{k+1}, i} }{   \V^{k+1}_{ {{p}^{k+1}} ,3 }   }
	\frac{  \p_{{3}} \eta_{{p}^{k+1}}(x^{k+1})}{\sqrt{g_{{p}^{k+1},33}(x^{k+1})}} 
	\Big] 	\\
	&\quad \cdot 
	\bigg[  \p_{j} \eta_{{p}^{k}} ( x^{k} ) - (t^{k}-t^{k+1})
	\frac{\p v^{k} }{\p \X^{k}_{{p}^{k},j}}
	\bigg] 
	\\
	&\quad + O_{\Omega}( \| \Phi\|_{C^{2}})   \Big\{ 1+ \frac{  | \V^{k+1}_{p^{k+1},i}|  }{|\mathbf{v}_{p^{k+1},3}^{k+1}|} \Big\}(t^k - t^{k+1})^{2}  	\\
	&\quad\quad\quad \times\big(1 + (t^k - t^{k+1}) |\mathbf{v}^{k}_{p^{k}}| \big) e^{\| \Phi \|_{C^{2}}(t^{k} - t^{k+1})^{2} },     \\
	\end{split}
	\end{equation}
	\begin{equation} \label{Est-3}
	\begin{split}
	&\frac{\p  {\V}^{k+1}_{{p}^{k+1}, i}}{\p \X^{k}_{{p}^{k},j}}  \\
	&=  
	\frac{\p v^{k}}{\p \X^{k}_{{p}^{k}, j}} \cdot \frac{\p_{i} \eta_{{p}^{k+1}} (x^{k+1})}{\sqrt{g_{{p}^{k+1}, ii} (x^{k+1})}}
	+   {v^{k}} \cdot \sum_{\ell=1}^{2} \frac{\p \X^{k+1}_{{p}^{k+1}, \ell}}{\p \X^{k}_{{p}^{k}, j}} \frac{\p}{\p \X^{k+1}_{{p}^{k+1}, \ell}}
	\Big( \frac{\p_{i} \eta_{{p}^{k+1}} (x^{k+1})}{\sqrt{g_{{p}^{k+1}, ii} (x^{k+1})}}
	\Big)		\\
	&  
	\quad + 
	O_{\Omega}(\|\Phi\|_{C^{2}})
	\Big\{
	\sum_{j} \Big| \frac{\p (t^{k} - t^{k+1})}{\p \X^{k}_{p^{k}, j}} \Big| \\
	&\quad + (t^{k } -t^{k+1}) (1+ |\V^{k}_{p^{k}}| (t^{k} - t^{k+1})) e^{\|\Phi \|_{C^{2}} (t^{k} - t^{k+1})^{2}} + \sum_{\ell } \Big| \frac{\p \X^{k+1}_{p^{k+1}, \ell }}{\p \X^{k }_{p^{k }, j }} \Big| 
	\Big\},
	\\
	\end{split}
	\end{equation}
	\begin{equation} \label{Est-4}
	\begin{split}
	&\frac{\p  {\V}^{k+1}_{{p}^{k+1}, 3}}{\p \X^{k}_{{p}^{k },j}} \\
	&=  
	- \frac{\p v^{k}}{\p \X^{k}_{{p}^{k}, j}} \cdot \frac{\p_{3} \eta_{{p}^{k+1}} (x^{k+1})}{\sqrt{g_{{p}^{k+1}, 33} (x^{k+1})}}
	-   {v^{k}} \cdot \sum_{\ell=1}^{2} \frac{\p \X^{k+1}_{{p}^{k+1}, \ell}}{\p \X^{k}_{{p}^{k}, j}} \frac{\p}{\p \X^{k+1}_{{p}^{k+1}, \ell}}
	\Big( \frac{\p_{3} \eta_{{p}^{k+1}} (x^{k+1})}{\sqrt{g_{{p}^{k+1}, 33} (x^{k+1})}}
	\Big)    \\
	&\quad +
	O_{\Omega}(\|\Phi\|_{C^{2}})
	\Big\{
	\sum_{j} \Big| \frac{\p (t^{k} - t^{k+1})}{\p \X^{k}_{p^{k}, j}} \Big| 	\\
	&\quad + (t^{k } -t^{k+1}) (1+ |\V^{k}_{p^{k}}| (t^{k} - t^{k+1})) e^{\|\Phi \|_{C^{2}} (t^{k} - t^{k+1})^{2}}
	+ \sum_{\ell } \Big| \frac{\p \X^{k+1}_{p^{k+1}, \ell }}{\p \X^{k }_{p^{k }, j }} \Big| 
	\Big\},      
	\end{split}
	\end{equation}
	where
	\begin{equation}\label{V_X}
	\frac{\p v^{k}_{i}}{\p \X^{k}_{{p}^{k},j}} = \sum_{\ell=1}^{3} \V^{k}_{{p}^{k},\ell} \sum_{r (\neq \ell)}
	\frac{\sqrt{g_{{p}^{k}, rr} (x^{k}) }}{ \sqrt{g_{{p}^{k}, \ell\ell} (x^{k}) }
	}
	\Gamma_{{p}^{k}, \ell j}^{r} (x^{k}) \frac{\p_{r} \eta_{{p}^{k},i} (x^{k})}{\sqrt{g_{{p}^{k}, rr} (x^{k})}}.
	\end{equation}
	
	\noindent For $i=1,2,$ and $   j=1,2,3, $
	\begin{equation} \label{Est-5}
	\begin{split} 
	\frac{\p (t^k - t^{k+1}) }{\p \V^{k }_{{p}^{k} , j}} &    =   
	\frac{   (t^k - t^{k+1}) }{\V^{k+1}_{{p}^{k+1}, 3}  } 
	\bigg[
	\frac{\p_{j} \eta_{{p}^{k}}  (x ^{k}  ) }{\sqrt{g_{{p}^{k},jj} (x ^{k}  )}}
	+  O_{\Omega}(
	\|\Phi \|_{C^{2}} ) (t^k - t^{k+1})^{2}  e^{\| \nabla_{x}^{2 } \Phi \|_{\infty} |t^{k} - t^{k+1}|^{2}}   \bigg] 	\\
	&\quad 
	\cdot 	\frac{\p_{3} \eta_{{p}^{k+1}} (x^{k+1})}{ \sqrt{g_{{p}^{k+1},33} (x^{k+1})}}, 
	\end{split}
	\end{equation}
	\begin{equation} \label{Est-6}
	\begin{split} 
	\\
	\frac{\p \X^{k+1}_{{p}^{k+1}, i}}{\p \V^{k}_{{p}^{k}, j}} 
	& =   - (t^k - t^{k+1}) \frac{\p_{j} \eta_{{p}^{k}}  (x^{k}) }{\sqrt{g_{{p}^{k}, jj} (x^{k})} } \cdot 
	\frac{1}{\sqrt{g_{{p}^{k+1}, ii} (x^{k+1})  }}
	\Big[
	\frac{\p_{i} \eta_{{p}^{k+1}} (x^{k+1})  }{\sqrt{g_{{p}^{k+1}, ii}  (x^{k+1})  }} 	\\
	&\quad 
	+\frac{\V^{k+1}_{{p}^{k+1}, i}}{\V^{k+1}_{{p}^{k+1},3}} \frac{\p_{3} \eta_{{p}^{k+1}} (x^{k+1}) 
	}{\sqrt{g_{{p}^{k+1}, 33}} (x^{k+1})  }
	\Big]    \\
	&\quad + O_{\Omega}( \|\Phi \|_{C^{2}} )
	\Big(1 + \frac{|\V^{k+1}_{p^{k+1},i}|}{|\V^{k+1}_{p^{k+1},3}|} \Big)  (t^k - t^{k+1})^3  
	e^{\|\Phi \|_{C^{2}} (t^k - t^{k+1})^2}
	,
	\end{split}
	\end{equation}
	\begin{equation} \label{Est-7}
	\begin{split} 
	\\
	%
	%
	%
	%
	\frac{\p {\V}^{k+1}_{{p}^{k+1},{i} }}{\p {\V}^{k}_{{p}^{k},j}}  
	& =  \sum_{\ell=1}^{2} \frac{\p \X^{k+1}_{{p}^{k+1}, \ell}}{\p \V^{k}_{{p}^{k}, j}}
	\p_{\ell} \Big( \frac{\p_{i} \eta_{{p}^{k+1}} }{\sqrt{g_{{p}^{k+1},ii}}}\Big)\Big|_{x^{k+1}} \cdot v^{k} + \frac{\p_{i} \eta_{{p}^{k+1} } (x^{k+1}) }{\sqrt{g_{{p}^{k+1}, ii} (x^{k+1})}} \cdot \frac{\p_{j} \eta_{{p}^{k}} (x^{k })}{\sqrt{g_{{p}^{k}, jj} (x^{k })}} \\
	&\quad +  O_{\Omega}( \|\Phi \|_{C^{2}} )(t^{k} - t^{k+1})^{2} \Big(1 + \frac{|\V^{k+1}_{p^{k+1}, i}|}{|\V^{k+1}_{p^{k+1},3}|} \Big) 	\\
	&\quad\quad \times 
	\big( 1+  O_{\Omega}(\|\Phi \|_{C^{2}}) (t^{k} - t^{k+1})^{2} \big) e^{ \|\Phi \|_{C^{2}} (t^{k} - t^{k+1})^{2}}\\
	&\quad + O_{\Omega}( \|\Phi \|_{C^{2}} ) \frac{|t^{k}- t^{k+1}|}{|\V^{k+1}_{p^{k+1}, 3}|} 	\\
	&\quad\quad \times \big(1+ 
	O_{\Omega}(
	\|\Phi \|_{C^{2}} )(t^{k}-t^{k+1})^{2 }
	e^{ \|\Phi \|_{C^{2}} (t^{k} - t^{k+1})^{2}}
	\big) ,
	\end{split}
	\end{equation}

	\begin{equation} \label{Est-8}
	\begin{split}
	\frac{\p {\V}^{k+1}_{{p}^{k+1},{3} }}{\p {\V}^{k}_{{p}^{k},j}}  
	& =  - \sum_{\ell=1}^{2} \frac{\p \X^{k+1}_{{p}^{k+1}, \ell}}{\p \V^{k}_{{p}^{k}, j}}
	\p_{\ell} \Big( \frac{\p_{3} \eta_{{p}^{k+1}} }{\sqrt{g_{{p}^{k+1},33}}}\Big)\Big|_{x^{k+1}} \cdot v^{k} - \frac{\p_{3} \eta_{{p}^{k+1} } (x^{k+1}) }{\sqrt{g_{{p}^{k+1}, 33} (x^{k+1})}} \cdot \frac{\p_{j} \eta_{{p}^{k}} (x^{k })}{\sqrt{g_{{p}^{k}, jj} (x^{k })}} \\
	&\quad +  O_{\Omega}( \|\Phi \|_{C^{2}} )(t^{k} - t^{k+1})^{2}   
	\big( 1+ O_{\Omega}( \|\Phi \|_{C^{2}} ) (t^{k} - t^{k+1})^{2} \big) e^{ \|\Phi \|_{C^{2}} (t^{k} - t^{k+1})^{2}}\\
	&\quad +  O_{\Omega}( \|\Phi \|_{C^{2}} ) \frac{|t^{k}- t^{k+1}|}{|\V^{k+1}_{p^{k+1}, 3}|} \big(1+ 
	O_{\Omega}(
	\|\Phi \|_{C^{2}} )(t^{k}-t^{k+1})^{2 }
	e^{ \|\Phi \|_{C^{2}} (t^{k} - t^{k+1})^{2}}
	\big) 
	.\end{split}\end{equation}
\end{lemma}
\noindent\textit{Remark. } Note that we do not need the convexity (\ref{convexity_eta}) or the smallness of the size of $\Phi$ in Lemma \ref{Jac_billiard} .

\begin{proof} \textit{Proof of (\ref{Est-1})}. By the definitions (\ref{eta}), (\ref{specular_cycles}), and our setting (\ref{x^k}) and (\ref{E_Ham}),
	\begin{equation}\begin{split} \label{position identity}
	&\eta_{{p}^{k+1}}(\X^{k+1}_{{p}^{k+1},1},\X^{k+1}_{{p}^{k+1},2},0)\\
	&=	\eta_{{p}^{k}}(\X^{k}_{{p}^{k},1},\X^{k}_{{p}^{k },2},0) + \int_{t^{k}}^{ t^{k+1}} v^{k } - \int_{t^{k}}^{ t^{k+1} }\int_{t^{k}}^{   s} \nabla\Phi(  \tau, X( \tau; t^{k }, x^{k }, v^{k })) \dd\tau \dd s   .
	\end{split}\end{equation}
	We take $\frac{\p}{\p \X^{k}_{{p}^{k},j}}$ to above equality for $j=1,2$ to get
	\begin{equation} \label{diff pos iden}
	\begin{split}
		&  \sum_{l=1,2} \frac{\p \X^{k+1}_{{p}^{k+1},l}}{\p \X^{k}_{{p}^{k},j}} \frac{\p\eta_{{p}^{k+1}}}{\p \X^{k+1}_{{p}^{k+1},l}}\Big\vert_{x^{k+1}}   
		=  - (t^{k}-t^{k+1}) \frac{\p v^{k}}{\p \X_{p^{k},j}^{k}}  \\
		&\quad\quad\quad\quad -   \frac{\p (t^{k }-t^{k+1}) }{\p {\X^{k}_{{p}^{k},j}}} \Big\{ v^{k }  - \int^{t^{k+1}}_{t^{k}} \nabla \Phi (s, X(s;t^{k}, x^{k}, v^{k})) \dd s\Big\} 	
		\\
		&\quad\quad\quad\quad + \Big\{
		\p_{j} \eta_{{p}^{k}}( \X^{k}_{{p}^{k},1}, \X^{k}_{{p}^{k},2}, 0 ) - 
		\int^{  t^{k+1} }_{t^{k}} \dd s \int^{ s}_{t^{k}} \dd \tau \big(\frac{\p X( \tau) }{\p \X^{k}_{{p}^{k},j}} \cdot \nabla_{x}  \big) \big(  \nabla_{x} \Phi (  \tau) \big)
		\Big\} ,
	\end{split}
	\end{equation}
	and then take an inner product with $\frac{\p_{{3}} \eta_{{p}^{k+1}}}{\sqrt{g_{{p}^{k+1},33}}}\Big|_{x^{k+1}}  $ to have 
	\begin{equation} \label{diff pos iden dot 3}
	\begin{split}
		&  \sum_{l=1,2} \frac{\p \X^{k+1}_{{p}^{k+1},l}}{\p \X^{k}_{{p}^{k},j}} \frac{\p\eta_{{p}^{k+1}}}{\p \X^{k+1}_{{p}^{k+1},l}}\Big\vert_{x^{k+1}}  \cdot
		\frac{\p_{{3}} \eta_{{p}^{k+1}} }{\sqrt{g_{{p}^{k+1},33}}}   \Big\vert_{x^{k+1}} 
		\\
		&=  - (t^{k}-t^{k+1}) \frac{\p v^{k}}{\p \X_{p^{k},j}^{k}} \cdot \frac{\p_{3} \eta_{{p}^{k+1}}}{\sqrt{g_{{p}^{k+1},33}} }  \Big|_{x^{k+1}}  \\
		& -   \frac{\p (t^{k }-t^{k+1}) }{\p {\X^{k}_{{p}^{k},j}}} \Big\{ v^{k }  - \int^{t^{k+1}}_{t^{k}} \nabla \Phi (s, X(s;t^{k}, x^{k}, v^{k})) \dd s\Big\} 	
		\cdot \frac{ \p_{3} \eta_{{p}^{k+1}}}{\sqrt{g_{{p}^{k+1},33}}}\Big|_{ x^{k+1}} \\
		& + \Big\{
		\p_{j} \eta_{{p}^{k}}( \X^{k}_{{p}^{k},1}, \X^{k}_{{p}^{k},2}, 0 ) - 
		\int^{  t^{k+1} }_{t^{k}} \dd s \int^{ s}_{t^{k}} \dd \tau \big(\frac{\p X( \tau) }{\p \X^{k}_{{p}^{k},j}} \cdot \nabla_{x}  \big) \big(  \nabla_{x} \Phi (  \tau) \big)
		\Big\} \cdot \frac{ \p_{3} \eta_{{p}^{k+1}}}{
			\sqrt{g_{p^{k+1},33}}
		}\Big|_{x^{k+1}} ,
	\end{split}
	\end{equation}
	where we abbreviated $X(s)=X(s;t^{k},x^{k},v^{k})$, $V(s)=V(s;t^{k},x^{k},v^{k})$, and \\ $\Phi(s) = \Phi(s,X(s;t^{k},x^{k},v^{k}))$. Due to (\ref{orthonormal_eta}) the LHS equals zero. Now we consider the RHS. From (\ref{v_v}), we prove (\ref{V_X}). 
	 We also note that 
	\begin{equation}\label{down_V}
	\lim_{s\downarrow t^{k+1}} V(s;t^{k}, x^{k}, v^{k}) = v^{k} - \int^{t^{k+1}}_{t^{k}} \nabla \Phi (s, X(s;t^{k}, x^{k}, v^{k})) \dd s.
	\end{equation}
	Therefore, from (\ref{specular_cycles}) and (\ref{x^k}), 
	$$ \big\{ v^{k }  - \int^{t^{k+1}}_{t^{k}} \nabla \Phi (s, X(s;t^{k}, x^{k}, v^{k})) \dd s\Big\}\cdot \frac{ \p_{3} \eta_{{p}^{k+1}}}{\sqrt{g_{{p}^{k+1},33}}}\big|_{ x^{k+1}} = -\V^{k+1}_{p^{k+1},3}. $$  
	From (\ref{diff pos iden dot 3}),
	\begin{equation} \label{Estimate 1}
	\begin{split}
	\\
	\frac{\p (t^{k}-t^{k+1})}{\p {\X^{k}_{{p}^{k},j}}} &=
	- 
	\frac{1}{ \V^{k+1}_{ {{p}^{k+1}} ,3 }  }
	\frac{  \p_{3} \eta_{{p}^{k+1}}(x^{k+1})}{\sqrt{g_{{p}^{k+1},33}(x^{k+1})}} \cdot 
	\bigg[  \p_{j} \eta_{{p}^{k}} ( \X_{{p}^{k},1}^{k}, \X_{{p}^{k},2}^{k}, 0 )   \\
	&-
	(t^{k}-t^{k+1})  \sum_{\ell=1}^{3} \V_{{p}^{k}, \ell}^{k}
	\sum_{r(\neq \ell)} \frac{\sqrt{ g_{{p}^{k},rr} (x^{k})}}{ \sqrt{ g_{{p}^{k}, \ell\ell} (x^{k})}}  
	\Gamma_{{p}^{k}, \ell j}^{r} (x^{k}) \frac{\p_{r} \eta_{{p}^{k}}(x^{k}) }{\sqrt{g_{{p}^{k} ,rr} (x^{k})}} 
	\bigg] \\
	& +  { \frac{1}{ \V^{k+1}_{ p^{k+1}, 3 } }\frac{  \p_{3} \eta_{{p}^{k+1}}(x^{k+1})}{\sqrt{g_{{p}^{k+1},33}(x^{k+1})}}
		\cdot \int^{ t^{k+1} }_{t^{k}} \dd s \int^{ s}_{t^{k}} \dd \tau \big(\frac{\p X(  \tau) }{\p \X_{{p}^{k},j}^{k}} \cdot \nabla_{x}  \big) \big(  \nabla_{x} \Phi ( \tau) \big) }  .
	\end{split}	
	\end{equation}
	Now we consider integrand $\frac{\p X (  \tau; t^{k}, x^{k}, v^{k})}{\p \X^{k}_{{p}^{k },j}}$. From (\ref{position identity}) for $t^{k+1} < \tau \leq t^{k}$ ,
	\begin{equation}\label{down_X}
	X( \tau; t^{k}, x^{k}, v^{k}) = \eta_{{p}^{k}}( \X^{k}_{{p}^{k},1}, \X^{k}_{{p}^{k},2}, 0) +   v^k(\tau- t^{k}) - \int_{t^{k}}^{\tau} \dd s\int_{t^{k}}^{ s} \dd s^{\prime} \nabla\Phi( s',X( s';t^{k},x^{k},v^{k}))   .
	\end{equation}
	By the direct computations, for $j=1,2,$
	\begin{equation*}
	\begin{split}	
		\sup_{\tau \leq s^{\prime} \leq t^{k}} \bigg| \frac{\p X(  s^{\prime}; t^{k}, x^{k}, v^{k}) }{\p \X_{{p}^{k},j}^{k}} \bigg| 
		& \leq   \bigg| \p_{j} \eta_{{p}^{k}}( \X^{k}_{{p}^{k},1}, \X^{k}_{{p}^{k},2},0)\bigg| + |\tau - t^{k}| \bigg|  \frac{\p v^{k}}{\p \X_{{p}^{k},j}^{k}} \bigg| 	\\
		&\quad + \int_{t^{k}}^{\tau} |s - t^{k}| \| \Phi\|_{C^{2}} \sup_{\tau \leq s^{\prime} \leq t^{k}} \bigg| \frac{\p X(  s^{\prime}; t^{k}, x^{k}, v^{k}) }{\p \X_{{p}^{k},j}^{k}} \bigg| \dd s.
	\end{split}
	\end{equation*}
	By Gronwall's inequality and (\ref{V_X}),
	\begin{equation} \label{dX/dx}
	\sup_\tau \bigg| \frac{\p X( \tau; t^{k}, \mathbf{x}^{k}_{p^{k}}, \mathbf{v}^{k}_{p^{k}}) }{\p \X_{{p}^{k},j}^{k}} \bigg| 
	\leq  O_{\Omega}(1) (1 + |t^{k} - \tau|| \mathbf{v}^{k}_{p^{k}}|) e ^{   { \| \Phi\|_{C^{2}}|t^{k} - \tau|^2 } /2   } .
	\end{equation} 
	Using (\ref{Estimate 1}) and (\ref{dX/dx}), we complete the proof of (\ref{Est-1}). 
	\vspace{4pt}
	
	\noindent \textit{Proof of (\ref{Est-2})}. We take inner product with $\frac{\p_{{i}} \eta_{{p}^{k+1}}}{g_{{p}^{k+1},ii}}\Big|_{x^{k+1}}$ to (\ref{diff pos iden}) to have 
	\begin{equation} \label{diff pos iden ii}
	\begin{split}
	&  \sum_{l=1,2} \frac{\p \X^{k+1}_{{p}^{k+1},l}}{\p \X^{k}_{{p}^{k},j}} \frac{\p\eta_{{p}^{k+1}}}{\p \X^{k+1}_{{p}^{k+1},l}}\Big\vert_{x^{k+1}} \cdot \frac{\p_{{i}} \eta_{{p}^{k+1}}}{g_{{p}^{k+1},ii}}\Big|_{x^{k+1}} = \frac{\p \X^{k+1}_{{p}^{k+1},i}}{\p \X^{k}_{{p}^{k},j}} 
	\\
	&=  - (t^{k}-t^{k+1}) \frac{\p v^{k}}{\p \X_{p^{k},j}^{k}} \cdot \frac{\p_{{i}} \eta_{{p}^{k+1}}}{g_{{p}^{k+1},ii}}\Big|_{x^{k+1}}  \\
	&\quad  -   \frac{\p (t^{k }-t^{k+1}) }{\p {\X^{k}_{{p}^{k},j}}} \Big\{ v^{k }  - \int^{t^{k+1}}_{t^{k}} \nabla \Phi (s, X(s;t^{k}, x^{k}, v^{k})) \dd s\Big\}  \cdot \frac{\p_{{i}} \eta_{{p}^{k+1}}}{g_{{p}^{k+1},ii}}\Big|_{x^{k+1}}	
	\\
	&\quad + \Big\{
	\p_{j} \eta_{{p}^{k}}( \X^{k}_{{p}^{k},1}, \X^{k}_{{p}^{k},2}, 0 ) - 
	\int^{  t^{k+1} }_{t^{k}} \dd s \int^{ s}_{t^{k}} \dd \tau \big(\frac{\p X( \tau) }{\p \X^{k}_{{p}^{k},j}} \cdot \nabla_{x}  \big) \big(  \nabla_{x} \Phi (  \tau) \big)
	\Big\} \cdot \frac{\p_{{i}} \eta_{{p}^{k+1}}}{g_{{p}^{k+1},ii}}\Big|_{x^{k+1}} .
	\end{split}
	\end{equation}
	Since
	$$ \big\{ v^{k }  - \int^{t^{k+1}}_{t^{k}} \nabla \Phi (s, X(s;t^{k}, x^{k}, v^{k})) \dd s\Big\}\cdot \frac{ \p_{i} \eta_{{p}^{k+1}}}{g_{{p}^{k+1},ii}}\big|_{ x^{k+1}} = - {\V^{k+1}_{p^{k+1},i} \over \sqrt{g_{{p}^{k+1},ii}} }, $$  
	from (\ref{orthonormal_eta}) and (\ref{Est-1}),
	\begin{equation} \label{Estimate 2}
	\begin{split}
		\frac{\p \X_{{p}^{k+1},i}^{k+1}}{\p { \X_{{p}^{k},j}^{k}}}   
		&=    \frac{1}{   \V^{k+1}_{ {{p}^{k+1}} ,3 }   }
		\frac{  \p_{{3}} \eta_{{p}^{k+1}}(x^{k+1})}{\sqrt{g_{{p}^{k+1},33}(x^{k+1})}} \cdot 
		\bigg[  \p_{j} \eta_{{p}^{k}} ( x^{k} ) - (t^{k}-t^{k+1})
		\frac{\p v^{k} }{\p \X^{k}_{{p}^{k},j}}
		\bigg] 
		\frac{ \V^{k+1}_{p^{k+1}, i} }{\sqrt{g_{p^{k+1}, ii}}} \Big|_{x^{k+1}} \\
		&				+ \frac{\p_{i} \eta_{p^{k+1}} }{ {g_{p^{k+1}, ii} }} \Big|_{x^{k+1}}  \cdot  \Big[\p_{j} \eta_{p^{k}} (x^{k})
		- (t^{k}-t^{k+1})  \frac{\p v^{k}}{\p \X^{k}_{{p}^{k}, j}}  \Big]\\
		&+ O_{\Omega}(     \|\Phi\|_{C^{2}} ) \frac{  | \V^{k+1}_{p^{k+1},i}|  }{|\mathbf{v}_{p^{k+1},3}^{k+1}|}(t^k - t^{k+1})^2  \big(1 + (t^k - t^{k+1}) |\mathbf{v}^{k}_{p^{k}}| \big) e^{\| \Phi \|_{C^{2}}(t^{k} - t^{k+1})^{2} /2 }\\
		&+ O_{\Omega}(     \| \Phi\|_{C^{2}} ) (t^{k}-t^{k+1})^{2} \big(1 + (t^k - t^{k+1}) |\mathbf{v}^{k}_{p^{k}}| \big) e^{\| \Phi \|_{C^{2}}(t^{k} - t^{k+1})^{2} /2 }.
	\end{split}
	\end{equation}
	This ends the proof of (\ref{Est-2}).

	\vspace{4pt}
	
	\noindent \textit{Proof of (\ref{Est-3}) and (\ref{Est-4})}. From (\ref{specular_cycles}) and (\ref{x^k}),			
	\begin{equation}  
	\begin{split}\label{v_123}
	{\V}^{k+1}_{{p}^{k+1},i}   
	&=  \frac{ \p_{i} \eta_{{p}^{k+1}} }{\sqrt{g_{{p}^{k+1},ii}}}  \Big|_{x^{k+1}} \cdot  \lim_{s \downarrow  t^{k+1}}V(s; t^{k}, x^{k},v^{k}), \ \ \ \text{for} \ \  i=1,2,\\
	{\V}^{k+1}_{{p}^{k+1},3}  & =  
	- \frac{ \p_{3} \eta_{{p}^{k+1}} }{\sqrt{g_{{p}^{k+1},33}}}  \Big|_{x^{k+1}} \cdot  \lim_{s \downarrow  t^{k+1}}V(s; t^{k}, x^{k},v^{k}).
	\end{split}
	\end{equation}
	For $i,j=1,2$, from (\ref{v_123}), 
	\begin{equation*}\begin{split}
	\frac{\p \V^{k+1}_{{p}^{k+1}, i}}{\p \X^{k}_{{p}^{k},j}}  =&
	\frac{\p_{i} \eta_{{p}^{k+1} } (x^{k+1}) }{\sqrt{g_{{p}^{k+1}, ii} (x^{k+1})}} \cdot \Big[\frac{\p v^{k}}{\p \X^{k}_{{p}^{k}, j}}
	+ \nabla_{x} \Phi (t^{k+1}; t^{k},x^{k}, v^{k}) \frac{\p (t^{k} - t^{k+1})}{\p \X^{k}_{p^{k}, j }}  \\
	& \quad  - \int^{t^{k+1}}_{t^{k}}  \big(\p_{\X^{k}_{p^{k}, j }} X(s ) \cdot \nabla_{x} \big) \nabla_{x} \Phi (s, X(s;t^{k}, x^{k}, v^{k})) \dd s 
	\Big]
	\\
	&\quad				+
	\sum_{\ell=1}^{2} 
	\frac{\p \X^{k+1}_{{p}^{k+1}, \ell}}{\p \X^{k}_{{p}^{k}, j}} \frac{\p}{\p \X^{k+1}_{{p}^{k+1}, \ell}}\Big(  \frac{\p_{i} \eta_{{p}^{k+1}}  }{\sqrt{g_{{p}^{k+1}, ii}    }  }  \Big)\Big|_{x^{k+1}} \cdot   \lim_{s \downarrow  t^{k+1}}V(s; t^{k}, x^{k},v^{k}).  \end{split}\end{equation*}	
	And for $j=1,2,$
	\begin{equation*} \begin{split}
	\frac{\p \V^{k+1}_{{p}^{k+1}, 3}}{\p \X^{k}_{{p}^{k},j}}  =&
	-\frac{\p_{3} \eta_{{p}^{k+1} } (x^{k+1}) }{\sqrt{g_{{p}^{k+1}, 33} (x^{k+1})}} \cdot 
	\Big[\frac{\p v^{k}}{\p \X^{k}_{{p}^{k}, j}}
	+ \nabla_{x} \Phi (t^{k+1}; t^{k},x^{k}, v^{k}) \frac{\p (t^{k} - t^{k+1})}{\p \X^{k}_{p^{k}, j }} \\
	& \ \ \ \ \ \ \ \ \ \ \ \ \ \ \ \ \ \ \ \ \ \  \ \ \ \ \ \ \
	- \int^{t^{k+1}}_{t^{k}}  \big(\p_{\X^{k}_{p^{k}, j }} X(s ) \cdot \nabla_{x} \big) \nabla_{x} \Phi (s, X(s;t^{k}, x^{k}, v^{k})) \dd s 
	\Big]
	\\
	&
	-
	\sum_{\ell=1}^{2} 
	\frac{\p \X^{k+1}_{{p}^{k+1}, \ell}}{\p \X^{k}_{{p}^{k}, j}} \frac{\p}{\p \X^{k+1}_{{p}^{k+1}, \ell}}\Big(  \frac{\p_{3} \eta_{{p}^{k+1}}  }{\sqrt{g_{{p}^{k+1}, 33}    }  }  \Big)\Big|_{x^{k+1}} \cdot  \lim_{s \downarrow  t^{k+1}}V(s; t^{k}, x^{k},v^{k}).\end{split}
	\end{equation*}
	From (\ref{down_V})	and (\ref{dX/dx}), we prove (\ref{Est-3}) and (\ref{Est-4}).				

	\vspace{4pt}
	
	Now we consider (\ref{Est-5})-(\ref{Est-8}) for $v-$derivatives.
	
	\vspace{4pt}
	
	\noindent\textit{Proof of (\ref{Est-5})}. We take $\frac{\p}{\p \V^{k}_{{p}^{k},j}}$ to (\ref{position identity}) for $j=1,2,3$ to get
	\begin{equation} \label{v diff pos iden}
	\begin{split}
	&  \sum_{l=1,2} \frac{\p \X^{k+1}_{{p}^{k+1},l}}{\p \V^{k}_{{p}^{k},j}} \frac{\p\eta_{{p}^{k+1}}}{\p \X^{k+1}_{{p}^{k+1},l}}\Big\vert_{x^{k+1}}   
	=  - (t^{k}-t^{k+1}) \frac{\p v^{k}}{\p \V_{p^{k},j}^{k}}  \\
	&\quad\quad\quad\quad -   \frac{\p (t^{k }-t^{k+1}) }{\p {\V^{k}_{{p}^{k},j}}} \Big\{ v^{k }  - \int^{t^{k+1}}_{t^{k}} \nabla \Phi (s, X(s;t^{k}, x^{k}, v^{k})) \dd s\Big\} 	
	\\
	&\quad\quad\quad\quad + \Big\{
	 - 
	\int^{  t^{k+1} }_{t^{k}} \dd s \int^{ s}_{t^{k}} \dd \tau \big(\frac{\p X( \tau) }{\p \V^{k}_{{p}^{k},j}} \cdot \nabla_{x}  \big) \big(  \nabla_{x} \Phi (  \tau) \big)
	\Big\} ,
	\end{split}
	\end{equation}
	and then take an inner product with $\frac{\p_{{3}} \eta_{{p}^{k+1}}}{\sqrt{g_{{p}^{k+1},33}}}\Big|_{x^{k+1}}  $ to have

	\begin{equation}  \label{v diff pos iden dot 3}
	\begin{split}
	& \sum_{l=1,2} \frac{\p \X^{k+1}_{{p}^{k+1},l}}{\p \V^{k}_{{p}^{k},j}} \frac{\p\eta_{{p}^{k+1}}}{\p \X^{k+1}_{{p}^{k+1},l}}\Big\vert_{x^{k+1}}  \cdot   \frac{\p_{3} \eta _{{p}^{k+1}} }{\sqrt{g_{{p}^{k+1}, 33} }} \Big|_{x^{k+1}}  \\
	&=	\Big\{ - ( t^{k} - t^{k+1})  \frac{\p v^{k }}{\p \V^{k}_{p^{k}, j}} - \frac{\p( t^{k} - t^{k+1})}{\p \V^{k}_{p^{k}, j}} \lim_{s \downarrow t^{k+1}}V(s;t^{k}, x^{k}, v^{k}) \Big\} \cdot  \frac{\p_{3} \eta _{{p}^{k+1}} }{\sqrt{g_{{p}^{k+1}, 33} }} \Big|_{x^{k+1}} \\
	&\quad  + O_{\O}(\|\Phi \|_{C^{2}}) |t^{k} - t^{k+1}|^{2}  \sup_{s} \big| \frac{\p X(s)}{\p \V^{k}_{p^{k},j}} \big| .
	\end{split}
	\end{equation}
	Due to (\ref{orthonormal_eta}), the LHS equals zero. Now we consider the RHS.
	From (\ref{v_v}),
	
	
	\begin{equation} \label{dv^0/dv^0}
	\frac{\p v^{k}}{\p {\V}^{k}_{{p}^{k},j}} =  \frac{\p_{j} \eta_{{p}^{k}}  (\X_{{p}^{k}, 1}^{k}, \X_{{p}^{k}, 2}^{k}, 0 ) }{\sqrt{g_{{p}^{k},jj} (\X_{{p}^{k}, 1}^{k}, \X_{{p}^{k}, 2}^{k}, 0 )}} .
	\end{equation} 
	
	Now we consider $\sup_{s} \big| \frac{X(s;t^{k}, x^{k}, v^{k})}{\p \V^{k}_{p^{k}, j}} \big|$. From (\ref{down_X}), for $j=1,2,3$,
	\[\big| \frac{\p X(s) }{\p \V_{{p}^{k},j}^{k}} \big| 
	\leq   |t^{k} - s | \big|  \frac{\p v^{k}}{\p \V_{{p}^{k},j}^{k}} \big| + \| \nabla_{x}^{2}\Phi\|_{\infty}\int^{t^{k}}_{ s} |t^{k} - \tau |  \big| \frac{\p X(\tau) }{\p \V_{{p}^{k},j}^{k}} \big| \dd \tau.
	\]
	By Gronwall's inequality and (\ref{dv^0/dv^0}), for $t^{k+1}\leq s\leq t^{k}$,
	\begin{equation} \label{dX/dv}
	\begin{split}
	\Big| \frac{\p X(s; t^{k}, x^{k}, v^{k}) }{\p \V_{{p}^{k},j}^{k}} \Big| \leq  | t^{k}-s| \Big|  \frac{\p v^{k}}{\p \V_{{p}^{k},j}^{k}} \Big|  
	e^{\| \Phi \|_{C^{2}}  | t^{k}-s|^{2}  /2}  
	\lesssim_{\Omega}   | t^{k}-s| 			e^{\| \Phi \|_{C^{2}}  | t^{k}-s|^{2} /2 }  .  
	\end{split}
	\end{equation}
	Using (\ref{v_123}), (\ref{v diff pos iden dot 3}), (\ref{dv^0/dv^0}), and (\ref{dX/dv}), we prove (\ref{Est-5}).
	
	\vspace{4pt}
	
	\noindent\textit{Proof of (\ref{Est-6})}. For $i=1,2$ and $j=1,2,3$ , we take inner product with $\frac{\p_{{i}} \eta_{{p}^{k+1}}}{g_{{p}^{k+1},ii}}\Big|_{x^{k+1}}$ to (\ref{v diff pos iden}) to have 
	\begin{eqnarray*}
		\frac{\p \X^{k+1}_{p^{k+1}, i}}{\p \V^{k}_{p^{k}, j}} 
		&=&
		\Big\{- \frac{\p (t^{k} - t^{k+1})}{\p \V^{k}_{p^{k}, j}} \lim_{s\downarrow t^{k+1}} V(s; t^{k}, x^{k}, v^{k})- (t^{k} - t^{k+1}) \frac{\p v^{k}}{\p \V^{k}_{p^{k}, j}}
		\Big\}
		\cdot \frac{\p_{{i}} \eta_{{p}^{k+1}}}{g_{{p}^{k+1},ii}}\Big|_{x^{k+1}} 	\\
		&& + 
		O_{\Omega}( \| \Phi \|_{C^{2}} )|t^{k} - t^{k+1}|^{2} \sup_{s } \big| \frac{\p X(s)}{\p \V^{k}_{p^{k},j}} \big|.
	\end{eqnarray*}
	From (\ref{dv^0/dv^0}), (\ref{dX/dv}), and (\ref{Est-5}), we prove (\ref{Est-6}).

	\vspace{4pt}
	
	\noindent\textit{Proof of (\ref{Est-7}) and (\ref{Est-8}) }. For $i=1,2$ and $j=1,2,3$, from (\ref{v_123}), 
	\begin{eqnarray*}
		\frac{\p {\V}^{k+1}_{{p}^{k+1},{i} }}{\p {\V}^{k}_{{p}^{k},j}}  &=&
		\sum_{\ell=1}^{2} \frac{\p \X^{k+1}_{{p}^{k+1}, \ell}}{\p \V^{k}_{{p}^{k}, j}}
		\p_{\ell}\Big( \frac{\p_{i} \eta_{{p}^{k+1}} }{\sqrt{g_{{p}^{k+1},ii}}}\Big)\Big|_{x^{k+1}} \cdot  \lim_{s \downarrow t^{k+1}} V(s;t^{k}, x^{k}, v^{k}) \\
		&& + \frac{\p_{i} \eta_{{p}^{k+1} } (x^{k+1}) }{\sqrt{g_{{p}^{k+1}, ii} (x^{k+1})}} \cdot \Big(
		\frac{\p v^{k}}{\p \V^{k}_{{p}^{k},j}}  + \frac{\p (t^{k}-t^{k+1})}{\p \V^{k}_{{p}^{k},j}} \nabla\Phi(t^{k+1}, x^{k+1})  \\
		&& -  \int_{t^{k}}^{ t^{k+1} } (\frac{\p X(s)}{\p \V^{k}_{{p}^{k},j} }\cdot\nabla)\nabla \Phi(s,X(s)) \dd s \Big) 
		\\
		&=& \sum_{\ell=1}^{2} \frac{\p \X^{k+1}_{{p}^{k+1}, \ell}}{\p \V^{k}_{{p}^{k}, j}}
		\p_{\ell} \Big( \frac{\p_{i} \eta_{{p}^{k+1}} }{\sqrt{g_{{p}^{k+1},ii}}}\Big)\Big|_{x^{k+1}} \cdot v^{k} + \frac{\p_{i} \eta_{{p}^{k+1} } (x^{k+1}) }{\sqrt{g_{{p}^{k+1}, ii} (x^{k+1})}} \cdot \frac{\p_{j} \eta_{{p}^{k}} (x^{k })}{\sqrt{g_{{p}^{k}, jj} (x^{k })}} \\
		&& 
		+ O_{\Omega}( \| \Phi \|_{C^{2}} ) (t^{k} - t^{k+1})\Big|  \frac{\p \X^{k+1}_{p^{k+1}}}{\p \V^{k}_{p^{k} ,j }} \Big|
		+ O_{\Omega}( \| \Phi \|_{C^{2}} ) \Big| \frac{\p (t^{k} - t^{k+1})}{\p \V^{k}_{p^{k},j}} \Big| \\
		&& +O_{\Omega}( \| \Phi \|_{C^{2}} ) (t^{k } -t^{k+1}) \sup_{s} \Big| \frac{\p X(s)}{\p \V^{k}_{p^{k}, j}} \Big|.
	\end{eqnarray*}
	From (\ref{Est-5}), (\ref{Est-6}), and (\ref{dX/dv}), we prove (\ref{Est-7}). The proof of (\ref{Est-8}) is also very similar as above from (\ref{v_123}).\end{proof}

\begin{lemma} \label{global to local} 
	Assume that $x\in \Omega$ (interior point) and $\xb(t,x,v)$ is in the neighborhood of $p^{1} \in \p\Omega$. Then locally, 
	\begin{equation} \label{Est--1}
	\begin{split}
		\frac{\p \tb }{\p x_{j}} &= \frac{1}{\V^{1}_{{p^{1}},3}} \Big[
		-e_{j} +  
		O_\Omega  ( \| \Phi \|_{C^{2}} ) |\tb|^2 e ^{   { \| \nabla_{x}^{2}\Phi\|_{\infty}}  (\tb)^2 /2 }  
		\Big] \cdot  \frac{\p_{3} \eta_{{p^{1}}} (x^{1})}{\sqrt{g_{{p^{1}}, 33} (x^{1})}}, \quad j=1,2, \\
	\end{split}
	\end{equation}	
	\begin{equation} \label{Est--2}
	\begin{split}	
		\frac{\p \tb}{\p v_{j}}
		&= \frac{1}{\V^{1}_{{p^{1}},3}}
		\Big[ \tb e_{j } 
		- \int^{t-\tb}_{t}  \int^{s}_{t}  \Big( \frac{\p X( \tau)}{\p v_{j}} \cdot \nabla_{x}  \Big)
		\big( \nabla_{x} \Phi ( \tau)\big) \dd \tau\dd s \Big] 
		\cdot \frac{\p_{3} \eta_{{p^{1}}}(x^{1})}{\sqrt{g_{{p^{1}},33} (x^{1}) }} \\
		&=  \frac{\tb}{\V^{1}_{{p^{1}},3}} \Big[ 
		e_{j} +  
		O_\Omega  ( \| \Phi \|_{C^{2}}  )
		|\tb|^{3}  e^{ \| \Phi \|_{C^{2}} (\tb)^{2} /2 }
		\Big] \cdot  \frac{\p_{3} \eta_{{p^{1}}  } (x^{1})}{\sqrt{g_{{p^{1}}, 33} (x^{1}) }},  \quad j=1,2,3,
	\end{split}
	\end{equation}
	\begin{equation} \label{Est--3}
	\begin{split}	
		\frac{\p \X^{1}_{{p^{1}}, i}}{\p x_{j} } 
		&=
		\Big[ e_{j} + 
		O_\Omega  (
		\| \Phi \|_{C^{2}} )
		\big(1 + \frac{ |\V_{p^{1},i}^{1}| }{ |\V_{p^{1},3}^{1}| } \big)  |\tb|^2 e ^{   { \| \nabla_{x}^{2}\Phi\|_{\infty}}  (\tb)^2 /2 }
		\Big] 	\\	
		&\quad \cdot \frac{1}{\sqrt{g_{{p^{1}},ii} (x^{1})  }}
		\Big[ \frac{\p_{i} \eta_{{p^{1}}} (x^{1})   }{ \sqrt{g_{{p^{1}}, ii}  (x^{1})  }   }
		+ \frac{\V^{1}_{{p^{1}}, i}   }{\V_{{p^{1}},3}   } \frac{\p_{3} \eta_{{p^{1}}   }  (x^{1}) }{\sqrt{g_{{p^{1}}, 33} (x^{1})  }}  \Big],  \\
	\end{split}
	\end{equation}
	\begin{equation} \label{Est--4}
	\begin{split}		
		\frac{\p \X^{1}_{{p^{1}},i}}{\p v_{j}}
		&=
		\big[- \tb e_{j} 
		- \int^{t-\tb}_{t} 
		\int^{s}_{t} 
		\Big( \frac{\p X(\tau)}{\p v_{j}} \cdot \nabla_{x} 	\Big)\nabla \Phi_{x}(\tau )
		\dd \tau
		\dd s
		\big] 	\\
		&\quad 
		\cdot \frac{1}{\sqrt{g_{{p^{1}},ii}  (x^{1})  }}
		\Big[ \frac{\p_{i} \eta_{{p^{1}}} (x^{1})  }{\sqrt{g_{{p^{1}},ii}(x^{1})}   } { +} \frac{\V^{1}_{{p^{1}},i}}{\V^{1}_{{p^{1}},3}} \frac{\p_{3} \eta_{{p^{1}}}(x^{1})   }{\sqrt{g_{{p^{1}},33}  (x^{1}) }}\Big]
		\\
		&=
		- \tb \big[
		e_{j} 
		+   O_\Omega  (\| \Phi \|_{C^{2}}) |\tb|^{2}
		e^{ \| \nabla_{x}^{2} \Phi \|_{\infty} (\tb)^{2} /2}
		\big] 	\\
		&\quad \cdot \frac{1}{\sqrt{g_{{p^{1}},ii}  (x^{1})  }}
		\Big[ \frac{\p_{i} \eta_{{p^{1}}} (x^{1})  }{\sqrt{g_{{p^{1}},ii}(x^{1})}   } { +} \frac{\V^{1}_{{p^{1}},i}}{\V^{1}_{{p^{1}},3}} \frac{\p_{3} \eta_{{p^{1}}}(x^{1})   }{\sqrt{g_{{p^{1}},33}  (x^{1}) }}\Big]
		,  
	\end{split}
	\end{equation}
	\begin{equation} \label{Est--5}
	\begin{split}	
		\frac{\p \V^{1}_{{p^{1}}, i}}{\p x_{j}} 
		&=  
		-\frac{\p_{i} \eta_{{p^{1}}}(x^{1})   }{\sqrt{g_{{p^{1}},ii} (x^{1})  }  } \cdot \Big[ \int^{t-\tb}_{t} \Big( \frac{\p X(s)}{\p x_{j}} \cdot \nabla_{x} \Big) ( \nabla_{x} \Phi (s) ) \dd s  \Big] 	\\
		&
		\quad\quad\quad +  \sum_{\ell=1}^{2} \frac{\p \X^{1}_{{p^{1}}, \ell}}{\p x_{j}}  
		\p_{\ell}\big(  \frac{\p_{i} \eta_{{p^{1}}}}{\sqrt{g_{{p^{1}}, ii}}}  \big) 
		\Big|_{x^{1}}\cdot V(t-\tb)   \\
		&=
		\sum_{\ell=1}^{2} \frac{\p \X^{1}_{{p^{1}}, \ell}}{\p x_{j}}  
		\p_{\ell}\big(  \frac{\p_{i} \eta_{{p^{1}}}}{\sqrt{g_{{p^{1}}, ii}}}  \big) 
		\Big|_{x^{1}}\cdot  v	
		+    O_\Omega  (\| \Phi \|_{C^{2}}	)
		\tb (1+ |v|\tb) e^{ \|\nabla_{x}^{2} \Phi \|_{\infty } (\tb)^{2} /2 } \\
		& + 
		\| \nabla_{x} \Phi \|_{\infty} |\tb|			\Big(1 + \frac{|\V^{1}_{p^{1},i}|}{|\V^{1}_{p^{1},3}|}\Big) 
		\Big(
		1+  O_\Omega  (\| \Phi \|_{C^{2}})
		(1+ \tb |v|) ( \tb)^{2} e^{\| \nabla_{x}^{2} \Phi \|_{\infty} (\tb)^{2} /2 }
		\Big)				
		,\\
	\end{split}
	\end{equation}
	\begin{equation} \label{Est--6}
	\begin{split}	
		\frac{\p \V^{1}_{{p^{1}}, i}}{\p v_{j}} 
		&=
		\frac{\p_{i} \eta_{p^{1}}(x^{1}) }{\sqrt{g_{p^{1},ii}(x^{1})  }} \cdot e_{j}
		+  
		\sum_{\ell=1}^{2} \frac{\p \X^{1}_{{p^{1}}, \ell}}{\p v_{j}} 
		\p_{\ell}
		\Big( \frac{\p_{i} \eta_{{p^{1}}}}{\sqrt{g_{{p^{1}}, ii}}}  \Big)\Big|_{x^{1}} \cdot V(t-\tb)
		\\
		&\quad + 
		\frac{\p_{i} \eta_{{p^{1}}} (x^{1})}{\sqrt{g_{{p^{1}},ii} (x_{1})}}
		\cdot 
		\Big[ - \nabla_{x} \Phi (t-\tb; X(t-\tb;t,x,v)) \frac{\p \tb}{\p v_{j}}  	\\
		&\quad -\int^{t-\tb}_{t} \Big(   \frac{\p X (s)}{\p v_{j}} \cdot \nabla_{x} \Big) \nabla_{x} \Phi ( s ) \dd s  \Big]\\
		&= 
		\frac{\p_{i} \eta_{p^{1}}(x^{1}) }{\sqrt{g_{p^{1},ii}(x^{1})  }} \cdot e_{j}
		+  
		\sum_{\ell=1}^{2} \frac{\p \X^{1}_{{p^{1}}, \ell}}{\p v_{j}} 
		\p_{\ell}
		\Big( \frac{\p_{i} \eta_{{p^{1}}}}{\sqrt{g_{{p^{1}}, ii}}}  \Big)\Big|_{x^{1}} \cdot  v	\\
		&
		+   O_\Omega  ( \| \Phi \|_{C^{2}} )
		\big(1+\frac{|\V_{p^{1}}^{1}|}{|\V^{1}_{p^{1},3}|}  \big)|\tb|\big(1+  O_\Omega  ( \| \Phi \|_{C^{2}} )(\tb)^{2} e^{\| \nabla_{x}^{2} \Phi \|_{\infty} (\tb)^{2} /2}\big).
	\end{split}
	\end{equation}
	Here, $e_{j}$ is the $j^{th}$ directional unit vector in $\mathbb{R}^{3}$. \\ 
	\noindent	Moreover,
	\begin{equation} \label{Est--7}
	\begin{split}
		\frac{\p |\mathbf{v}^{1}_{p^{1}}|}{\p x_{j}} &=
		O_{\O}(\| \nabla_{x} \Phi \|_{\infty}) \frac{1+ O_{\O} (\| \nabla_{x}^{2} \Phi \|_{\infty}) (1+ \tb |v|) |\tb|^{2} e^{ \| \nabla_{x}^{2 } \Phi \|_{\infty} (\tb)^{2}}
		}{|\mathbf{v}^{1}_{p^{1},3}|}  \\
		&\quad +O_{\O} (\| \Phi \|_{C^{2}}) \tb e^{ \| \nabla_{x}^{2} \Phi \|_{\infty} (\tb)^{2} /2},
	\end{split}
	\end{equation}
	\begin{equation} \label{Est--8}
	\begin{split}
		\frac{\p |\mathbf{v}^{1}_{p^{1}}|}{\p v_{j}} &=  \lim_{s \downarrow t^{1}} \frac{V_{j}(s;t,x,v)}{|V(s;t,x,v)|}
		+ O_{\O} (\| \Phi \|_{C^{2}})  (\tb)^{2} e^{ \| \nabla_{x}^{2} \Phi \|_{\infty} (\tb)^{2} /2} \\
		&\quad + O_{\O}(\| \nabla_{x} \Phi \|_{\infty}) \frac{\tb}{|\mathbf{v}^{1}_{p^{1},3}|}\Big\{1+ O_{\O} (\| \nabla_{x}^{2} \eta \|_{\infty}) |\tb|^{3} e^{\| \nabla_{x}^{2} \Phi \|_{\infty} (\tb)^{2}}
		\Big\}.  \\
	\end{split}
	\end{equation}
\end{lemma}	

\begin{proof}
We have
\begin{equation}\label{global down_V}
\lim_{s\downarrow t^{1}} V(s;t, x, v) = v - \int^{t^{1}}_{t} \nabla \Phi (s, X(s;t, x, v)) \dd s,
\end{equation}
\begin{equation}\label{global down_X}
X( \tau; t, x, v) = x +   v(\tau- t) - \int_{t}^{\tau} \dd s\int_{t}^{ s} \dd s^{\prime} \nabla\Phi( s',X( s';t,x,v))   .
\end{equation}
Especially, when $\tau=t^{1}$, we get
\begin{equation}\label{global position identity}
X( t^{1}; t, x, v) = x +   v(t^{1}- t) - \int_{t}^{t^{1}} \dd s\int_{t}^{ s} \dd s^{\prime} \nabla\Phi( s',X( s';t,x,v))   .
\end{equation}
From (\ref{global down_V}), we have
\begin{equation} \label{global V over x}
\begin{split}
\lim_{s \downarrow t^{1}} \frac{\p V(s;t,x,v)}{\p x_{j}} &= { \p\tb \over \p x_{j}} \nabla\Phi(t^{1};X(t^{1};t,x,v)) - \int_{t}^{t^{1}} \big( \frac{\p X(s)}{\p x_{j}}\cdot\nabla_{x} \big) \nabla \Phi(s) \dd s,
\end{split}
\end{equation}
and from (\ref{global V over x}),
\begin{equation} \label{global X over x}
\begin{split}	
\sup_{\tau \leq s^{\prime} \leq t} \bigg| \frac{\p X(  s^{\prime}; t, x, v) }{\p x_{j}} \bigg| 
& \leq   1 + \int_{t}^{\tau} |s - t| \| \Phi\|_{C^{2}} \sup_{\tau \leq s^{\prime} \leq t} \bigg| \frac{\p X(  s^{\prime}; t, x, v) }{\p x_{j}} \bigg| \dd s.
\end{split}
\end{equation}
By Gronwall's inequality 
\begin{equation} \label{global dX/dx}
\sup_{\tau \leq s^{\prime} \leq t} \bigg| \frac{\p X(  s^{\prime}; t, x, v) }{\p x_{j}} \bigg| 
\leq  O_{\Omega}(1) e ^{   { \|\nabla^{2} \Phi\|_{\infty}|t - \tau|^2 } /2   } .
\end{equation} 

Similarly, from (\ref{global down_V}), we have
\begin{equation} \label{global V over v}
\begin{split}
\lim_{s \downarrow t^{1}} \frac{\p V(s;t,x,v)}{\p v_{j}} &= e_{j} + { \p\tb \over \p v_{j}} \nabla\Phi(t^{1};X(t^{1};t,x,v)) - \int_{t}^{t^{1}} \big( \frac{\p X(s)}{\p v_{j}}\cdot\nabla_{x} \big) \nabla \Phi(s) \dd s,
\end{split}
\end{equation}
and from (\ref{global down_X}),
\begin{equation} \label{global X over v}
\begin{split}	
\sup_{\tau \leq s^{\prime} \leq t} \bigg| \frac{\p X(  s^{\prime}; t, x, v) }{\p v_{j}} \bigg| 
& \leq   |\tau-t| + \int_{t}^{\tau} |s - t| \| \Phi\|_{C^{2}} \sup_{\tau \leq s^{\prime} \leq t} \bigg| \frac{\p X(  s^{\prime}; t, x, v) }{\p v_{j}} \bigg| \dd s.
\end{split}
\end{equation}
By Gronwall's inequality 
\begin{equation} \label{global dX/dv}
\sup_{\tau \leq s^{\prime} \leq t} \bigg| \frac{\p X(  s^{\prime}; t, x, v) }{\p v_{j}} \bigg| 
\leq  O_{\Omega}(1) |\tau-t| e ^{   { \|\nabla^{2} \Phi\|_{\infty}|t - \tau|^2 } /2   } .
\end{equation} 

To prove (\ref{Est--1}) - (\ref{Est--6}), these estimates are very similar with those of Lemma \ref{Jac_billiard}. We are suffice to choose global euclidean coordinate instead of $\eta_{p^{k}}$. Therefore we should replace
\begin{equation} \label{replace}
	\eta_{p^{k+1}} \rightarrow \eta_{p^{1}},\quad \eta_{p^{k}} \rightarrow x,\quad t^{k} \rightarrow t,\quad t^{k+1} \rightarrow t-\tb=t^{1},\quad \p_{x_{j}} x = e_{j}.
\end{equation}

\textit{Proof of (\ref{Est--1})}. For $j=1,2$, we apply $\p x_{j}$ to (\ref{global position identity}) and take $\cdot \frac{\p_{3} \eta_{p^{1}}}{\sqrt{g_{p^{1},33}}} \Big|_{x^{1}}$. In this case,  we have ${\p v \over \p x_{j} } = 0$. Then we get the following.
\begin{equation} \label{instead Estimate 1}
\begin{split}
\frac{\p \tb}{\p x_{j}} &=
- 
\frac{1}{ \V^{1}_{ {{p}^{1}} ,3 }  }
\frac{  \p_{3} \eta_{{p}^{1}}(x^{1})}{\sqrt{g_{{p}^{1},33}(x^{1})}} \cdot 
  e_{j}   \\
&\quad +  { \frac{1}{ \V^{1}_{ p^{1}, 3 } }\frac{  \p_{3} \eta_{{p}^{1}}(x^{1})}{\sqrt{g_{{p}^{1},33}(x^{1})}}
	\cdot \int^{ t^{1} }_{t} \dd s \int^{ s}_{t} \dd \tau \big(\frac{\p X(  \tau) }{\p \X_{{p},j}} \cdot \nabla_{x}  \big) \big(  \nabla_{x} \Phi ( \tau) \big) }  .
\end{split}	
\end{equation}
We combine this with (\ref{global X over x}) to get (\ref{Est--1}). \\

\textit{Proof of (\ref{Est--2})}. For $j=1,2,3$, we apply $\p v_{j}$ to (\ref{global position identity}) and take $\cdot \frac{\p_{3} \eta_{p^{1}}}{\sqrt{g_{p^{1},33}}} \Big|_{x^{1}}$. Then we get
\begin{equation}  \label{instead v diff pos iden dot 3}
\begin{split}
0 &= \sum_{l=1,2} \frac{\p \X^{1}_{{p}^{1},l}}{\p v_{j}} \frac{\p\eta_{{p}^{1}}}{\p \X^{k}_{{p}^{1},l}}\Big\vert_{x^{1}}  \cdot   \frac{\p_{3} \eta _{{p}^{1}} }{\sqrt{g_{{p}^{1}, 33} }} \Big|_{x^{1}}  \\
&=	\Big\{ - ( t - t^{1}) e_{j} - \frac{\p( t - t^{1})}{\p v_{j} } \lim_{s \downarrow t^{1}}V(s;t, x, v) \Big\} \cdot  \frac{\p_{3} \eta _{{p}^{1}} }{\sqrt{g_{{p}^{1}, 33} }} \Big|_{x^{1}} \\
&\quad  + O_{\O}(\|\Phi \|_{C^{2}}) |t - t^{1}|^{2}  \sup_{s} \big| \frac{\p X(s)}{\p v_{j} } \big| .
\end{split}
\end{equation}
We use (\ref{global down_V}) and (\ref{global X over v}) to get (\ref{Est--2}). \\		

\textit{Proof of (\ref{Est--3})}. For $i,j=1,2$, we apply $\p x_{j}$ to (\ref{global position identity}) and take $\cdot \frac{\p_{i} \eta_{p^{1}}}{\sqrt{g_{p^{1},ii}}} \Big|_{x^{1}}$. And then we use (\ref{global X over x}) to get

\begin{equation} \label{instead Estimate 2}
\begin{split}
	\frac{\p \X_{{p}^{1},i}^{1}}{\p  x_{j} }   
	&=    \frac{1}{   \V^{1}_{ {{p}^{1}} ,3 }   }
	\frac{  \p_{{3}} \eta_{{p}^{1}}(x^{1})}{\sqrt{g_{{p}^{1},33}(x^{1})}} \cdot 
	 e_{j} 
	\frac{ \V^{1}_{p^{1}, i} }{\sqrt{g_{p^{1}, ii}}} \Big|_{x^{1}}  + \frac{\p_{i} \eta_{p^{1}} }{ {g_{p^{1}, ii} }} \Big|_{x^{1}}  \cdot  e_{j}  \\
	&+ O_{\Omega}(     \|\Phi\|_{C^{2}} ) \frac{  |\mathbf{v}_{p^{1},i}^{1}|  }{|\mathbf{v}_{p^{1},3}^{1}|}( t - t^{1})^2 e^{\| \Phi \|_{C^{2}}(t - t^{1})^{2} /2 }  \\
	&+ O_{\Omega}(     \| \Phi\|_{C^{2}} ) (t-t^{1})^{2}  e^{\| \Phi \|_{C^{2}}(t - t^{1})^{2} /2 }.
\end{split}
\end{equation}
This yields (\ref{Est--3}). \\

\textit{Proof of (\ref{Est--4})}. For $i=1,2$ and $j=1,2,3$, we apply $\p v_{j}$ to (\ref{global position identity}) and take $\cdot \frac{\p_{i} \eta_{p^{1}}}{\sqrt{g_{p^{1},ii}}} \Big|_{x^{1}}$. 	
\begin{eqnarray*}
	\frac{\p \X^{1}_{p^{1}, i}}{\p v_{j} } 
	&=&
	\Big\{- \frac{\p (t - t^{1})}{\p v_{j} } \lim_{s\downarrow t^{1}} V(s; t, x, v)- (t - t^{1}) \frac{\p v}{\p v_{j} }
	\Big\}
	\cdot \frac{\p_{{i}} \eta_{{p}^{1}}}{g_{{p}^{1},ii}}\Big|_{x^{1}} 	\\
	&& + 
	O_{\Omega}( \| \Phi \|_{C^{2}} )|t - t^{1}|^{2} \sup_{s } \big| \frac{\p X(s)}{\p \V^{k}_{p^{k},j}} \big|.
\end{eqnarray*}
And then we use (\ref{global X over x}) to get (\ref{Est--4}). \\

\textit{Proof of (\ref{Est--5})}. For $i=1,2$ and $j=1,2$, we apply $\p x_{j}$ to \begin{equation}  
	\begin{split}\label{global v_123}
	{\V}^{1}_{{p}^{1},i}   
	&=  \frac{ \p_{i} \eta_{{p}^{1}} }{\sqrt{g_{{p}^{1},ii}}}  \Big|_{x^{1}} \cdot  \lim_{s \downarrow  t^{1}}V(s; t , x ,v ), \ \ \ \text{for} \ \  i=1,2,\\
	{\V}^{1}_{{p}^{1},3}  & =  
	- \frac{ \p_{3} \eta_{{p}^{1}} }{\sqrt{g_{{p}^{1},33}}}  \Big|_{x^{1}} \cdot  \lim_{s \downarrow  t^{1}}V(s; t , x ,v ).
	\end{split}
	\end{equation}
	For $i,j=1,2$, from (\ref{v_123}), 
	\begin{equation*}
	\begin{split}
		\frac{\p \V^{1}_{{p}^{1}, i}}{\p x_{j} }  =&
		\frac{\p_{i} \eta_{{p}^{1} } (x^{1}) }{\sqrt{g_{{p}^{1}, ii} (x^{1})}} \cdot \Big[\nabla_{x} \Phi (t^{1}; t ,x , v ) \frac{\p (t  - t^{1})}{\p x_{j} }  \\
		& \quad  - \int^{t^{1}}_{t }  \big(\p_{x_{j}} X(s ) \cdot \nabla_{x} \big) \nabla_{x} \Phi (s, X(s;t , x , v )) \dd s 
		\Big]
		\\
		&\quad				+
		\sum_{\ell=1}^{2} 
		\frac{\p \X^{1}_{{p}^{1}, \ell}}{ \p x_{j} } \frac{\p}{\p \X^{1}_{{p}^{1}, \ell}}\Big(  \frac{\p_{i} \eta_{{p}^{1}}  }{\sqrt{g_{{p}^{1}, ii}    }  }  \Big)\Big|_{x^{1}} \cdot   \lim_{s \downarrow  t^{1}}V(s; t , x ,v ).  
	\end{split}
	\end{equation*}	
	And for $j=1,2,$
	\begin{equation*} \begin{split}
			\frac{\p \V^{1}_{{p}^{1}, 3}}{\p x_{j} }  =&
			-\frac{\p_{3} \eta_{{p}^{1} } (x^{1}) }{\sqrt{g_{{p}^{1}, 33} (x^{1})}} \cdot 
			\Big[ \nabla_{x} \Phi (t^{1}; t ,x , v ) \frac{\p (t  - t^{1})}{\p x_{j} } \\
			&\quad  
			- \int^{t^{1}}_{t }  \big(\p_{x_{j} } X(s ) \cdot \nabla_{x} \big) \nabla_{x} \Phi (s, X(s;t , x , v )) \dd s 
			\Big]
			\\
			&
			-
			\sum_{\ell=1}^{2} 
			\frac{\p \X^{1}_{{p}^{1}, \ell}}{\p x_{j} } \frac{\p}{\p \X^{1}_{{p}^{1}, \ell}}\Big(  \frac{\p_{3} \eta_{{p}^{1}}  }{\sqrt{g_{{p}^{1}, 33}    }  }  \Big)\Big|_{x^{1}} \cdot  \lim_{s \downarrow  t^{1}}V(s; t , x ,v ).\end{split}
	\end{equation*}
	From (\ref{global down_V}), (\ref{Est--3}), and (\ref{Est--1}), we prove (\ref{Est-5}).\\

\textit{Proof of (\ref{Est--6})}. Similar as above we apply $\p v_{j}$ to (\ref{global v_123}) and then use (\ref{global down_V}), (\ref{Est--4}), and (\ref{Est--2}). We skip detail. \\

\textit{Proof of (\ref{Est--7})}. Note that $|\V_{p^{1}}^{1}| = \lim_{s \downarrow t^{1}} {|V(s;t,x,v)|}$ and 

\begin{equation*}
\begin{split}
	2 |\V_{p^{1}}^{1}| \frac{\p |\V_{p^{1}}^{1}| }{ \p x_{j} } = 2 \lim_{s \downarrow t^{1}}  V(s;t,x,v) \cdot \lim_{s \downarrow t^{1}}  \p_{x_{j}}V(s;t,x,v) ,
\end{split}
\end{equation*}
we have
\begin{equation} \label{norm over x}
\begin{split}
		\frac{\p |\V_{p^{1}}^{1}| }{ \p x_{j} } &= \lim_{s \downarrow t^{1}} { V(s;t,x,v) \over |V(s;t,x,v)| } \cdot \lim_{s \downarrow t^{1}}  \p_{x_{j}}V(s;t,x,v) .
\end{split}
\end{equation}
We combine (\ref{norm over x}), (\ref{global V over x}), (\ref{Est--1}), and (\ref{global dX/dx}) to derive (\ref{Est--7}). \\

\textit{Proof of (\ref{Est--8})}. We perform similar process as above with $\p v_{j}$ to get
\begin{equation} \label{norm over v}
\begin{split}
\frac{\p |\V_{p^{1}}^{1}| }{ \p v_{j} } &= \lim_{s \downarrow t^{1}} { V(s;t,x,v) \over |V(s;t,x,v)| } \cdot \lim_{s \downarrow t^{1}}  \p_{v_{j}}V(s;t,x,v) .
\end{split}
\end{equation}

We combine (\ref{norm over v}), (\ref{global V over v}), (\ref{Est--2}), and (\ref{global dX/dv}) to derive (\ref{Est--8}).

\end{proof}
	
\vspace{4pt}

\begin{lemma} \label{trans trajec}
We define $(\mathbf{X}_{{p}} (s;t,x,v), \mathbf{V}_{{p}} (s;t,x,v))$ as 
\begin{equation}\begin{split}\label{XV}
\eta_{{p}} (\mathbf{X}_{{p}} (s;t,x,v))  &:=   X(s;t,x,v) , \\
\mathbf{V}_{{p},i} (s;t,x,v )  &:=  \frac{\p_{i} \eta_{{p}} (  
	\mathbf{X}_{{p}} (s;t,x,v)
	)}{
	\sqrt{
		g_{{p},ii} (\mathbf{X}_{{p}} (s;t,x,v))
	}
}   \cdot V(s;t,x,v).
\end{split}\end{equation}
Then we have
\begin{equation}\begin{split}\label{dot_X}
&\dot{\mathbf{X}}_{{p},i} (s;t,x,v) \\
&= \sum_{j} 
\Big(
\frac{\p_{i} \eta_{{p},j}  (\mathbf{X}_{{p},1}(s), \mathbf{X}_{{p},2}(s),0)  }{g_{{p},ii }  (\mathbf{X}_{{p},1}(s), \mathbf{X}_{{p},2}(s),0)  }
+ O_{\| \eta \|_{C^{2}}}( |\mathbf{X}_{{p}, 3} (s)|)
\Big)
V_{j} (s;t,x,v)\\
&= \frac{1}{\sqrt{g_{{p},ii} (  \mathbf{X}_{{p} } (s;t,x,v)  )}}
\mathbf{V}_{{p}, i}(s;t,x,v) + O_{\| \eta \|_{C^{2}}}(\max_{s} |\mathbf{X}_{{p}, 3}(s )| \max_{s} |V(s )|),\end{split}
\end{equation}
\begin{equation}
\begin{split}\label{dot_V}
&\dot{\mathbf{V}}_{p,i} (s;t,x,v) \\
&=
-\sum_{m =1}^{3}\sum_{n (\neq i)}\sum_{\ell =1}^{2} \frac{1}{\sqrt{g_{p, \ell\ell} }
}
\p_{\ell} \Big( \frac{\p_{n} \eta_{p,m}}{\sqrt{g_{p,nn}}}\Big)\frac{\p_{i} \eta_{p,m}}{\sqrt{g_{p,ii}}}\Big|_{ ( \mathbf{X}_{p,1} (s), \mathbf{X}_{p,2} (s), 0  )} 	\\
&\quad\quad\times 
{\mathbf{V}}_{p,\ell} (s;t,x,v)
\mathbf{V}_{p,n} (s;t,x,v)
+ O(\| \eta \|_{C^{3}}) \big\{ \max_{s} |\mathbf{X}_{p,3} (s)| \max_{s} |V(s)|^{2} 	\\
&\quad +  \max_{s}|\mathbf{V}_{p,3} (s)| \max_{s}| {V}  (s)| + \| \nabla_{x}^{2} \Phi \|_{\infty} \big\} .
\end{split}
\end{equation}
\end{lemma}

\begin{proof}	
\textit{Proof of (\ref{dot_X}) }.
	From (\ref{E_Ham}),
	\begin{equation}\begin{split}\notag
	\sum_{\ell} \p_{\ell} \eta_{{p}, i} ( \mathbf{X}_{{p}} (s;t,x,v))
	\dot{ \mathbf{X}}_{{p},\ell}  (s;t,x,v)   = \dot{X}_{i}(s;t,x,v) = V_{i}(s;t,x,v). 
	\end{split}
	\end{equation}
	Note that from (\ref{orthonormal_eta}), for $|\mathbf{X}_{{p},3}(s)|\ll 1$,
	\begin{equation}\label{inverse}
	\begin{split}
	\big(\nabla \eta_{{p}} (\mathbf{X}_{{p},1} (s),  \mathbf{X}_{{p},2}(s), \mathbf{X}_{{p},3}(s))\big)^{-1}_{i,j} &= 
	\frac{\p_{i} \eta_{{p},j} (\mathbf{X}_{{p},1}(s), \mathbf{X}_{{p},2}(s),0) }{g_{{p},ii } (\mathbf{X}_{{p},1}(s), \mathbf{X}_{{p},2}(s),0) }  \\
	&\quad + O(\| \eta \|_{C^{2}}) |\mathbf{X}_{{p}, 3}(s)|,\\
	\left( \begin{array}{c}
	\frac{\p_{1} \eta_{p}}{\sqrt{g_{p,11}}} (\mathbf{X}_{{p} ,1}(s), \mathbf{X}_{{p} ,2}(s),0)\\ 
	\frac{\p_{2} \eta_{p}}{\sqrt{g_{p,22}}} (\mathbf{X}_{{p} ,1}(s), \mathbf{X}_{{p} ,2}(s),0)\\ 
	\frac{\p_{3} \eta_{p}}{\sqrt{g_{p,33}}} (\mathbf{X}_{{p} ,1}(s), \mathbf{X}_{{p} ,2}(s),0)
	\end{array}
	\right)^{-1} _{i,j}
	&= 
	\frac{\p_{j} \eta_{{p},i} (\mathbf{X}_{{p},1}(s), \mathbf{X}_{{p},2}(s),0) }{\sqrt{g_{{p},jj} (\mathbf{X}_{{p},1}(s), \mathbf{X}_{{p},2}(s),0) }}.
	\end{split}
	\end{equation}
We apply these to (\ref{XV}) and use (\ref{E_Ham}) to get (\ref{dot_X}). \\
	
\textit{Proof of (\ref{dot_V}) }.	
	\noindent From (\ref{XV}), (\ref{inverse}), and (\ref{dot_X}),
	\begin{eqnarray*}
		&&\dot{\mathbf{V}}_{p,i}(s;t,x,v) \\
		&=& \frac{d}{ds} \Big[
		\frac{\p_{i} \eta_{p} ( \mathbf{X}_{p,1} (s), \mathbf{X}_{p,2} (s), 0  )}{
			\sqrt{g_{p,ii} ( \mathbf{X}_{p,1} (s), \mathbf{X}_{p,2} (s), 0  )}
		}\cdot V(s;t,x,v)  
		\Big]  \\
		&& + O(\| \eta \|_{C^{3}}) \max_{s} |\mathbf{X}_{p,3} (s)| \max_{s} |V(s)|^{2}  \\
		&& + O(\| \eta \|_{C^{2}}) \max_{s}|\mathbf{V}_{p,3} (s)| \max_{s}|\mathbf{V}  (s)| + O(\| \eta \|_{C^{2}}) \| \nabla_{x}^{2} \Phi \|_{\infty} \max_{s}|\mathbf{X}_{p,3} (s)|
		\\
		&
		=
		& 
		\sum_{m,n=1}^{3}\sum_{\ell =1}^{2} \frac{1}{\sqrt{g_{p, \ell\ell} }
		}
		\p_{\ell} \Big(
		\frac{\p_{i} \eta_{p,m}}{\sqrt{g_{p,ii}}}
		\Big)
		\frac{\p_{n} \eta_{p,m}}{\sqrt{g_{p,nn}}}\Big|_{ ( \mathbf{X}_{p,1} (s), \mathbf{X}_{p,2} (s), 0  )}
		{\mathbf{V}}_{p,\ell} (s;t,x,v)
		\mathbf{V}_{p,n} (s;t,x,v)
		\\
		&&+ O(\| \eta \|_{C^{3}}) \max_{s} |\mathbf{X}_{p,3} (s)| \max_{s} |V(s)|^{2} + O(\| \eta \|_{C^{2}}) \max_{s}|\mathbf{V}_{p,3} (s)| \max_{s}|\mathbf{V}  (s)| 	\\
		&&+ O(\| \eta \|_{C^{2}}) \| \nabla_{x}^{2} \Phi \|_{\infty} ,
	\end{eqnarray*}
	where we have used
	\begin{equation}\notag
	V_{m} (s;t,x,v) = \sum_{n=1}^{3} \frac{\p_{n} \eta_{p,m}}{\sqrt{g_{p,nn}}}\Big|_{ ( \mathbf{X}_{p,1} (s), \mathbf{X}_{p,2} (s), 0  )} \mathbf{V}_{p,n} (s;t,x,v)
	+  C_{\Omega} \max_{s} |\mathbf{X}_{p,3} (s)  | \max_{s} |V(s)|.
	\end{equation}
	In the case of $i=n$, we have $\sum_{m=1}^{3}\p_{\ell} \Big(
	\frac{\p_{i} \eta_{p,m}}{\sqrt{g_{p,ii}}}
	\Big)
	\frac{\p_{i} \eta_{p,m}}{\sqrt{g_{p,ii}}}\Big|_{ ( \mathbf{X}_{p,1} (s), \mathbf{X}_{p,2} (s), 0  )} =0$. Moreover,
	\[
	\sum_{m=1}^{3}  \p_{\ell} \Big(
	\frac{\p_{i} \eta_{p,m}}{\sqrt{g_{p,ii}}}
	\Big)
	\frac{\p_{n} \eta_{p,m}}{\sqrt{g_{p,nn}}}\Big|_{ ( \mathbf{X}_{p,1} (s), \mathbf{X}_{p,2} (s), 0  )}
	= -\sum_{m=1}^{3}  
	\p_{\ell} \Big( \frac{\p_{n} \eta_{p,m}}{\sqrt{g_{p,nn}}}\Big)
	\frac{\p_{i} \eta_{p,m}}{\sqrt{g_{p,ii}}}
	\Big|_{ ( \mathbf{X}_{p,1} (s), \mathbf{X}_{p,2} (s), 0  )}.
	\]
	This finishes the proof for (\ref{dot_V}).
\end{proof}

	\begin{lemma} \label{uniform number of bounce} (i) Suppose $\Omega$ be a bounded open domain in $\R^{3}$. If $|v|\geq \frac{1}{N}$ and $\| \nabla \Phi \|_{\infty} < \frac{\delta}{3 \mathrm{diam} (\Omega)N^{2}}$ for $1 \ll N$ and $0< \delta \ll 1$. Here $\mathrm{diam} (\Omega):= \max_{x,y \in \Omega} |x-y|$. Then 
	\begin{equation} \label{ub tb}
		\tb(t,x,v) \leq 3N \mathrm{diam} (\Omega). 
	\end{equation}
		
		\noindent (ii) Assume the convexity in (\ref{convexity_eta}). Suppose $\frac{1}{N} \leq |v^{k}  | \leq N,$ $\| \nabla \Phi \|_{\infty} <\frac{\delta}{3 \mathrm{diam} (\Omega)N^{2}}$ for $1 \ll N$, and $0< \delta \ll {1\over N}  \ll 1$. If either $\frac{| v^{k} \cdot n( x^{k})|}{|v^{k}|}\ll 1$ or  $\frac{| v^{k+1}\cdot n(x^{k+1})|}{|v^{k+1}|}\ll 1$, then we have the following estimates:
		\begin{equation}\label{upper_delta_t}
		|v^{k} |  (t^{k}-t^{k+1})  \  \lesssim_{\Omega}   \   \min \Big\{\frac{|v^{k} \cdot n(x^{k})|}{|v^{k}|}, \frac{|v^{k+1} \cdot n(x^{k+1})|}{|v^{k+1}|}
		\Big\},
		\end{equation} 
		\begin{equation}\label{lower_delta_t}
		|v^{k} |  (t^{k}-t^{k+1})  \  \gtrsim_{\Omega}   \   \min \Big\{\frac{|v^{k} \cdot n(x^{k})|}{|v^{k}|}, \frac{|v^{k+1} \cdot n(x^{k+1})|}{|v^{k+1}|}
		\Big\}.
		\end{equation}

	\end{lemma}
	\begin{proof}
	\textit{Proof of (\ref{ub tb})}. 
	Note that if $|y-x|> \mathrm{diam}(\Omega)$ and $x \in \bar{\Omega}$ then $y \notin \bar{\Omega}$. If $s_{*} = t- 3N \mathrm{diam}(\Omega)$, then 
	$$
	|X(s_{*};t,x,v)-x| \geq 
	\big\{ |v|   - \| \nabla \Phi \|_{\infty} \frac{|s_{*} -t|  }{2}
	\big\} |s^{*}-t|\geq \frac{1}{2N}3N \mathrm{diam}(\Omega) = \frac{3}{2}\mathrm{diam}(\Omega).
	$$ 
	From (\ref{backward_exit}), therefore, 
	$$
	\tb(t,x,v)= \sup \{s \geq 0: X(\tau;t,x,v) \in \Omega \ \text{for all }   \tau \in (t-s,t) \}\leq 3N \mathrm{diam}(\Omega).
	$$  
		
	\vspace{4pt}
		
	\noindent\textit{Proof of (\ref{upper_delta_t}).}
		%
		Firstly we consider the case of $|v^{k}||t^{k+1} - t^{k}| > \delta$ for $0< \delta \ll1$. If $\|\nabla \Phi \|_{\infty} \leq \frac{2\delta}{N^{2}}$, then 
		\begin{equation} \label{least dist}
		\begin{split}
			|X(t^{k+1})- X(t^{k})| &\geq |v^{k}| |t^{k}-t^{k+1}| - \| \nabla \Phi \|_{\infty} \frac{|t^{k}-t^{k+1}| ^{2}}{2}  \\
			&\geq  |v^{k}| |t^{k}-t^{k+1}|
			\Big\{
			1- \frac{\| \nabla\Phi \|_{\infty} |t^{k}-t^{k+1}|}{2|v^{k}|}
			\Big\} \geq \frac{|v^{k}| |t^{k}-t^{k+1}|}{2} \geq \frac{\delta}{2},
		\end{split}
		\end{equation}
		where we have used the fact that
		\begin{equation}
		\begin{split}	
			\frac{ \|\nabla\Phi\|_{\infty} |t^{k} - t^{k+1}| }{ 2|v^{k}| } \leq \frac{2\delta}{N^{2}} \frac{3N \mathrm{diam}(\O)}{2/N} \leq 3\delta \mathrm{diam}(\O) \ll 1.
		\end{split}
		\end{equation}
		On the other hand, note that
		\begin{equation} \label{normal dist}
		\begin{split}
			&n(X(t^{k})) \cdot \big( X(t^{k+1}) - X(t^{k})\big)  \\ 
			&=
			n(X(t^{k})) \cdot  \Big(
			\lim_{s \uparrow t^{k}}V(s) (t^{k+1}-t^{k})
			+ \int^{t^{k+1}}_{t^{k}}  \int^{s}_{t^{k}} - \nabla \Phi (\tau, X(\tau)) \dd \tau \dd s 
			\Big)	\\
			&= \mathbf{v}^{k}_{{p}^{k},3} (t^{k+1} - t^{k})
			+ \| \nabla \Phi \|_{\infty} \frac{|t^{k+1}-t^{k}|^{2}}{2}.
		\end{split}
		\end{equation}
		From the convexity, the LHS has a lower bound $C_{\Omega} |X(t^{k+1}) - X(t^{k})|^{2}$. Therefore, if $\| \nabla \Phi \|_{\infty} \leq \frac{2\delta}{N^{2}}$, then from (\ref{least dist}) and (\ref{normal dist}),
		\begin{equation*}
		\begin{split}
			\frac{ \mathbf{v}^{k}_{{p}^{k},3}}{| {v}^{k}|} &\geq \frac{1}{ |{v}^{k}| }
			\Big( \frac{ C_{\Omega} }{ |t^{k}- t^{k+1}| } \Big( \frac{ |v^{k}||X(t^{k}) - X(t^{k+1})| }{2} \Big)^{2}   
			- 
			\| \nabla \Phi \|_{\infty} \frac{|t^{k}-t^{k-1}|}{2} \Big)  \\
			&\geq \Big(
			\frac{C_{\Omega}}{4} - \frac{\| \nabla \Phi \|_{\infty}}{ 2|v^{k}|^{2} } \Big) |v^{k}||t^{k} - t^{k+1}| 	\\
			&\geq \Big(
			\frac{C_{\Omega}}{4} - \frac{ 2\delta / N }{ 2/N } \Big) |v^{k}||t^{k} - t^{k+1}| 	\\
			&\geq \Big(
			\frac{C_{\Omega}}{4} - \delta \Big) |v^{k}||t^{k} - t^{k+1}| 	\\
			&\geq 
			\frac{C_{\Omega}}{8} |v^{k}||t^{k} - t^{k+1}|. 	\\
		\end{split}
		\end{equation*}

		\noindent Secondly we consider the case of $|v^{k}||t^{k+1} - t^{k}|  \leq  \delta$ for $0 < \delta \ll 1$. Then $|X(t^{k}) - X(s)| \leq |v^{k}||t^{k+1} - t^{k}| + \frac{\| \nabla \Phi \|_{\infty}}{2}|t^{k+1} - t^{k}| ^{2} \ll 1$ and therefore we may assume that $X(s)$ can be parametrized by ${p}^{k}-$coordinate for all $s \in [t^{k+1}, t^{k}]$. From (\ref{dot_X}),
		\begin{eqnarray*}
			\max_{s}|\mathbf{X}_{{p}^{k},3} (s)|  &\leq& |\mathbf{v}_{{p}^{k},3}^{k}| |t^{k+1}- t^{k}|  \\
			&& + O_{\| \eta \|_{C^{2}}}( \max_{s} |V(s)| |t^{k}-t^{k+1}| )\times \max_{s} |\mathbf{X}_{{p}^{k},3} (s)| .
		\end{eqnarray*}	
		On the RHS, we control $max_{s} |V(s)|$ by 
		\[
			\max_{s} |V(s)| \leq |v^{k}| + \| \nabla \Phi \|_{\infty} \leq 2 |v^{k}| \quad\text{for}\quad \| \nabla \Phi \|_{\infty} \leq \frac{1}{N} \leq |v^{k}|,
		\] 
		so we have 
		\begin{equation} \label{max_V_s}
			\max_{s} |V(s)|  |t^{k}-t^{k+1}| \leq 2 \delta
		\end{equation}
		and 
		\begin{equation}\label{max_X_3}
		\max_{s}|\mathbf{X}_{{p}^{k},3} (s)| \lesssim_{\delta} |\mathbf{v}_{{p}^{k},3}^{k}| |t^{k+1}- t^{k}| .
		\end{equation}
		From (\ref{dot_V}) and (\ref{max_X_3}),
		\begin{eqnarray*}
			\max_{s} |\mathbf{V}_{{p}^{k},3} (s)| &\leq& 
			|\mathbf{v}^{k}_{{p}^{k},3}|
			+4|v^{k}|^{2 } |t^{k}-t^{k+1}| +  \| \nabla \Phi \|_{\infty}  | t^{k}-t^{k+1} | \\
			&& + 4 | \mathbf{v}^{k}_{{p}^{k},3} | |v^{k}|^{2}| t^{k} - t^{k+1} |^{2}  \\
			&& +  O_{\| \eta \|_{C^{3}}}(1) \max_{s} |\mathbf{V}_{{p}^{k},3}(s)| |v^{k}||t^{k}-t^{k+1}|.
		\end{eqnarray*}
		Now we use $|v^{k} | |t^{k+1}- t^{k} | \leq \delta \ll 1$ to have
		\begin{equation}\label{max_V_3}
			\max_{s} |\mathbf{V}_{{p}^{k},3} (s)| \leq 2 |\mathbf{v}^{k}_{{p}^{k},3}| + 4 |v^{k}|^{2} |t^{k}-t^{k+1} | + \| \nabla \Phi \|_{\infty}  | t^{k} - t^{k+1} | .
		\end{equation}
		
		\noindent Now we integrate (\ref{dot_X}) on $t^{k+1}\leq s\leq t^{k}$ and then use (\ref{dot_V}) to have
		\begin{equation}
		\begin{split}
			& \V_{p^{k},3}^{k} ( t^{k} - t^{k+1} ) 	\\
			&= - \int^{t^{k}}_{t^{k+1}}  \int^{s}_{t^{k} } 
			\sum_{m,n=1}^{2} \mathbf{v}_{{p}^{k},m}^{k} \mathbf{v}_{{p}^{k},n}^{k} \frac{ \p_{m} \p_{n} \eta_{{p}^{k}}}{\sqrt{g_{{p}^{k}, nn} }} 
			\cdot \frac{\p_{3} \eta_{{p}^{k}}
			}{\sqrt{g_{{p}^{k}, 33}}} \Big|_{(
			\mathbf{X}_{{p}^{k} ,1}(s), \mathbf{X}_{{p}^{k},2 }(s),0)}
			\dd \tau \dd s 	\\
			&+ O_{\| \eta  \|_{C^{2}}}(1) \Big[ \max_{s} |\mathbf{X}_{{p}^{k},3} (s) | \max_{s} |V(s)| |t^{k}-t^{k+1}|  +  |t^{k+1}-t^{k}|^{3} \max_{s} |V(s)|^{3} 
			\\
			&
			+  |t^{k}-t^{k+1}|^{2}  \big\{  \| \nabla \Phi \|_{\infty} + 
			\max_{s}  |\mathbf{X}_{{p}^{k},3} (s) |  \max_{s} |V(s)|^{2} 	
			+ \max_{s}  |\mathbf{V}_{{p}^{k},3} (s) |  \max_{s} |V(s)| 
			\big\} \Big],	
		\end{split}
		\end{equation}
		and we use convexity (\ref{convexity_eta}), (\ref{max_X_3}), (\ref{max_V_3}), and (\ref{max_V_s}) to derive
		\begin{equation} 
		\begin{split}
			& |\V_{p^{k},3}^{k}| ( t^{k} - t^{k+1} ) 	\\
			&\geq C_{\O}  \frac{ (t^{k} - t^{k+1})^{2} }{2} \sum_{m=1,2} 
				| \V_{p^{k},m}^{k} |^{2} 	\\ 
			&- O_{ \|\eta\|_{C^{1}} } \Big[ | \V_{p^{k},3}^{k} | | t^{k} - t^{k+1} |
				\underbrace{ \max_{s} |V(s)|  | t^{k} - t^{k+1} | }_{\leq 2\delta} + \underbrace{ 2\delta  | t^{k} - t^{k+1} |^{2} \max_{s} |V(s)|^{2} }_{(*)_{1}} 	\\
			&+  | t^{k} - t^{k+1} |^{2} \Big\{ \underbrace{ \|\nabla\Phi\|_{\infty} }_{(*)_{2}} + \underbrace{ | \V_{p^{k},3}^{k} | 
				| t^{k} - t^{k+1} | \max_{s} |V(s)|^{2} }_{(*)_{3}} \\
			&+ 2\Big( \underbrace{2 | \V_{p^{k},3}^{k} |}_{(*)_{4}} + \underbrace{4 |v^{k}|^{2}| t^{k} - t^{k+1} | }_{(*)_{5}}	+ \underbrace{\|\nabla\Phi\|_{\infty} | t^{k} - t^{k+1} |}_{(*)_{6}}	\Big) \Big\} \Big].
		\end{split}
		\end{equation}
		For $(*)_{1}$, we decomposed $|V(s)|$ by $\{ \V_{p^{k},\ell} \}_{\ell=1,2,3}$ and then $\sum_{\ell=1,2} | \V_{p^{k},\ell} |^{2}$ part is absorbed by $C_{\O}  \frac{ (t^{k} - t^{k+1})^{2} }{2} \sum_{m=1,2}| \V_{p^{k},m}^{k} |^{2} $. $\|\V_{p^{k},3}^{k}|^{2}|$ is absorbed by LHS by the fact $|v^{k} | |t^{k+1}- t^{k} | \leq \delta \ll 1$ .
		\\
		For $(*)_{2}$, since $|v^{k}| \geq 1$,
		\begin{equation} 
		\begin{split}
			\|\nabla\Phi\|_{\infty}| t^{k} - t^{k+1} |^{2} &\leq \big( |\V_{p^{k},3}^{k}| |v^{k}| + |\V_{p^{k},\parallel}^{k}|^{2} \big) N^{2} \|\nabla\Phi\|_{\infty}| t^{k} - t^{k+1} |^{2} 	\\
			&\quad + |\V_{p^{k},\parallel}^{k}|^{2} N^{2} \|\nabla\Phi\|_{\infty}| t^{k} - t^{k+1} |^{2} 	\\
			&\leq \underbrace{ N^{2} \|\nabla\Phi\|_{\infty} }_{\lesssim O(1)} \underbrace{ |v^{k}| | t^{k} - t^{k+1} |}_{\leq \delta \ll 1} |\V_{p^{k},3}^{k}| | t^{k} - t^{k+1} | ,
		\end{split}
		\end{equation}
		so absorbed by LHS. For $(*)_{3}$, it is also absorbed by LHS from (\ref{max_V_s}) . For $(*)_{4}$, it is also absorbed by LHS from the facts $|v^{k} | |t^{k+1}- t^{k} | \leq \delta \ll 1$ and $\delta \ll {1\over N}$. For $(*)_{5}$, we perform decomposition as we did in $(*)_{1}$ and apply $|v^{k} | |t^{k+1}- t^{k} | \leq \delta \ll 1$ and $\delta \ll {1\over N}$ to be absorbed by LHS and $C_{\O}  \frac{ (t^{k} - t^{k+1})^{2} }{2} \sum_{m=1,2}| \V_{p^{k},m}^{k} |^{2} $.  For $(*)_{6}$, it is also absorbed by LHS by similar as $(*)_{2}$ case. Finally we conclude (\ref{upper_delta_t}).

	
	\vspace{4pt}

	\noindent\textit{Proof of (\ref{lower_delta_t}).}  
	Assume that $x^{k+1}$ and $x^{k}$ are close enough, \\
	i.e. $|x^{k+1} - x^{k}| \leq    \|\Phi\|_{C^{1}}^{1/2}  \ll 1$. From
	\begin{equation}\begin{split}\label{eta_x}
	& \eta_{{p}^{k}} (\mathbf{x}^{k+1 }_{{p}^{k}}) - \eta_{{p}^{k}} (\mathbf{x}^{k }_{{p}^{k}}) \\
	= & 
	\int^{- (t^{k} - t^{k+1})}_{0}
	V(t^{k} + s ; t^{k}, \eta_{{p}^{k}} (\mathbf{x}^{k}_{{p}^{k}}), v^{k}) \dd s
	\\
	= & \  
	v^{k} (t^{k+1 } - t^{k}) - \int^{ - (t^{k}-t^{k+1})}_{0}  \int^{s}_{0} \nabla \Phi (t^{k} +\tau; X( t^{k} +\tau; t^{k}, \eta_{{p}^{k}} (\mathbf{x}^{k}_{{p}^{k}}), v^{k} ))
	\dd \tau
	\dd s,\end{split}
	\end{equation}
	we have 
	\[
	\eta_{{p}^{k}} (\mathbf{x}^{k+1 }_{{p}^{k}}) - \eta_{{p}^{k}} (\mathbf{x}^{k }_{{p}^{k}}) = v^{k} (t^{k+1}-t^{k})  + O(\| \Phi \|_{C^{1}}) |t^{k+1} - t^{k}|^{2}.
	\]
	By the expansion, $\eta_{{p}^{k}} (\mathbf{x}^{k+1 }_{{p}^{k}}) - \eta_{{p}^{k}} (\mathbf{x}^{k }_{{p}^{k}})=(\mathbf{x}^{k+1}_{{p}^{k+1}} - \mathbf{x}^{k}_{{p}^{k}} ) \cdot \nabla \eta_{{p}^{k}} (\mathbf{x}^{k+1 }_{{p}^{k}})$. For $|t^{k+1}-t^{k}| \leq 1$, $|v^{k}| \geq \frac{1}{N}$, and $\| \Phi \|_{C^{2}} \leq \frac{1}{4N}$ for $N \gg 1$,
	\begin{equation}\label{upper_delta_x}
	|\mathbf{x}^{k+1}_{{p}^{k+1}} - \mathbf{x}^{k}_{{p}^{k}}   |
	\leq \big|  (\nabla \eta_{{p}^{k}} (\mathbf{x}^{k+1 }_{{p}^{k}})  )^{-1} \big| |t^{k+1} - t^{k}| \{ |v^{k}| + O(\| \Phi \|_{C^{1}})  \}\lesssim_{\Omega,N} |v^{k}| |t^{k+1} - t^{k}|.
	\end{equation}

	On the other hand, from $(\ref{eta_x}) \cdot n_{{p}^{k}} (\mathbf{x}^{k+1 }_{{p}^{k}})$,
	we have
	\begin{eqnarray*}
		[\eta_{{p}^{k}} (\mathbf{x}^{k+1 }_{{p}^{k}}) - \eta_{{p}^{k}} (\mathbf{x}^{k }_{{p}^{k}})] \cdot n_{{p}^{k}} (\mathbf{x}^{k+1 }_{{p}^{k}})
		= \mathbf{v}^{k}_{{p}^{k},3} (t^{k+1}- t^{k}) + O(\| \Phi \|_{C^{2}}) |t^{k+1 } - t^{k}|^{2}.
	\end{eqnarray*}
	By the expansion, the LHS equals
	\begin{eqnarray*}
		&&[\eta_{{p}^{k}} (\mathbf{x}^{k+1 }_{{p}^{k}}) - \eta_{{p}^{k}} (\mathbf{x}^{k }_{{p}^{k}}) ]\cdot n_{{p}^{k}} (\mathbf{x}^{k+1 }_{{p}^{k}}) 	\\
		&=& [ (\mathbf{x}^{k+1}_{{p}^{k+1}} - \mathbf{x}^{k}_{{p}^{k}} ) \cdot \nabla \eta_{{p}^{k}} (\mathbf{x}^{k+1 }_{{p}^{k}}) ] \cdot n_{{p}^{k}} (\mathbf{x}^{k+1 }_{{p}^{k}}) + O(\| \eta \|_{C^{2}})  |\mathbf{x}^{k+1}_{{p}^{k+1}} - \mathbf{x}^{k}_{{p}^{k}}|^{2} 	\\
		&\lesssim_{\O}& |\mathbf{x}^{k+1}_{{p}^{k+1}} - \mathbf{x}^{k}_{{p}^{k}}|^{2},
	\end{eqnarray*}
	where we have used the fact that $\nabla \eta_{{p}^{k+1}} (\mathbf{x}^{k+1}_{{p}^{k}}) \perp n_{{p}^{k}} (\mathbf{x}^{k+1}_{{p}^{k}})$. Therefore, if \\ $|\mathbf{v}^{k}_{{p}^{k},3}|> \e$ and $\| \Phi \|_{C^{2}} \ll_{\e} 1$,
	\begin{equation}\begin{split}\label{lower_delta_x}
	|\mathbf{x}^{k+1}_{{p}^{k+1}} - \mathbf{x}^{k}_{{p}^{k}}|^{2}
	&\gtrsim_{\Omega} \big| \mathbf{v}^{k}_{{p}^{k},3} (t^{k+1}- t^{k}) + O(\| \Phi \|_{C^{2}}) |t^{k+1 } - t^{k}|^{2} \big| \\
	&\gtrsim_{\Omega} \big\{|\mathbf{v}^{k}_{{p}^{k},3}| -  O(\| \Phi \|_{C^{2}})   \big\} | t^{k+1}- t^{k}  | \\
	&\gtrsim_{\Omega} |\mathbf{v}^{k}_{{p}^{k},3}| | t^{k+1}- t^{k}  |.
	\end{split}
	\end{equation}
	From (\ref{upper_delta_x}) and (\ref{lower_delta_x}), we prove (\ref{lower_delta_t}) when $x^{k+1}$ and $x^{k}$ are close enough.
	
	Assume $x^{k+1}$ and $x^{k}$ is not close, i.e. $|x^{k+1} - x^{k}| \geq    \|\Phi\|_{C^{1}}^{1/2}$. From (\ref{E_Ham}) and $|t^{k} -t^{k+1}| \leq 1$, $|v^{k}| \geq \frac{1}{N}$, and $\| \Phi \|_{C^{2}} \leq \frac{1}{4N}$ for $N \gg 1$,
	\begin{eqnarray*}
		|t^{k} - t^{k+1}| |v^{k}|  \geq 
		|x^{k+1} - x^{k}| - O(\| \Phi \|_{C^{1}}) |t^{k} - t^{k+1}|^{2}   \gtrsim   \| \Phi \|_{C^{1}}^{1/2}.
	\end{eqnarray*}
	This prove $(\ref{lower_delta_t})$. 
	\end{proof}


\begin{lemma}	\label{velocity_lemma}
	Assume (\ref{eta}) and (\ref{convexity_eta}) hold. Suppose $x \in \bar{\Omega},$ $\frac{1}{N} \leq |v  | \leq N,$ $\| \nabla \Phi \|_{\infty} <\frac{\delta}{3 \mathrm{diam} (\Omega)N^{2}}$ for $1 \ll N$, and $0< \delta \ll {1\over N} \ll 1$. Assume $t \in [M,M+1]$ for $M \in \mathbb{N}$. For all $i \in \mathbb{N}$ with $t^{i} \in [M-1, t]$, 
	\begin{equation}\label{velocity_le_1}
	\begin{split}
	&\max\big\{ 1  - C_{\Omega} |v^{k} | |t^{k}- t^{k+1}|,  c_{\Omega, N}\big\}{\V^{k }_{p^{k },3}} 	\\
	&\quad \leq {\V^{k+1}_{p^{k+1},3}} \leq \min \big\{ 1  + C_{\Omega} |v^{k} | |t^{k}- t^{k+1}|, C_{\Omega, N}\big\}{\V^{k }_{p^{k },3}}
	, 
	\end{split}
	\end{equation}
	and 
	\begin{equation}\label{velocity_le}
	\begin{split}
	&\prod_{j=1}^{k} \max\big\{ 1  - C_{\Omega} |v^{j} | |t^{j}- t^{j+1}|, c_{\Omega, N}\big\}{\V^{1}_{p^{1 },3}}
	\leq {\V^{k+1}_{p^{k+1},3}} 	\\
	&\quad \leq
	\prod_{j=1}^{k}
	\min \big\{ 1  + C_{\Omega} |v^{j} | |t^{j}- t^{j+1}|, C_{\Omega, N}\big\}{\V^{1}_{p^{1},3}}.
	\end{split}
	\end{equation}	
	Moreover,
	\begin{equation}\label{upper_k}
	\sup \big\{ k \in \mathbb{N}: |t-t^{k}| \leq 1 \big\}\lesssim_{\Omega, N, \delta} 1.
	\end{equation}

\end{lemma}	
\begin{proof}\noindent\textit{Step 1. } We claim that if $\V^{k }_{{p}^{k },3}\ll |v^{k }|$, then 
	\begin{equation}\label{tb_v3}
	\begin{split}
	t^{k } -t^{k+1}  &=  \frac{ -2\V^{k+1 }_{{p}^{k+1 },3}}{\sum_{m,n=1}^{2}   \frac{\p_{m} \p_{n} \eta_{{p}^{k+1 }}}{\sqrt{g_{{p}^{k+1 }, mm}}\sqrt{g_{{p}^{k+1 }, nn}}}\Big|_{x^{k +1}} \cdot
		\frac{\p_{3} \eta_{{p}^{k+1 }}}{\sqrt{g_{{p}^{k+1 },33}}} \Big|_{x^{k+1 }} 
		\V^{k+1 }_{{p}^{k +1},m}\V^{k +1}_{{p}^{k +1},n}}   \\
	&\quad + C_{\Omega}\frac{|t^{k }-t^{k+1}|}{|v^{k }|^{2}}\Big\{
	\| \nabla \Phi \|_{\infty} +  |\V^{k }_{{p}^{k },3}| |v^{k }| 
	\Big\}.\end{split}
	\end{equation}	 
	Due to (\ref{upper_delta_t}) and its proof, if $\V^{k }_{{p}^{k },3}\ll |v^{k }|$, then $X(s;t^{k+1 }, x^{k+1 }, v^{k+1 }) \sim  x^{k +1} \sim p^{k+1}$ for all $t^{k+1} \leq s \leq t^{k}$. By the expansion of \\ $(\mathbf{X}_{{p}^{k+1}}(s; t^{k+1}, x^{k+1}, v^{k+1}), \mathbf{V}_{{p}^{k+1}} (s; t^{k+1}, x^{k+1}, v^{k+1}))$ in (\ref{dot_V}) around $s=t^{k+1}$, 
	\begin{equation}\notag
	\begin{split} 
	& \dot{\mathbf{V}}_{p^{k+1},3} (s;t^{k+1},x^{k+1},v^{k+1}) \\
	&=
	- \sum_{n =1}^{2}\sum_{\ell =1}^{2} 
	\frac{  \p_{\ell}  \p_{n} \eta_{p^{k+1} }}{\sqrt{g_{p^{k+1}, \ell\ell} }\sqrt{g_{p^{k+1},nn}}}\Big|_{ x^{k+1}} \cdot \frac{\p_{3} \eta_{p^{k+1} }}{\sqrt{g_{p^{k+1},33}}}\Big|_{ x^{k+1}}
	{\V}_{p^{k+1},\ell}^{k+1}
	\V_{p^{k+1},n}^{k+1}
	\\
	&\quad +
	O (\| \eta \|_{C^{3}})  \big\{\max_{s} |V(s)|^{3} |t^{k} -t^{k+1}| 
	+\| \nabla_{x} \Phi \|_{\infty } \max_{s} |V(s)||t^{k} -t^{k+1}| \\
	&\quad +  \max_{s} |\mathbf{X}_{p^{k+1},3} (s)| \max_{s} |V(s)|^{2} 
	\big\}
	\\
	&\quad + O(\| \eta \|_{C^{2}}) \big\{\max_{s}|\mathbf{V}_{p^{k+1},3} (s)| \max_{s}| {V}  (s)|   + \| \nabla_{x}^{2} \Phi \|_{\infty}\big\} .
	\end{split}
	\end{equation}
	Note that from Lemma \ref{uniform number of bounce} and (\ref{max_X_3}), the last three lines of above are bounded by $
	|v^{k+1}| |\V^{k+1}_{p^{k+1}}| + \| \nabla_{x}^{2} \Phi\|_{\infty}$.
	Then from (\ref{dot_X}), (\ref{dot_V}), (\ref{max_X_3}), and (\ref{max_V_3}), 
	\begin{eqnarray*}
		&& - \frac{(t^{k }-t^{k+1})^{2}}{2}  \sum_{m,n=1}^{2} 
		\frac{\p_{m} \p_{n} \eta_{{p}^{k+1}}}{\sqrt{g_{{p}^{k+1}, mm}}\sqrt{g_{{p}^{k+1}, nn}}}\Big|_{x^{k+1 }}
		\cdot
		\frac{\p_{3} \eta_{{p}^{k+1}}}{\sqrt{g_{{p}^{k+1},33}}} \Big|_{x^{k+1 }} 
		\V^{k +1}_{{p}^{k+1 },m}\V^{k+1 }_{{p}^{k+1 },n}\\
		&& = \V^{k +1}_{{p}^{k+1 },3} (t^{k }-t^{k+1})
		+ O_{\| \eta \|_{C^{3}}} (|t^{k }-t^{k+1}|^{2})\big\{
		\| \nabla \Phi \|_{\infty} +  |\V^{k +1}_{{p}^{k+1 },3}| |v^{k +1}|
		\big\}.
	\end{eqnarray*}
	This proves (\ref{tb_v3}). 
	%
	
	\vspace{4pt}
	
	\noindent\textit{Step 2. } We claim that for $\V^{k+1}_{p^{k+1},3} \ll |v^{k+1}|$,
	\begin{equation}\label{pv_pv}
	\frac{
		\p \V_{ {p}^{k+1},3}^{k+1}
	}{\p \V_{ {p}^{k },3}^{k }} = 1 + O_{\Omega}  \big( \| \Phi \|_{C^{1}} |t^{k }- t^{k+1}|\big)+ O _{\Omega} \big(|v^{k+1}| (t^{k}-t^{k+1})\big)
	.
	\end{equation}	 
	
	\noindent From Lemma \ref{Jac_billiard},
	\begin{eqnarray*}
		&&	\frac{
			\p \V_{{p}^{k+1},3}^{k+1}
		}{\p \V^{k }_{{p}^{k },3}} \\
		&=& - \sum_{\ell=1}^{2} 
		\Big\{
		- (t^k - t^{k+1}) \frac{\p_{3} \eta_{{p}^{k}}  (x^{k}) }{\sqrt{g_{{p}^{k}, 33} (x^{k})} } \cdot 
		\frac{1}{\sqrt{g_{{p}^{k+1}, \ell\ell} (x^{k+1})  }}
		\Big[
		\frac{\p_{  \ell} \eta_{{p}^{k+1}} (x^{k+1})  }{\sqrt{g_{{p}^{k+1},  \ell \ell}  (x^{k+1})  }} 	\\
		&+& 
		\frac{\V^{k+1}_{{p}^{k+1}, \ell}}{\V^{k+1}_{{p}^{k+1},3}} \frac{\p_{3} \eta_{{p}^{k+1}} (x^{k+1}) 
		}{\sqrt{g_{{p}^{k+1}, 33} (x^{k+1})}  }
		\Big]    
		\Big\} 
		\p_{\ell}\Big( \frac{\p_{3} \eta_{ {p}^{k+1}} }{\sqrt{g_{ {p}^{k+1},33}}}\Big)\Big|_{x^{k+1}}  \cdot v^{k}  - \frac{\p_{3} \eta_{ {p}^{k+1} }  }{\sqrt{g_{ {p}^{k+1}, 33}  }} \Big|_{x^{k+1}}\cdot \frac{\p_{3} \eta_{ {p}^{k }}}{\sqrt{g_{ {p}^{k },33}}} \Big|_{x^{k }}   \\
		&+&  O_\Omega(   \|\Phi \|_{C^2} )|v^{k}|
		\Big(1 + \frac{|\V^{k+1}_{p^{k+1},i}|}{|\V^{k+1}_{p^{k+1},3}|} \Big)  (t^k - t^{k+1})^3  
		e^{\| \Phi \|_{C^2} (t^k - t^{k+1})^2}
		\\
		&=& \frac{\p_{3} \eta_{ {p}^{k+1} }  }{\sqrt{g_{ {p}^{k+1}, 33}  }} \Big|_{x^{k+1}}\cdot \frac{\p_{3} \eta_{ {p}^{k }}}{\sqrt{g_{ {p}^{k },33}}} \Big|_{x^{k }}  \\
		&&\times \underbrace{
			\Big\{ 
			-1 + \frac{t^{k } - t^{k+1}}{\V^{k+1}_{{p}^{k+1},3}} \sum_{\ell=1}^{2} \frac{{\V}^{k+1}_{{p}^{k+1},\ell}}{\sqrt{g_{{p}^{k+1}, \ell\ell}}}
			\p_{\ell}\Big( \frac{\p_{3} \eta_{{p}^{k+1}} }{\sqrt{g_{{p}^{k+1},33}}}\Big)\Big|_{x^{k+1}} \cdot v^k 
			\Big\}
		}_{(*)}\\
		&& 
		+ 
		O_\Omega(  1)|v^{k}|(t^{k}- t^{k+1})  
		\bigg|
		\frac{\p_{3} \eta_{{p}^{k}}   }{\sqrt{g_{{p}^{k}, 33}  } } 
		\Big|_{x^{k}}
		\cdot  
		\frac{\p_{  \ell} \eta_{{p}^{k+1}}   }{ {g_{{p}^{k+1},  \ell \ell}    }}  \Big|_{x^{k+1}}
		\bigg|   \\
		&&+ 
		O_\Omega(  \| \Phi \|_{C^2}  )
		\Big(1 + \frac{|\V^{k+1}_{p^{k+1},i}|}{|\V^{k+1}_{p^{k+1},3}|} \Big) |v^{k}| (t^k - t^{k+1})^3  
		e^{\| \Phi \|_{C^2} (t^k - t^{k+1})^2}.
	\end{eqnarray*}
	
	\noindent Consider $(*)$. For $\ell, j=1,2$, from (\ref{orthonormal_eta}),
	\begin{eqnarray*}
		&& \p_{\ell} \Big( \frac{\p_{3} \eta_{{p}^{k+1}} }{\sqrt{g_{{p}^{k+1},33}}}\Big) \cdot  \frac{\p_{j} \eta_{{p}^{k+1}} }{\sqrt{g_{{p}^{k+1},jj}}} \Big|_{x^{k+1}} 	\\
		&& =   \underbrace{\p_{\ell} \Big( \frac{\p_{3} \eta_{{p}^{k+1}} }{\sqrt{g_{{p}^{k+1},33}}} \cdot  \frac{\p_{j} \eta_{{p}^{k+1}} }{\sqrt{g_{{p}^{k+1},jj}}}\Big) \Big|_{x^{k+1}} }_{=0}-   \frac{\p_{3} \eta_{{p}^{k+1}} }{\sqrt{g_{{p}^{k+1},33}}}  \Big|_{x^{k+1}}\cdot   \frac{ \p_{\ell}\p_{j} \eta_{{p}^{k+1}} }{\sqrt{g_{{p}^{k+1},jj}}}  \Big|_{x^{k+1}},
	\end{eqnarray*}
	and hence
	\begin{eqnarray*}
		&& \p_{\ell} \Big( \frac{\p_{3} \eta_{{p}^{k+1}} }{\sqrt{g_{{p}^{k+1},33}}}\Big) 
		\Big|_{x^{k+1}}\cdot v^{k } 	\\ 
		&=&\p_{\ell}\Big( \frac{\p_{3} \eta_{{p}^{k+1}} }{\sqrt{g_{{p}^{k+1},33}}}\Big) \cdot v^{k+1}+ O_{\| \eta \|_{C^{2}}} ( \| \Phi \|_{C^{1}}) |t^{k }- t^{k+1}| \\
		&=& -   \sum_{j=1}^{2}\frac{\p_{3} \eta_{{p}^{k+1}} }{\sqrt{g_{{p}^{k+1},33}}} \cdot   \frac{ \p_{\ell}\p_{j} \eta_{{p}^{k+1}} }{\sqrt{g_{{p}^{k+1},jj}}}  \Big|_{x^{k+1}} \V^{k+1}_{{p}^{k+1}, j}+ O_{\| \eta \|_{C^{2}}} (  \| \Phi \|_{C^{1}})|t^{k }- t^{k+1}| .
	\end{eqnarray*}
	Combining with (\ref{tb_v3}), we conclude that 
	\begin{eqnarray*}
		(*) =  -1 + 2 + 
		O_{\| \eta \|_{C^{2}}}  (  \| \Phi \|_{C^{1}}) \frac{|v^{k +1}|^{2}|t^{k }- t^{k+1}|^{2}}{|\V^{k+1}_{{p}^{k+1},3}|}  =  1 + O_{\| \eta \|_{C^{2}}}  ( \| \Phi \|_{C^{1}} |t^{k }- t^{k+1}|).
	\end{eqnarray*}
	
	\noindent Note that 
	\begin{eqnarray*}
		&&\frac{\p_{3} \eta_{ {p}^{k+1} }  }{\sqrt{g_{ {p}^{k+1}, 33}  }} \Big|_{x^{k+1}}\cdot \frac{\p_{3} \eta_{ {p}^{k }}}{\sqrt{g_{ {p}^{k },33}}} \Big|_{x^{k }}
		= 1 + O_\Omega(  1) \max_{s} |V(s)| (t^{k} - t^{k+1}),\\
		&&  \frac{\p_{\ell} \eta_{ {p}^{k+1} }  }{\sqrt{g_{ {p}^{k+1}, \ell\ell}  }} \Big|_{x^{k+1}}\cdot \frac{\p_{3} \eta_{ {p}^{k }}}{\sqrt{g_{ {p}^{k },33}}} \Big|_{x^{k }}
		=  O_\Omega(  1) \max_{s} |V(s)| (t^{k} - t^{k+1}), \ \ \text{for} \ \ell=1,2.
	\end{eqnarray*}
	Altogether we prove (\ref{pv_pv}).
	
	\vspace{4pt}
	
	\noindent\textit{Step 3. } We prove (\ref{velocity_le_1}). 
	For $\V^{k+1}_{p^{k+1},3}\ll |v^{k+1}|$, by the expansion and (\ref{pv_pv}),
	\begin{eqnarray*}
		&&\V_{{p}^{k+1},3}^{k+1} (t^{k }, \X^{k }_{{p}^{k }}; \V^{k }_{{p}^{k },1},\V^{k }_{{p}^{k},2},  \V^{k }_{{p}^{k },3})  \\
		&=&\V_{{p}^{k+1},3}^{k+1} (t^{k }, \X^{k }_{{p}^{k }}; \V^{k }_{{p}^{k },1},\V^{k }_{{p}^{k },2},  0)   + \int^{\V^{k }_{{p}^{k },3}}_{0} \frac{\p \V^{k+1}_{{p}^{k+1},3}}{\p \V^{k }_{{p}^{k },3}} (t^{k }, \X^{k }_{{p}^{k }}; \V^{k }_{{p}^{k},1},\V^{k }_{{p}^{k },2},  \tau) \dd \tau	\\
		&=& 0 + \V^{k }_{{p}^{k },3}  \times (\ref{pv_pv}).
	\end{eqnarray*}
	This proves $\V^{k+1}_{{p}^{k+1},3} = 
	\big(
	1+  O_{\| \eta \|_{C^{2}}} |v^{k }| |t^{k } - t^{k+1}|
	\big)
	\V^{k }_{{p}^{k },3}$. Now we consider the case of $\V^{k+1}_{p^{k+1},3}\gtrsim |v^{k+1}|$. Clearly $\V^{k}_{p^{k},3} \leq |\V^{k}_{p^{k} }| \leq |v^{k+1}| + \| \nabla_{x} \Phi \|_{\infty} |t^{k}-t^{k+1}| \leq \frac{1}{2} |v^{k+1}| \lesssim \V^{k+1}_{p^{k+1},3}$ for sufficiently small $\| \nabla_{x} \Phi \|_{\infty}$. These prove (\ref{velocity_le_1}). Then we prove (\ref{velocity_le}) by induction in $k$. ALso, the proof of (\ref{upper_k}) is direct consequence of (\ref{velocity_le}). 
	\begin{equation}\notag
	\begin{split}
	\V^{k+1}_{p^{k+1},3} &\geq \big(1+ C_{\Omega} |v^{k} | |t^{k}-t^{k+1}|\big)^{-1}\V^{k }_{p^{k },3} \geq  e^{-C_{\Omega} |v^{k} | |t^{k}-t^{k+1}|}\V^{k }_{p^{k },3}   \\
	&\geq e^{- C_{\Omega}\sum_{i=1}^{k} |v^{i} | |t^{i}-t^{i+1}|}
	\V^{1 }_{p^{1 },3} \geq e^{-C_{\Omega} N } \delta.
	\end{split}
	\end{equation} 
\end{proof}

\begin{lemma}\label{det_billiard}Assume $\frac{1}{N} \leq |v  | \leq N,$ $\| \nabla \Phi \|_{\infty} <\frac{\delta}{3 \mathrm{diam} (\Omega)N^{2}}$ for $1 \ll N$, and $0< \delta \ll {1\over N} \ll 1$. Also we assume $|t^{k}-t^{k+1}| \leq 1$. Then 

	\begin{equation}\notag 
	\begin{split}
	& \bigg| \det \left[\begin{array}{cc} \nabla_{\X^{k}_{{p}^{k}}} \X^{k+1}_{{p}^{k+1}} & \nabla_{ {\V}^{k}_{{p}^{k}}} \X^{k+1}_{{p}^{k+1}}\\
	\nabla_{\X^{k}_{{p}^{k}}}  {\V}^{k+1}_{{p}^{k+1}}
	& \nabla_{ {\V}^{k}_{{p}^{k}} }  {\V}^{k+1}_{{p}^{k+1}}
	\end{array}\right]_{5\times 5} \bigg| 	 \\ 
	&=\Big(1+ O_{\Omega,N}(  \|   \Phi \|_{C^{2}} )
	\Big) \frac{ \sqrt{g_{{p}^{k },11} (x^{k}) }  \sqrt{g_{{p}^{k },22}  (x^{k})  }}{\sqrt{g_{{p}^{k+1 },11}  (x^{k+1})  }  \sqrt{g_{{p}^{k +1},22}  (x^{k+1})  }}	
	\frac{
		|\V^{k }_{p^{k },3}|
	}{|\V^{k+1}_{p^{k+1},3}|},
	\end{split}
	\end{equation}
		for the mapping $(\X^{k}_{p^{k},1}, \X^{k}_{p^{k},2}, \V^{k}_{p^{k}}) \mapsto (\X^{k+1}_{p^{k+1},1}, \X^{k}_{p^{k+1},2}, \V^{k}_{p^{k+1}})$.

	%
\end{lemma}
\begin{proof} From Lemma \ref{Jac_billiard} and Lemma \ref{uniform number of bounce},
	\begin{equation}\begin{split}\notag
	&
	\left[\begin{array}{c|c} \nabla_{ \X^{k}_{{p}^{k}}} \X^{k+1}_{{p}^{k+1},1} 
	& \nabla_{ \V^{k}_{{p}^{k}}} \X^{k+1}_{{p}^{k+1},1}\\ 
	\nabla_{ \X^{k}_{{p}^{k}}} \X^{k+1}_{{p}^{k+1},2} & \nabla_{ \V^{k}_{{p}^{k}}} \X^{k+1}_{{p}^{k+1},2}\\ 				 
	\hline
	\nabla_{\X^{k}_{{p}^{k}}} \V^{k+1}_{{p}^{k+1} }
	& \nabla_{ \V^{k}_{{p}^{k}} } \V_{{p}^{k+1} }^{k+1} 
	\end{array}\right]_{5\times 5}   \\
	&= 
	\underbrace{
		{\left[\begin{array}{c|c}
			\bigg[\frac{\p \X^{k+1}_{{p}^{k+1},i}}{\p {\X^{k}_{{p}^{k}, j}} }\bigg]_{i=1,2, j=1,2 }
			& \bigg[\frac{\p  \X^{k+1}_{{p}^{k+1},i}}{\p{\V^{k}_{{p}^{k}, j}}}\bigg]_{i=1,2, j=1,2,3}
			\\ 
			\hline 
			\substack{	\frac{\p_{i} \eta_{{p}^{k+1} } }{\sqrt{g_{{p}^{k+1}, ii}  }}\Big|_{x^{k+1}} \cdot \frac{\p  v^{k} }{\p \X^{k}_{{p}^{k}, j}} 
				\\ +
				\sum_{\ell=1}^{2} 
				\frac{\p \X^{k+1}_{{p}^{k+1}, \ell}}{\p \X^{k}_{{p}^{k}, j}} \p_{\ell}\Big(  \frac{\p_{i} \eta_{{p}^{k+1}}  }{\sqrt{g_{{p}^{k+1}, ii} } }  \Big) \Big|_{x^{k+1}}  \cdot v^{k}   }
			& 
			\substack{ \sum_{\ell=1}^{2} \frac{\p \X^{k+1}_{{p}^{k+1}, \ell}}{\p \V^{k}_{{p}^{k}, j}}
				\p_{\ell}\Big( \frac{\p_{i} \eta_{{p}^{k+1}} }{\sqrt{g_{{p}^{k+1},ii}}}\Big) \Big|_{x^{k+1}}  \cdot v^{k}\\
				+ \frac{\p_{i} \eta_{{p}^{k+1}} (x^{k+1})}{\sqrt{g_{{p}^{k+1}, ii} (x^{k+1})}} \cdot \frac{\p_{j} \eta_{{p}^{k}} (x^{k})}{\sqrt{g_{{p}^{k}, jj} (x^{k})}}	}\\ \hline
			\substack{ -\frac{\p_{3} \eta_{{p}^{k+1} }   }{\sqrt{g_{{p}^{k+1}, 33}  }} \Big|_{x^{k+1}}  \cdot \frac{\p  v^{k} }{\p \X^{k}_{{p}^{k}, j}}\\
				-
				\sum_{\ell=1}^{2} 
				\frac{\p \X^{k+1}_{{p}^{k+1}, \ell}}{\p \X^{k}_{{p}^{k}, j}} \p_{\ell}\Big(  \frac{\p_{3} \eta_{{p}^{k+1}}  }{\sqrt{g_{{p}^{k+1}, 33}   }  }  \Big) \Big|_{x^{k+1}}  \cdot v^k} &  \substack{
				-\sum_{\ell=1}^{2} \frac{\p \X^{k+1}_{{p}^{k+1}, \ell}}{\p \V^{k}_{{p}^{k}, j}}
				\p_{\ell}\Big( \frac{\p_{3} \eta_{{p}^{k+1}} }{\sqrt{g_{{p}^{k+1},33}}}\Big)\Big|_{x^{k+1}}  \cdot v^{k}\\
				- \frac{\p_{3} \eta_{{p}^{k+1}} (x^{k+1})}{\sqrt{g_{{p}^{k+1}, 33} (x^{k+1})}} \cdot \frac{\p_{j} \eta_{{p}^{k}} (x^{k})}{\sqrt{g_{{p}^{k}, jj} (x^{k})}}}
			\end{array} \right]}  }_{:=\mathcal{A}}  \\
	& \ \ \ \ \tiny{ + \left[ 
		\begin{array}{c|c}	
		0  & 0  \\	 \hline
		C_{N}\| 	\Phi \|_{C^{2}}	 & C_{N}\| 	\Phi \|_{C^{2}}		
		\end{array}					 \right]_{5\times 5}} .  \\
	\end{split}\end{equation}
	Now for $i=1$ and $i=2$, we multiply $\p_{\ell} \Big( \frac{\p_{i} \eta_{{p}^{k+1}}}{\sqrt{g_{{p}^{k+1},ii}}}
	\Big)\Big|_{x^{k+1}} \cdot v^{k}$ to $\ell^{\text{th}}$ row for $\ell=1,2$, and then subtract this to the $(i+2)^{\text{th}}$ row.  Similarly, we multiply $\p_{\ell}\Big( \frac{\p_{3} \eta_{{p}^{k+1}}}{\sqrt{g_{{p}^{k+1},ii}}}
	\Big) \Big|_{x^{k+1}}  \cdot v^{k}$ to $\ell^{\text{th}}$ row for $\ell=1,2$ and then subtract this to the $5^{\text{th}}$ row.  
	{ Hence, rewriting first two rows using Lemma \ref{Jac_billiard}, the resulting row echelon form of matrix $\mathcal{A}$ is } 
	\begin{equation}  \notag
	\begin{split}
	&  \tiny{ \left[\begin{array}{c|c}
		\Big[ \p_{j} \eta_{{p}^{k}} (x^{k}) - (t^{k} - t^{k+1}) \frac{\p v^{k} }{\p \X^{k}_{{p}^{k}, j}}
		\Big] 
		&  -(t^{k} - t^{k+1})
		\frac{\p_{j} \eta_{{p}^{k}} }{\sqrt{g_{{p}^{k},jj}  }} \Big|_{x^{k}}
		\\
		\cdot  \frac{1}{\sqrt{g_{{p}^{k+1},ii}  }}	\Big[ \frac{\p_{i} \eta_{{p}^{k+1}} }{ \sqrt{g_{{p}^{k+1} ,ii}  }}	 +  \frac{\V^{k+1}_{{p}^{k+1}, i}  }{\V^{k+1}_{{p}^{k+1}, 3}  } 
		\frac{\p_{3} \eta_{{p}^{k+1}}  }{ \sqrt{g_{{p}^{k+1}, 33} }}
		\Big]\Big|_{x^{k+1}}
		&
		\cdot  \frac{  1}{\sqrt{g_{{p}^{k+1},ii}  }}
		\bigg[ \frac{\p_{i} \eta_{{p}^{k+1}}  }{\sqrt{g_{{p}^{k+1},ii}  }}
		+\frac{ {\V}^{k+1}_{{p}^{k+1},i} }{ {\V}^{k+1}_{{p}^{k+1},3}  }
		\frac{\p_{3} \eta_{{p}^{k+1}} }{\sqrt{g_{{p}^{k+1},33}  }}	\bigg]  \Big|_{x^{k+1}}
		\\ \hline
		\frac{\p_{i} \eta_{{p}^{k+1} }   }{\sqrt{g_{{p}^{k+1}, ii}  }} \Big|_{x^{k+1}}\cdot \frac{\p v^{k}  }{\p \X^{k}_{{p}^{k}, j}}
		& 
		\frac{\p_{i} \eta_{{p}^{k+1}} (x^{k+1})}{\sqrt{g_{{p}^{k+1}, ii} (x^{k+1})}} \cdot \frac{\p_{j} \eta_{{p}^{k} } (x^{k})}{\sqrt{g_{{p}^{k}, jj} (x^{k})}}
		\\
		-\frac{\p_{3} \eta_{{p}^{k+1} }   }{\sqrt{g_{{p}^{k+1}, 33}  }}\Big|_{x^{k+1}} \cdot \frac{\p v^{k} }{\p \X^{k}_{{p}^{k}, j}}
		& 
		-  \frac{\p_{3} \eta_{{p}^{k+1}} (x^{k+1})}{\sqrt{g_{{p}^{k+1},33} (x^{k+1})}} \cdot \frac{\p_{j} \eta_{{p}^{k} } (x^{k})}{\sqrt{g_{{p}^{k}, jj} (x^{k})}}
		\end{array} \right] }   \\
	& \tiny{+ 
		\left[ 
		\begin{array}{c|c}	
		0  & 0  \\	 \hline
		C_{N}\| 	\Phi \|_{C^{2}}	 & C_{N}\| 	\Phi \|_{C^{2}}				
		\end{array}					 \right]_{5\times 5}}
	. 
	\end{split}
	\end{equation}
	
	\noindent Then we use $3-5$ rows to remove the following parts in $1-2$ rows:
	\[ {
		\left[\begin{array}{c|c}  
		\substack{  - (t^{k} - t^{k+1}) \frac{\p v^{k} }{\p \X^{k}_{{p}^k, j}}  \\ \cdot  \frac{1}{\sqrt{g_{{p}^{k+1},ii}  }}
			\Big[ \frac{\p_{i} \eta_{{p}^{k+1}} }{ \sqrt{g_{{p}^{k+1} ,ii}  }}
			+ \frac{\V^{k+1}_{{p}^{k+1}, i}  }{\V^{k+1}_{{p}^{k+1}, 3}  } 
			\frac{\p_{3} \eta_{{p}^{k+1}}  }{ \sqrt{g_{{p}^{k+1}, 33} }}
			\Big]}
		& \substack{  -(t^{k} - t^{k+1})
			\frac{\p_{j} \eta_{{p}^k}    }{\sqrt{g_{{p}^{k},jj} }} \\
			\cdot  \frac{  1}{\sqrt{g_{{p}^{k+1},ii}  }}
			\bigg[ \frac{\p_{i} \eta_{{p}^{k+1}}  }{\sqrt{g_{{p}^{k+1},ii}  }}
			+\frac{ {\V}^{k+1}_{{p}^{k+1},i} }{ {\V}^{k+1}_{{p}^{k+1},3}  }
			\frac{\p_{3} \eta_{{p}^{k+1}} }{\sqrt{g_{{p}^{k+1},33}  }}
			\bigg]  }
		\end{array}\right].}
	\]
	{ Via this process, we obtain the following row echelon form of matrix $\mathcal{A}$.  }  
	\begin{equation}\label{echelon}
	\begin{split}
	& 	  \left[\begin{array}{c|c}
	\begin{array}{c}
	\p_{j} \eta_{{p}^k}     (x^{k})
	\frac{1}{\sqrt{g_{{p}^{k+1},ii}  }}  \\
	\cdot \Big[ \frac{\p_{i} \eta_{{p}^{k+1}} }{ \sqrt{g_{{p}^{k+1} ,ii}  }}
	+  \frac{\V^{k+1}_{{p}^{k+1}, i}  }{\V^{k+1}_{{p}^{k+1}, 3}  } 
	\frac{\p_{3} \eta_{{p}^{k+1}}  }{ \sqrt{g_{{p}^{k+1}, 33} }}
	\Big]\Big|_{x^{k+1}}
	\end{array}
	&  0 
	\\ \hline
	\frac{\p_{i} \eta_{{p}^{k+1} } }{\sqrt{g_{{p}^{k+1}, ii}  }} \Big|_{x^{k+1}} \cdot \frac{\p  v^{k} }{\p \X^{k}_{{p}^k, j}}
	& 
	\frac{\p_{i} \eta_{{p}^{k+1}} (x^{k+1})}{\sqrt{g_{{p}^{k+1}, ii} (x^{k+1})}} \cdot \frac{\p_{j} \eta_{{p}^k } (x^{k})}{\sqrt{g_{{p}^k, jj} (x^{k})}}
	\\
	-\frac{\p_{3} \eta_{{p}^{k+1} } }{\sqrt{g_{{p}^{k+1}, 33}  }} \Big|_{x^{k+1}}  \cdot \frac{\p v^{k} }{\p \X^{k}_{{p}^k, j}}
	& 
	-  \frac{\p_{3} \eta_{{p}^{k+1}} (x^{k+1})}{\sqrt{g_{{p}^{k+1},33} (x^{k+1})}} \cdot \frac{\p_{j} \eta_{{p}^k } (x^{k})}{\sqrt{g_{{p}^k, jj} (x^{k})}}
	\end{array} \right] \\
	& + \tiny{ \left[ 
	\begin{array}{c|c}
	O_{\Omega}(\| \Phi \|_{C^{2}}) |v^{k} | (t^{k} - t^{k+1}) \Big(1+ \frac{|\V^{k+1}_{p^{k+1}}|}{|\V^{k+1}_{p^{k+1},3}|}\Big)&O_{\Omega}(\| \Phi \|_{C^{2}}) |v^{k} | (t^{k} - t^{k+1}) \Big(1+ \frac{|\V^{k+1}_{p^{k+1}}|}{|\V^{k+1}_{p^{k+1},3}|}\Big)\\ \hline
	O_{\Omega, N}(\| \Phi \|_{C^{2}}) &O_{\Omega, N}(\| \Phi \|_{C^{2}}) 
	\end{array}
	\right]_{5\times 5}  }.
	\end{split}
	\end{equation}
	
	{ Note that row echelon operation preserves determinant. Therefore, we compute determinants of two matrices. Determinant of the lower right $3\times 3$ block of the first matrix in (\ref{echelon}) is given by  } 
	\begin{equation} \label{LR 33}
	\begin{split}
	-1 &= \det \left[\begin{array}{c} \frac{\p_{i} \eta_{{p}^{k+1}} (x^{k+1})}{\sqrt{g_{{p}^{k+1}, ii} (x^{k+1})}} \cdot \frac{\p_{j} \eta_{{p}^k } (x_{{p}^{k})}}{\sqrt{g_{{p}^{k}, jj} (x^{k})}} \\
	-  \frac{\p_{3} \eta_{{p}^{k+1}} (x^{k+1})}{\sqrt{g_{{p}^{k+1},33} (x^{k+1})}} \cdot \frac{\p_{j} \eta_{{p}^k } (x^{k})}{\sqrt{g_{{p}^k, jj} (x^{k})}}
	\end{array}\right]_{3\times 3}   \\   
	&= 
	\det \left[\begin{array}{c}
	\frac{\p_{1} \eta_{{p}^{k+1}} (x^{k+1})}{\sqrt{g_{{p}^{k+1}, 11} (x^{k+1})}} \\
	\frac{\p_{2} \eta_{{p}^{k+1}} (x^{k+1})}{\sqrt{g_{{p}^{k+1}, 22} (x^{k+1})}} \\
	-  \frac{\p_{3} \eta_{{p}^{k+1}} (x^{k+1})}{\sqrt{g_{{p}^{k+1},33} (x^{k+1})}} 
	\end{array}\right]_{3\times 3}   
	\det \left[\begin{array}{ccc} 
	\frac{\p_{1} \eta_{{p}^k } (x^{k})}{\sqrt{g_{{p}^k, 11} (x^{k})}} &
	\frac{\p_{2} \eta_{{p}^k } (x^{k})}{\sqrt{g_{{p}^k, 22} (x^{k})}}
	&
	\frac{\p_{3} \eta_{{p}^k } (x^{k})}{\sqrt{g_{{p}^k, 33} (x^{k})}}
	\end{array}\right]_{3\times 3}  .
	\end{split}
	\end{equation}
	
	\noindent In order to evaluate the determinant of upper left $2\times 2$ matrix, we use a basic linear algebra result: Let $A_{1}, A_{2}, B_{1}, B_{2} \in \R^{3}$. Then 
	\begin{equation}\label{identity_cross_det}
	\Big| \det  \left(\begin{array}{cc} 
	A_{1} \cdot B_{1} & A_{1} \cdot B_{2} \\
	A_{2} \cdot B_{1} & A_{2} \cdot B_{2}
	\end{array}\right)\Big| = \big|(A_{1} \times A_{2}) \cdot (B_{1} \times B_{2})   \big|.
	\end{equation}
	
	{ \noindent From (\ref{identity_cross_det}), the determinant of upper left $2\times 2$ submatrix of the first matrix in (\ref{echelon}) equals 		
	} 
	\begin{eqnarray}  
	&&\Big|
	\big( \p_{1} \eta_{{p}^{k}} \times \p_{2} \eta_{{p}^{k}}
	\big)\big|_{x^{k}} \cdot  
	\Big(
	\frac{1}{\sqrt{g_{{p}^{k+1},11}  }}
	\Big[ \frac{\p_{1} \eta_{{p}^{k+1}} }{ \sqrt{g_{{p}^{k+1} ,11}  }}\Big|_{x^{k+1}}
	+ \frac{\V^{k+1}_{{p}^{k+1}, 1}  }{\V^{k+1}_{{p}^{k+1}, 3}  } 
	\frac{\p_{3} \eta_{{p}^{k+1}}  }{ \sqrt{g_{{p}^{k+1}, 33} }}\Big|_{x^{k+1}}
	\Big] 
	\notag
	\\
	&&
	\ \ \ \ \ \ \  \ \  \ \   \ \  \ \ \ \ \ \ \  \ \  \ \ \ \ \ \ \  \ \ 
	\times 
	\frac{1}{\sqrt{g_{{p}^{k+1},22}  }}
	\Big[ \frac{\p_{2} \eta_{{p}^{k+1}} }{ \sqrt{g_{{p}^{k+1} ,22}  }}\Big|_{x^{k+1}}
	+ \frac{\V^{k+1}_{{p}^{k+1}, 2}  }{\V^{k+1}_{{p}^{k+1}, 3}  } 
	\frac{\p_{3} \eta_{{p}^{k+1}}  }{ \sqrt{g_{{p}^{k+1}, 33} }}\Big|_{x^{k+1}}
	\Big]
	\Big)
	\Big|
	\notag
	\\
	&&\quad = \frac{ \sqrt{g_{{p}^{k },11} (x^{k})} \sqrt{g_{{p}^{k },22}  (x^{k})  }}{\sqrt{g_{{p}^{k+1 },11}  (x^{k+1}) } \sqrt{g_{{p}^{k +1},22}  (x^{k+1})  }} 	\\
	&&\quad\quad  \times \Big|
	n_{{p}^{k}} (x^{k})\cdot 
	\Big(
	n_{{p}^{k+1}}-\frac{\V^{k+1}_{{p}^{k+1},2}}{ \V^{k+1}_{{p}^{k+1},3}} \frac{\p_{2} \eta_{{p}^{k+1}}  }{ \sqrt{g_{{p}^{k+1}, 22} }} 
	- \frac{\V^{k+1}_{{p}^{k+1},1}}{ \V^{k+1}_{{p}^{k+1},3}} \frac{\p_{1} \eta_{{p}^{k+1}}  }{ \sqrt{g_{{p}^{k+1}, 11} }}
	\Big)
	\Big|
	\notag
	\\
	&&\quad = \frac{ \sqrt{g_{{p}^{k },11} (x^{k}) g_{{p}^{k },22}  (x^{k})  }}{\sqrt{g_{{p}^{k+1 },11}  (x^{k+1}) g_{{p}^{k +1},22}  (x^{k+1})  }}
	\frac{|
		\V^{k}_{{p}^{k},3}
		|
		+ \| \nabla \Phi \|_{\infty} |t^{k}- t^{k+1}|
	}{| \V^{k+1}_{{p}^{k+1},3}|}. \label{UL 22}
	\end{eqnarray}
	
	\noindent Since determinant of the second matrix in (\ref{echelon}) is size of $\|\Phi\|_{C^2}$, we finish the proof from (\ref{echelon}), (\ref{LR 33}), and (\ref{UL 22}).	
\end{proof}

\begin{lemma} \label{5X5}
	We define, for all $k$,
	\begin{equation}\label{hat_v}
	|\V^{k}_{p^{k}}| = \sqrt{ ({\V}^{k}_{{p}^{k},1})^{2} + ({\V}^{k}_{{p}^{k},2})^{2} + ({\V}^{k}_{{p}^{k},3})^{2}},	
	\	\  \hat{\V}^{k}_{{p}^{k},1} = \frac{\V^{k}_{{p}^{k},1}}{|\V^{k}_{{p}^{k}}|},   \ \ 
	\hat{\V}^{k}_{{p}^{k},2} = \frac{\V^{k}_{{p}^{k},2}}{|\V^{k}_{{p}^{k}}|}, 
	\end{equation}	
	where $\V^{k}_{p^{k}}= \V^{k}_{p^{k}}(t,x,v)$ are defined in (\ref{x^k}).
	Assume (\ref{convexity_eta}), $\frac{1}{N} \leq |v| \leq N$, $\|  \Phi \|_{C^{2}_{x}} < \frac{\delta_{1}}{3 \text{diam} (\Omega) N^{2}}$ for $1 \ll N$, $0 < \delta_{1} \ll \frac{1}{N} \ll 1$, and $|\V^{1}_{p^{1},3}(t,x,v)|> \delta_{2}>0$. If $|t-t^{k}| \leq 1$, then  \begin{equation}\label{Jac_hat}
	\bigg|\det\left[\begin{array}{cc|cc} 
	\frac{\p \X^{k}_{{p}^{k},1}}{\p \X^{1}_{{p}^{1},1}} & \frac{\p \X^{k}_{{p}^{k},1}}{\p \X^{1}_{{p}^{1},2}} & \frac{\p \X^{k}_{{p}^{k},1}}{\p \hat{\V}^{1}_{{p}^{1},1}}  & \frac{\p \X^{k}_{{p}^{k},1}}{\p  \hat{\V}^{1}_{{p}^{1},2}}  \\
	\frac{\p \X^{k}_{{p}^{k},2}}{\p \X^{1}_{{p}^{1},1}} & \frac{\p \X^{k}_{{p}^{k},2}}{\p \X^{1}_{{p}^{1},2}} & \frac{\p \X^{k}_{{p}^{k},2}}{\p \hat{\V}^{1}_{{p}^{1},1}}  & \frac{\p \X^{k}_{{p}^{k},2}}{\p \hat{\V}^{1}_{{p}^{1},2}}  \\ \hline
	\frac{\p \hat{\V}^{k}_{{p}^{k},1}}{\p \X^{1}_{{p}^{1},1}} & \frac{\p  \hat{\V}^{k}_{{p}^{k},1}}{\p \X^{1}_{{p}^{1},2}} & \frac{\p   \hat{\V}^{k}_{{p}^{k},1}}{\p \hat{\V}^{1}_{{p}^{1},1}}  & \frac{\p   \hat{\V}^{k}_{{p}^{k},1}}{\p \hat{\V}^{1}_{{p}^{1},2}}  \\
	\frac{\p   \hat{\V}^{k}_{{p}^{k},2}}{\p \X^{1}_{{p}^{1},1}} & \frac{\p   \hat{\V}^{k}_{{p}^{k},2}}{\p \X^{1}_{{p}^{1},2}} & \frac{\p   \hat{\V}^{k}_{{p}^{k},2}}{\p \hat{\V}^{1}_{{p}^{1},1}}  & \frac{\p   \hat{\V}^{k}_{{p}^{k},2}}{\p \hat{\V}^{1}_{{p}^{1},2}}  
	\end{array}\right] \bigg| \  >  \ \epsilon_{ \Omega,N,  \delta_{1} ,\delta_{2}} >0,
	\end{equation}	
	where $ t^{1}  =   t^{1} (t,x,v),$ $\X^{1}_{p^{1},i}  =  \X^{1}_{p^{1},i} (t,x,v),$ $\hat{\V}^{1}_{p^{1},i} = \hat{\V}^{1}_{p^{1},i} (t,x,v)$, and 
	\begin{eqnarray*}
	\X^{k}_{{p}^{k},i} &=&   \X^{k}_{{p}^{k},i} (t^{1}, \X^{1}_{p^{1},1}, \X^{1}_{p^{1},2},  \hat{\V}^{1}_{p^{1},1},  \hat{\V}^{1}_{p^{1},2}, |\V^{1}_{p^{1}}|), \\ 
	\hat{\V}^{k}_{{p}^{k},i} &=&   \hat{\V}^{k}_{{p}^{k},i} (t^{1}, \X^{1}_{p^{1},1}, \X^{1}_{p^{1},2},  \hat{\V}^{1}_{p^{1},1},  \hat{\V}^{1}_{p^{1},2}, |\V^{1}_{p^{1}}|).
	\end{eqnarray*}
	Here, the constant $\epsilon_{ \Omega,N,   \delta_{1}, \delta_{2}}>0$ does not depend on $t$ and $x$.\end{lemma}
	\begin{proof} 
		
	\textit{Step 1.} We compute 

	\begin{equation} \label{similar form}
	\begin{split}
		J^{i+1}_{i} &:= \frac{\p ( \X^{i+1}_{{p}^{i+1},1}, \X^{i+1}_{{p}^{i+1},2}, \hat{\V}^{i+1}_{{p}^{i+1},1} , \hat{\V}^{i+1}_{{p}^{i+1},2} ,  |  {\V}^{i+1}_{{p}^{i+1} } |)}{\p (
		\X^{i }_{{p}^{i },1}, \X^{i }_{{p}^{i },2}, \hat{\V}^{i }_{{p}^{i },1} , \hat{\V}^{i }_{{p}^{i },2} ,  |  {\V}^{i }_{{p}^{i } } |
		)} 	\\
		&=
		\underbrace{ \frac{\p ( \X^{i}_{{p}^{i},1}, \X^{i}_{{p}^{i},2}, {\V}^{i}_{{p}^{i}} )}{\p (
			\X^{i }_{{p}^{i },1}, \X^{i }_{{p}^{i },2}, \hat{\V}^{i }_{{p}^{i },1} , \hat{\V}^{i }_{{p}^{i },2} ,  |  {\V}^{i }_{{p}^{i } } |
			)} 	
		}_{=Q_{i}}
		\underbrace{
		\frac{\p ( \X^{i+1}_{{p}^{i+1},1}, \X^{i+1}_{{p}^{i+1},2}, {\V}^{i+1}_{{p}^{i+1}}  )}{\p (
			\X^{i }_{{p}^{i },1}, \X^{i }_{{p}^{i },2}, {\V}^{i }_{{p}^{i }}
			)}
		}_{=P_{i}}	\\
		&\quad \times 	
		\underbrace{ \frac{\p ( \X^{i+1}_{{p}^{i+1},1}, \X^{i+1}_{{p}^{i+1},2}, \hat{\V}^{i+1}_{{p}^{i+1},1} , \hat{\V}^{i+1}_{{p}^{i+1},2} ,  |  {\V}^{i+1}_{{p}^{i+1} } |)}{\p (
			\X^{i+1 }_{{p}^{i+1 },1}, \X^{i+1 }_{{p}^{i+1 },2}, {\V}^{i+1 }_{{p}^{i+1 }}	)} 
		}_{=Q_{i+1}} .
		\\
	\end{split}
	\end{equation} 		
	For $Q_{i}$,
	\begin{equation} \label{Qi}
	\begin{split}
		Q_{i} 
		&= 
		\left[\begin{array}{cc|ccc} 
		1 & 0 &0 & 0 & 0\\
		0 &1 & 0&  0 & 0\\ \hline
		0& 0 &    { \p\V_{p^{i},1}^{i} \over \p\hat{\V}_{p^{i},1}^{i} } & { \p\V_{p^{i},1}^{i} \over \p\hat{\V}_{p^{i},2}^{i} } & { \p\V_{p^{i},1}^{i} \over \p|{\V}_{p^{i}}^{i}| }  \\
		0& 0 &  { \p\V_{p^{i},2}^{i} \over \p\hat{\V}_{p^{i},1}^{i} }&  { \p\V_{p^{i},2}^{i} \over \p\hat{\V}_{p^{i},2}^{i} } & { \p\V_{p^{i},2}^{i} \over \p |{\V}_{p^{i}}^{i}| }  \\
		0& 0 & 		 { \p\V_{p^{i},3}^{i} \over \p\hat{\V}_{p^{i},1}^{i} }
		& { \p\V_{p^{i},3}^{i} \over \p\hat{\V}_{p^{i},2}^{i} } & { \p\V_{p^{i},3}^{i} \over \p |{\V}_{p^{i}}^{i}| }
		\end{array}\right] 
		=
		\left[\begin{array}{cc|ccc} 
		1 & 0 &0 & 0 & 0\\
		0 &1 & 0&  0 & 0\\ \hline
		0& 0 &  |\V_{p^{i}}^{i} |  & 0 & { \p\V_{p^{i},1}^{i} \over \p|{\V}_{p^{i}}^{i}| }  \\
		0& 0 &  0 &  |\V_{p^{i}}^{i} | & { \p\V_{p^{i},2}^{i} \over \p |{\V}_{p^{i}}^{i}| }  \\
		0& 0 & 		 { \p\V_{p^{i},3}^{i} \over \p\hat{\V}_{p^{i},1}^{i} }
		& { \p\V_{p^{i},3}^{i} \over \p\hat{\V}_{p^{i},2}^{i} } & { \p\V_{p^{i},3}^{i} \over \p |{\V}_{p^{i}}^{i}| }
		\end{array}\right]. 	\\
	\end{split}
	\end{equation}
	For $Q_{i+1}$,
	\begin{equation} \label{Qi+1}
	\begin{split}
	Q_{i+1} 
	&= 
	\left[\begin{array}{cc|ccc} 
	1 & 0 &0 & 0 & 0\\
	0 &1 & 0&  0 & 0\\ \hline
	0& 0 &    { \p\hat{\V}_{p^{i+1},1}^{i+1} \over \p{\V}_{p^{i+1},1}^{i+1} } & { \p\hat{\V}_{p^{i+1},1}^{i+1} \over \p{\V}_{p^{i+1},2}^{i+1} } & { \p\hat{\V}_{p^{i+1},1}^{i+1} \over \p{\V}_{p^{i+1},3}^{i+1} }  \\
	0& 0 &  { \p\hat{\V}_{p^{i+1},2}^{i+1} \over \p{\V}_{p^{i+1},1}^{i+1} }&  { \p\hat{\V}_{p^{i+1},2}^{i+1} \over \p{\V}_{p^{i+1},2}^{i+1} } & { \p\hat{\V}_{p^{i+1},2}^{i+1} \over \p {\V}_{p^{i+1},3}^{i+1} }  \\
	0& 0 & 		 { \p|\V_{p^{i+1}}^{i+1}| \over \p{\V}_{p^{i+1},1}^{i+1} }
	& { \p|\V_{p^{i+1}}^{i+1}| \over \p{\V}_{p^{i+1},2}^{i+1} } & { \p|\V_{p^{i+1}}^{i+1}| \over \p {\V}_{p^{i+1},3}^{i+1} }
	\end{array}\right] 	\\
	&=
	\left[\begin{array}{cc|ccc} 
	1 & 0 &0 & 0 & 0\\
	0 &1 & 0&  0 & 0\\ \hline
	0& 0 &  |\V_{p^{i+1}}^{i+1} |^{-1}  & 0 & { \p\hat{\V}_{p^{i+1},1}^{i+1} \over \p{\V}_{p^{i+1},3}^{i+1} }  \\
	0& 0 &  0 &  |\V_{p^{i+1}}^{i+1} |^{-1} & { \p\hat{\V}_{p^{i+1},2}^{i+1} \over \p {\V}_{p^{i+1},3}^{i+1} }  \\
	0& 0 & 		 { \p|\V_{p^{i+1}}^{i+1}| \over \p{\V}_{p^{i+1},1}^{i+1} }
	& { \p|\V_{p^{i+1}}^{i+1}| \over \p{\V}_{p^{i+1},2}^{i+1} } & { \p|\V_{p^{i+1}}^{i+1}| \over \p {\V}_{p^{i+1},3}^{i+1} }
	\end{array}\right]. 	\\
	\end{split}
	\end{equation}
	Note that for $\ell=1,2$,
	\begin{equation} \label{1323}
	\begin{split}
		\frac{ \p \hat{\V}_{p^{i+1},{\ell}}^{i+1} }{ \p \V_{p^{i+1},{3}}^{i+1} } 
		&=
		\V_{p^{i+1},{\ell}}^{i+1} \frac{\p}{\p \V_{p^{i+1},{3}}^{i+1}} \big({1\over |\V_{p^{i+1}}^{i+1}|} \big) = \frac{ \V_{p^{i+1},{\ell}}^{i+1} \V_{p^{i+1},{3}}^{i+1} }{ |\V_{p^{i+1}}^{i+1}|^{3} } ,
	\end{split}
	\end{equation}
	and for $k=1,2,3$,
	\begin{equation} \label{313233}
	\begin{split}
	\frac{ \p |{\V}_{p^{i+1}}^{i+1}| }{ \p \V_{p^{i+1},{k}}^{i+1} } 
	&=
	- \frac{ \V_{p^{i+1},{k}}^{i+1} }{ |\V_{p^{i+1}}^{i+1}| }.
	\end{split}
	\end{equation}
	From (\ref{Qi+1}), (\ref{1323}), and (\ref{313233}),
	\begin{equation} \label{det Qi+1}
	\begin{split}
		\det Q_{i+1} &= \frac{1}{|\V_{p^{i+1}}^{i+1}|} \Big( -\frac{\V_{p^{i+1},{3}}^{i+1}}{ |\V_{p^{i+1}}^{i+1}|^{2} } + \frac{ (\V_{p^{i+1},{2}}^{i+1})^{2}\V_{p^{i+1},{3}}^{i+1} }{ |\V_{p^{i+1}}^{i+1}|^{4} } \Big) + \frac{ \V_{p^{i+1},{1}}^{i+1} \V_{p^{i+1},{3}}^{i+1} }{ |\V_{p^{i+1}}^{i+1}|^{3} } \frac{ \V_{p^{i+1},{1}}^{i+1} }{ |\V_{p^{i+1}}^{i+1}|^{2} }  \\
		&= - \frac{ (\V_{p^{i+1},{3}}^{i+1})^{3} }{ |\V_{p^{i+1}}^{i+1}|^{5} } .
	\end{split}
	\end{equation}
	By taking inverse, we get
	\begin{equation} \label{det Qi}
	\begin{split}
	\det Q_{i} &= - \frac{ |\V_{p^{i}}^{i}|^{5} } { (\V_{p^{i},{3}}^{i})^{3} }.
	\end{split}
	\end{equation}
	From (\ref{similar form}), (\ref{det Qi}), (\ref{det Qi+1}), and Lemma \ref{det_billiard}, we get
	\begin{equation}\notag 
	\begin{split}
	& \bigg| \det \left[\begin{array}{cc} \nabla_{\X^{k}_{{p}^{k}}} \X^{k+1}_{{p}^{k+1}} & \nabla_{ {\V}^{k}_{{p}^{k}}} \X^{k+1}_{{p}^{k+1}}\\
	\nabla_{\X^{k}_{{p}^{k}}}  {\V}^{k+1}_{{p}^{k+1}}
	& \nabla_{ {\V}^{k}_{{p}^{k}} }  {\V}^{k+1}_{{p}^{k+1}}
	\end{array}\right]_{5\times 5} \bigg| 	 \\ 
	&=\Big(1+ O_{\Omega,N}(  \|   \Phi \|_{C^{2}} )
	\Big) \frac{ \sqrt{g_{{p}^{k },11} (x^{k}) }  \sqrt{g_{{p}^{k },22}  (x^{k})  }}{\sqrt{g_{{p}^{k+1 },11}  (x^{k+1})  }  \sqrt{g_{{p}^{k +1},22}  (x^{k+1})  }}	
	\frac{
		|\V^{k }_{p^{k },3}|
	}{|\V^{k+1}_{p^{k+1},3}|},
	\end{split}
	\end{equation}
	\begin{equation} \label{det Jii+1}
	\begin{split}
		| \det J_{i}^{i+1} | &= | \det Q_{i} \det P_{i} \det Q_{i+1} |  \\
		&=  \frac{ |\V_{p^{i}}^{i}|^{5} } { (\V_{p^{i},{3}}^{i})^{3} } 
		\Big(1+ O_{\Omega,N}(  \|   \Phi \|_{C^{2}} )
		\Big) \frac{ \sqrt{g_{{p}^{i },11} }  \sqrt{g_{{p}^{i },22}  } \big|_{x^{i}} }{\sqrt{g_{{p}^{i+1 },11}   }  \sqrt{g_{{p}^{i +1},22}  } \big|_{x^{i+1}}  }	
		\frac{
			|\V^{i }_{p^{i },3}|
		}{|\V^{i+1}_{p^{i+1},3}|}
		\frac{ (\V_{p^{i+1},{3}}^{i+1})^{3} }{ |\V_{p^{i+1}}^{i+1}|^{5} }  \\
		&= 
		\Big(1+ O_{\Omega,N}(  \|   \Phi \|_{C^{2}} )
		\Big) \frac{ \sqrt{g_{{p}^{i },11} }  \sqrt{g_{{p}^{i },22}  } \big|_{x^{i}} }{\sqrt{g_{{p}^{i+1 },11}   }  \sqrt{g_{{p}^{i +1},22}  } \big|_{x^{i+1}}  }	
		\frac{ |\V_{p^{i+1},{3}}^{i+1}|^{2} }{ |\V_{p^{i},{3}}^{i}|^{2} } 
		\\
		&\quad + O_{\Omega,N}(  \|   \Phi \|_{C^{2}} ).
	\end{split}
	\end{equation}
	Therefore,
	\begin{equation} \label{det J1k}
	\begin{split}
	| \det J_{1}^{k} | 
	&= 
	\Big(1+ O_{\Omega,N}(  \|   \Phi \|_{C^{2}} )
	\Big) \frac{ \sqrt{g_{{p}^{1 },11} }  \sqrt{g_{{p}^{1 },22}  } \big|_{x^{1}} }{\sqrt{g_{{p}^{k },11}   }  \sqrt{g_{{p}^{k},22}  } \big|_{x^{k}}  }	
	\frac{ |\V_{p^{k},{3}}^{k}|^{2} }{ |\V_{p^{1},{3}}^{1}|^{2} } 
	\\
	&\quad + O_{\Omega,N}(  \|   \Phi \|_{C^{2}} ).
	\end{split}
	\end{equation}
	
	\noindent \textit{Step 2.} From (\ref{dX/dv}),
	\begin{equation} \label{hatv part1}
	\begin{split}
	& 2| \V^{i+1}_{p^{i+1}}  |\frac{\p | \V^{i+1}_{p^{i+1}}  |}{\p  \V^{i}_{p^{i}, n}}	\\	
	&=\frac{ \p |V(t^{i+1};t^{i}, x^{i}, v^{i}) |^{2}}{\p \V^{i}_{p^{i}, n}}
	= 2 \frac{\p V(t^{i+1};t^{i}, x^{i}, v^{i})  }{\p \V^{i}_{p^{i}, n}} \cdot  V(t^{i+1};t^{i}, x^{i}, v^{i})
	\\
	&= 
	2 \Big(
	\frac{\p_{n} \eta_{p^{i}}}{\sqrt{g_{p^{i}, nn}}}\Big|_{x^{i}}
	+ \frac{\p (t^{i} - t^{i+1})}{\p \V^{i}_{p^{i}, n}} \nabla_{x} \Phi (t^{i+1}, x^{i+1})  	\\
	&\quad + O_{\Omega}(\nabla_{x}^{2} \Phi \|_{\infty}) |t^{i} - t^{i+1}|^{2} e^{\| \nabla_{x}^{2} \Phi \|_{\infty} (t^{i } - t^{i+1})^{2}}
	\Big) 
	\quad \cdot 
	V(t^{i+1};t^{i}, x^{i}, v^{i})
	\\
	&= 2 \V^{i}_{p^{i},n} + 2 \Big| \frac{\p (t^{i} - t^{i+1})}{\p \V^{i}_{p^{i},n}}\Big| \| \nabla_{x} \Phi \|_{\infty} |\V^{i+1}_{p^{i+1}}|  	
	+ O_{\Omega} (\| \nabla_{x} \Phi \|_{\infty}) (t^{i} - t^{i+1})   \\
	&\quad +O_{\Omega}( \| \nabla_{x}^{2} \Phi \|_{\infty} )|v^{i}|(t^{i} - t^{i+1})^{2}
	e^{\| \nabla_{x}^{2} \Phi \|_{\infty} (t^{i} - t^{i+1})^{2}}.
	\end{split}
	\end{equation}
	Then by Lemma \ref{Jac_billiard} and $\Big|\frac{\p (t^{i} - t^{i+1})}{\p \V^{i}_{p^{i},n}}\Big|\lesssim_{\Omega, N} 1$, we get
	\begin{equation}\label{|v|_v}
	\frac{\p | \V^{i+1}_{p^{i+1}}  |}{\p  \V^{i}_{p^{i}, n}} = \frac{\V^{i}_{p^{i},n}}{|\V^{i+1}_{p^{i+1}}|}+ O_{\Omega,N} (\| \Phi \|_{C^{2}}   ), \ \ \ \text{for} \ \ n=1,2.
	\end{equation}
	
	From (\ref{dX/dv}), for $n=1,2$,
	\begin{eqnarray*}
		2| \V^{i+1}_{p^{i+1}}  |  \frac{\p | \V^{i+1}_{p^{i+1}}  |}{\p  \X^{i}_{p^{i}, n}}	 	&=&\frac{ \p |V(t^{i+1};t^{i}, x^{i}, v^{i}) |^{2}}{\p \X^{i}_{p^{i}, n}}
		= 2 \frac{\p V(t^{i+1};t^{i}, x^{i}, v^{i})  }{\p \X^{i}_{p^{i}, n}} \cdot  V(t^{i+1};t^{i}, x^{i}, v^{i})
		\\
		&= &
		2 \Big(
		\frac{\p (t^{i} - t^{i+1})}{\p \X^{i}_{p^{i}, n}} \nabla_{x} \Phi (t^{i+1}, x^{i+1})  \\
		&& + O_{\Omega}( \|\nabla_{x}^{2} \Phi \|_{\infty}) |t^{i} - t^{i+1}| e^{\| \nabla_{x}^{2} \Phi \|_{\infty} (t^{i } - t^{i+1})^{2}}
		\Big) 
		\cdot 
		V(t^{i+1};t^{i}, x^{i}, v^{i})
		\\ 
		&\leq& O_{N,\Omega } (\| \Phi \|_{C^{2}} ),
	\end{eqnarray*}	
	where we have used $\Big| \frac{\p (t^{i} - t^{i+1})}{\p \X^{i}_{p^{i}, n}}   \Big|\lesssim_{N, \Omega}  1$ for $n=1,2$ from Lemma \ref{Jac_billiard}.	 This proves 
	\begin{equation}\label{|v|_x}
	\frac{\p | \V^{i+1}_{p^{i+1}}  |}{\p  \X^{i}_{p^{i}, n}}	 
	= O_{N,\Omega} (\| \Phi  \|_{C^{2}})   , \ \ \ \text{for} \ \ n=1,2.
	\end{equation}
	
	Meanwhile,
	\begin{equation}\label{|v|_|v|}
	\begin{split}
	\frac{  \p |\V^{i+1}_{p^{i+1}}|}{\p |\V^{i}_{p^{i}}|  }  =&
	\sum_{\ell=1}^{2} \hat{\V}_{p^{i}, \ell}^{i} \frac{\p |\V^{i+1}_{p^{i+1}}|}{\p \V^{i}_{p^{i}, \ell}}  + \sqrt{1- (\hat{\V}^{i}_{p^{i},1})^{2} - (\hat{\V}^{i}_{p^{i},2})^{2}} \frac{\V^{i}_{p^{i},3}}{|\V^{i+1}_{p^{i+1}}|} + O_{\Omega, N} (\|\Phi \|_{C^{2}})\\
	=&\sum_{\ell=1}^{2} \hat{\V}_{p^{i}, \ell}^{i}
	\frac{ \V^{i}_{p^{i}, \ell}}{| \V^{i}_{p^{i} }|} + \sqrt{1- (\hat{\V}^{i}_{p^{i},1})^{2} - (\hat{\V}^{i}_{p^{i},2})^{2}} \frac{\V^{i}_{p^{i},3}}{|\V^{i }_{p^{i }}|} + O_{\Omega, N} (\| \Phi \|_{C^{2}}) \\
	= &1+ O_{\Omega, N} (\|\Phi \|_{C^{2}}) ,
	\end{split}\end{equation}
	
	\noindent \textit{Step 3.} From (\ref{det J1k}), (\ref{|v|_v}), (\ref{|v|_x}), and (\ref{|v|_|v|}),
	\begin{equation} 
	\begin{split}
		| \det J_{1}^{k} | 
		&= 
		\Big(1+ O_{\Omega,N}(  \|   \Phi \|_{C^{2}} )
		\Big) \frac{ \sqrt{g_{{p}^{1 },11} }  \sqrt{g_{{p}^{1 },22}  } \big|_{x^{1}} }{\sqrt{g_{{p}^{k },11}   }  \sqrt{g_{{p}^{k},22}  } \big|_{x^{k}}  }	
		\frac{ |\V_{p^{k},{3}}^{k}|^{2} }{ |\V_{p^{1},{3}}^{1}|^{2} } \\
		&\quad + O_{\Omega,N}(  \|   \Phi \|_{C^{2}} ) \\
		&=
		\bigg|\det\left[\begin{array}{c|c|c} 
		\frac{\p (\X^{k}_{{p}^{k},1}, \X^{k}_{{p}^{k},2}) }{\p (\X^{1}_{{p}^{1},1}, \X^{1}_{{p}^{1},2}) } & \frac{\p (\X^{k}_{{p}^{k},1}, \X^{k}_{{p}^{k},2}) }{\p (\hat{\V}^{1}_{{p}^{1},1}, \hat{\V}^{1}_{{p}^{1},2}) } & (*)  \\ \hline
		\frac{\p (\hat{\V}^{k}_{{p}^{k},1}, \hat{\V}^{k}_{{p}^{k},2}) }{\p (\X^{1}_{{p}^{1},1}, \X^{1}_{{p}^{1},2}) } & \frac{\p (\hat{\V}^{k}_{{p}^{k},1}, \hat{\V}^{k}_{{p}^{k},2}) }{\p (\hat{\V}^{1}_{{p}^{1},1}, \hat{\V}^{1}_{{p}^{1},2}) } & (*)  \\ \hline
		(*)  & (*) & 1 + (*)
		\end{array}\right]_{5\times 5} \bigg| ,\quad (*):= O_{\Omega,N}(  \|   \Phi \|_{C^{2}} ) \\
		&= 
		\bigg|\det\left[\frac{\p (\X^{k}_{{p}^{k},1}, \X^{k}_{{p}^{k},2}, \hat{\V}^{k}_{{p}^{k},1}, \hat{\V}^{k}_{{p}^{k},2}) }{\p (\X^{1}_{{p}^{1},1}, \X^{1}_{{p}^{1},2}, \hat{\V}^{1}_{{p}^{1},1}, \hat{\V}^{1}_{{p}^{1},2}) }\right]_{4\times 4}  \bigg| + O_{\Omega,N}(  \|   \Phi \|_{C^{2}} ) .
	\end{split}
	\end{equation}
	Note that from (\ref{upper_k}), $k\lesssim_{\O, N, \delta_{1,2}}1$ and $|\mathbf{v}^{k}_{p^{k},3}| \lesssim_{\O, N, \delta_{1,2}} |\mathbf{v}^{1}_{p^{1},3}|$. Therefore, we conclude (\ref{Jac_hat}).

	\end{proof}

				\section{Transversality via the geometric decomposition and triple iterations}

				\begin{lemma}\label{decom_lemma} Assume $Y: (y_{1}, y_{2}) \mapsto Y(y_{1}, y_{2})  \in \R^{3}$ is a $C^{1}$-map locally. For any $t,s \geq 0$ with $s \in [t-1,t]$, $|n(x^{1}(t,Y(y_{1},y_{2}),v)) \cdot v^{1}(t,Y(y_{1},y_{2}),v)|> \delta$, $\frac{1}{N}\leq |v| \leq N$, $\frac{1}{N} \leq |v_{3}|$, $t^{k+1}(t,Y(y_{1},y_{2}),v)<s<t^{k}(t,Y(y_{1},y_{2}),v)$, and
					$\| \nabla \Phi \|_\infty < \frac{\delta}{ 3 \text{diam} (\Omega) N^2}$,
					we have 
					\begin{equation} \label{X_|v|}
					\p_{|v| } \big[ X_{i} (s;t,  Y (y_{1},y_{2}),v)  \big] =  
					-   ( t - s)  \sum_{\ell=1}^{3} \frac{\p_{\ell} \eta_{{p}^{k},i} }{\sqrt{g_{{p}^{k},\ell\ell}}} ( \mathbf{x}_{{p}^{k}, 1}^{k}, \mathbf{x}_{{p}^{k},2}^{k}, 0) \frac{\mathbf{v}_{{p}^{k},\ell}^{k} 
					}{| \mathbf{v}_{{p}^{k}}^{k} |} 
					+ O_{\delta, N}(\| \Phi \|_{C^{2}}), 
					\end{equation} 
					and for $\p \in \{\p_{\hat{v}_{1}}, \p_{\hat{v}_{2}}, \p_{y_{1}}, \p_{y_{2}}\}$,
					\begin{equation}\label{p_X}
					\begin{split} 
					& \p \big[ X_{i} (s;t, Y(y_{1},y_{2}),v)  \big] \\
					=&-  \p  t^{k} |\mathbf{v}_{{p}^{k}}^{k} |  \sum_{\ell=1}^{3} \frac{\p_{\ell} \eta_{{p}^{k},i} }{\sqrt{g_{{p}^{k},\ell\ell}}} (\mathbf{x}_{{p}^{k}, 1}^{k}, \mathbf{x}_{{p}^{k},2}^{k}, 0) 
					\frac{ \mathbf{v}_{{p}^{k}\ell}^{k}	}{| \mathbf{v}_{{p}^{k}}^{k} |} 
					+ \sum_{\ell=1}^{2} \p \mathbf{x}^{k}_{{p}^{k},\ell}  \p_{\ell} \eta_{{p}^{k},i} (\mathbf{x}^{k}_{{p}^{k},1 },\mathbf{x}^{k}_{{p}^{k},2},0)  \\
					& 
					- (t^{k}-s)| \mathbf{v}_{{p}^{k}}^{k} | 
					\sum_{j=1}^{2}
					\bigg( \sum_{\ell=1}^3  \frac{\p}{\p {\mathbf{x}^{k}_{{p}^{k},j}}}
					\Big[  \frac{\p_\ell \eta_{{p}^{k},i}}{\sqrt{g_{{p}^{k},\ell\ell}}} \Big](\mathbf{x}^k_{{p}^{k},1}, \mathbf{x}^k_{{p}^{k},2},0) \hat{\mathbf{v}}_{{p}^{k},\ell}^k \bigg)
					{\p \mathbf{x}^k_{{p}^{k},j}}
					\\
					&
					- (t^{k}-s)| \mathbf{v}_{{p}^{k}}^{k} | 
					\sum_{j=1}^{2}\Big[ \frac{\p_j \eta_{{p}^{k} ,i}}{\sqrt{g_{{p}^{k},11}}}(\mathbf{x}^k_{{p}^{k},1}, \mathbf{x}^k_{{p}^{k},2}, 0)  -  \frac{\p_3 \eta_{{p}^{k},i}}{\sqrt{g_{{p}^{k},33}}}(\mathbf{x}^k_{{p}^{k},1}, \mathbf{x}^k_{{p}^{k},2}, 0)\frac{\hat{\mathbf{v}}_{{p}^{k},j}^{k} }{ \hat{\mathbf{v}}_{{p}^{k},3}^{k} } \Big] {\p  \hat{\mathbf{v}}_{{p}^{k},j}^{k}  }
					\\
					&
					+ O_{\delta, N}(\| \Phi \|_{C^{2}}).
					\end{split}
					\end{equation}
					Here $t^{k}= t^{k} (t, Y(y_{1},y_{2}), v), \mathbf{x}^{k}_{p^{k}}= \mathbf{x}^{k}_{p^{k}} (t, Y(y_{1},y_{2}), v),$ and $ \mathbf{v}^{k}_{p^{k}}= \mathbf{v}^{k}_{p^{k}}(t, Y(y_{1},y_{2}), v)$.

				\end{lemma}
				\begin{proof}
					\noindent\textit{Step 1. } We claim that 
					\begin{equation}\label{p|v|=0}
					\begin{split} 			
					& \frac{ \p \big( (t^{j} - t^{j+1}) |v^{k} | \big) }{\p |v|} = O_{\Omega, N} (\| \Phi \|_{C^{2}}), \ \ \ 
					\frac{ \p   |v^{k} |  }{\p |v|} =1  + O_{\Omega, N} (\| \Phi \|_{C^{2}}),  \\
					& \frac{\p \X^{j}_{p^{j}, i}}{\p |v|}=O_{\Omega, N} (\| \Phi \|_{C^{2}}), \ \ \ 
					\frac{\p \hat{\V}^{j}_{p^{j}, i}}{\p |v|}=O_{\Omega, N} (\| \Phi \|_{C^{2}})
					.	
					\end{split}
					\end{equation}

					By the chain rule, 
					\begin{equation}\label{matrix_|v|}
					\begin{split}
					\left[\begin{array}{ccccc}
					\nabla_{y_{1},y_{2}, \hat{v}_{1}, \hat{v}_{2}, |v|}t^{k}
					\\
					\nabla_{y_{1},y_{2}, \hat{v}_{1}, \hat{v}_{2}, |v|}   \X^{k}_{p^{k}} \\	
					\nabla_{y_{1},y_{2}, \hat{v}_{1}, \hat{v}_{2}, |v|}   \hat{\V}^{k}_{p^{k}} \\
					\nabla_{y_{1},y_{2}, \hat{v}_{1}, \hat{v}_{2}, |v|} |\V^{k}_{p^{k}}| 		\end{array} \right] 
					&= 	
					\left[
					\begin{array}{c}
					\nabla_{ t^{k-1}, \X^{k-1}_{p^{k-1}} , \hat{\V}^{k-1} _{p^{k-1}} ,  |{\V} ^{k-1}_{p^{k-1}}|
					}t^{k}
					\\
					\nabla_{ t^{k-1}, \X^{k-1}_{p^{k-1}} , \hat{\V}^{k-1} _{p^{k-1}} ,  |{\V}^{k-1} _{p^{k-1}}|
					} \X^{k}_{p^{k}}
					\\
					\nabla_{ t^{k-1}, \X^{k-1}_{p^{k-1}} , \hat{\V}^{k-1} _{p^{k-1}} ,  |{\V}^{k-1} _{p^{k-1}}|
					}
					\hat{\V}^{k}_{p^{k}}
					\\
					\nabla_{ t^{k-1}, \X^{k-1}_{p^{k-1}} , \hat{\V}^{k-1} _{p^{k-1}} ,  |{\V}^{k-1} _{p^{k-1}}|
					}  |{\V}^{k}_{p^{k}}|
					\end{array}
					\right]\times\cdots 	\\
					\cdots \times &
					\left[
					\begin{array}{c}
					\nabla_{ t^{1}, \X^{1}_{p^{1}} , \hat{\V}^{1} _{p^{1}} ,  |{\V} ^{1}_{p^{1}}|
					}t^{2}
					\\
					\nabla_{ t^{1}, \X^{1}_{p^{1}} , \hat{\V}^{1} _{p^{1}} ,  |{\V} ^{1}_{p^{1}}|
					} \X^{2}_{p^{2}}
					\\
					\nabla_{ t^{1}, \X^{1}_{p^{1}} , \hat{\V}^{1} _{p^{1}} ,  |{\V} ^{1}_{p^{1}}|
					}
					\hat{\V}^{2}_{p^{2}}
					\\
					\nabla_{ t^{1}, \X^{1}_{p^{1}} , \hat{\V}^{1} _{p^{1}} ,  |{\V} ^{1}_{p^{1}}|
					}  |{\V}^{2}_{p^{2}}|
					\end{array}
					\right]  
					\left[\begin{array}{c}
					\nabla_{y_{1},y_{2}, \hat{v}_{1}, \hat{v}_{2}, |v|}t^{1}
					\\
					\nabla_{y_{1},y_{2}, \hat{v}_{1}, \hat{v}_{2}, |v|}   \X^{1}_{p^{1}} \\	
					\nabla_{y_{1},y_{2}, \hat{v}_{1}, \hat{v}_{2}, |v|}   \hat{\V}^{1}_{p^{1}} 	\\	
					\nabla_{y_{1},y_{2}, \hat{v}_{1}, \hat{v}_{2}, |v|}   |\V^{1}_{p^{1}}| \end{array} \right]
					.
					\end{split}
					\end{equation}
					
					We claim that 	
					\begin{equation}\label{est_j_{j+1}}
					\begin{split}
					\left[
					\begin{array}{c}
					\nabla_{ t^{j}, \X^{j}_{p^{j}} , \hat{\V}^{j} _{p^{j}} ,  |{\V} ^{j}_{p^{j}}|
					}t^{j+1}
					\\
					\nabla_{ t^{j}, \X^{j}_{p^{j}} , \hat{\V}^{j} _{p^{j}} ,  |{\V}^{j} _{p^{j}}|
					} \X^{j+1}_{p^{j+1}}
					\\
					\nabla_{ t^{j}, \X^{j}_{p^{j}} , \hat{\V}^{j} _{p^{j}} ,  |{\V}^{j} _{p^{j}}|
					}
					\hat{\V}^{j+1}_{p^{j+1}}
					\\
					\nabla_{ t^{j}, \X^{j}_{p^{j}} , \hat{\V}^{j} _{p^{j}} ,  |{\V}^{j} _{p^{j}}|
					}  |{\V}^{j+1}_{p^{j+1}}|
					\end{array}
					\right]
					&= 
					\left[
					\begin{array}{c|cccc|c}
					1 & 0 & 0 & 0 &  0 & - \frac{t^{j+1}}{|\V^{j}_{p^{j}}|}\\ \hline
					0 &  &    &   &   &    	0	\\
					0 &   &    &     &   &   	0	\\
					0 &   & O_{\O, N,\delta}(1)  &   &    &   	0	\\
					0 &  &   &    & &   	0	\\
					0 &  &    &    &  & 	1	\\
					\end{array}
					\right] 
					+ O_{\Omega, N} (\| \Phi \|_{C^{2}}).
					\end{split}
					\end{equation}	
					Once (\ref{est_j_{j+1}}) is proven, from the chain rule (\ref{matrix_|v|}) and Lemma \ref{global to local}, we conclude (\ref{p|v|=0}).

					From (\ref{E_Ham}), 
					\begin{eqnarray*}
						v^{j}(t^{j}-t^{j+1}) &=&  \eta_{p^{j+1}}(\X^{j+1}_{p^{j+1},1}, \X^{j+1}_{p^{j+1},2},0)  -\eta_{p^{j }}(\X^{j}_{p^{j},1}, \X^{j}_{p^{j},2},0) \\ 
						&&- \int^{t^{j+1}}_{t^{j}} \int^{s}_{t^{j}}
						\nabla_{x} \Phi (\tau, X(\tau;t^{j},x^{j},v^{j}))
						\dd \tau \dd s.
					\end{eqnarray*} 
					Taking $\frac{\p}{\p {t^{j}}}$ directly to the above equality, we derive 
					\begin{eqnarray*}
						v^{j} \Big(1- \frac{\p t^{j+1}}{\p t^{j}}\Big) = 
						- \frac{\p t^{j+1}}{\p t^{j}} \int^{t^{j+1}}_{t^{j}} \nabla_{x} \Phi (\tau, X(\tau;t^{j},x^{j},v^{j}))
						\dd \tau
						+  ({t^{j+1}} -{t^{j}}) \nabla_{x} \Phi ( t^{j}, x^{j}),
					\end{eqnarray*}
					and $\frac{\p t^{j+1}}{\p t^{j} } = 1+ \| \nabla_{x} \Phi \|_{\infty} {|t^{j}-t^{j+1}|}/{|\V^{j+1}_{p^{j+1} }|}.$ Now from Lemma \ref{uniform number of bounce}, 
					$$|\mathbf{v}^{j+1}_{p^{j+1}}|= |v| + O(\| \nabla \Phi \|_\infty) |t-t^{j+1}|\geq \frac{1}{N}+   \frac{\delta \times 3N \text{diam}(\Omega) }{3 \text{diam}(\Omega) N^2}   \gtrsim \frac{1}{N}.$$ Therefore we conclude that 
					\begin{equation}\label{t_tt}
					\frac{\p t^{j+1}}{\p t^{j} } = 1+ O_{\Omega, N} (\| \Phi \|_{C^2}).\end{equation}

					From (\ref{position identity}), we derive 
					\begin{equation}\label{x_t}
					\begin{split}
					\frac{\p \X^{j+1} _{p^{j+1}, i}}{\p t^{j}} &=  
					\Big(\frac{\p t^{j+1}}{\p t^{j}} -1 \Big) \V^{j+1}_{p^{j+1}, i}
					+ \| \Phi \|_{C^{2}} |t^{j } -t^{j+1}|^{2} \sup_{t^{j+1} \leq \tau \leq t^{j}}\Big|
					\frac{\p X(\tau; t^{j}, x^{j}, v^{j})}{\p t^{j}}
					\Big|  \\
					&=    O_{\Omega, N } (\| \Phi \|_{C^{2}}),\end{split}
					\end{equation}
					where we have used the fact $\sup_{t^{j+1} \leq \tau \leq t^{j}}\Big|
					\frac{\p X(\tau; t^{j}, x^{j}, v^{j})}{\p t^{j}}
					\Big|\lesssim |v^{j}| + \| \nabla_{x} \Phi \|_{\infty} |t^{j} - t^{j+1}|\leq C_{N,\Omega}$, which is proved by the similar proof of (\ref{dX/dx}).
					
					\label{t_t}
					
					From (\ref{E_Ham}), we have $$|v^{j+1}|^{2} = |v^{j}|^{2} - 2 \int^{t^{j+1}}_{t^{j}} v^{j} \cdot \nabla_{x} \Phi (\tau, X(\tau; t^{j}, x^{j}, v^{j})) \dd \tau 
					+ \big| \int^{t^{j+1}}_{t^{j}}  \nabla_{x} \Phi (\tau )\big|^{2}.$$ Then 
					\begin{equation} \notag
					\begin{split}
					2|v^{j+1}| \frac{\p |v^{j+1}|}{\p t^{j}}  &=  
					-2 \frac{\p t^{j+1}}{\p t^{j}} v^{j} \cdot \nabla_{x} \Phi (t^{j+1})
					+ 2 v^{j} \cdot \nabla_{x} \Phi (t^{j})  \\
					&+ 2\Big\{
					\| \nabla_{x}^{2 } \Phi \|_{\infty} |v^{j}| (t^{j} -t^{j+1})+ 
					\sup_{t^{j+1} \leq \tau \leq t^{j}} \Big|\frac{\p X(\tau)}{\p t^{j}}\Big| 
					\Big\} \sup_{t^{j+1} \leq \tau \leq t^{j}} \Big|\frac{\p X(\tau)}{\p t^{j}}\Big| ,
					\end{split}
					\end{equation}
					and hence
					\begin{equation}\label{v_{t}}
					\frac{\p |v^{j+1}|}{\p t^{j}} = O_{\Omega, N } (\| \Phi \|_{C^{2}}).
					\end{equation}
					
					From (\ref{v_123}), we prove 
					\begin{equation}\begin{split}\label{hatv_{t}}
					\frac{\p \hat{\V}^{j+1}_{p^{j+1}, i} }{\p t^{j}}  =& O_{\| \eta\|_{C^{2}}} \Big(\Big| \frac{\p \X^{j}_{p^{j}}}{\p t^{j}}\Big|\Big) \{|v^{j} | + \| \nabla \Phi \|_{\infty} (t^{j}- t^{j+1})\} + O_{\| \eta \|_{C^{1}}} (\| \nabla_{x} \Phi \|_{\infty}) \\
					=& O_{N,\Omega} (\| \Phi \|_{C^{2}}).\end{split}
					\end{equation}
					
					We already have estimates for $\frac{\p t^{j+1}}{\p \X^{j}_{p^{j}, i}}$ in Lemma \ref{Jac_billiard}. 
					
					From Lemma \ref{Jac_billiard},
					\begin{eqnarray*}
						\frac{
							\p t^{j+1}}{\p  |{\V}^{j}_{p^{j}}|
						}
						=  - \frac{t^{j+1}}{|\V^{j}_{p^{j}}|} + O_{\Omega, N} (\| \Phi \|_{C^{2}})
						\quad\text{and}\quad
						\frac{\p t^{j+1}}{\p \hat{\V}^{j}_{p^{j}, i}}
						=   O_{\Omega, N} (\| \Phi \|_{C^{2}}).
					\end{eqnarray*} 
					
					Moreover, from the conditions $|n(x^{1}(t,Y(y_{1},y_{2}),v)) \cdot v^{1}(t,Y(y_{1},y_{2}),v)|> \delta$, $\frac{1}{N}\leq |v| \leq N$ and Lemma \ref{velocity_lemma}, (\ref{velocity_le}), and (\ref{upper_k}), we have 
					\begin{equation}\label{lower_bound_v_3}
					| \mathbf{v}^{j}_{p^{j},3}(t,Y(y_{1},y_{2}),v)| \gtrsim \delta.
					\end{equation}
					Then, from Lemma \ref{Jac_billiard},
					\begin{equation}\label{k is bounded}
					\Big|
					\frac{\p ( \X^{j+1}_{p^{j+1}},  \hat{\V}^{j+1} _{p^{j+1}} 
						)}{
						\p (  
						\X^{j }_{p^{j }},  \hat{\V}^{j } _{p^{j }},  |{\V}^{j }_{p^{j }}|
						)
					}\Big| = O_{\O, N, \delta} (1).
					\end{equation}
					From (\ref{t_tt}) to (\ref{k is bounded}), we prove (\ref{est_j_{j+1}}).

					\vspace{4pt}

					\noindent	\textit{Step 2. } Recall that, from (\ref{E_Ham}), for $t^{k+1} \leq s < t^{k} $
					\begin{eqnarray}
					&&X_{i} (s;t,Y(x_{1},x_{2}),v) 
					=	X_{i} (s; t^{k},  \mathbf{x}_{ {p}^{k}, 1}^{k}, \mathbf{x}_{ {p}^{k},2}^{k}, 0 ,  { \mathbf{v}_{ {p}^{k} }^{k}} )\notag\\
					&&= \eta_{ {p}^{k}, i} ( \mathbf{x}_{ {p}^{k}, 1}^{k},  \mathbf{x}_{ {p}^{k},2}^{k}, 0) 
					- (t^{k} - s)     | {v}^{k}   |  \hat{v}^{k}_{  ,i}  \notag \\
					&& \ \   \    
					- \int^{s }_{t^{k}} 
					\int^{\tau}_{t^{k}} 
					\p_{i}
					\Phi (   \tau^{\prime}; 
					X(  \tau^{\prime}; 
					t^{k}, \mathbf{x}^{k}_{ {p}^{k},1},  \mathbf{x}^{k}_{ {p}^{k},2}, 0 ,  \mathbf{v}^{k}_{{p}^{k}}
					)   ) \ \dd \tau^{\prime} \dd \tau,   \label{XV_{i}}	 \\
					&	&V_{i}(s;t,Y(x_{1},x_{2}),v) =
					V_{i}	( s; t^{k},  \mathbf{x}^{k}_{ {p}^{k},1} , \mathbf{x}^{k}_{ {p}^{k},2} ,0,  \mathbf{v}^{k}_{ {p}^{k}} 
					)\notag\\
					&&= 
					| {v}^{k}   |\hat{v}^{k}_{ ,i}  
					- \int^{s }_{t^{k}}   \p_{i}
					\Phi (   \tau ; 
					X(  \tau ; 
					t^{k}, \mathbf{x}^{k}_{ {p}^{k},1},  \mathbf{x}^{k}_{ {p}^{k},2}, 0 ,   \mathbf{v}^{k}_{ {p}^{k}} 
					)   )
					\ \dd \tau,\notag
					\end{eqnarray}
					where the specular cycles are defined in (\ref{specular_cycles}) as$$\big(t^{k},\mathbf{x}_{ {p}^{k}}^{k}, \mathbf{v}_{ {p}^{k}}^{k} \big)
					=\big(t^{k}(t,Y(y_{1},y_{2}),v),\mathbf{x}_{ {p}^{k}}^{k}(t,Y(y_{1},y_{2}),v),  \mathbf{v}_{ {p}^{k}}^{k} (t,Y(y_{1},y_{2}),v)\big).
					$$ 

					By the direct computations, for $\p = \p_{|v|},$
					\begin{eqnarray*}
						&& \p_{|v|} [X_{i} (s;t,Y(x_{1}, x_{2}),v)]\notag\\
						& 
						=&
						{ \sum_{\ell=1}^{2} \p_{|v|} \X^{k}_{{p}^{k}, \ell} \cdot \p_\ell \eta_{{p}^{k}, i} (\X_{{p}^{k},1}^{k}, \X_{{p}^{k},2}^{k},0)}
						\\
						&& 
						+ {\p_{|v|} \big[ (  t-t^{k})  |v ^{k}|\big]}
						\hat{v} ^{k}
						- { ( t - s)  \p_{|v|}  |  {v} ^{k}| }\hat{v} ^{k}
						\\
						&& 
						- (t^{k}-s)|v^{k} |  \p_{|v|} \big[ \hat{v} ^{k}\big] 
						\\
						&& 
						- \int^{s}_{t^k} \int^{\tau}_{t^k}  
						\Big(
						\p_{|v|} t^{k} \p_{t^k} X(\tau^{\prime}; t^k) + \sum_{\ell=1}^2 \p_{|v|} \mathbf{x}^k_{{p}^k, \ell} \p_{\mathbf{x}^k_{{p}^k,\ell}} X(\tau^{\prime}; t^k) 
						\\
						&& \quad+ \p_{|v|}  \mathbf{v}_{{p}^{k},\ell}^k  \p_{ \mathbf{v}^k_{\ell}} X(\tau^{\prime}; t^k) 
						\Big)
						\cdot \nabla \p_i \Phi
						(   \tau^{\prime} ; 
						X(  \tau^{\prime}  ; 
						t^{k}, \mathbf{x}^{k}_{{p}^{k},1},  \mathbf{x}^{k}_{{p}^{k},2}, 0 ,  \mathbf{v}_{{p}^{k}}^{k}
						)   )
						\dd\tau^{\prime} \dd\tau \notag\\
						&& 
						+ \p_{|v|} t^k  (s-t^k)  \lim_{\tau^{\prime} \uparrow t^{k}} 
						\p_i \Phi (\tau^{\prime} ; X(\tau^{\prime}; t^{k},  \mathbf{x}^k_{{p}^k,1},\mathbf{x}^k_{{p}^k,2},0, \mathbf{v}_{{p}^{k}}^{k}  )) ,	
					\end{eqnarray*}	
					where we have used the abbreviated notation $X(\tau^{\prime}; t^{k})$ for $X(\tau^{\prime}; t^{k},  \mathbf{x}^k_{{p}^k,1},\mathbf{x}^k_{{p}^k,2},0, \mathbf{v}_{{p}^{k}}^{k} )$. From (\ref{p|v|=0}), we bound the first, second, fourth, fifth, and the last line of RHS by $O_{\Omega, N}(\| \Phi \|_{C^{2}})$. Finally we apply (\ref{p|v|=0}) to the third line and conclude (\ref{X_|v|}).

					\vspace{4pt}
					
					\noindent\textit{Step 3. } First we compute $\p \hat{v}^{k}$ with any arbitrary derivative $\p$. Note that from (\ref{v_p}) and (\ref{normal_eta}), $\hat{\V}^{k}_{p^{k},3}>0$ and $\hat{\V}^{k}_{p^{k},3}
					= \sqrt{1- |\hat{ \mathbf{v}}_{{p}^{k},1}^k
						|^2 - |\hat{\mathbf{v}}_{{p}^{k},2}^k|^2} $. Therefore 	
					\begin{equation}\begin{split}\notag
					&\p \hat{v}^{k}\\
					=& \p \bigg[  \sum_{\ell =1}^2 \frac{\p_\ell \eta_{{p}^{k}}}{\sqrt{g_{{p}^{k},\ell\ell}}} (\mathbf{x}^k_{{p}^{k},1}, \mathbf{x}^k_{{p}^{k},2},0) \hat{\mathbf{v}}^{k}_{{p}^{k},\ell}  + \frac{\p_3 \eta_{{p}^{k}}}{\sqrt{g_{{p}^{k},33}}} (\mathbf{x}^k_{{p}^{k},1}, \mathbf{x}^k_{{p}^{k},2},0) \sqrt{1- |\hat{ \mathbf{v}}_{{p}^{k},1}^k
						|^2 - |\hat{\mathbf{v}}_{{p}^{k},2}^k|^2}  \bigg]
					\\
					=& \sum_{\ell=1}^3 \sum_{m=1}^2  {\p \mathbf{x}^k_{{p}^{k},m}} \p_m \Big[  \frac{\p_\ell \eta_{{p}^{k}}}{\sqrt{g_{{p}^{k},\ell\ell}}} \Big](\mathbf{x}^k_{{p}^{k},1}, \mathbf{x}^k_{{p}^{k},2},0) \hat{\mathbf{v}}^{k}_{{p}^{k},\ell}
					+ \sum_{\ell=1}^2  \frac{\p_\ell \eta_{{p}^{k}}}{\sqrt{g_{{p}^{k},\ell\ell}}} (\mathbf{x}^k_{{p},1}, \mathbf{x}^k_{{p},2},0)  \p [\hat{\mathbf{v}}^{k}_{{p}^{k},\ell}] \\
					&- \frac{\p_3 \eta_{{p}^{k}}}{\sqrt{g_{{p}^{k},33}}} 
					(\mathbf{x}^k_{{p}^{k},1}, \mathbf{x}^k_{{p}^{k},2},0) \frac{1}{  \sqrt{1- |\hat{\mathbf{v}}^k_{{p}^{k},1}|^2 - |\hat{\mathbf{v}}^k_{{p}^{k},2} |^2}   } \Big[ \hat{\mathbf{v}}^k_{{p}^{k},1}  \p [\hat{\mathbf{v}}^k_{{p}^{k},1} ]  +
					\hat{\mathbf{v}}^k_{{p}^{k},2}  \p [\hat{\mathbf{v}}^k_{{p}^{k},2} ] 
					\Big]\\
					=&
					\sum_{j=1}^{2}
					\bigg( \sum_{\ell=1}^3  \frac{\p}{\p {\mathbf{x}^{k}_{{p}^{k},j}}}
					\Big[  \frac{\p_\ell \eta_{{p}^{k}}}{\sqrt{g_{{p}^{k},\ell\ell}}} \Big](\mathbf{x}^k_{{p}^{k},1}, \mathbf{x}^k_{{p}^{k},2},0) \hat{\mathbf{v}}^k_{{p}^{k},\ell} \bigg)
					{\p \mathbf{x}^k_{{p}^{k},j}}
					\\
					&
					+
					\sum_{j=1}^{2}\Big[ \frac{\p_j \eta_{{p}^{k}}}{\sqrt{g_{{p}^{k},11}}}(\mathbf{x}^k_{{p}^{k},1}, \mathbf{x}^k_{{p}^{k},2}, 0)  -  \frac{\p_3 \eta_{{p}^{k}}}{\sqrt{g_{{p}^{k},33}}}(\mathbf{x}^k_{{p}^{k},1}, \mathbf{x}^k_{{p}^{k},2}, 0)\frac{\hat{\mathbf{v}}_{{p}^{k},j}^k }{ \hat{\mathbf{v}}_{{p}^{k},3}^{k} } \Big] {\p [\hat{\mathbf{v}}_{{p}^{k},j}^k ]} .
					\end{split}\end{equation}
					From (\ref{XV_{i}}), for $\p \in \{  \p_{\hat{v}_{1}}, \p_{\hat{v}_{2}}, \p_{y_{1}}, \p_{y_{2}}\}$,
					\begin{eqnarray*}
						&&\p   [X_i(s;t,Y(y_{1},y_{2}),v) ]\\
						&=& \sum_{\ell=1}^{2} \p \X^{k}_{{p}^{k},\ell} \cdot \p_{\ell} \eta_{{p}^{k},i} (\X^{k}_{{p}^{k},1 }, \X^{k}_{{p}^{k},2},0) -  \p  t^{k} |\V^{k}_{{p}^{k}}| 
						\hat{v}^{k}			
						\\
						&& 
						-
						(t^{k}-s)\p  |\V^{k}_{{p}^{k}}| \hat{v}^{k}	
						- (t^{k}-s)| \V^{k}_{{p}^{k}}| {\p} \hat{v}^{k}
						\\
						&&- \int^{s}_{t^k} \int^{\tau}_{t^k}  
						\Big(
						\p t^{k} \p_{t^k} X(\tau^{\prime}; t^k) + \sum_{\ell=1}^2 \p \mathbf{x}^k_{{p}^k, \ell} \p_{\mathbf{x}^k_{{p}^k,\ell}} X(\tau^{\prime}; t^k) + \p  \mathbf{v}_{{p}^{k},\ell}^k  (\mathbf{x}^k_{{p}^k}) \p_{ \mathbf{v}^k_{{p}^{k},\ell}} X(\tau^{\prime}; t^k) 
						\Big)\\
						&& \ \ \ \ \ \ \ \ \ \ \ \ \ 
						\cdot \nabla \p_i \Phi
						(   \tau^{\prime} ; 
						X(  \tau^{\prime}  ; 
						t^{k}, \mathbf{x}^{k}_{{p}^{k},1},  \mathbf{x}^{k}_{{p}^{k},2}, 0 ,  \mathbf{v}^{k}_{{p}^{k}} 
						)   )
						\dd\tau^{\prime} \dd\tau \\
						&& + \p t^k  (s-t^k)  \lim_{\tau^{\prime} \uparrow t^{k}} 
						\p_i \Phi (\tau^{\prime} ; X(\tau^{\prime}; t^{k},  \mathbf{x}^k_{{p}^k,1},\mathbf{x}^k_{{p}^k,2},0, \mathbf{v}^{k}_{{p}^{k}}  )) .
					\end{eqnarray*}
					We can easily conclude (\ref{p_X}) by (\ref{p|v|=0}) and \textit{Step 2.}\end{proof}

				\begin{definition}[Specular Basis]
					Recall the specular cycles $(t^{k},x^{k}, v^{k})$ in (\ref{specular_cycles}). Assume
					\begin{equation}\label{no_grazing_0}
					n(x^{k}) \cdot v^{k} \neq 0.
					\end{equation}
					Recall $\eta_{p^{k}}$ in (\ref{eta}). We define the \textit{specular basis}, which is an orthonormal basis of $\R^{3}$, as 
					\begin{equation}\label{orthonormal_basis}
					\begin{split}
					&\mathbf{ {e}}^{k}_{0}  := \frac{v^{k}}{|v^{k}|}  
					, \\ 
					&\mathbf{  {e}}^{k}_{\perp,1}
					:=  \mathbf{e}^{k}_{0} \times \frac{\p_{2} \eta_{p^{k}}(x^{k})}{ \sqrt{g_{{p}^{k},22}(x^{k})}   } \bigg{/}  \Big| \mathbf{e}^{k}_{0} \times \frac{\p_{2} \eta_{p^{k}} (x^{k})}{ \sqrt{g_{{p}^{k},22}(x^{k})}   }   \Big|
					, \\  &\mathbf{  {e}}^{k}_{\perp,2}   :=  \mathbf{e}^{k}_{0} \times  \mathbf{  {e}}^{k}_{\perp,1}.\end{split}
					\end{equation}
				\end{definition}

				\begin{definition}[Specular Matrix]
					
					For fixed $k \in \mathbb{N}$ and a $C^{1}$-map $Y: (y_{1}, y_{2} ) \mapsto Y(y_{1}, y_{2} )$, assume (\ref{no_grazing_0})
					with 	${x}^{k} ={x}^{k} (t, Y( {y}_{1},  {y}_{2}), |v|, \hat{v}_{1}, \hat{v}_{2})$ and \\
					${v}^{k} ={v}^{k} (t, Y( {y}_{1},  {y}_{2}), |v|, \hat{v}_{1}, \hat{v}_{2})$. We define the $4\times4$ specular transition matrix \\
					$\mathcal{S}^{k, p^{k}, Y}= \mathcal{S}^{k, p^{k}, Y}(t,y_{1},y_{2}, |v|, \hat{v}_{1}, \hat{v}_{2})$ as 
					\begin{equation}\label{specular_transition_matrix}
					\mathcal{S}^{k, p^{k}, Y}
					:= \left[\begin{array} {c|c}
					\mathcal{S}_{1}^{k, p^{k}, Y}	 &0_{2\times2}\\ \hline
					\mathcal{S}_{2}^{k, p^{k}, Y} & \mathcal{S}_{3}^{k, p^{k}, Y}
					\end{array} \right]_{4\times 4},
					\end{equation}
					where 
					\begin{equation}\begin{split}\notag
					\mathcal{S}_{1}^{k, p^{k}, Y} &: = 
					\left[\begin{array}{cc}	\p_{1} \eta_{ {p}^{k}} \cdot \mathbf{e}^{k}_{\perp,1} & 
					\p_{2} \eta_{{p}^{k}} \cdot \mathbf{e}^{k}_{\perp,1}\\
					\p_{1} \eta_{{p}^{k}} \cdot \mathbf{e}^{k}_{\perp,2} & 
					\p_{2} \eta_{{p}^{k}} \cdot \mathbf{e}^{k}_{\perp,2}
					\end{array}
					\right]_{2\times 2} ,\\
					\mathcal{S}_{2}^{k, p^{k}, Y} &: = 	
					\left[\begin{array}{cc}
					\Big(  \sum_{\ell=1}^{3} \p_{1} \big[ \frac{\p_{\ell} \eta_{{p}^{k}}}{\sqrt{g_{ {{p}^{k}},\ell\ell}}}  \big] \hat{\mathbf{v}}^{k}_{ {{p}^{k}},\ell}  \Big) \cdot \mathbf{e}^{k}_{\perp,1}   
					&   \Big(  \sum_{\ell=1}^{3} \p_{2} \big[ \frac{\p_{\ell} \eta_{{p}^{k}}}{\sqrt{g_{ {{p^{k}}},\ell\ell}}}  \big] \hat{\mathbf{v}}^{k}_{ { {p}},\ell}  \Big) \cdot \mathbf{e}_{\perp,1}  ^{k}
					\\
					\Big(  \sum_{\ell=1}^{3} \p_{1} \big[ \frac{\p_{\ell} \eta_{{p^{k}}}}{\sqrt{g_{ { {p^{k}}},\ell\ell}}}  \big] \hat{\mathbf{v}}^{k}_{ {{p^{k}}},\ell}  \Big) \cdot \mathbf{e}^{k}_{\perp,2}   & 
					\Big(  \sum_{\ell=1}^{3} \p_{2} \big[ \frac{\p_{\ell} \eta_{{p^{k}}}}{\sqrt{g_{ {{p^{k}}},\ell\ell}}}  \big] \hat{\mathbf{v}}^{k}_{ {{p^{k}}},\ell}  \Big) \cdot \mathbf{e}^{k}_{\perp,2}  
					\end{array} \right]_{2\times 2},\\
					\mathcal{S}_{3} ^{k, p^{k}, Y}& : = 
					\left[\begin{array}{cc}
					\Big[ \frac{\p_{1} \eta_{p^{k}}}{\sqrt{g_{p^{k},11}}} - \frac{\p_{3} \eta_{p^{k}}}{\sqrt{g_{p^{k},33}}}  \frac{\hat{\mathbf{v}}^{k}_{p^{k},1}}{\hat{\mathbf{v}}^{k}_{p^{k},3}} \Big] \cdot \mathbf{e}^{k}_{\perp,1}
					& \Big[ \frac{\p_{2} \eta_{p^{k}}}{\sqrt{g_{p^{k},22}}} - \frac{\p_{3} \eta_{p^{k}}}{\sqrt{g_{p^{k},33}}}  \frac{\hat{\mathbf{v}}^{k}_{p^{k},2}}{\hat{\mathbf{v}}^{k}_{p^{k},3}} \Big] \cdot \mathbf{e}^{k}_{\perp,1}\\
					\Big[ \frac{\p_{1} \eta_{p^{k}}}{\sqrt{g_{p^{k},11}}} - \frac{\p_{3} \eta_{p^{k}}}{\sqrt{g_{p^{k},33}}}  \frac{\hat{\mathbf{v}}^{k}_{p^{k},1}}{\hat{\mathbf{v}}^{k}_{p^{k},3}} \Big] \cdot \mathbf{e}^{k}_{\perp,2}
					& \Big[ \frac{\p_{2} \eta_{p^{k}}}{\sqrt{g_{p^{k},22}}} - \frac{\p_{3} \eta_{p^{k}}}{\sqrt{g_{p^{k},33}}}  \frac{\hat{\mathbf{v}}^{k}_{p^{k},2}}{\hat{\mathbf{v}}^{k}_{p^{k},3}} \Big] \cdot \mathbf{e}^{k}_{\perp,2}
					\end{array} \right]_{2\times 2},
					\end{split} \end{equation}
					where $\eta_{p^{k}}$ and $g_{p^{k}}$ are evaluated at $x^{k}(t, Y( {y}_{1},  {y}_{2}), |v|, \hat{v}_{1}, \hat{v}_{2})$. We also define the $4\times 4$ specular matrix $\mathcal{R}^{k,p^{k},Y} = \mathcal{R}^{k,p^{k},Y} (t,y_{1},y_{2}, |v|, \hat{v}_{1}, \hat{v}_{2})$ as
					\begin{equation}\label{specular_matrix}
					\mathcal{R}^{k, p^{k}, Y}: = 	\mathcal{S}^{k, p^{k}, Y}
					\frac{\p ( 
						\mathbf{x}^{k}_{ {p}^{k},1}, \mathbf{x}^{k}_{ {p}^{k},2},
						\hat{\mathbf{v}}^{k}_{ {p}^{k},1}, \hat{\mathbf{v}}^{k}_{ {p}^{k},1}
						)}{\p (y_{1},y_{2}, \hat{v}_{1}, \hat{v}_{2})}
					,
					\end{equation}	
					where $\mathbf{x}^{k}_{{p}^{k}}=\mathbf{x}^{k}_{{p}^{k}}(t, Y( {y}_{1},  {y}_{2}), |v|, \hat{v}_{1}, \hat{v}_{2})$, $\mathbf{v}^{k}_{{p}^{k}}=\mathbf{v}^{k}_{{p}^{k}}(t, Y( {y}_{1},  {y}_{2}), |v|, \hat{v}_{1}, \hat{v}_{2})$. \end{definition}

				Finally we state the result which is a crucial ingredient in the proofs of Lemma \ref{lemma rank 2} and Lemma \ref{lemma rank 3}. For $n \times m$ matrix $A$, we use the notation $A_{i,j}$ for the $(i,j)-$entry of the matrix $A$.	Recall that $e_{3} = (0,0,1) \in \R^{3}$ and $v_{3} = v\cdot e_{3}$.
				
				\begin{lemma}\label{nonzero_sub}  Let a $C^{1}$-map $Y: (y_{1}, y_{2} ) \mapsto Y(y_{1}, y_{2} ) \in \bar{\O}$ with $\| Y \|_{C^{1}} \lesssim 1$. Assume $\frac{1}{N} \leq |v| \leq N$, $\frac{1}{N} \leq |v_{3}|$, $\frac{1}{N}<|n(x^{1}(t, Y(y_{1},y_{2})  ,v)) \cdot e_{3}|$, and $\|  \Phi \|_{C^{2}_{x}} < \frac{\delta_{1}}{3 \text{diam} (\Omega) N^{2}}$ for $1 \ll N$ and $0 < \delta_{1}\ll 1$. We also assume non-grazing condition
					\begin{equation}\label{no_grazing_lemma}
					|v^{1}(t, Y(y_{1},y_{2}) ,v) \cdot n(x^{1}(t, Y(y_{1},y_{2}) ,v))|> \delta_{2}>0,
					\end{equation} 
					and non-degenerate condition 
					\begin{equation}\label{crossY_n}
					\Big|\Big(\frac{\p Y(y_{1},y_{2})}{\p y_{1}} \times \frac{\p Y(y_{1},y_{2})}{\p y_{2}}\Big) \cdot 
					R_{ x^{1}(t, Y(y_{1},y_{2}) ,v)  } v^{1}(t, Y(y_{1},y_{2}) ,v) 
					\Big| > \delta_{3}>0.
					\end{equation}
					
					Fix $k \in \mathbb{N}$ with $|t-t^{k}| \leq 1$. Then the following results hold: (i) There exists at least one $i \in \{1,2,3,4\}$ such that for some constant $\varrho_{\Omega, N , \delta_{1}, \delta_{2}, \delta_{3}}>0$ 
					\begin{equation}
					|\mathcal{R}_{i,3}^{k,{p}^{k},Y}   (t, Y(y_{1}, y_{2}) , v)| > \varrho_{\Omega, N , \delta_{1}, \delta_{2}, \delta_{3}} .\label{nonzero_sub1}
					\end{equation}
					
					\noindent(ii) There exist $i, j \in \{1,2,3,4\}$ with $i< j$ such that 
					\begin{equation}\label{nonzero_sub2}
					\Big|\det \left[\begin{array}{cc}
					\mathcal{R}_{3,i}^{k,{p}^{k},Y}&\mathcal{R}_{3,j}^{k,{p}^{k},Y} \\
					\mathcal{R}_{4,i}^{k,{p}^{k},Y}&\mathcal{R}_{4,j}^{k,{p}^{k},Y} 
					\end{array}\right] (t, Y(y_{1}, y_{2}) , v)\Big| > \varrho_{\Omega, N , \delta_{1}, \delta_{2},\delta_{3}} . 
					\end{equation}

				\end{lemma}

				\begin{proof}
					\textit{Step 1.} We claim that 
					\begin{equation}\label{R_lower}
					\big| \det \mathcal{R}^{k, p^{k}, Y} (t, Y(y_{1},y_{2}), v) \big|  \ \gtrsim_{\Omega, N, \delta_{1}, \delta_{2} , \delta_{3}} \ 1. 
					\end{equation}

					\noindent Note that from (\ref{specular_matrix}) and (\ref{specular_transition_matrix}),
					\begin{equation}
					\det (\mathcal{R}^{k,p^{k}, Y}) = \det (\mathcal{S}_{1}^{k,p^{k}, Y})
					\det (\mathcal{S}_{3}^{k,p^{k}, Y}) 
					\det \left(
					\frac{\p ( 
						\mathbf{x}^{k}_{ {p}^{k},1}, \mathbf{x}^{k}_{ {p}^{k},2},
						\hat{\mathbf{v}}^{k}_{ {p}^{k},1}, \hat{\mathbf{v}}^{k}_{ {p}^{k},1}
						)}{\p (y_{1},y_{2}, \hat{v}_{1}, \hat{v}_{2})}
					\right). \notag
					\end{equation}
					By (\ref{identity_cross_det}) and (\ref{orthonormal_basis}),  
					\begin{eqnarray*}
						\det (\mathcal{S}_{1}^{k,p^{k}, Y})
						&=&
						| ( \p_{1} \eta_{{p}^k} \times \p_{1} \eta_{{p}^k} ) \cdot ( \mathbf{e}^{k}_{\perp,1} \times \mathbf{e}^{k}_{\perp,2} ) |    \\
						&=& 
						\sqrt{
							g_{p^{k}, 11}(x^{k})g_{p^{k},22}(x^{k})
						}\Big|\frac{\mathbf{v}^{k}_{{p}^{k}} }{ |\mathbf{v}^{k}_{{p}^{k}} | } \cdot  {n}  (x^{k})\Big|
						=
						\sqrt{
							g_{p^{k}, 11}(x^{k})g_{p^{k},22}(x^{k})
						} \frac{ |\mathbf{v}^{k}_{{p}^{k},3}| }{ |\mathbf{v}^{k}_{{p}^{k}} | }   ,
					\end{eqnarray*}
					and	
					\begin{eqnarray*}
						&& \det (\mathcal{S}_{3}^{k, p^{k}, Y}) \\
						&=&
						\Big| \Big( \Big[ \frac{\p_{1} \eta_{{p}^k}}{\sqrt{g_{{p}^k,11}}} - \frac{\p_{3} \eta_{{p}^k}}{\sqrt{g_{{p}^k,33}}}  \frac{\hat{\V}^{k}_{{p}^k,1}}{\hat{\V}^{k}_{{p}^k,3}} \Big] 
						\times
						\Big[ \frac{\p_{2} \eta_{{p}^k}}{\sqrt{g_{{p}^k,22}}} - \frac{\p_{3} \eta_{{p}^k}}{\sqrt{g_{{p}^k,33}}}  \frac{\hat{\V}^{k}_{{p}^k,2}}{\hat{\V}^{k}_{{p}^k,3}} \Big] \Big) \cdot ( \mathbf{e}^{k}_{\perp,1} \times \mathbf{e}^{k}_{\perp,2} ) \Big|    \\
						&=& \frac{1}{ |\hat{\mathbf{v}}^k_{{p}^k,3}| } \Big| \Big(  \hat{\mathbf{v}}^k_{{p}^k,1}  \frac{\p_1\eta_{{p}^k}}{\sqrt{g_{{p}^k,11}}} + \hat{\mathbf{v}}^k_{{p}^k,2}  \frac{\p_2\eta_{{p}^k}}{\sqrt{g_{{p}^k,22}}} + \hat{\mathbf{v}}^k_{{p}^k,3}  \frac{\p_3\eta_{{p}^k}}{\sqrt{g_{{p}^k,33}}} \Big) \cdot ( \mathbf{e}^{k}_{\perp,1} \times \mathbf{e}^{k}_{\perp,2} ) \Big|    \\
						&=& \frac{1}{ |{\mathbf{v}}^k_{{p}^k,3}| } \mathbf{v}^k_{{p}^k} \cdot \frac{\mathbf{v}^k_{{p}^k}}{ |\mathbf{v}^k_{{p}^k}| } = \frac{|\mathbf{v}^k_{{p}^k}|}{|{\mathbf{v}}^k_{{p}^k,3}|} .
					\end{eqnarray*}
					
					\noindent From the chain rule and (\ref{Jac_hat}),
					\begin{equation} \label{angle k bounce}
					\begin{split}
					& \Big|  \det \left(
					\frac{\p ( 
						\mathbf{x}^{k}_{ {p}^{k},1}, \mathbf{x}^{k}_{ {p}^{k},2},
						\hat{\mathbf{v}}^{k}_{ {p}^{k},1}, \hat{\mathbf{v}}^{k}_{ {p}^{k},1}
						)}{\p (y_{1},y_{2}, \hat{v}_{1}, \hat{v}_{2})}
					\right) \Big| 
					\\
					&=    \Big| \det \left( \frac{ \p( \mathbf{x}^k_{{p}^k,1}, \mathbf{x}^k_{{p}^k,2}, \hat{\mathbf{v}}^k_{{p}^{k},1}, \hat{\mathbf{v}}^k_{{p}^{k},2} ) }{ \p( \mathbf{x}^1_{{p}^1,1}, \mathbf{x}^1_{{p}^1,2}, \hat{\mathbf{v}}^1_{{p}^{k},1}, \hat{\mathbf{v}}^1_{{p}^{k},2} ) } \right)
					\det \left( \frac{ \p( \mathbf{x}^1_{{p}^1,1}, \mathbf{x}^1_{{p}^1,2}, \hat{\mathbf{v}}^1_{{p}^{1},1}, \hat{\mathbf{v}}^1_{{p}^{1},2} ) }{ \p( y_1, y_2, \hat{v}_1, \hat{v}_2 ) } \right)   \Big| \\
					&\geq \epsilon_{\O,N, \delta_{1}, \delta_{2}, \delta_{3}} \Big| 
					\det
					\left( \frac{ \p( \mathbf{x}^1_{{p}^1,1}, \mathbf{x}^1_{{p}^1,2}, \hat{\mathbf{v}}^1_{{p}^{1},1}, \hat{\mathbf{v}}^1_{{p}^{1},2} ) }{ \p( y_1, y_2, \hat{v}_1, \hat{v}_2 ) } \right)
					\Big|. 
					\end{split}	
					\end{equation}	
					Note that
					\begin{eqnarray*}
						\frac{ \p( \mathbf{x}^1_{{p}^1,1}, \mathbf{x}^1_{{p}^1,2}, \hat{\mathbf{v}}^1_{{p}^{1},1}, \hat{\mathbf{v}}^1_{{p}^{1},2} ) }{ \p( y_1, y_2, \hat{v}_1, \hat{v}_2 ) } 
						=  \left[\begin{array}{cc|cc}
							\frac{\p Y}{\p y_{1}} \cdot \nabla_{x} \mathbf{x}^{1}_{p^{1},1} &   \frac{\p Y}{\p y_{2}} \cdot \nabla_{x} \mathbf{x}^{1}_{p^{1},1} & \frac{\p \mathbf{x}^{1}_{p^{1},1}}{\p \hat{v}_{1}} &  \frac{\p \mathbf{x}^{1}_{p^{1},1}}{\p \hat{v}_{2}} \\  
							\frac{\p Y}{\p y_{1}} \cdot \nabla_{x} \mathbf{x}^{1}_{p^{1},2} &  \frac{\p Y}{\p y_{2}} \cdot \nabla_{x} \mathbf{x}^{1}_{p^{1},2} & 
							\frac{\p \mathbf{x}^{1}_{p^{1},2}}{\p \hat{v}_{1}} &  \frac{\p \mathbf{x}^{1}_{p^{1},2}}{\p \hat{v}_{2}}
							\\ \hline
							\frac{\p Y}{\p y_{1}}  \cdot \nabla_{x} \hat{\mathbf{v}}^{1}_{{p}^{1}, 1} &   \frac{\p Y}{\p y_{2}} \cdot \nabla_{x} \hat{\mathbf{v}}^{1}_{{p}^{1},1}   &  \frac{\p \hat{\mathbf{v}}^{1}_{p^{1},1}}{\p \hat{v}_{1}} &  \frac{\p \hat{\mathbf{v}}^{1}_{p^{1},1}}{\p \hat{v}_{2}}  \\  
							\frac{\p Y}{\p y_{1}} \cdot \nabla_{x}  \hat{\mathbf{v}}^{1}_{{p}^{1},2}  & \frac{\p Y}{\p y_{2}}\cdot \nabla_{x}\hat{\mathbf{v}}^{1}_{{p}^{1},2}   &   \frac{\p \hat{\mathbf{v}}^{1}_{p^{1},2}}{\p \hat{v}_{1}} &  \frac{\p \hat{\mathbf{v}}^{1}_{p^{1},2}}{\p \hat{v}_{2}}
						\end{array} \right].
					\end{eqnarray*}
					From Lemma \ref{global to local} and (\ref{no_grazing_lemma}),
					\begin{eqnarray*}
						\nabla_{x} \mathbf{\hat{v}}_{p^{1},i}^{1}
						&=&\frac{1}{|\mathbf{v}^{1}_{p^{1}}|}
						\Big\{  \nabla_{x} \mathbf{x}_{p^{1} }^{1} \cdot \nabla \big( \frac{\p_{i} \eta_{p^{1}}  }{\sqrt{g_{p^{1} ,ii }  }   }   \big)\Big|_{x^{1}} \cdot v + O_{\O, N, \delta_{2}} ( \| \Phi \|_{C^{2}} )\Big\} +\frac{  O_{\O, N} (\| \Phi \|_{C^{2}})}{|\mathbf{v}^{1}_{p^{1}}|^{2} |\mathbf{v}^{1}_{p^{1},3}| }\\
						&=&  \nabla_{x} \mathbf{x}_{p^{1} }^{1} \cdot \nabla \big( \frac{\p_{i} \eta_{p^{1}}  }{\sqrt{g_{p^{1} ,ii }  }   }   \big)\Big|_{x^{1}} \cdot  \frac{v}{|\mathbf{v}^{1}_{p^{1}}|} + O_{\O, N , \delta_{1}, \delta_{2}} (\| \Phi \|_{C^{2}})
						,\\
						\nabla_{\hat{v}} \mathbf{\hat{v}}_{p^{1},i}^{1}
						&=&\frac{1}{|\mathbf{v}^{1}_{p^{1}}|} \nabla_{\hat{v}} \mathbf{v}^{1}_{p^{1},i}
						- \frac{\mathbf{v}^{1}_{p^{1},i} }{|\mathbf{v}^{1}_{p^{1}}|^{2}} \nabla_{\hat{v}} |\mathbf{v}^{1}_{p^{1}}|
						\\
						&=&  \frac{1}{|\mathbf{v}^{1}_{p^{1}}|}\Big\{ \nabla_{\hat{v}} \mathbf{x}^{1}_{p^{1}}
						\cdot \nabla \big( \frac{\p_{i} \eta_{p^{1} } }{\sqrt{g_{p^{1}, ii}  }} \big)\Big|_{x^{1}}
						\cdot v + O_{\O, N, \delta_{1}, \delta_{2}} (\| \Phi \|_{C^{2}})
						\Big\}  \\ 
						&&+ \frac{1}{|\mathbf{v}^{1}_{p^{1}}|} \Big\{
						|v| e_{j} \cdot \frac{\p_{i} \eta_{p^{1}}}{\sqrt{g_{p^{1}, ii}}}\Big|_{x^{1}}  + O(\| \Phi \|_{C^{1}}) \Big| \frac{\p \tb}{\p \mathbf{\hat{v}} _{j} } \Big|
						+ O_{N}(\| \Phi \|_{C^{2}}) \tb
						\Big\} \\
						&&+ O_{\O, N, \delta_{1}, \delta_{2}} (\| \Phi \|_{C^{2}}) \\
						&=&  \frac{|v|}{|\mathbf{v}^{1}_{p^{1}}|}  \frac{\p_{i} \eta_{p^{1}}}{\sqrt{g_{p^{1}, ii}}}\Big|_{x^{1}}  \cdot e_{j} +   \nabla_{\hat{v}} \mathbf{x}^{1}_{p^{1}}
						\cdot \nabla \big( \frac{\p_{i} \eta_{p^{1} } }{\sqrt{g_{p^{1}, ii}  }} \big)\Big|_{x^{1}}
						\cdot \frac{ v}{|\mathbf{v}^{1}_{p^{1}}|}  + O_{\O, N, \delta_{1}, \delta_{2}} (\| \Phi \|_{C^{2}}).
					\end{eqnarray*}
					Then by the Gaussian elimination,
					\begin{eqnarray*}
						&&\det \left(\frac{ \p( \mathbf{x}^1_{{p}^1,1}, \mathbf{x}^1_{{p}^1,2}, \hat{\mathbf{v}}^1_{{p}^{1},1}, \hat{\mathbf{v}}^1_{{p}^{1},2} ) }{ \p( y_1, y_2, \hat{v}_1, \hat{v}_2 ) } \right)\\
						&=& \det \Bigg(
						\left[\begin{array}{cc|cc}
							\frac{\p Y}{\p y_{1}} \cdot \nabla_{x} \mathbf{x}^{1}_{p^{1},1} &   \frac{\p Y}{\p y_{2}} \cdot \nabla_{x} \mathbf{x}^{1}_{p^{1},1} & \frac{\p \mathbf{x}^{1}_{p^{1},1}}{\p \hat{v}_{1}} &  \frac{\p \mathbf{x}^{1}_{p^{1},1}}{\p \hat{v}_{2}} \\  
							\frac{\p Y}{\p y_{1}} \cdot \nabla_{x} \mathbf{x}^{1}_{p^{1},2} &  \frac{\p Y}{\p y_{2}} \cdot \nabla_{x} \mathbf{x}^{1}_{p^{1},2} & 
							\frac{\p \mathbf{x}^{1}_{p^{1},2}}{\p \hat{v}_{1}} &  \frac{\p \mathbf{x}^{1}_{p^{1},2}}{\p \hat{v}_{2}}
							\\ \hline
							0 &0 & 
							\frac{\p_{1} \eta_{p^{1}}}{\sqrt{g_{p^{1},11}}} \cdot e_{1}
							& \frac{\p_{1} \eta_{p^{1}}}{\sqrt{g_{p^{1},11}}} \cdot e_{2} \\  
							0&0 &  
							\frac{\p_{2} \eta_{p^{2}}}{\sqrt{g_{p^{2},22}}} \cdot e_{1}
							&   \frac{\p_{2} \eta_{p^{2}}}{\sqrt{g_{p^{2},22}}} \cdot e_{2}
						\end{array} \right]  \\
						&&+ \big[ O_{\Omega,N,\delta_{1}, \delta_{2}} (\| \Phi \|_{C^{2}}) \big]_{4\times 4}  \Bigg).
					\end{eqnarray*}
					From (\ref{lower_bound_v_3}) and (\ref{Jac_billiard}), all the entries of above matrix is bound and hence the determinant of the Jacobian matrix equals
					\begin{equation}
					\begin{split}
					&\det \left(
					\left[\begin{array}{cc|cc}
					\frac{\p Y}{\p y_{1}} \cdot \nabla_{x} \mathbf{x}^{1}_{p^{1},1} &   \frac{\p Y}{\p y_{2}} \cdot \nabla_{x} \mathbf{x}^{1}_{p^{1},1} & \frac{\p \mathbf{x}^{1}_{p^{1},1}}{\p \hat{v}_{1}} &  \frac{\p \mathbf{x}^{1}_{p^{1},1}}{\p \hat{v}_{2}} \\  
					\frac{\p Y}{\p y_{1}} \cdot \nabla_{x} \mathbf{x}^{1}_{p^{1},2} &  \frac{\p Y}{\p y_{2}} \cdot \nabla_{x} \mathbf{x}^{1}_{p^{1},2} & 
					\frac{\p \mathbf{x}^{1}_{p^{1},2}}{\p \hat{v}_{1}} &  \frac{\p \mathbf{x}^{1}_{p^{1},2}}{\p \hat{v}_{2}}
					\\ \hline
					0 &0 & 
					\frac{\p_{1} \eta_{p^{1}}}{\sqrt{g_{p^{1},11}}} \cdot e_{1}
					& \frac{\p_{1} \eta_{p^{1}}}{\sqrt{g_{p^{1},11}}} \cdot e_{2} \\  
					0&0 &  
					\frac{\p_{2} \eta_{p^{2}}}{\sqrt{g_{p^{2},22}}} \cdot e_{1}
					&   \frac{\p_{2} \eta_{p^{2}}}{\sqrt{g_{p^{2},22}}} \cdot e_{2}
					\end{array} \right]
					\right)  \\
					&\quad + O_{\Omega,N,\delta_{1}, \delta_{2}} (\| \Phi \|_{C^{2}}).
					\end{split}
					\end{equation}

					\noindent From (\ref{identity_cross_det}), the determinant equals
					\begin{equation}\begin{split}\label{YY}
					\Big|
					\Big(\frac{\p Y}{\p y_{1}} \times \frac{\p Y}{\p y_{2}}\Big) \cdot 
					\big(
					\nabla_{x} \X^{1}_{p^{1},1} \times \nabla_{x} \X^{1}_{p^{1},2} 
					\big)
					\Big| \times 
					\Big|
					\Big(
					\frac{\p_{1} \eta_{p^{1}}}{\sqrt{g_{p^{1},11}}}
					\times \frac{\p_{2} \eta_{p^{1}}}{\sqrt{g_{p^{1},22}}}
					\Big) \cdot e_{3}
					\Big|\\
					+ O_{\Omega,N,\delta_{1}, \delta_{2}} (\| \Phi \|_{C^{2}}).
					\end{split}
					\end{equation}
					From (\ref{global to local})
					\begin{eqnarray*}
						&& {\nabla_{x} \X^{1}_{{p^{1}}, 1}}  \times  {\nabla_{x} \X^{1}_{{p^{1}}, 2}} \\
						&=& \frac{1}{\sqrt{g_{{p^{1}},11}(x^{1}) g_{{p^{1}},22}(x^{1}) } }
						\Big[
						n_{p^{1}}(x^{1})-  \frac{\V^{1}_{{p^{1}},1}}{\V^{1}_{{p^{1}},3}} \frac{\p_{1} \eta_{{p^{1}}   } (x^{1})  }{\sqrt{g_{{p^{1}}, 11} (x^{1})  }  } -  \frac{\V^{1}_{{p^{1}},2}}{\V^{1}_{{p^{1}},3}} \frac{\p_{2} \eta_{{p^{1}}   } (x^{1})  }{\sqrt{g_{{p^{1}}, 22} (x^{1})  }  }
						\Big]\\	
						&&	+ O_{\O}(\| \nabla_{x}^{2} \Phi \|_{\infty})(1+ \frac{|v|}{|\V^{1}_{{p^{1}},3}|}) \frac{|v|}{|\V^{1}_{{p^{1}},3}|}.
					\end{eqnarray*}
					From (\ref{normal_eta}) and (\ref{x^k}), the first line of above RHS equals		
					\[ \frac{1}{\sqrt{g_{{p^{1}},11}(x^{1}) g_{{p^{1}},22}(x^{1}) } } \frac{- R_{x^{1}}v^{1} }{\V^{1}_{p^{1}, 3}},\]
					while the second line is bounded by $O_{\O, N , \delta_{2}} (\| \nabla_{x}^{2} \Phi \|_{\infty}) \frac{|v|}{|\V^{1}_{{p^{1}},3}|}$ from (\ref{no_grazing_lemma}). From the assumptions of the lemma, including (\ref{crossY_n}), we derive a lower bound as 
					\begin{eqnarray*}
						(\ref{YY}) \gtrsim_{\O} \frac{\delta_{3} }{N} \times \frac{1}{N} + O(\| \Phi \|_{C^{2}}).
					\end{eqnarray*}By choosing sufficiently small $\| \Phi \|_{C^{2}}$ we prove (\ref{R_lower}).

					
					\vspace{4pt}
					
					\noindent\textit{Step 2.} Assume $|\mathcal{R}^{k, p^{k}, Y}_{i,3}| \ll 1$ {for all $i\in\{1,2,3,4\} $ }. Then
					\begin{eqnarray*}
						|\det \mathcal{R}^{k, p^{k}, Y} | \leq  \Big|\sum_{i=1}^{4} (-1)^{i+3}\mathcal{R}^{k, p^{k}, Y}_{i,3} M_{i,3}\Big| \leq 4 \max_{i}|M_{i,3}|
						\times \max_{i} |\mathcal{R}^{k, p^{k}, Y}_{i,3}|,
					\end{eqnarray*}
					where the minor $M_{i,j}$ is defined to be the determinant of the $3 \times 3-$matrix that results from $\mathcal{R}^{k, p^{k}, Y}$ by removing the $i$th row and the $j$th column. Note that \\
					$|M_{i,3}| \lesssim_{\Omega, N , \delta_{1}, \delta_{2}} 1$. From (\ref{R_lower}) we prove (\ref{nonzero_sub1}). 
					
					\vspace{4pt}
					
					\noindent\textit{Step 3.} Note that 
					\begin{equation}\begin{split}\notag
					|\det \mathcal{R}^{k,p^{k},Y}|   &\leq  12\max_{i} | \mathcal{R}_{1,i}^{k,p^{k},Y}| \times    \max_{i} | \mathcal{R}_{2,i}^{k,p^{k},Y}| \times \max_{i,j }\Big|\det \left[\begin{array}{cc}
					\mathcal{R}_{3,i}^{k,p^{k},Y}&\mathcal{R}_{3,j}^{k,p^{k},Y} \\
					\mathcal{R}_{4,i}^{k,p^{k},Y}&\mathcal{R}_{4,j}^{k,p^{k},Y} 
					\end{array}\right]\Big|
					\\
					& \lesssim_{\Omega, N, \delta_{1}, \delta_{2}} 
					\max_{i,j }\Big|\det \left[\begin{array}{cc}
					\mathcal{R}_{3,i}^{k,p^{k},Y}&\mathcal{R}_{3,j}^{k,p^{k},Y} \\
					\mathcal{R}_{4,i}^{k,p^{k},Y}&\mathcal{R}_{4,j}^{k,p^{k},Y} 
					\end{array}\right]\Big|.
					\end{split}\end{equation}
					From (\ref{R_lower}), we prove (\ref{nonzero_sub2}).\end{proof}

				\begin{lemma}\label{zero_poly}
					Assume that $a(z),b(z),c(z)$ are $C^{0,\gamma}$-functions of $z \in \R^{n}$ locally. We consider $G(z,s) := a(z)s^{2} + b(z)s + c(z)$.
					
					(i) Assume $|a| \geq \min |a| >0$. Define 
					\begin{equation}\begin{split}\label{psi123}
					&  \varphi_{1}(z): = \frac{-b(z)}{2a( z)}, \  \ \varphi_{2}(z) :=\mathbf{1}_{b^{2}-4ac >0} \frac{-b(z) + \sqrt{b^{2} (z)- 4a(z)c(z)}}{2a(z)},\\
					&\varphi_{3}(z) := \mathbf{1}_{b^{2}(z)-4a(z)c(z) >0}\frac{-b(z) -\sqrt{b^{2} (z)- 4a(z)c(z)}}{2a(z)}.\end{split}
					\end{equation}
					Then $\varphi_{i}(z) \in C^{0,\gamma}$ with $\| \varphi_{i}\|_{C^{0,\gamma}} \leq C( \min |a|, \| a\|_{C^{0,\gamma}}, \| b \|_{C^{0,\gamma}}, \| c\|_{C^{0,\gamma}})$ for $i=1,2,3$ such that if $|s| \leq 1$ and $\min_{i=1,2,3}|s-\varphi_{i}(z)|> \delta$, then $|G(z,s)|\gtrsim \min |a| \times \delta^{2}$.
					
					(ii) Assume $a\equiv 0$ and $\min |b| >0$. Define 
					\begin{equation}\label{psi4}
					\varphi_{4} ( z ) : = \frac{- c(z)}{ b(z)}.
					\end{equation}
					Then $\varphi_{4}(z) \in C^{0,\gamma}$ with $\| \varphi_{4}\|_{C^{0,\gamma}} \leq C(\min |b|  , \| b \|_{C^{0,\gamma}}, \| c \|_{C^{0,\gamma}})$. Moreover, if $|s| \leq 1$ and $|s- \varphi_{4}(z)|> \delta$, then $|G(z,s)| \gtrsim \min |b| \times \delta$.

					(iii) Assume $a\equiv 0$ and $\min |c| >0$. Define 
					\begin{equation}\label{psi5}
					\varphi_{5}(z) : =  \mathbf{1}_{|b(z) | > \frac{\min |c|}{4}}  \frac{- c(z)}{ b(z)}.
					\end{equation}
					Then $\varphi_{5} (z) \in C^{0,\gamma}$ with $\| \varphi_{5}\|_{C^{0,\gamma}} \leq C(\min |b|  , \| b \|_{C^{0,\gamma}}, \| c \|_{C^{0,\gamma}})$. Moreover, if $|s| \leq1$ and $|s- \varphi_{5} (z)|>\delta$, then $|G(z,s)|\geq \min \big\{ \frac{\min |c|}{2},  \frac{\min |c|}{4} \times  \delta \big\}
					$.
					

				\end{lemma}

				\begin{proof}
					We consider $(i)$. Without loss of generality we may assume that \\
					$a\geq  \min a >0$. Clearly if $a(z)\geq  \min a >0$ then $\varphi_{i}$ is $C^{0,\gamma}$ and \\
					$\| \varphi_{i}\|_{C^{0,\gamma}} \leq C( \min a, \| a\|_{C^{0,\gamma}}, \| b \|_{C^{0,\gamma}}, \| c\|_{C^{0,\gamma}})$ for $i=1,2,3$.
					
					We claim that 
					$$
					\min_{|s- \varphi_{1}|> \delta} \big\{ |G(z, s+ \delta) - G(z,s)|, |G(z,s) - G(z, s- \delta) |\big\}\gtrsim \min a \times  \delta^{2}.
					$$
					Since $G(z,s)$ is symmetric around $s= \varphi_{1}$, it suffices to prove above estimate for $s \geq  \varphi_{1}$. Firstly, we consider the difference $G(z,s+\delta) - G(z,s)$ for $s\geq - \frac{b}{2a}$ and $\delta>0$. Note $\p_{s}  [G(z,s+\delta) - G(z,s) ] = 2a\delta>0$. Therefore, for any $z$, $\min_{s\geq - \frac{b}{2a} }  [G(z,s+\delta)- G(z,s) ]= [G(z, - \frac{b}{2a}+  \delta)- G(z,- \frac{b}{2a} ) ]\geq \min a \times  \delta^{2}$. Secondly, we consider $G(z,s)-G(z,s-\delta)$ for $s\geq - \frac{b}{2a}+\delta$ and $\delta>0$. Since $\p_{s}  [G(z,s+\delta) - G(z,s) ] = 2a\delta>0$, $\min_{s\geq - \frac{b}{2a}+\delta}  [ G(z,s)-G(z,s-\delta) ]= [G(z, - \frac{b}{2a}+  \delta)- G(z,- \frac{b}{2a} ) ]\geq   \min a \times \delta^{2}$. These prove the claim.
					
					Finally we consider $\varphi_{2}$ and $\varphi_{3}$. We split the cases into two with small number $\delta$: $\delta < {\sqrt{b^{2}-4ac} \over 2a}$ and $\delta \geq {\sqrt{b^{2}-4ac} \over 2a}$. \\
					\textit{Case 1}. If $\delta < {\sqrt{b^{2}-4ac} \over 2a} = {\varphi_{2}-\varphi_{3}\over 2}$,  
					\begin{equation*}
						\{ s : \ \min_{i=2,3}|s-\varphi_{i}| > \delta \} = \{ s < \varphi_{3} - \delta \} \cup \{ \varphi_{3} + \delta < s < \varphi_{2} - \delta \} \cup \{ \varphi_{2} + \delta < s \} ,	
					\end{equation*}
					where $\{ \varphi_{3} + \delta < s < \varphi_{2} - \delta \}$ is not empty set. For $\{ s > \varphi_{2} + \delta \}$, 
					\begin{equation*}
						\begin{split}
							|G(z,s)| &= \int_{\varphi_{2}}^{s} \p_{s} G(z,t) \dd t = \int_{\varphi_{2}}^{s} (2at + b) \dd t  \\
							&= \int_{\varphi_{2}}^{s} 2a(t - \frac{b}{2a}) \dd t = \int_{\varphi_{2}}^{s} 2a(t - \varphi_{1}) \dd t ,\quad t-\varphi_{1} := r,  \\
							&= \int_{\varphi_{2}-\varphi_{1}}^{s-\varphi_{1}} 2ar \dd t \geq (\min a) (s-\varphi_{2})  (s-\varphi_{1} + \varphi_{2} - \varphi_{1})  \\
							&\geq (\min a) \delta \big(s-\varphi_{2} + (\varphi_{2} - \varphi_{1}) \big) \geq  (\min a) \delta \big(s-\varphi_{2}\big)	\\
							&\geq (\min a) \delta^{2}.
						\end{split}
					\end{equation*}	
					By symmetry, we get same estimate for $\{ s < \varphi_{3} - \delta \}$ case. 
					
					On the other hand, for $\{ \varphi_{3} + \delta < s < \varphi_{2} - \delta \}$, we suffices to consider \\
					$\{ \varphi_{1} \leq s < \varphi_{2} - \delta \}$ because $\{ \varphi_{3} + \delta < s \leq \varphi_{1} \}$ case is same by symmetry. We have a lower bound as
					\begin{equation*}
						\begin{split}
							|G(z,s)| &= \int_{s}^{\varphi_{2}} \p_{s} G(z,t) \dd t = \int_{s}^{\varphi_{2}} 2a(t - \varphi_{1}) \dd t ,\quad t-\varphi_{1} := r,  \\
							&= \int_{s-\varphi_{1}}^{\varphi_{2}-\varphi_{1}} 2ar \dd t \geq (\min a) (\varphi_{2}-s)  (\varphi_{2} - \varphi_{1} + s-\varphi_{1})  \\
							&\geq (\min a) \delta \big( \delta + (s - \varphi_{1}) \big) \\
							&\geq (\min a) \delta^{2}.
						\end{split}
					\end{equation*}	
					
					\noindent \textit{Case 2}. If $\delta \geq {\sqrt{b^{2}-4ac} \over 2a} = {\varphi_{2}-\varphi_{3}\over 2}$, \\
					\begin{equation*}
						\{ s : \ \min_{i=2,3}|s-\varphi_{i}| > \delta \} = \{ s < \varphi_{3} - \delta \} \cup 	\{ s > \varphi_{2} + \delta \} .	
					\end{equation*}
					Note that $\{ \varphi_{3} + \delta < s < \varphi_{2} - \delta \}$ is empty set becasue, if $s - \varphi_{3} = |s - \varphi_{3}| > \delta$,
					\[
					\varphi_{2} - s = |s - \varphi_{2}| = (\varphi_{2} - \varphi_{3}) + (\varphi_{3} - s) < 2\delta + (-\delta) = \delta.  
					\]
					For $\{ s < \varphi_{3} - \delta \}$ and $\{ s > \varphi_{2} + \delta \}$, we already checked that $|G(z,s)|\gtrsim \min |a| \times \delta^{2}$ holds in $(1)$.  
					
					Finally we conclude that 
					$$
					|G(z,s)| = |G(z,s) - G(z, \varphi_{i})|
					\gtrsim \min a  \times \delta^{2} \ \  \text{for } \min_{i=2,3}|s-\varphi_{i}| > \delta,
					$$  
					when $-1 < s < 2$.
					
					\vspace{4pt}
					
					Now we consider $(ii)$. Clearly $\varphi_{4}$ is $C^{0,\gamma}$ for this case. And
					\[
					|G(z,s)| \geq \min \{| b(z)( \frac{-c(z)}{b(z)} + \delta) + c(z)|,| b(z)( \frac{-c(z)}{b(z)}  \delta) + c(z)|\} \geq \min |b| \times \delta. 
					\]
					
					Now we consider $(iii)$. First, if $|b| < \frac{\min |c|}{2}$ then $|\varphi_{5}(z)|\geq \frac{ |c(z)|}{ \min|c|/2 } \geq 2$. Therefore,
					\[
					|G(z,s)| \geq \min \{ |G(z,1)|, |G(z,-1)|\} \geq |c(z)| - |b(z)|
					\geq   \frac{ \min|c|}{2}.
					\]
					Consider the case of $|b| > \frac{\min |c|}{4}$. If $|s- \varphi_{5}(s)|> \delta$ then 
					\begin{eqnarray*}
						|G(z,s)| &\geq& \min \big\{ |  b(z)( \frac{-c(z)}{b(z)} + \delta)+ c(z)  | ,  |  b(z)( \frac{-c(z)}{b(z)} - \delta)+ c(z)  |\big\}\\
						&=& \min |b| \times \delta \geq \frac{\min|c|}{2} \times \delta.
					\end{eqnarray*} \end{proof}

					\begin{lemma} \label{lemma rank 2}
						Assume $\Omega$ is $C^{3}$ (\ref{eta}) and convex (\ref{convexity_eta}), and $\Phi$ is $C^{2,\gamma}_{t,x}$ for some $0<\gamma<1$. We also assume that $\| \Phi \|_{C^{2}_{x}} \ll \delta_{1}$. Let $t^{0}  \geq 0$, $x^{0} \in \bar{\O}$, $v^{0} \in\R^{3}$, and 
						\begin{equation}\begin{split}\label{con_v0}
						\frac{1}{N} \leq |v^{0}| \leq N,  \ 
						\frac{1}{N} \leq |v^{0}_{3}|, \ \frac{1}{N} \leq| n(x^{1}) \cdot e_{3}|, \ \text{and} \  {| n(x^{1}) \cdot v^{1} |} > \delta_{2}>0
						,\end{split}
						\end{equation} 
						where $(x^{1},v^{1}) =\big(x^{1}(t^{0},x^{0},v^{0}),v^{1}(t^{0},x^{0},v^{0})\big)$.

						Fix $k\in \mathbb{N}$ with $ t^{k}\geq t-1$. Then there exists $\e>0$ and finitely many $C^{0,\gamma}$-functions $\psi^{k}_{i}  : B_{\varepsilon}(t,x,v)\rightarrow \R$ with $\|\psi^{k}_{i}  \|_{C^{0,\gamma}_{t,x}}\lesssim_{\delta_{1}, \delta_{2}, \O, N }
						1$ and there exists a constant $\epsilon_{\delta_{1}, \delta_{2}, \O, N }
						>0$, 
						\begin{equation}\label{det_nonzero}
						\begin{split}
						& \text{if } \  \min_{i}|s- \psi_{i}^{k}(t,x,v)|>\delta_{*}\\
						&
						\ \ \ \ \  \text{and} \ (s;t,x,v) \in [ \max\{t-1, t^{k+1}\}, \min \{t- \frac{1}{N}, t^{k}\}] \times  B_{\e}(t^{0},x^{0},v^{0}) , \\
						& \text{then} \  \big|{\p_{ |v|} X} (s;t,x, v ) \times {\p_{\hat{v}_{1}} X}(s;t,x, v ) \big| >   \epsilon_{\delta_{1}, \delta_{2}, \O, N, \delta_{*} }
						.\end{split}\end{equation}
						Here $\hat{v}_{1}= v_{1}/|v|, \ \hat{v}_{2}= v_{2}/|v|.$ 
					\end{lemma}
					It is important that this lower bound $\epsilon_{\delta_{1}, \delta_{2}, \O, N }$	does not depend on time $ t $.

					\begin{proof}
						
						\textit{Step 1. }For $(t^{0},x^{0},v^{0})$ in the assumption we choose $(t,x,v)$ with \\ $|(t,x,v)- (t^{0},x^{0},v^{0})| \ll 1$. 
						
						For each $x$ with $|x-x^{0}|\ll 1$, we set a $C^{1}$-map $Y_{x}: (y_{1}, y_{2}) \mapsto Y(y_{1}, y_{2}) \in \R^{3}$ such that 
						\begin{equation}\label{Y_x}
						Y_{x}(y_{1}, y_{2} ): = x  + y_{1} \mathbf{e}^{0}_{\perp,1} (t^{0},x^{0},v^{0})
						+   y_{2} \mathbf{e}^{0}_{\perp,2} (t^{0},x^{0},v^{0}).
						\end{equation}
						We claim that 
						\begin{equation}\label{lower_Y}
						\Big|\Big(
						\frac{\p Y_{x}(y_{1},y_{2})}{\p y_{1}} \times  \frac{ \p Y_{x} (y_{1},y_{2})}{\p y_{2}}\Big) \cdot R_{x^{1} (t, Y_{x}(y_{1}, y_{2}), v)} v^{1} (t, Y_{x}(y_{1}, y_{2}), v) \Big|\gtrsim_{\O, N, \delta_{1}, \delta_{2}}1.
						\end{equation}
						Using the definition of the \textit{specular basis} (\ref{orthonormal_basis}), we equate LHS of (\ref{lower_Y}) to   \begin{equation}\begin{split}\notag
						&  \big| \big(\mathbf{e}^{0}_{\perp,1} (t^{0},x^{0},v^{0}) \times  \mathbf{e}^{0}_{\perp,2} (t^{0},x^{0},v^{0})\big) \cdot R_{x^{1} (t^{0}, x^{0}, v^{0})} v^{1} (t^{0}, x^{0}, v^{0}) \big|	\\
						&=  \Big| \frac{v^{0}(t^{0},x^{0},v^{0})}{|v^{0}(t^{0},x^{0},v^{0})|}
						\cdot \lim_{s\downarrow t^{1}} V(s;t^{0},x^{0}, v^{0}) \Big| .
						\end{split}\end{equation} 
						For a small potential, we conclude (\ref{lower_Y}).

						\vspace{4pt}

						\noindent\textit{Step 2. }Fix $k$ with $|t^{k}(t,x,v) -t| \leq 1$. Then we fix the orthonormal basis 	$\big\{ \mathbf{e}^{k}_{0}  ,  \mathbf{e}_{\perp,1}^{k}  ,  \mathbf{e}_{\perp,2}^{k}  \big\}$ of (\ref{orthonormal_basis}) with $x^{k} = x^{k}(t,x,v)$, $v^{k}=v^{k}(t,x,v)$. Note that this orthonormal basis $\big\{ \mathbf{e}^{k}_{0}  ,  \mathbf{e}_{\perp,1}^{k}  ,  \mathbf{e}_{\perp,2}^{k}  \big\}$ depends on $(t,x,v)$.


						For $t^{k+1} < s < t^{k}$, recall the forms of $\frac{\p X(s)}{ \p |v|}$ and $\frac{\p X(s)}{\p \hat{v}_{j}}$ in (\ref{X_|v|}) and (\ref{p_X}) of Lemma \ref{decom_lemma}, where 
						\[X(s) = X(s; t^{k}, x^{k}, v^{k}).
						\]
						Recall the \textit{specular matrix} (\ref{specular_matrix}) with $Y=Y_{x}$ in (\ref{Y_x}). Using the \textit{specular basis} (\ref{orthonormal_basis}) and the \textit{specular matrix} (\ref{specular_matrix}), we rewrite (\ref{X_|v|}) and (\ref{p_X}) as 
						\begin{eqnarray}
						&
						&
						\left[\begin{array}{ccc}
						\frac{\p X (s)}{\p |v|}\cdot\mathbf{e}_0^{k} &  \frac{\p X(s)}{\p \hat{v}_{1}}\cdot\mathbf{e}_0^{k} &  \frac{\p X(s)}{\p \hat{v}_{2}}\cdot\mathbf{e}_0^{k}   \\
						\frac{\p X(s)}{\p |v|}\cdot\mathbf{e}_{\perp,1}^{k} &  \frac{\p X(s)}{\p \hat{v}_{1}}\cdot\mathbf{e}_{\perp,1}^{k} &  \frac{\p X(s)}{\p \hat{v}_{2}}\cdot\mathbf{e}_{\perp,1} ^{k}  \\ 
						\frac{\p X(s)}{\p |v|}\cdot\mathbf{e}_{\perp,2} ^{k}&  \frac{\p X(s)}{\p \hat{v}_{1}}\cdot\mathbf{e}_{\perp,2} ^{k}&  \frac{\p X(s)}{\p \hat{v}_{2}}\cdot\mathbf{e}_{\perp,2}  ^{k}	
						\end{array}\right]\notag \\
						&=& 
						\left[\begin{array}{c|cc}
						-(t-s) & 
						\substack{	
							-  |\mathbf{v}^{k}_{{p}^{k}}|    \nabla_{\hat{v}_{1}, \hat{v}_{2}}   t^{k}+
							\nabla_{\hat{v}_{1},\hat{v}_{2}} \mathbf{x}^{k}_{{p}^{k}, \ell} \p_{\ell} \eta_{{p}^{k}}\big|_{x^{k}} \cdot \mathbf{e}_{0}^{k}
							\\ - (t^{k}-s)| \mathbf{v}^{k}_{{p}^{k}}  | 
							\sum_{j=1}^{2}
							\bigg( \sum_{\ell=1}^3 \p_{j}  
							\Big(  \frac{\p_\ell \eta_{{p}^{k}}}{\sqrt{g_{{p}^{k},\ell\ell}}} \Big)\Big|_{x^{k}} \hat{\mathbf{v}}^k_{{p}^{k},\ell} \bigg)
							\nabla_{\hat{v}_{1}, \hat{v}_{2}} \mathbf{x}^k_{{p}^{k},j} 
						}
						\\ \hline
						\mathbf{0}_{2,1} &     
						\left[\begin{array}{cc}
						\mathcal{R}_{1,3}^{k,{p}^{k},Y} & 	\mathcal{R}_{1,4}^{k,{p}^{k},Y} \\
						\mathcal{R}_{2,3}^{k,{p}^{k},Y}  & 	\mathcal{R}_{2,4}^{k,{p}^{k},Y} 
						\end{array} \right]
						- (t^{k}-s) |\mathbf{v}^{k}_{{p}^{k}}|
						\left[\begin{array}{cc}
						\mathcal{R}_{3,3} ^{k,{p}^{k},Y} & 	\mathcal{R}_{3,4}^{k,{p}^{k},Y} \\
						\mathcal{R}_{4,3}^{k,{p}^{k},Y}  & 	\mathcal{R}_{4,4}^{k,{p}^{k},Y} 
						\end{array} \right]
						\end{array}\right]
						\label{sub_R}
						\\
						&&
						\quad + O_{\Omega, N, \delta_{2} } (\| \Phi \|_{C^{2}}).\notag
						\end{eqnarray} 
						From (\ref{lower_bound_v_3}) and (\ref{Jac_billiard}), all the entries of above matrix is bound. By the direct computation, 
						\begin{equation}\begin{split}\label{comp_X_times_X}
						& \p_{|v| } X(s) \times \p_{\hat{v}_{1}} X  (s)  \\
						=&  
						- (t-s)    
						\big\{ \mathcal{R}^{k,{p}^{k},Y}_{1,3} - (t^{k}-s) |\mathbf{v}^{k}_{{p}^{k}}| \mathcal{R}^{k,{p}^{k},Y}_{3,3} 
						\big\} \mathbf{e}_{\perp,2} ^{k}\\
						&
						+  (t-s)   
						\big\{ \mathcal{R}^{k,{p}^{k},Y}_{2,3} - (t^{k}-s) |\mathbf{v}^{k}_{{p}^{k}}| \mathcal{R}^{k,{p}^{k},Y}_{4,3} 
						\big\} \mathbf{e}_{\perp,1}^{k}+ O_{\Omega, N, \delta_{2}} (\| \Phi \|_{C^{2}}).\end{split}
						\end{equation}
						Here $\mathcal{R}^{k,p^{k},Y}_{i,j}, t^{k}, \mathbf{v}^{k}_{p^{k}},$ and $\mathbf{e}^{k}_{\perp, i}$ depend on $(t,x,v)$ but not $s$.

						\vspace{4pt}
						
						\noindent\textit{Step 3. }Recall Lemma \ref{nonzero_sub}. From (\ref{con_v0}) and (\ref{lower_Y}) we can choose non-zero contants $\delta_{1}, \delta_{2},$ and $\delta_{3}$ for a large $N$. Applying Lemma \ref{nonzero_sub} and (\ref{nonzero_sub1}), we conclude that, for some $i \in \{1,2,3,4\}$, 
						\begin{equation}\label{lower_R0}
						|\mathcal{R}^{k,p^{k}, Y}_{i,3}(t^{0},x^{0},v^{0})|> \varrho_{\Omega, N, \delta_{1}, \delta_{2}  }>0.
						\end{equation}
						
						Now we claim that $\mathcal{R}^{k,p^{k}, Y}_{i,j} (t,x,v) \in C^{0,\gamma}_{t,x,v}$ if $|(t,x,v) - (t^{0}, x^{0}, v^{0})| \ll1$. Note that since the domain is convex (\ref{convexity_eta}) and $|n(x^{1}(t^{0}, x^{0}, v^{0})) \cdot v^{1}(t^{0}, x^{0}, v^{0})|> \delta_{2}$ in (\ref{con_v0}), utilizing Lemma \ref{velocity_lemma}, we deduce that if $|(t,x,v) - (t^{0}, x^{0}, v^{0})| \ll1$ then \\ $|n(x^{l}) \cdot v^{l}|\gtrsim \delta_{2}$ for all $1\leq l \leq k$. By Lemma \ref{Jac_billiard}, $(t^{l},x^{l},v^{l})$ is $C^{0,\gamma}$ for all $1\leq l \leq k$. Hence, from (\ref{specular_transition_matrix}) and (\ref{specular_matrix}), we conclude our claim.
						
						Finally we choose a small constant $\e>0$ such that, for some $i\in \{1,2,3,4\}$ satisfying (\ref{lower_R0}),
						\begin{equation}\label{lower_R}
						|\mathcal{R}^{k,p^{k}, Y}_{i,3}(t ,x ,v )|> \frac{\varrho_{\Omega, N, \delta_{1}, \delta_{2}  }}{2} \ \  \ \text{for } |(t,x,v) - (t^{0}, x^{0}, v^{0})| < \e.
						\end{equation} 
						
						\vspace{4pt}
						
						\noindent\textit{Step 4. }  With $N\gg1$, from (\ref{lower_R}), we divide the cases into the follows
						\begin{equation} \label{case1_R}
						|\mathcal{R}^{k,p^{k}, Y}_{i,3}|> \frac{\varrho_{\Omega, N, \delta_{1}, \delta_{2} } }{2}
						\ \ \text{for some } \ i \in \{1,2\},
						\end{equation}
						and 
						\begin{equation}\notag
						|\mathcal{R}^{k,{p}^{k},Y}_{j,3}| \geq  \frac{\varrho_{\Omega, N, \delta _{1}, \delta_{2}   }  }{2} \ \ \text{for some } \ j \in \{3,4\}.
						\end{equation}
						
						\noindent We split the first case (\ref{case1_R}) further into two cases as  
						\begin{equation}\label{case1_R-1}
						\min_{i=1,2} |\mathcal{R}^{k,p^{k}, Y}_{i,3}|> \frac{\varrho_{\Omega, N, \delta_{1}, \delta_{2} } }{2}   \ \   \text{and} \ \ \max_{i=1,2} |\mathcal{R}^{k,p^{k}, Y}_{i+2,3}| < \frac{\varrho_{\Omega, N, \delta_{1}, \delta_{2} } }{4N},
						\end{equation}
						and 	
						\begin{equation}\notag\label{case1_R-2}
						\min_{i=1,2} |\mathcal{R}^{k,p^{k}, Y}_{i,3}|> \frac{\varrho_{\Omega, N, \delta _{1}, \delta_{2}} }{2}   \ \   \text{and} \ \ \max_{i=1,2} |\mathcal{R}^{k,p^{k}, Y}_{i+2,3}| \geq \frac{\varrho_{\Omega, N, \delta_{1}, \delta_{2} } }{4N}.
						\end{equation}
						
						\noindent Set the other case\begin{equation}\label{case2_R}
						|\mathcal{R}^{k,{p}^{k},Y}_{j,3}| \geq  \frac{\varrho_{\Omega, N, \delta_{1}, \delta_{2}    }  }{2} \ \ \text{for some } \ j \in \{3,4\}.
						\end{equation}
						Then clearly (\ref{case1_R-1}) and (\ref{case2_R}) cover all the cases.

						
						\vspace{4pt}
						
						\noindent\textit{Step 5. } We consider the case of (\ref{case1_R-1}). 
						Then, from (\ref{comp_X_times_X}), 
						\begin{equation}\label{det_expansion}
						\begin{split}
						&|  \p_{|v| } X (s)\times \p_{\hat{v}_{1}} X(s)  |\\
						&\geq  \big| |v^{k}| \mathcal{R}^{k,p^{k}, Y}_{i+2,3} (t^{k}-s)   - \mathcal{R}^{k,p^{k}, Y}_{i ,3}    \big|(t-s) 
						+O_{\O,N,\delta} (\|  \Phi \|_{C^{2}})\\
						&= \big|
						\underbrace{| {v}^{k} | \mathcal{R}_{i+2,3}^{k,p^{k}, Y} (t-s)  + \big[ - \mathcal{R} ^{k,p^{k}, Y}_{i ,3} + (t^{k}-t) |v^{k}| \mathcal{R} ^{k,p^{k}, Y}_{i+2,3} \big] } \big|(t-s)+ O_{\O,\delta_{2},N} (\|  \Phi \|_{C^{2}}).
						\end{split}
						\end{equation}
						
						We define 
						\begin{equation}\label{tilde_s}
						\tilde{s} = t-s,
						\end{equation}
						and set 
						\begin{equation}\label{abc_R}
						a\equiv 0 , \quad b:= |{v}^{k} | \mathcal{R} ^{k,p^{k}, Y}_{i+2,3} , \quad\text{and}\quad  c:=- 
						\mathcal{R} ^{k,p^{k}, Y}_{i ,3} + (t^{k}-t) | {v}^{k} | \mathcal{R}_{i+2,3} ^{k,p^{k}, Y}.
						\end{equation}
						Note that $\mathcal{R}^{k,p^{k}, Y}_{i,3},$ $\mathcal{R}^{k,p^{k}, Y}_{i+2,3}$, $|v^{k}|$, and $t^{k}$ only depend on $(t,x,v)$.
						
						Hence we regard the underbraced term of (\ref{det_expansion}) as an affine function of $\tilde{s}$
						\begin{equation}\label{affine_tilde_s}
						b(t,x,v) \tilde{s}+ c(t,x,v).
						\end{equation}
						Note that from (\ref{case1_R-1})
						\begin{equation}\notag
						|c(t,x,v)| \geq  \frac{\varrho_{\Omega, N, \delta  _{1}, \delta_{2}}}{2}  - N \frac{ \varrho_{\Omega, N, \delta  _{1}, \delta_{2}
							}
						}{4N}
						\geq \frac{ \varrho_{\Omega, N, \delta _{1}, \delta_{2}
							}}{4}.
							\end{equation}
							Now we apply $(iii)$ of Lemma \ref{zero_poly}. With $\varphi_{5} (t,x,v)$ in (\ref{psi5}), if $|\tilde{s} - \varphi_{5} (t,x,v)|> \delta_{*}$ then $|b(t,x,v) \tilde{s}+ c(t,x,v)| \geq \frac{\varrho_{\Omega, N, \delta 
								}}{4} \times \delta_{*}$. We set
								\begin{equation}\label{psi5_phi5}
								\psi_{5} (t,x,v) = t- \varphi_{5} (t,x,v) .
								\end{equation}
								From (\ref{tilde_s}), 
								\begin{equation}\label{lower_bound_1}
								\text{if}  \ |s - \psi_{5} (t,x,v)|> \delta_{*}, \ \text{then} \  |b(t,x,v)  (t-s)+ c(t,x,v)| \geq \frac{\varrho_{\Omega, N, \delta  _{1}, \delta_{2}
									}}{4} \times \delta_{*}.
									\end{equation}
									
									%
									%
									%
									%
									

									\vspace{4pt}
									
									Now we consider the case of (\ref{case2_R}). From (\ref{comp_X_times_X}),
									\begin{equation}\label{det_expansion2}
									\begin{split}
									|\p_{|v|} X(s) \times \p_{\hat{v}_{1}}X(s)|
									&\geq \big| 
									| {v}^{k} | \mathcal{R}^{k,{p}^{k},Y}_{j,3} (t-s) 	\\
									&\quad + \big[ -  \mathcal{R}^{k,{p}^{k},Y}_{j-2,3}
									+ (t^{k}-t)| {v}^{k} | \mathcal{R}^{k,{p}^{k},Y}_{j,3} \big]
									\big| (t-s) + O_{\O, N ,\delta} (\| \Phi \|_{C^{2}}).
									\end{split}
									\end{equation} 
									We set $\tilde{s}$ as (\ref{tilde_s}) and
									\begin{equation}\label{abc_R2} 
									a\equiv 0 , \quad b:= |{v}^{k} | \mathcal{R} ^{k,p^{k}, Y}_{j,3} , \quad\text{and}\quad  c:=- 
									\mathcal{R} ^{k,p^{k}, Y}_{j-2 ,3} + (t^{k}-t) | {v}^{k} | \mathcal{R}_{j,3} ^{k,p^{k}, Y}.
									\end{equation}

									From (\ref{case2_R}) and (\ref{abc_R2})
									\begin{equation}\notag
									|b(t,x,v)| \geq   \frac{\varrho_{\Omega, N, \delta _{1}, \delta_{2} }}{8N^{2}}.
									\end{equation}
									We apply $(ii)$ of Lemma \ref{zero_poly} to this case: With $\varphi_{4}(t,x,v)$ in (\ref{psi4}), we set
									\begin{equation}\label{psi4_phi4}
									\psi_{4} (t,x,v) = t- \varphi_{4} (t,x,v),
									\end{equation}	
									and 
									\begin{equation}\label{lower_bound_2}
									\text{if}  \ |s - \psi_{4} (t,x,v)|> \delta_{*}, \ \text{then} \  |b(t,x,v)  (t-s)+ c(t,x,v)| \gtrsim  \frac{\varrho_{\Omega, N, \delta _{1}, \delta_{2} }}{8N^{2}}
									\times
									\delta_{*}
									.
									\end{equation}

									\vspace{4pt}

									Finally, from (\ref{lower_bound_1}), (\ref{det_expansion}), (\ref{lower_bound_2}), and (\ref{det_expansion2}), we conclude the proof of Lemma \ref{lemma rank 2}.\end{proof}

								\begin{lemma} \label{lemma rank 3} 
									Assume $\Omega$ is $C^{3}$ (\ref{eta}) and convex (\ref{convexity_eta}), and $\Phi$ is $C^{2,\gamma}_{t,x}$ for some $0<\gamma<1$. Assume the conditions of the statement of Lemma \ref{nonzero_sub}. 
									
									Let a $C^{1}$-map $Y_{x}: (y_{1},y_{2}) \mapsto Y_{x}(y_{1},y_{2}) \in \bar{\O}$ with $Y_{x}(0, 0) = x$ and $\| Y\|_{C^{1}_{x,y_{1},y_{2}}}\lesssim1$. We assume that 
									\begin{equation}
									\Big|\Big(\frac{\p Y_{x^{0}}(0,0)}{\p y_{1}} \times \frac{\p Y_{x^{0}}(0,0)}{\p y_{2}}\Big) \cdot 
									R_{ x^{1}(t, x^{0} ,v^{0})  } v^{1}(t, x^{0} ,v^{0})  \Big| > \delta_{3}>0.
									\end{equation} 
									%
									%
									
									For $k\in \mathbb{N}$ with $ t^{k}\geq t-1$, there exists $\e>0$ and finitely many $C^{0,\gamma}$-functions $\psi^{k}_{i}  : B_{\varepsilon}(t,x,v)\rightarrow \R$ with $\|\psi^{k}_{i}  \|_{C^{0,\gamma}_{t,x}}\lesssim 1$, and there exists a constant $\epsilon_{\delta_{1}, \delta_{2},\delta_{3}, N, \O }
									>0$ and $\{\zeta_{1}, \zeta_{2}\} \subset \{  \hat{v} _{1},  \hat{v}_{2}, y_{1}, y_{2}\}$ such that 
									\begin{eqnarray} 
									&& \text{if } \  \min_{i}|s- \psi_{i}^{k}(t, 
									Y_{x}(y_{1}, y_{2})
									,v)|>\delta_{*}\nonumber\\
									&& \ \ \ 
									\  \text{and} \ (s;t,Y_{x}(y_{1}, y_{2}),v) \in [ \max\{t-1, t^{k+1}\}, \min \{t- \frac{1}{N}, t^{k}\}] \times  B_{\e}(t^{0},x^{0},v^{0}), \label{lower_2} \\
									&&\text{then}  \  \det \left(
									\frac{\p X (s; t, Y_{x}(y_{1}, y_{2}), |v|,\hat{v} _{1},  \hat{v}_{2}  )}{\p (|v |,\zeta_{1}, \zeta_{2} )}\right)> \epsilon_{\delta_{1}, \delta_{2},  \delta_{3},N,  \O, \delta_{*}}
									>0
									.\nonumber\end{eqnarray}
								\end{lemma}
								\begin{proof} 
									\textit{Step 1.} Recall the \textit{specular basis} $\{ \mathbf{e}^{k}_{0},\mathbf{e}^{k}_{\perp,1},\mathbf{e}^{k}_{\perp,2} \}$ in (\ref{orthonormal_basis}) with \\
									$x^{k} = x^{k} (t, Y_{x}(y_{1},y_{2}), |v|, \hat{v}_{1}, \hat{v}_{2})$ and $v^{k} = v^{k} (t, Y_{x}(y_{1},y_{2}), |v|, \hat{v}_{1}, \hat{v}_{2})$,
									\begin{eqnarray*}
										&&\frac{\p X (s; t, Y_{x}(y_{1},y_{2}), |v|, \hat{v}_{1}, \hat{v}_{2})}{ \p (|v |, y_{1}, y_{2}, \hat{v}_{1} , \hat{v}_{2} )} \\
										&=& \left[\begin{array}{ccc}
											\mathbf{e}_{0}^{k} & \mathbf{e}_{\perp,1}^{k} & \mathbf{e}_{\perp,2}^{k}
										\end{array}\right]
										\underbrace{\left[\begin{array}{ccccc}
												\frac{\p X}{\p |v |} \cdot \mathbf{e}_{0}^{k} & \frac{\p X}{\p y_{1}} \cdot \mathbf{e}_{0}^{k} & \frac{\p X}{\p y_{2}} \cdot \mathbf{e}_{0}^{k}& \frac{\p X}{\p \hat{v} _{1}} \cdot \mathbf{e}_{0}^{k} & \frac{\p X}{\p \hat{v} _{2}} \cdot \mathbf{e}_{0}^{k} 
												\\
												\frac{\p X}{\p |v |} \cdot \mathbf{e}_{\perp,1}^{k} & \frac{\p X}{\p y_{1}} \cdot \mathbf{e}_ {\perp,1}^{k} & \frac{\p X}{\p y_{2}} \cdot \mathbf{e}_ {\perp,1}^{k}& \frac{\p X}{\p \hat{v}_{1}} \cdot \mathbf{e}_ {\perp,1} ^{k}& \frac{\p X}{\p \hat{v}_{2}} \cdot \mathbf{e}_ {\perp,1}^{k}
												\\
												\frac{\p X}{\p |v |} \cdot \mathbf{e}_{\perp,2}^{k} & \frac{\p X}{\p y_{1}} \cdot \mathbf{e}_ {\perp,2} ^{k}& \frac{\p X}{\p y_{2}} \cdot \mathbf{e}_ {\perp,2}^{k} & \frac{\p X}{\p \hat{v}_{1}} \cdot \mathbf{e}_ {\perp,2} ^{k}& \frac{\p X}{\p \hat{v}_{2}} \cdot \mathbf{e}_ {\perp,2}^{k}
											\end{array} \right]} .
									\end{eqnarray*}
									From (\ref{X_|v|}) and (\ref{p_X}), using the \textit{specular basis} (\ref{orthonormal_basis}) and the \textit{specular matrix} (\ref{specular_matrix}), we rewrite the underbraced term as
									\begin{eqnarray*}
										{  \left[\begin{array}{c|c}
												- (t-s) & \substack{	
													-  | {v}^{k} |    \nabla_{\hat{v} _{1}, \hat{v} _{2},y_{1}, y_{2}}   t^{k}+
													\nabla_{\hat{v} _{1}, \hat{v} _{2}, y_{1}, y_{2}}   \mathbf{x}^{k}_{ {p}^{k}, \ell} \p_{\ell} \eta_{ {p}^{k}} \cdot \mathbf{e}_{0}^{k}
													\\ - (t^{k}-s) | {v}^{k} |  
													\sum_{j=1}^{2}
													\bigg( \sum_{\ell=1}^3  \frac{\p}{\p {\mathbf{x}^{k}_{ {p}^{k},j}}}
													\Big[  \frac{\p_\ell \eta_{ {p}^{k}}}{\sqrt{g_{ {p}^{k},\ell\ell}}} \Big] \hat{\mathbf{v}}^k_{{p}^{k},\ell} \bigg)
													\nabla_{\hat{v} _{1}, \hat{v} _{2}, y_{1}, y_{2}} \mathbf{x}^k_{ {p}^{k},j} 
												}\\ \hline
												\substack{ 0 \\ 
													0 }
												& *_{2 \times 4} 
											\end{array} \right]} 
										+ O_{\delta,N} (\| \Phi \|_{C^{2}}),
									\end{eqnarray*}
									where the lower right $2 \times 4$-submatrix equals
									\begin{equation}\label{R_24}\begin{split}
									&\left[\begin{array}{cccc}
									\mathcal{R}_{1,1}^{k,{p}^{k},Y_{x}} & 	\mathcal{R}_{1,2}^{k,{p}^{k},Y_{x}}  & 	\mathcal{R}_{1,3}^{k,{p}^{k},Y_{x}} & 	\mathcal{R}_{1,4}^{k,{p}^{k},Y_{x}}\\
									\mathcal{R}_{2,1} ^{k,{p}^{k},Y_{x}}& 	\mathcal{R}_{2,2} ^{k,{p}^{k},Y_{x}} & 	\mathcal{R}_{2,3}^{k,{p}^{k},Y_{x}} & 	\mathcal{R}_{2,4}^{k,{p}^{k},Y_{x}}
									\end{array} \right]\\
									&	- (t^{k}-s) | {v}^{k} |
									\left[\begin{array}{cccc}
									\mathcal{R}_{3,1} ^{k,{p}^{k},Y_{x}}& 	\mathcal{R}_{3,2}^{k,{p}^{k},Y_{x}} &	\mathcal{R}_{3,3} ^{k,{p}^{k},Y_{x}}& 	\mathcal{R}_{3,4}^{k,{p}^{k},Y_{x}}\\
									\mathcal{R}_{4,1} ^{k,{p}^{k},Y_{x}}& 	\mathcal{R}_{4,2}^{k,{p}^{k},Y_{x}}	 & 	\mathcal{R}_{4,3} ^{k,{p}^{k},Y_{x}}& 	\mathcal{R}_{4,4}^{k,{p}^{k},Y_{x}}
									\end{array} \right].
									\end{split}\end{equation}
									Here $\mathcal{R}^{k,p^{k},Y_{x}}_{i,j}$ is defined in (\ref{specular_matrix}) with $x^{k} = x^{k} (t, Y_{x}(y_{1},y_{2}), |v|, \hat{v}_{1}, \hat{v}_{2})$ and \\ $v^{k} = v^{k} (t, Y_{x}(y_{1},y_{2}), |v|, \hat{v}_{1}, \hat{v}_{2})$.
									
									\vspace{8pt}
									
									\noindent\textit{Step 2.} From Lemma \ref{nonzero_sub}, there exist $i<j$ such that (\ref{nonzero_sub2}) holds. We choose $\zeta_{1}, \zeta_{2}$ to be the $i^{th}$ component and $j^{th}$ component of $\{y_{1}, y_{2}, \hat{v}_{1}, \hat{v}_{2}\}$. For the sake of simplicity, we abuse the notation as $$\left[\begin{array}{cc} 
									\mathcal{R}^{k,{p}^{k},Y_{x}}_{3, \zeta_{1}} & \mathcal{R}^{k,{p}^{k},Y_{x}}_{3,  \zeta_{2}}  \\
									\mathcal{R}^{k,{p}^{k},Y_{x}}_{4,  \zeta_{1}} & \mathcal{R}^{k,{p}^{k},Y_{x}}_{4,  \zeta_{2}} 
									\end{array}\right]=\left[\begin{array}{cc} \mathcal{R}^{k,{p}^{k},Y_{x}}_{3,i} & \mathcal{R}^{k,{p}^{k},Y_{x}}_{3,j}  \\
									\mathcal{R}^{k,{p}^{k},Y_{x}}_{4,i} & \mathcal{R}^{k,{p}^{k},Y_{x}}_{4,j} 
									\end{array}\right].$$ 
									
									\noindent Note that 
									\begin{eqnarray*}
										&&\det \left(\frac{\p X(s; t, Y_{x}(y_{1},y_{2}), |v|, \hat{v}_{1}, \hat{v}_{2})}{\p (|v|, \zeta_{1}, \zeta_{2})}\right)\\
										&=& \det \Bigg(
											{\left[\begin{array}{c|c}
													 (s-t) & \substack{	
														-  | {v}^{k} |    \nabla_{\zeta_{1}, \zeta_{2}}   t^{k}+
														\nabla_{\zeta_{1}, \zeta_{2}}   \mathbf{x}^{k}_{ {p}^{k}, \ell} \p_{\ell} \eta_{ {p}^{k}} \cdot \mathbf{e}_{0}^{k}
														\\ - (t^{k}-s) | {v}^{k} |  
														\sum_{j=1}^{2}
														\bigg( \sum_{\ell=1}^3  \frac{\p}{\p {\mathbf{x}^{k}_{ {p}^{k},j}}}
														\Big[  \frac{\p_\ell \eta_{ {p}^{k}}}{\sqrt{g_{ {p}^{k},\ell\ell}}} \Big] \hat{\mathbf{v}}^k_{{p}^{k},\ell} \bigg)
														\nabla_{\zeta_{1}, \zeta_{2}} \mathbf{x}^k_{ {p}^{k},j} 
													}\\ \hline
													\substack{ 0 \\  \\
														0 }
													&
													\left[\begin{array}{cc} 
														\mathcal{R}^{k,{p}^{k},Y_{x}}_{1, \zeta_{1}} & \mathcal{R}^{k,{p}^{k},Y_{x}}_{1,  \zeta_{2}}  \\
														\mathcal{R}^{k,{p}^{k},Y_{x}}_{2,  \zeta_{1}} & \mathcal{R}^{k,{p}^{k},Y_{x}}_{2,  \zeta_{2}} 
													\end{array}\right]
													- (t^{k}-s)|v^{k}|
													\left[\begin{array}{cc} 
														\mathcal{R}^{k,{p}^{k},Y_{x}}_{3, \zeta_{1}} & \mathcal{R}^{k,{p}^{k},Y_{x}}_{3,  \zeta_{2}}  \\
														\mathcal{R}^{k,{p}^{k},Y_{x}}_{4,  \zeta_{1}} & \mathcal{R}^{k,{p}^{k},Y_{x}}_{4,  \zeta_{2}} 
													\end{array}\right] 
												\end{array} \right] }  \\
											&& + \big[ O_{\delta,N} (\| \Phi \|_{C^{2}}) \big]_{3\times 3}
											\Bigg). 
									\end{eqnarray*}
									
									\noindent From (\ref{lower_bound_v_3}) and (\ref{Jac_billiard}), all the entries of above matrix is bound and hence the determinant of above matrix equals
									%
									\begin{equation}\label{sub_det}
									\begin{split}
									&-(t-s) \underbrace{\det \left(   \left[\begin{array}{cc} 
										\mathcal{R}^{k,{p}^{k},Y_{x}}_{1, \zeta_{1}} & \mathcal{R}^{k,{p}^{k},Y_{x}}_{1,  \zeta_{2}}  \\
										\mathcal{R}^{k,{p}^{k},Y_{x}}_{2,  \zeta_{1}} & \mathcal{R}^{k,{p}^{k},Y_{x}}_{2,  \zeta_{2}} 
										\end{array}\right]
										- (t^{k}-s)|v^{k}|
										\left[\begin{array}{cc} 
										\mathcal{R}^{k,{p}^{k},Y_{x}}_{3, \zeta_{1}} & \mathcal{R}^{k,{p}^{k},Y_{x}}_{3,  \zeta_{2}}  \\
										\mathcal{R}^{k,{p}^{k},Y_{x}}_{4,  \zeta_{1}} & \mathcal{R}^{k,{p}^{k},Y_{x}}_{4,  \zeta_{2}} 
										\end{array}\right]\right)} \\
									&\quad + O(\| \Phi  \|_{C^{2}}).
									\end{split}
									\end{equation}
									The underbraced term equals
									\begin{equation}\begin{split}\label{quadratic}
									&(t^{k}-s)^{2}| {v}^{k} |^{2}\det
									\left(\left[\begin{array}{cc}
									\mathcal{R}^{k,{p}^{k},Y_{x}}_{3,\zeta_{1}} & \mathcal{R}^{k,{p}^{k},Y_{x}}_{3,\zeta_{2}}\\
									\mathcal{R}^{k,{p}^{k},Y_{x}}_{4,\zeta_{1}} & \mathcal{R}^{k,{p}^{k},Y_{x}}_{4,\zeta_{2}}
									\end{array}\right] \right)+ \det \left(\left[\begin{array}{cc}
									\mathcal{R}^{k,{p}^{k},Y_{x}}_{1,\zeta_{1}} & \mathcal{R}^{k,{p}^{k},Y_{x}}_{1,\zeta_{2}}\\
									\mathcal{R}^{k,{p}^{k},Y_{x}}_{2,\zeta_{1}} & \mathcal{R}^{k,{p}^{k},Y_{x}}_{2,\zeta_{2}}
									\end{array}\right]  \right)
									\\
									&  \ \  
									- (t^{k}-s) | {v}^{k} | \Big(
									\mathcal{R}^{k,{p}^{k},Y_{x}}_{3, \zeta_{1}} \mathcal{R}^{k,{p}^{k},Y_{x}}_{2, \zeta_{2}} + \mathcal{R}^{k,{p}^{k},Y_{x}}_{1, \zeta_{1}} \mathcal{R}^{k,{p}^{k},Y_{x}}_{4, \zeta_{2}}
									\\
									& \ \ 
									- \mathcal{R}^{k,{p}^{k},Y_{x}}_{1, \zeta_{2}} \mathcal{R}^{k,{p}^{k},Y_{x}}_{4, \zeta_{1}}
									- \mathcal{R}^{k,{p}^{k},Y_{x}}_{3, \zeta_{2}} \mathcal{R}^{k,{p}^{k},Y_{x}}_{2, \zeta_{1}}
									\Big).\end{split}
									\end{equation}
									We define 
									\[
									\tilde{s} = t^{k} (t,Y_{x}(y_{1},y_{2}),v) -s.
									\]
									And we regard (\ref{quadratic}) as a quadratic polynomial of $\tilde{s}$. Then the coefficient of $\tilde{s}^{2}$ is $$ |v^{k}|^{2}\Big|\det
									\left(\left[\begin{array}{cc}
									\mathcal{R}^{k,{p}^{k},Y_{x}}_{3,\zeta_{1}} & \mathcal{R}^{k,{p}^{k},Y_{x}}_{3,\zeta_{2}}\\
									\mathcal{R}^{k,{p}^{k},Y_{x}}_{4,\zeta_{1}} & \mathcal{R}^{k,{p}^{k},Y_{x}}_{4,\zeta_{2}}
									\end{array}\right] \right)\Big| ,$$
									where depends only on $(t, Y_{x}(y_{1},y_{2}),v)$. From (\ref{nonzero_sub2}) and $|v|\geq \frac{1}{N}$, we have a lower bound of $
									\frac{\delta_{2}}{N^{2}}.$

									Now we apply $(iii)$ of Lemma \ref{zero_poly}: There exist $C^{1}-$functions $\psi_{1} (t, Y_{x}(y_{1},y_{2}),v)$, $\psi_{2}(t, Y_{x}(y_{1},y_{2}),v)$, $\psi_{3}(t, Y_{x}(y_{1},y_{2}),v)$ such that if $|\tilde{s} - \psi_{i} (t, Y_{x}(y_{1},y_{2}),v)|> \delta_{2}$ for all $i=1,2,3$ then the absolute value of $(\ref{quadratic})$ has a positive lower bound. Set 
									\[
									\psi_{i}   = t^{k} - \phi_{i} .
									\]
									Using $|t-s|> \frac{1}{N}$ and (\ref{sub_det}) we prove (\ref{lower_2}).
								\end{proof}	
								
								
								\begin{lemma}\label{OV}
									Assume $\Omega$ is convex in (\ref{convexity_eta}) and $\|\Phi \|_{C^{1}} \ll 1$. Choose $N\gg 1$, $0< \delta \ll1$ and then choose $\delta_{1}= \delta_{1} (\Omega,N, \delta, \| \nabla_{x}\Phi \|_{\infty})>0$ as in (\ref{delta_1}). There exists collections of open subsets $\{\mathcal{O}_{i}\}_{i=1}^{I_{\Omega,N, \delta,\delta_{1}}}$ of $\O$ and $\{\mathcal{V}_{i}(\mathbf{q}_{1},\mathbf{q}_{2})\}_{i=1}^{I_{\Omega,N, \delta,\delta_{1}}}$ of $\mathbb{R}^{3}$, where $\mathbf{q}_{1}$ and $\mathbf{q}_{2}$ are two independent vectors in $\mathbb{R}^{3}$, with $I_{\Omega,N, \delta,\delta_{1}}< \infty$ such that $\bar{\Omega} \subset\bigcup_{i} \mathcal{O}_{i}$ and $\int_{\R^{3} \backslash \mathcal{V}_{i}(\mathbf{q}_{1},\mathbf{q}_{2})} e^{- |v|^{2}/100} \dd v \leq O_{\Omega}(\frac{1}{N}) + O_{\Omega}(\delta_{1})$. Moreover,
									\begin{equation}\label{finite_k_V}
									K_{i}:= \sup \{ k \in \mathbb{N}: t^{k}(t,x,v) \geq T, \ (t,x,v) \in [T,T+1] \times \mathcal{O}_{i} \times \R^{3} \backslash \mathcal{V}_{i}(\mathbf{q}_{1},\mathbf{q}_{2}) \}
									< \infty.
									\end{equation}
									If $(x,v) \in \mathcal{O}_{i} \times \R^{3} \backslash \mathcal{V}_{i}(\mathbf{q}_{1},\mathbf{q}_{2})$ for some $i$, then 
									\begin{equation}\label{nv_lower}
									|n(x^{1}(t,x,v)) \cdot v^{1}(t,x,v)|> \min \big\{ \frac{\delta_{1}}{4},
									C_{\Omega,N, \| \nabla_{x}\Phi \|_{\infty}} \delta
									\big\}
									\end{equation}
									and 
									\begin{equation}\label{degene_lower}
									\big| (\mathbf{q}_{1} \times \mathbf{q}_{2}) \cdot v \big| \geq  \frac{1}{N}.									\end{equation}
								\end{lemma}
								\begin{proof}We construct $\mathcal{O}_{i}$ and $\mathcal{V}_{i}(\mathbf{q}_{1},\mathbf{q}_{2})$. Choose $\mathfrak{x} \in \bar{\Omega}$ and $v \in \R^{3}$ with $\frac{1}{N} \leq |v| \leq N$, $\frac{1}{N}\leq |v_{3}|$ for $N\gg1$. We split the cases $|v||t-t^{1}(t,\mathfrak{x},v)|\geq \delta$ and $|v||t-t^{1}(t, \mathfrak{x},v)|\leq 2\delta$ for some $0<\delta\ll1$.  For the first case, from (\ref{upper_delta_t}),
									\[
									\delta\leq |v||t-t^{1}(t,\mathfrak{x},v)| \lesssim_{\Omega} \frac{|v^{1}(t,\mathfrak{x},v) \cdot n(x^{1}(t,\mathfrak{x},v))|}{|v^{1}(t,\mathfrak{x},v)|},
									\]
									and hence $|v^{1}(t,\mathfrak{x},v) \cdot n(x^{1} (t,\mathfrak{x},v))|>C_{\Omega,N, \| \nabla_{x}\Phi \|_{\infty}} \delta$. For the second case,
									\begin{eqnarray*}
										&&|n(x^{1}(t,\mathfrak{x},v)) \cdot v^{1}(t,\mathfrak{x},v)| \\ 
										&=&|n(\mathfrak{x}) \cdot v| + O_{\| \eta \|_{C^{3}}} (|x^{1} (t,\mathfrak{x},v) - \mathfrak{x}|) \times \{|v| + \| \nabla_{x} \Phi \|_{\infty}\}
										+ \| \nabla_{x} \Phi \|_{\infty}\\
										&=& |n(\mathfrak{x}) \cdot v| + O_{N,\| \eta \|_{C^{3}}} (\delta )  
										+ O_{N,\| \eta \|_{C^{3}}} (\| \nabla_{x} \Phi \|_{\infty}),
									\end{eqnarray*}
									where we have used the fact $|x^{1} (t,\mathfrak{x},v) - \mathfrak{x}|= |v||t-t^{1}| + \| \nabla_{x} \Phi \|_{\infty}|t-t^{1}|^{2}$. Let us choose 
									\begin{equation}\label{delta_1}
									\delta_{1} =2 \big|O_{N,\| \eta \|_{C^{3}}} (\delta )  
									+  O_{N,\| \eta \|_{C^{3}}} (\| \nabla_{x} \Phi \|_{\infty})\big| \ \ \text{ for } \ \ \delta\ll_{N,\Omega} 1, \  \| \nabla_{x} \Phi \|_{\infty} \ll_{N,\Omega} 1.
									\end{equation}
									Then $ |n(x^{1}(t,\mathfrak{x},v)) \cdot v^{1}(t,\mathfrak{x},v)|  \geq \frac{\delta_{1}}{2}$ for $|n(\mathfrak{x}) \cdot v|\geq \delta_{1}$. And, condition (\ref{degene_lower}) is independent to position $x$. Note that, from Lemma \ref{global to local}, $(t^{1},x^{1},v^{1})$ is continuous locally. Therefore, we can choose $r_{\mathfrak{x}}>0$ such that if $x \in B(\mathfrak{x}, r_{\mathfrak{x}}) \cap \bar{\Omega}$, $\frac{1}{N} \leq |v| \leq N, \frac{1}{N} \leq |v_{3}|$, $|n(\mathfrak{x}) \cdot v|\geq 2\delta_{1}$, and $\big|\big( \mathbf{q}_{1} \times \mathbf{q}_{2} \big) \cdot v \big| \geq \frac{1}{N}  $, then we have (\ref{nv_lower}) and (\ref{degene_lower}). Since $\bar{\Omega}$ is a compact subset of $\R^{3}$, we extract finite points $\{\mathfrak{x}_{i}\}_{i=1}^{I_{\Omega,N ,\delta, \delta_{1}}}$ with $I_{\Omega,N ,\delta, \delta_{1}}< \infty$ such that $\{B(\mathfrak{x}_{i}, r_{\mathfrak{x}_{i}})\}_{i=1}^{I_{\Omega,N ,\delta, \delta_{1}}}$ is an open covering of $\bar{\Omega}$. We define
									
									\begin{equation}\label{mathcal_O}
									\begin{split} 
									\mathcal{O}_{i}  &: = B(\mathfrak{x}_{i}, r_{\mathfrak{x}_{i}}), \\
									\mathcal{V}_{i}(\mathbf{q}_{1}, \mathbf{q}_{2})   &:= \big\{v \in \R^{3}: |v| \leq \frac{1}{N} \ \ \text{or} \ \ 
									|v|\geq N \ \ \text{or} \ \ |v_{3}| \leq \frac{1}{N}  \ \ \text{or} \ \ 
									|n(\mathfrak{x}) \cdot v| \leq 2 \delta_{1}, \\
									& \quad \quad \text{or} \ \ \big| \big( \mathbf{q}_{1} \times \mathbf{q}_{2} \big) \cdot v \big| \leq \frac{1}{N}  
									\big\},
									\end{split}
									\end{equation}
									for some two independent vectors $\mathbf{q}_{1}, \mathbf{q}_{2}$ in $\mathbb{R}^{3}$. Clearly we already proved that if $(x,v) \in \mathcal{O}_{i} \times \R^{3} \backslash \mathcal{V}_{i}(\mathbf{q}_{1}, \mathbf{q}_{2})$ for some $i=1,2,\cdots, I_{\Omega,N,\delta,\delta_{1}}$ then we have (\ref{nv_lower}). Moreover, $\int_{\mathcal{V}_{i}(\mathbf{q}_{1}, \mathbf{q}_{2})}
									e^{-|v|^{2}/100}  \dd v< O(\frac{1}{N}) + O(\delta_{1})$ from our construction. From (\ref{upper_k}), we prove (\ref{finite_k_V}).
								\end{proof}
								
								
								Now we are ready to prove the main theorem.
								
								\begin{theorem}\label{prop_full_rank} 
									Fix any arbitrary $(t,x,v) \in [T,T+1] \times \Omega \times \R^{3}$. Recall  $M,\delta, \delta_{1}$ and $\mathcal{O}_{i}, \mathcal{V}_{i}(\hat{\mathbf{e}}_{1},\hat{\mathbf{e}}_{2})$, which are chosen in Lemma \ref{OV}. For each $i=1,2, \cdots, I_{\Omega,N, \delta,\delta_{1}}$, there exists $\delta_{2}>0$ and $C^{0,\gamma}$-function $\psi^{\ell_{0}, \vec{\ell}, i, k}$ for $k \leq K_{i}$ in (\ref{finite_k_V}) where $\psi^{\ell_{0}, \vec{\ell}, i, k}$ is defined locally around $(T+ \delta_{2} \ell_{0}, X(T+ \delta_{2} \ell_{0};t,x,v), \delta_{2} \vec{\ell})$ with $(\ell_{0},\vec{\ell}) = (\ell_{0}, \ell_{1}, \ell_{2},\ell_{3}) \in \{ 0,1, \cdots, \lfloor\frac{1}{\delta_{2}}\rfloor+1\} \times \{- \lfloor\frac{N}{\delta_{2}}\rfloor-1, \cdots, 0 , \cdots, \lfloor\frac{N}{\delta_{2}}\rfloor+1 \}^{3}$ and \\
									$ \| \psi^{\ell_{0}, \vec{\ell}, i, k}  \|_{C^{0,\gamma}}   \leq C_{N,\Omega,\delta,\delta_{1}, \delta_{2}, \| \Phi \|_{C^{2,\gamma}} }< \infty$. 
									
									Moreover, if 
									\begin{equation}\label{Xs_OVi}
									(X(s;t,x,v),u) \in \mathcal{O}_{i} \times \R^{3} \backslash \mathcal{V}_{i}(\hat{\mathbf{e}}_{1}, \hat{\mathbf{e}}_{2})
									\ \ \text{for} \ \ i=1,2,\cdots, I_{\Omega,N, \delta,\delta_{1}},
									\end{equation}
									
									\begin{equation}\label{su_ell0}
									(s,u) \in [T+ (\ell_{0} -1) \delta_{2},T+ (\ell_{0} +1) \delta_{2} ] \times B(\delta_{2} \vec{\ell}; 2\delta_{2}),
									\end{equation}
									
									\begin{equation}\label{sprime_k}
									\begin{split}
									s^{\prime} &\in \big[ t^{k+1} (T+ \delta_{2} \ell_{0}; X(T+ \delta_{2} \ell_{0} ;t,x,v), \delta_{2} \vec{\ell})+\frac{1}{N} \\
									&\quad\quad\quad\quad , t^{k } (T+ \delta_{2} \ell_{0}; X(T+ \delta_{2} \ell_{0} ;t,x,v), \delta_{2} \vec{\ell})-\frac{1}{N}\big], 
									\end{split}
									\end{equation}
									and 
									\begin{equation}\label{sprime_psi}
									|s^{\prime} -   \psi^{\ell_{0}, \vec{\ell}, i, k}  (T+ \delta_{2} \ell_{0} , X( T+ \delta_{2}\ell_{0} ; t,x,v), \delta_{2} \vec{\ell} ) | > N^{2} (1 + \| \psi^{\ell_{0}, \vec{\ell}, i, k} \|_{C^{0,\gamma}}) (\delta_{2})^{\gamma} ,
									\end{equation}
									then 
									\begin{equation}\label{lower_1}
									\big| \p_{|u| } X(s^{\prime} ; s, X(s;t,x,v), u) \times   \p_{\hat{u}_{1}} X(s^{\prime} ; s, X(s;t,x,v), u)
									\big| > \epsilon_{\Omega, N,\| \Phi \|_{C^{2}} ,\delta_{1}, \delta_{2},\delta_{2}}.
									\end{equation}
									Here $\epsilon_{\Omega, N,\| \Phi \|_{C^{2}} ,\delta_{1}, \delta_{2},\delta_{2}}>0$ does not depend on $T, t,x,v$.

									For each $j=1,2,\cdots, I_{\Omega,N,\delta,\delta_{1}}$ in Lemma \ref{OV}, there exists $\delta_{3}>0$ and $C^{0,\gamma}$-functions 
									\begin{equation}\label{three_psi}
									\psi^{ \ell_{0}, \vec{\ell}, i, k, j, m_{0}, \vec{m}, k^{\prime}}_{1}, \psi^{\ell_{0}, \vec{\ell}, i, k, j, m_{0}, \vec{m}, k^{\prime}}_{2}, \psi^{ \ell_{0}, \vec{\ell}, i, k, j, m_{0}, \vec{m}, k^{\prime}}_{3},\end{equation}
									for $k^{\prime} \leq K^{j}$ in (\ref{finite_k_V}) where $\psi^{ \ell_{0}, \vec{\ell}, i, k, j, m_{0}, \vec{m}, k^{\prime}}_{n}$ is defined locally around $(T+ \delta_{3}m_{0} ; X(T+ \delta_{3}m_{0}; T+ \delta_{2} \ell_{0}, X ( T+ \delta_{2} \ell_{0}; t,x,v   ), \delta_{2} \vec{\ell}  ), \delta_{3} \vec{m} )$ for some $(m_{0}, \vec{m}) = (m_{0}, m_{1}, m_{2},m_{3})  \in\{ 0,1, \cdots, \lfloor\frac{1}{\delta_{3}}\rfloor+1\} \times \{- \lfloor\frac{N}{\delta_{3}}\rfloor-1, \cdots, 0 , \cdots, \lfloor\frac{N}{\delta_{3}}\rfloor+1 \}^{3}$ with $0< \delta_{3} \ll1$.
									
									Moreover, if we assume (\ref{Xs_OVi}), (\ref{su_ell0}), (\ref{sprime_k}), (\ref{sprime_psi}),
									\begin{equation}\label{sprime_j}
									\begin{split}
									&(X(s^{\prime};s, X(s;t,x,v),u), u^{\prime}) \in \mathcal{O}_{j} \times \R^{3} \backslash \mathcal{V}_{j}(\p_{|u|}X, \p_{\hat{u}_{1}}X), 	\\
									&\quad\quad \ \ \text{ for some } \  j=1,2,\cdots, I_{\Omega,N,\delta,\delta_{1}} \ \ \text{ in Lemma \ref{OV},}
									\end{split}
									\end{equation}
									\begin{equation}\label{sprimeprime_kprime}
									\begin{split}
									s^{\prime\prime}  \in& \  \Big[t^{k^{\prime}+1}  (T+ \delta_{3}m_{0} ; X(T+ \delta_{3}m_{0}; T+ \delta_{2} \ell_{0}, X ( T+ \delta_{2} \ell_{0}; t,x,v   ), \delta_{2} \vec{\ell}  ),  \delta_{3} \vec{m} ) + \frac{1}{N} \\ &   \ \ 
									\ \quad , t^{k^{\prime} } \underbrace{ (T+ \delta_{3}m_{0} ; X(T+ \delta_{3}m_{0}; T+ \delta_{2} \ell_{0}, X ( T+ \delta_{2} \ell_{0}; t,x,v   ), \delta_{2} \vec{\ell}  ), \delta_{3} \vec{m} ) }_{(**)} - \frac{1}{N}\Big],
									\end{split}
									\end{equation} 
									and
									\begin{equation}\begin{split}\label{sprimeprime_psi}
									&\min_{n=1,2,3}|s^{\prime\prime} - \psi^{ \ell_{0}, \vec{\ell}, i, k, j, m_{0}, \vec{m}, k^{\prime}}_{n}  (**)  |					>		N^{2} (1+ \max_{n=1,2,3} \| \psi^{ \ell_{0}, \vec{\ell}, i, k, j, m_{0}, \vec{m}, k^{\prime}}_{n}  \|_{C^{0,\gamma}} ) (\delta_{3})^{\gamma},\end{split}\end{equation}
									where $(**)$ is defined in (\ref{sprimeprime_kprime}). Then for each $\ell_{0}, \vec{\ell}, i, k, j, m_{0}, \vec{m}, k^{\prime}$, we can choose two distinct variables $\{\zeta_{1}, \zeta_{2}\} \subset \{|u|, \hat{u}_{1}, 
									\hat{u}_{1}^{\prime} , \hat{u}_{2}^{\prime}  \}$ such that \\ $(|u^{\prime}|, \zeta_{1}, \zeta_{2} ) \mapsto X(s^{\prime\prime};s^{\prime},X( s^{\prime}; s, X(s;t,x,v),u), u^{\prime} ) $ is one-to-one locally and 
									\begin{equation}\label{lower_zeta12}
									\Big|\det \left( \frac{\p X(s^{\prime\prime};s^{\prime},X( s^{\prime}; s, X(s;t,x,v),u), u^{\prime} ) }{\p (|u^{\prime}|, \zeta_{1}, \zeta_{2})} \right)
									\Big| > \epsilon^{\prime}_{\Omega, N, \| \Phi \|_{C^{2}}, \delta_{1}, \delta_{2}, \delta_{3} }.
									\end{equation}
									Here $\epsilon^{\prime}_{\Omega, N, \| \Phi \|_{C^{2}}, \delta_{1}, \delta_{2}, \delta_{3} }>0$ does not depend on $T,t,x,v$.
								\end{theorem}

								\begin{proof}
									%
									%
									%
									%
									%
									\noindent\textit{Step 1. } Fix any arbitrary $(t,x,v) \in [T,T+1] \times \Omega \times \R^{3}$. Assume that \\
									$s\in [T,t]$ and $(X(s;t,x,v), u) \in \mathcal{O}_{i} \times \R^{3} \backslash \mathcal{V}_{i}(\hat{\mathbf{e}}_{1}, \hat{\mathbf{e}}_{2})$ for some $i$, where $\mathbf{e}_{1}$ and $\mathbf{e}_{2}$ are standard unit vector $(1,0,0)$ and $(0,1,0)$ in global coordinate. Due to Lemma \ref{OV}, \\
									$(X(s^{\prime}; s,X(s;t,x,v),u), V(s^{\prime}; s,X(s;t,x,v),u) )$ is well-defined for all $s^{\prime} \in [T,s]$ and $|n(x^{k}(s, X(s;t,x,v), u)) \cdot v^{k }(s, X(s;t,x,v), u)|\gtrsim_{\Omega, N}1$ for all $k$ with\\
									$|t-t^{k}(s, X(s;t,x,v), u)| \leq 1$.

									We note that, from $X(s;t,x,v) = X(\bar{s};t,x,v) + \int^{s}_{\bar{s}} V(\tau; t,x,v) \dd \tau$,
									\begin{equation}\begin{split}\label{diff_psi_k}
									&|\psi^{k} (s,X(s;t,x,v), u) - \psi^{k} (\bar{s},X(\bar{s};t,x,v), \bar{u})| \\
									&\leq   \   \|  \psi^{k} \|_{C^{0,\gamma}_{t,x,v}}  \{ |s- \bar{s}|^{\gamma} + |X(s;t,x,v)- X(\bar{s};t,x,v)| ^{\gamma} + |u- \bar{u}|^{\gamma} \}\\
									&\leq \ \| \psi^{k} \|_{C^{0,\gamma} _{t,x,v}}  \{ |s- \bar{s}|^{\gamma}   + 
									(1+ N^{\gamma} + \| \nabla_{x} \Phi \|_{\infty}^{\gamma})
									|u- \bar{u}|^{\gamma} \}.\end{split}
									\end{equation} 
									For $0<\delta_{2}\ll1$ we split 
									\begin{equation}\notag
									\begin{split}
									[T,T+1] &= \bigcup_{\ell_{0}=0}^{ [\delta_{2}^{-1}]+1 } \big[T+ (\ell_{0}-1) \delta_{2}, T+  (\ell_{0}+1) \delta_{2}  \big] , \\  
									\mathbb{R}^{3} \backslash \mathcal{V}_{i}(\hat{\mathbf{e}}_{1}, \hat{\mathbf{e}}_{2}) &=  \bigcup_{|\ell_{i}|=0}^{ [N /\delta_{2}^{-2}]+1 }
									B\big( (\ell_{1} \delta_{2}, \ell_{2} \delta_{2}, \ell_{3} \delta_{2} );2\delta_{2}\big)
									\cap  \  \mathbb{R}^{3} \backslash \mathcal{V}_{i}(\hat{\mathbf{e}}_{1}, \hat{\mathbf{e}}_{2}). 
									\end{split}
									\end{equation}
									From (\ref{diff_psi_k}), if 
									$$(s,u) \in \big[T+ (\ell_{0}-1) \delta_{2}, T+  (\ell_{0}+1) \delta_{2}  \big] \times \{ B\big( (\ell_{1} \delta_{2}, \ell_{2} \delta_{2}, \ell_{3} \delta_{2} );2\delta_{2}\big)
									\cap  \  \mathbb{R}^{3} \backslash \mathcal{V}_{i} \}, $$ 
									then 
									\begin{eqnarray*}
										&&	|\psi^{k} ( T+ \ell_{0} \delta, X(T+ \ell_{0}  \delta;t,x,v), (\ell_{1} \delta, \ell_{2} \delta, \ell_{3} \delta)  ) - \psi^{k} (s, X(s;t,x,v), u) | \\
										&&\leq
										\| \psi^{k} \|_{C^{0,\gamma}}(2+ N^{\gamma} + \| \nabla_{x} \Phi \|_{\infty}^{\gamma})
										( \delta_{2})^{\gamma}.
									\end{eqnarray*}
									Therefore, if (\ref{sprime_psi}) holds then 
									\begin{eqnarray}
									&&	|s^{\prime} - \psi^{k} (s, X(s;t,x,v),u)| \nonumber\\
									&\geq& |s^{\prime} - \psi^{k} ( T+ \ell_{0} \delta, X(T+ \ell_{0}  \delta;t,x,v), (\ell_{1} \delta, \ell_{2} \delta, \ell_{3} \delta)  )|\notag\\
									&&
									- |\psi^{k} ( T+ \ell_{0} \delta, X(T+ \ell_{0}  \delta;t,x,v), (\ell_{1} \delta, \ell_{2} \delta, \ell_{3} \delta)  ) - \psi^{k} (s, X(s;t,x,v), u) |
									\label{sprime-psi_lower}
									\\ 
									&\gtrsim&(N^{2}- N^{\gamma})\| \psi^{k} \|_{C^{0,\gamma}}( \delta_{2})^{\gamma} \gtrsim_{N}  \| \psi^{k} \|_{C^{0,\gamma}}( \delta_{2})^{\gamma}.\notag
									\end{eqnarray}

									Consider the mapping $u \mapsto X(s^{\prime}; s,X(s;t,x,v),u)$. Note that from Lemma \ref{OV} we verify the condition of Lemma \ref{lemma rank 2}. Applying Lemma \ref{lemma rank 2}, we construct $C^{0,\gamma}$-function $\psi^{k}: B_{\varepsilon} (s, X(s;t,x,v), u) \rightarrow \R$ for $k \leq K^{i}$ such that if \\
									$|s^{\prime} - \psi^{k}(s,X(s;t,x,v),u)|>(\delta_{2})^{\gamma}$, then $$|\p_{|u|} X(s^{\prime}; s,X(s;t,x,v),u) \times \p_{\hat{u}_{1}}X(s^{\prime}; s,X(s;t,x,v),u)|> \epsilon_{\Omega,N, \| \Phi \|_{C^{2}}, \delta_{1}, (\delta_{2})^{\gamma}  }>0. $$
									Clearly if (\ref{sprime_psi}) holds, then from (\ref{sprime-psi_lower}), we have  $|s^{\prime} - \psi^{k}(s,X(s;t,x,v),u)|>(\delta_{2})^{\gamma}$.

									\vspace{4pt}
									
									\noindent\textit{Step 2. } Assume all the conditions of (\ref{Xs_OVi})-(\ref{sprime_psi}) and (\ref{sprime_j}). Applying Lemma \ref{lemma rank 3}, we construct (\ref{three_psi}). As (\ref{sprime-psi_lower}),
									\begin{equation}\begin{split}\label{sprime-psi_lower2}
									& |  \psi( s^{\prime}, X(s^{\prime};s, X(s;t,x,v), u), u^{\prime}  )
									-   \psi( \bar{s^{\prime}}, X(\bar{s^{\prime}};s, X(s;t,x,v), u), \bar{u^{\prime}}  )|\\
									\leq& \  \| \psi \|_{C^{0,\gamma}} \{ |s^{\prime}-  \bar{s^{\prime}}|^{\gamma}  +
									(1+ N^{\gamma} + \| \nabla_{x} \Phi \|_{\infty}^{\gamma})
									|u^{\prime}-  \bar{u^{\prime}}|^{\gamma}  \}.
									\end{split}\end{equation}
									
									\noindent For $0< \delta_{3} \ll 1$, we split
									\begin{equation}\notag
									\begin{split}
									[T,T+1] 
									&= \bigcup_{m_{0}=0}^{ [\delta_{3}^{-1}]+1 } \big[T+ (m_{0}-1) \delta_{3}, T+  (m_{0}+1) \delta_{3}  \big] , 
									\\
									\mathbb{R}^{3} \backslash \mathcal{V}_{j}(\p_{|u|}X, \p_{\hat{u}_{1}}X ) &=  \bigcup_{|m_{i}|=0}^{ [N /\delta_{3}^{-2}]+1 }
									B\big( (m_{1} \delta_{3}, m_{2} \delta_{3}, m_{3} \delta_{3} );2\delta_{3}\big)
									\cap  \  \mathbb{R}^{3} \backslash \mathcal{V}_{j}(\p_{|u|}X, \p_{\hat{u}_{1}}X ). 
									\end{split}
									\end{equation}
									From (\ref{sprime-psi_lower2}) if 
									\begin{equation*}
									\begin{split}
										(s^{\prime}, u^{\prime}) &\in [T+ (m-1) \delta, T+ (m+1) \delta]  \\
										&\quad \times \{ B((m_{1} \delta, m_{2} \delta, m_{3} \delta); 2 \delta ) \cap \R^{3} \backslash \mathcal{V}_{j}(\p_{|u|}X, \p_{\hat{u}_{1}}X ) \}  , 	\\ 
										(X(s;t,x,v), u ) &\in 
										\mathcal{O}_{i} \times \R^{3 }\backslash \mathcal{V}_{i}(\hat{\mathbf{e}}_{1}, \hat{\mathbf{e}}_{2}), \\
										( X(s^{\prime};s, X(s;t,x,v), u), u^{\prime}) &\in \mathcal{O}_{j} \times \R^{3} \backslash \mathcal{V}_{j}(\p_{|u|}X, \p_{\hat{u}_{1}}X ),
									\end{split}
									\end{equation*}
									then
									\begin{eqnarray*}
										&&| s^{\prime} -  \phi  ( T+ \ell \delta, X(T+ \ell \delta;t,x,v), \vec{\ell} \delta  ) | \gtrsim_{N} \|\phi \|_{C^{0,\gamma}  } \delta^{\gamma}, \\
										&&|s^{\prime\prime}- \psi  ( T+ m \delta, X(T+ m \delta; T+\ell \delta, X(T+\ell \delta; t,x,v ), \vec{m}\delta  ), \vec{\ell} \delta )  | 
										\gtrsim_{N} \delta^{\gamma}.
									\end{eqnarray*}
									\\
									Consider the mapping $$(u,u^{\prime}) \mapsto X(
									s^{\prime\prime}; s^{\prime}, X(s^{\prime};s, X(s;t,x,v), u), u^{\prime}).$$
									Note that from Lemma \ref{OV} we verify the condition of Lemma \ref{lemma rank 3}. For each $i,j$ and $\ell_{0}, \ell_{1},\ell_{2},\ell_{3}$ and $m_{0}, m_{1},m_{2},m_{3}$, applying Lemma \ref{lemma rank 3}, we can choose two variables $\{ \zeta_{1}, \zeta_{2}\} \subset\{ |u|,\hat{u}_{1}, 
									\hat{u}^{\prime}_{1},  \hat{u}^{\prime}_{2}  \}$ so that (\ref{lower_zeta12}) holds.
									
									\end{proof}

		\section{A time-dependent potential}
								
								\begin{theorem}[Local existence]\label{local_existence} For a sufficiently small $\delta_{0}>0$ and $\delta_{\phi}>0$ there exists $T^{*} >0$ such that if $
								\|wf_{0}\|_{\infty}\leq\delta_{0}$ and $\|\phi\|_{C^{1}}\leq \delta_{\phi}$, then there exists a unique solution $f(t,x,v)$ to (\ref{E_eqtn}) in $[0, T^{*}  ) \times \O \times \R^{3}$ such that 
								\begin{equation}\label{local_est}
									\sup_{0 \leq t \leq T^{*} } \| w f (t) \|_{\infty} \leq 
									2 (\delta_{0} + C \delta_{\phi})
									,
									\end{equation} and $ \| w f (t) \|_{\infty}$ is continuous over $[0, T^{*}  )$. If $F_{0}= \mu_{E} + \sqrt{\mu}_{E} f_{0} \geq 0$, then $
									F = \mu_{E} + \sqrt{\mu_{E}} f \geq 0.$
								\end{theorem}
								
								\begin{proof}
									For the proof we use a sequence of $F^{0}\equiv 0$ and for $\ell \geq 0$
									\begin{equation}\notag
									\begin{split}
									&\p_{t } F^{\ell+1} + v\cdot \nabla_{x} F^{\ell+1} - \nabla_{x} (\phi + \Phi) \cdot \nabla_{v} F^{\ell+1} \\
									&\quad\quad = Q_{+} (F^{\ell},F^{\ell}) - \nu(F^{\ell}) F^{\ell+1},  \ \ F|_{t=0} = F_{0}, \\
									&F^{\ell+1} (t,x,v) = F^{\ell+1} (t,x,R_{x}v) \ \ \ \text{on}  \ \  \p\O.
									\end{split}
									\end{equation}
									Note that $$\frac{d}{ds} e^{- \int^{t}_{s} \nu(F^{\ell}) (\tau, X(\tau;t,x,v), X(\tau;t,x,v) ) \dd \tau }F^{\ell+1} (s, X(s), V(s)) = Q_{+} (F^{\ell}, F^{\ell}) (s, X(s), V(s)),$$ where $X(s)= X(s;t,x,v), V(s):= V(s;t,x,v)$ satisfying (\ref{E_Ham}) and (\ref{specular_cycles}). Note that if $F^{\ell} \geq 0$, then $\nu(F^{\ell}) \geq 0$ and $Q_{+} (F^{\ell}, F^{\ell}) \geq 0$. Therefore, if $F^{\ell} \geq 0$ and $F_{0}\geq 0$, then 
									\begin{equation}\label{positive_F}
									F^{\ell+1} \geq 0 \ \ \text{ for all} \  \ell.
									\end{equation}

									\noindent From $F^{\ell+1} = \mu_{E} + \sqrt{\mu_{E}} f^{\ell+1}$,
									\begin{equation} \label{iteration scheme}
									\begin{split}
									&\p_t f^{\ell+1} + v\cdot\nabla_x f^{\ell+1} - \nabla_x (\phi +\Phi)\cdot\nabla_v f^{\ell+1}  + e^{-\Phi} \nu f^{\ell+1} + \frac{f^{\ell+1}}{2} v\cdot\nabla_x \phi\\ & =
									e^{-\Phi }Kf^{\ell } 
									-\sqrt{\mu_E} v\cdot\nabla_x \phi +e^{-\frac{\Phi }{2}} \Gamma_{+}(f^{\ell},f^{\ell})-e^{-\frac{\Phi }{2}} \Gamma_{-}(f^{\ell},f^{\ell+1}).\end{split}
									\end{equation}
									For $h^{\ell} := w f^{\ell}$
									\begin{equation} \label{equation_h}
									\begin{split}
									&  \p_t h^{\ell+1} + v\cdot\nabla_x h^{\ell+1} - \nabla_x( \phi + \Phi)\cdot\nabla_v h^{\ell+1} \\
									&\quad\quad + \frac{h^{\ell+1}}{w} \nabla (\phi+\Phi)\cdot\nabla_v w  + e^{-\Phi } \nu h^{\ell+1} + \frac{h^{\ell+1}}{2} v\cdot\nabla_x  \phi   \\
									&=  e^{-\Phi } K_w h^{\ell } - w\sqrt{\mu_E }v\cdot\nabla_x \phi + w e^{-\frac{\Phi }{2}} \Gamma_{+}(\frac{h^{\ell}}{w},\frac{h^{\ell}}{w}) 
									- w e^{-\frac{\Phi }{2}} \Gamma_{-}(\frac{h^{\ell}}{w},\frac{h^{\ell+1}}{w}) 
									.\end{split}
									\end{equation}
									We claim that we can choose $0<T^{*} \ll1 $ such that for all $\ell$
									\begin{equation}\label{local_uniform}
									\sup_{0 \leq t \leq T^{*}}\| h^{\ell}(t) \|_{\infty} \leq 2 ( \delta_{0}+C \delta_{\phi}).
									\end{equation}

									\noindent We define,
									\begin{equation}\label{E_G}\begin{split}
									E(v,t,x) &:= \exp \Big\{- \int_{s}^{t} \nu_{E} (\tau, X(\tau;t,x,v), V(\tau;t,x,v)) \dd \tau \Big\}  \\
									&:= \exp \Big\{ -\int_{s}^{t} \big[ e^{ -\Phi (X(\tau))}\nu(V(\tau))+\frac{1}{2} V(\tau )\cdot\nabla \phi(\tau,X(\tau))  \\
									&\quad\quad \quad\quad + \frac{1}{w}\nabla_x ( \phi(\tau,X(\tau)) + \Phi(X(\tau )) \cdot\nabla_v w(V(\tau )) \big]\dd \tau \Big\} ,  \\
									G^{\ell+1} &:=   -  w \sqrt{\mu_E}V\cdot\nabla_x \phi + w e^{-\frac{\Phi }{2}} \Gamma_{+}(\frac{h^{\ell}}{w},\frac{h^{\ell}}{w})- w e^{-\frac{\Phi }{2}} \Gamma_{-}(\frac{h^{\ell}}{w},\frac{h^{\ell+1}}{w}) .
									\end{split}\end{equation}
									Along the trajectory,
									%
									\begin{equation}\begin{split}\notag
									&  \frac{\dd}{\dd s} \Big( E(v,t,s) h^{\ell+1}(s,X(s;t,x,v),V(s;t,x,v)) \Big)\\
									&=  {E(v,t,s) } \big[ e^{-\Phi (X(s))} K_w h^{\ell+1}  + G^{\ell+1} \big](s,X(s;t,x,v),V(s;t,x,v)).\end{split}
									\end{equation}
									By integrating from $0$ to $t$, we obtain
									\begin{equation} \begin{split}\label{Duhamel_once}
									h^{\ell+1}(t,x,v) =& E(v,t,0) h^{\ell+1}(0, X(0), V(0)) + \int^{t}_{0} E(v,t,s) G^{\ell+1}(s)   \dd s\\
									&  + \int^{t}_{0} E(v,t,s) e^{-\Phi (X(s))} \int_{\R^{3}} k_{w}(u,V(s)) h^{\ell+1}(s,  X(s;t,x,v),u)   \dd u  \dd s .\end{split}
									\end{equation}
									
									\noindent From (\ref{nu0}), 
									\begin{equation}\label{est_nu_E}
									\langle  V(\tau;t,x,v)\rangle\lesssim_{\phi, \Phi} \nu_{E}(\tau, X(\tau;t,x,v), V(\tau;t,x,v)) \lesssim_{\phi, \Phi} \langle  V(\tau;t,x,v)\rangle.
									\end{equation}
									Recall the standard estimates (see Lemma 4 and Lemma 5 in \cite{GKTT1})
									\begin{equation}\label{est_kw}
									\int_{\R^{3}} |k_{w} (v,u)| \dd u \leq C_{K}\langle v\rangle^{-1}, \ \ \ 
									\Big| w \Gamma_{\pm} (\frac{h_{1}}{w}, \frac{h_{2}}{w})(v) \Big| \lesssim 
									\langle v\rangle |h_{1}| |h_{2}|.
									\end{equation}
									Therefore,
									\begin{equation}\label{est_G}
									\begin{split}
									|G^{\ell+1}  (s;t,x,v)| &\lesssim_{\Phi}  \| \nabla_{x} \phi (s )  \|_{\infty} e^{- \frac{|V(s)|^{2}}{8}} \\
									&\quad + \langle V(s;t,x,v) \rangle \{ \| h^{\ell}(s) \|_{\infty} + \| h^{\ell+1}(s) \|_{\infty}  \} \| h^{\ell}(s) \|_{\infty}. 
									\end{split}
									\end{equation}

									\noindent From (\ref{est_nu_E}) and (\ref{est_G}), we deduce that 
									\begin{equation}\notag
									\begin{split}
									\sup_{0\leq t \leq T}\| h^{\ell+1}(t) \|_{\infty} \leq& \  \delta_{0}+  C\delta_{\phi} +C_{K}  T\sup_{0 \leq s \leq t} \| h^{\ell+1} (s) \|_{\infty}\\
									&+C \big\{
									\sup_{0 \leq s \leq t} \| h^{\ell+1} (s) \|_{\infty}  + \sup_{0 \leq s \leq t} \| h^{\ell } (s) \|_{\infty} \big\} \sup_{0 \leq s \leq t} \| h^{\ell } (s) \|_{\infty}.
									\end{split}
									\end{equation}
									Choose $T^{*}>0$ such that $C_{K}T^{*} \ll 1$. Then from (\ref{local_uniform}) for $h^{\ell}$
									\begin{equation}\notag
									\Big(1- \frac{1}{10} - 2 (\delta_{0} + C \delta_{\phi})
									\Big) \times \sup_{0\leq t \leq T}\| h^{\ell+1}(t) \|_{\infty} \leq
									\delta_{0}+  C\delta_{\phi} + C\big( \delta_{0}+  C\delta_{\phi}\big)^{2},
									\end{equation}
									and we prove the same upper bound of (\ref{local_uniform}) for $h^{\ell+1}$ for sufficiently small $\delta_{0}$ and $\delta_{\phi}$.
									
									We can show that $h^{\ell}$ is a Cauchy sequence in $L^{\infty} ([0,T^{*}); L^{\infty}(\O \times\R^{3}))$ by repeat the argument with $h^{\ell+1} - h^{\ell}$. Then we pass a limit $\ell \rightarrow \infty$ to prove the existence and (\ref{local_est}). Using (\ref{positive_F}) and this limit we prove $F \geq 0$. Assume $h_{1}$ and $h_{2}$ solve the same equation (\ref{equation_h}). Following the same proof of (\ref{local_uniform}) we prove that $\sup_{0 \leq t \leq T^{*}}\| h_{1}- h_{2} \|_{\infty}\leq o(1) \sup_{0 \leq t \leq T^{*}}\| h_{1}- h_{2} \|_{\infty}$. Hence $h_{1} \equiv h_{2}$ and we conclude the uniqueness. 
									
									For $0<\e\ll1$, from (\ref{equation_h}) with $h=h^{\ell+1}$
									\begin{eqnarray*} 
										\|h(t+ \e)\|_{\infty} - \|h( t)\|_{\infty}
										&\leq&  \|h(t+ \e) -h( t)\|_{\infty} \\
										&\lesssim& \e \{ \| h _{0} \|_{\infty} + \| h \|_{\infty}+ \| h \|_{\infty}^{2} + \| \phi \|_{C^{1}}  \}.
									\end{eqnarray*}
									Hence $\| w f(t) \|_{\infty}$ is continuous on $[0, T^{*})$.\end{proof}

								\begin{lemma}For $w(v)= (1+ |v|)^{\beta}$ with $\beta>2$,
									\begin{equation}\label{est_Gamma_2}
									\begin{split}
									\Big|\int_{\R^{3}} \Gamma_{+}(\psi, \psi)   \varphi \dd v
									\Big|
									&\lesssim 
									\|  w \psi \| _{\infty}
									\| \psi \|_{L^{2}(\R^{3})}
									\|   \varphi  \|_{L^{2}(\R^{3})},\\
									\Big|\int_{\R^{3}} \Gamma_{-}(\psi, \psi)   \varphi \dd v
									\Big|
									&\lesssim 
									\|  w \psi \| _{\infty} \big\{
									\| \psi \|_{L^{2}(\R^{3})}
									\|   \varphi  \|_{L^{2}(\R^{3})}
									+ \| (\mathbf{I} - \mathbf{P}) \psi \|_{\nu} \| \varphi \|_{\nu}
									\big\}.
									\end{split}
									\end{equation}
									
								\end{lemma}
								\begin{proof}
									Via the well-known Carleman representation (for example see (32) in \cite{Kim11}), we have 
									\begin{equation}\begin{split}\notag
									\Gamma_{+} & (\psi, \psi) (v) =  \frac{1}{\sqrt{\mu(v)} }Q_{+} (\sqrt{\mu }\psi, \sqrt{\mu }\psi) (v)\\
									&= 
									2 \int_{\R^{3}} \psi(v^{\prime}) \frac{
										1}{|v-v^{\prime}|^{2}}
									\int_{E_{vv^{\prime}}}
									\psi (v_{1}^{\prime})  e^{- \frac{|-v+ v^{\prime} + v_{1}^{\prime}|^{2}}{4}}
									B(2 v-v^{\prime } - v_{1}^{\prime}, \frac{v^{\prime} - v_{1}^{\prime}}{|v^{\prime} - v_{1}^{\prime}|})
									\dd v_{1}^{\prime} 
									\dd v^{\prime},
									\end{split}\end{equation}
									where $E_{vv^{\prime}}$ is a hyperplane containing $v \in\R^{3}$ and perpendicular to $\frac{v^{\prime}-v}{|v^{\prime}-v|} \in \mathbb{S}^{2}$, i.e.
									\[
									E_{vv^{\prime}} : = \{ v_{1}^{\prime} \in \R^{3}: (v_{1}^{\prime} - v ) \cdot (v^{\prime} - v) =0  \}.
									\]
									For the internal integration over $E_{vv^{\prime}}$, using Lemma 6 and (34) in \cite{Kim11}, we bound it above as 
									\begin{equation}\notag
									\int_{E_{vv^{\prime}}} \cdots \dd v_{1}^{\prime}
									\lesssim \| w \psi \|_{\infty}  
									\frac{ 1+|v-v^{\prime}|  }{ w(v-v^{\prime})} 
									\lesssim \| w \psi \|_{\infty}  \langle v-v^{\prime} \rangle^{-(\beta-1)} ,
									\end{equation}
									where we have used $w(v_{1}^{\prime})^{-1}e^{- \frac{|-v+ v^{\prime} + v_{1}^{\prime}|^{2}}{4}}\lesssim w( v- v^{\prime})^{-1}$. Note that \\ $\int_{\R^{3}} \frac{ \langle v-v^{\prime} \rangle^{-(\beta-1)} }{|v-v^{\prime}|^{2}  }  \dd v^{\prime}\lesssim 1$ for $\beta>2$. Hence we conclude that 
									\begin{equation}\begin{split}\notag
									&\big|\int_{\R^{3}} \Gamma_{+} (\psi, \psi) \varphi
									\dd v\big|\\
									&\lesssim 
									\| w \psi \|_{\infty}\int_{\R^{3}} \int_{\R^{3}}
									\frac{ \langle v-v^{\prime} \rangle^{-(\beta-1)} }{|v-v^{\prime}|^{2}  }
									|\psi(v^{\prime})|
									|\varphi(v)|
									\dd v^{\prime} \dd v\\
									&\lesssim  \| w \psi \|_{\infty}
									\Big[ \int_{\R^{3}}
									\Big(
									\int_{\R^{3}} 
									\frac{\langle v-v^{\prime} \rangle^{-( \beta-1)}}{|v-v^{\prime}|^{2}}
									\dd v
									\Big)
									|\psi(v^{\prime})|^{2} \dd v^{\prime}\Big]^{\frac{1}{2}} \\
									&\quad \times
									\Big[ \int_{\R^{3}}
									\Big(
									\int_{\R^{3}} 
									\frac{\langle v-v^{\prime} \rangle^{-( \beta-1)}}{|v-v^{\prime}|^{2}}
									\dd v^{\prime}
									\Big)
									|\varphi(v)|^{2} \dd v \Big]^{\frac{1}{2}}\\
									&\lesssim  \| w \psi \|_{\infty}
									\| \psi \|_{L^{2}(\R^{3})}  \| \varphi \|_{L^{2}(\R^{3})}.
									\end{split}
									\end{equation}
									
									\noindent For the $\Gamma_{-}$ estimate, we have 
									\begin{eqnarray*}
										\int_{\R^{3}}| \Gamma_{-} (\psi , \psi ) (v)\varphi(v) |  \dd v& \leq &\int_{\R^{3}} \int_{\R^{3}} |v-u| |\psi  (u) |\sqrt{\mu(u)} \dd u \times |\psi(v)|| \varphi(v)| \dd v\\
										&\lesssim&  \| w \psi  \|_{\infty}   \int_{\R^{3}}   \langle v\rangle \{ | \mathbf{P}\psi (v)|  + |( \mathbf{I}-  \mathbf{P})\psi (v)|  \}| \varphi(v)|\dd v\\
										&\lesssim& \| w \psi  \|_{\infty}
										\big\{ \| \psi \|_{L^{2} (\R^{3})} \| \varphi \|_{L^{2} (\R^{3})}
										+ \| (\mathbf{I} - \mathbf{P}) \psi\|_{\nu} \| \varphi \|_{\nu}
										\big\},
									\end{eqnarray*}
									where we have used the fact, for all $p \in [1, \infty]$,
									\[
									\| \langle v \rangle  \mathbf{P} \psi \|_{L^{p} (\R^{3})}
									\lesssim \Big\|   \langle v\rangle ^{3} \sqrt{\mu(v)}
									\int_{\R^{3}} \psi(u) \langle u\rangle ^{2} \sqrt{\mu(u)} \dd u\Big\|_{L^{p}(\R^{3})}
									\lesssim \| \psi \|_{L^{p}(\R^{3})} .
									\]
								\end{proof}
								
								\begin{lemma}Let $f$ solve (\ref{E_eqtn}). Then there exists a constant $C>0$ not depending on $f_{0}, f$ and $\phi $ such that, for all $t \geq 0$, 
									\begin{equation}\label{f_growth}
									\|  f(t) \|_{2}^{2} \leq C\Big( 
									\| f(0) \|_{2}^{2}
									+\int^{t}_{0} \| \phi(s) \|_{\infty}   \Big) \times \Big(
									1+ C(\| \phi \|_{\infty} + \| w f \|_{\infty}) t e^{ C(   \| \phi \|_{\infty}  + \| w f \|_{\infty} ) t }
									\Big).
									\end{equation}
								\end{lemma}	
								\begin{proof}
									\begin{eqnarray*}
										&&\| f(t) \|_{2}^{2} +\int^{t}_{0} \iint_{\O \times \R^{3}} e^{- \Phi }f L f
										\\
										&= &
										\| f(0) \|_{2}^{2}-\underbrace{\int^{t}_{0} \iint_{\O \times \R^{3}} \frac{v \cdot \nabla_{x} \phi}{2}|f|^{2}}_{(I)}
										- \underbrace{\int^{t}_{0}\iint _{\O \times \R^{3}}v\cdot \nabla_{x} \phi \mathbf{P}f  \sqrt{\mu_{E}}}_{(II)}
										\\
										&& + \underbrace{\int^{t}_{0}\iint_{\O \times \R^{3}} e^{- \frac{\Phi}{2}}\Gamma (f,f) (\mathbf{I} - \mathbf{P})f}_{(III)}.
									\end{eqnarray*}
									By the decomposition $f= \mathbf{P} f + (\mathbf{I} - \mathbf{P})f$ and the strong decay-in-$v$ of $\mathbf{P} f$ in (\ref{E_Pf}),
									\begin{equation}\notag\begin{split}
									(I)  &\leq \| \phi \|_{\infty}\Big\{ \int^{t}_{0} \iint_{\O\times\R^{3}} |v| |\mathbf{P}f|^{2}
									+  \int^{t}_{0} \| (\mathbf{I} - \mathbf{P})f \|_{\nu}^{2}
									\Big\} 	\\
									&\lesssim \| \phi \|_{\infty}
									\Big\{ \int^{t}_{0} \| f \|_{2}^{2} +  \int^{t}_{0} \| (\mathbf{I} - \mathbf{P})f \|_{\nu}^{2}
									\Big\},\\
									(II) & \lesssim \| \phi \|_{\infty} \int^{t}_{0} \| f\|_{2}^{2}
									+ \int^{t}_{0} \| \phi (s) \|_{\infty} \dd s.
									\end{split}\end{equation}
									From (\ref{est_Gamma_2})
									\begin{equation}\notag
									\begin{split}
									(III)  &\lesssim \int^{t}_{0}\int_{\O}\| w f(s,x, \cdot) \|_{\infty}
									\| (\mathbf{I} - \mathbf{P}) f (s,x, \cdot) \|_{\nu} ^{2}   \dd x \dd s
									\\
									&\quad +   \int^{t}_{0}\int_{\O}\| w f(s,x, \cdot) \|_{\infty}
									\|  f (s,x, \cdot) \|_{2} ^{2}   \dd x \dd s
									\\
									&\lesssim \| w f \|_{\infty} \int^{t}_{0} \| (\mathbf{I} - \mathbf{P}) f (s)\|_{\nu}^{2} \dd s+ 
									\| w f \|_{\infty} \int^{t}_{0} \| f (s ) \|_{2}^{2} \dd s.
									\end{split}
									\end{equation}
									Using (\ref{semi-positive}) and collecting the terms, we deduce that, for some constant $C>0$, 
									\begin{equation}
									\begin{split}
									\| f(t) \|_{2}^{2}  \leq & \ \| f(t) \|_{2}^{2} 
									+ \big( \delta_{L} -   \| \phi \|_{\infty} -\| w f \|_{\infty} \big) \int^{t}_{0} \| (\mathbf{I} - \mathbf{P}) f \|_{\nu}^{2}\\
									\leq  & \ \| f(0)\|_{2}^{2}  + C  \big( \| \phi \|_{\infty}  +
									\| w f \|_{\infty}
									\big)
									\int^{t}_{0} \| f \|_{2}^{2} + C\int^{t}_{0} \| \phi \|_{\infty} . 
									\end{split}
									\end{equation}
									By the Gronwall's inequality we conclude (\ref{f_growth}).\end{proof}
								
								\begin{lemma}Assume $F= \mu_{E} + \sqrt{\mu_{E}}f$ solves (\ref{Boltzmann_phi}) and satisfies (\ref{evol_energy}). Assume (\ref{exp_phi}) and 
									\begin{equation}\label{large_lambda_phi}
									\lambda_{\phi} \gg \delta_{\phi}+ \| w f \|_{\infty}.
									\end{equation}
									Then 
									\begin{equation}\label{excess_energy}
									\begin{split}
									&\Big|\iint_{\O\times\R^{3}}  \Big( \frac{|v|^{2}}{2} + \Phi (x)  \Big) F(t,x,v) \dd x \dd v -  \iint_{\O\times\R^{3}}  \Big( \frac{|v|^{2}}{2} + \Phi (x)  \Big) F_{0}(x,v) \dd x \dd v\Big|\\
									&\lesssim\frac{\delta_{\phi}}{\lambda_{\phi}} \{1 + \| f(0) \|_{2}^{2} + \| w f \|_{\infty}\}.
									\end{split}
									\end{equation}
									
								\end{lemma}
								\begin{proof}
									The proof is a direct consequence of previous two lemmas ((\ref{evol_energy}), (\ref{f_growth})) and the exponential decay-in-time of $\phi(t)$ in (\ref{exp_phi}): The difference is bounded by
									\begin{equation}\notag
									\begin{split}
									&\leq \int^{t}_{0}\iint_{\O \times \R^{3}} 
									\big\{
									\mu_{E}(x,v)+
									\sqrt{\mu_{E} (x,v)} |f(s,x,v)|
									\big\}
									|v| |\nabla_{x} \phi(s,x) | \dd v \dd x \dd s
									\\
									&\leq
									\delta_{\phi}  \int^{t}_{0} e^{-\lambda_{\phi}s} 
									\bigg\{
									1+  C\Big( 
									\| f(0) \|_{2}^{2}
									+\int^{t}_{0} \| \phi(s) \|_{\infty}   \Big) \\
									&\quad \times \Big(
									1+ C(\| \phi \|_{\infty} + \| w f \|_{\infty}) s e^{ C(   \| \phi \|_{\infty}  + \| w f \|_{\infty} ) s}
									\Big)
									\bigg\} \dd s\\
									&\leq
									\delta_{\phi} \int^{t}_{0} e^{-\lambda_{\phi} s}
									\Big\{
									1+ C  \big(\| f(0) \|_{2}^{2} + \delta_{\phi} / \lambda_{\phi} \big)
									\big(
									1+ C(  \delta_{\phi} + \| w f \|_{\infty}) s e^{C( \delta_{\phi} +  \| w f \|_{\infty} )s}
									\big)
									\Big\}\dd s\\
									&\leq 
									\delta_{\phi}/ \lambda_{\phi} \big\{1+ C  \big(\| f(0) \|_{2}^{2} + \delta_{\phi} / \lambda_{\phi} \big) \big\}+    C
									\delta_{\phi }(  \delta_{\phi} + \| w f \|_{\infty}) \int^{t}_{0} e^{-  [\lambda_{\phi} -C( \delta_{\phi} +  \| w f \|_{\infty} )  ]s}\\
									&\lesssim \frac{\delta_{\phi}}{\lambda_{\phi}} \{1 + \| f(0) \|_{2}^{2} + \| w f \|_{\infty}\}.
									\end{split}
									\end{equation}
								\end{proof}
								
								\begin{lemma}[\cite{bounded_boltzmann, Kim11}]
									Recall $\mu_{E}$ in (\ref{E_Max}). Then
									\begin{equation}\label{est_via_entropy}
									\begin{split}
									|F- \mu_{E}|  \mathbf{1}_{|F- \mu_{E}|\geq \bar{\delta} \mu_{E}} &\leq
									\frac{4}{\bar{\delta}} \Big\{ 
									(F  \ln F  - \mu_{E} \ln \mu_{E} ) \\
									&\quad - (F - \mu_{E}) + \Big(  \frac{|v|^{2}}{2} + \Phi(x) \Big) (F - \mu_{E}) 
									\Big\} .
									\end{split}
									\end{equation}
									%
									%
								\end{lemma}
								\begin{proof} 
									The proof is based on the proof of Lemma 4 of \cite{Kim11} and the argument in page 147 of \cite{bounded_boltzmann}. 
									
									By the Taylor expansion, for $t,s > 0$
									\begin{equation}\label{tlnt_exp}
									\frac{1}{\max \{ t,s \}} \frac{|t-s|^{2}}{2}
									\leq \int^{t}_{s}  \int^{s_{1}}_{s} \frac{1}{s_{2}} \dd  s_{2} \dd  s_{1}
									=t\ln t - s \ln s -  (1+ \ln s ) (t-s)   .
									\end{equation}
									
									Note that if $F(t,x,v) - \mu_{E} (x,v)  \geq \bar{\delta} \mu_{E} (x,v)$ with $0 \leq \bar{\delta} \ll1$, then $F\geq (1+\bar{\delta})\mu_{E}$ and hence
									\[ 
									{\max \{ F ,  \mu_{E}  \} }=   {(1+ \bar{\delta}) \mu_{E}  }.
									\]
									If $ \mu_{E} (x,v)-F(t,x,v)  \geq - \bar{\delta} \mu_{E} (x,v)$ then $(1+ \bar{\delta})\mu_{E} \geq F$ and  
									\[
									\max \{ F ,  \mu_{E}   \} \leq (1+ \bar{\delta}) \mu_{E}  .
									\]
									
									
									\noindent Therefore, if $|F- \mu_{E}| \geq \bar{\delta} \mu_{E}$ then 
									\begin{equation}\notag
									\frac{|F- \mu_{E}|}{\max\{ F, \mu_{E}  \}}  \times \frac{|F- \mu_{E}|}{2}\geq \frac{\bar{\delta} \mu_{E}}{ (1+ \bar{\delta}) \mu_{E}} \times \frac{|F- \mu_{E}|}{2} \geq \frac{\bar{\delta}}{2}  \times \frac{|F- \mu_{E}|}{2}
									\geq \frac{\bar{\delta}}{4}|F- \mu_{E}|
									.
									\end{equation}
									
									%
									%
									
									\noindent Hence, from (\ref{tlnt_exp}), we deduce that 
									\begin{equation}\notag
									\begin{split}
									& |F- \mu_{E}|  \mathbf{1}_{|F- \mu_{E}|\geq \bar{\delta} \mu_{E}} \\
									&\leq \frac{4}{\bar{\delta}} \frac{1}{\max\{ F, \mu_{E} \}} \frac{|F- \mu_{E}|^{2}}{2}
									\\
									&\leq \frac{4}{\bar{\delta}} \big\{ F(t ) \ln F(t )  - \mu_{E} \ln \mu_{E} -(1+ \ln \mu_{E} ) (F - \mu_{E}(x,v))\big\}\\
									&\leq \frac{4}{\bar{\delta}} \big\{ 
									(F  \ln F  - \mu_{E} \ln \mu_{E})  -(F - \mu_{E}) + (  {|v|^{2}}/2 + \Phi(x) ) (F - \mu_{E}) 
									\},
									\end{split}
									\end{equation}
									where we have used $\ln \mu_{E} =  -    (  \frac{|v|^{2}}{2}+ \Phi(x)  )$. 
								\end{proof}


								\begin{proof}[\textbf{Proof of Theorem \ref{theorem_time}}]
									Denote 
									\[
									T_{1} : = \sup\Big\{ t \geq 0:  \| w f(t) \|_{\infty} \leq 2(\delta_{0} + C \delta_{\phi}) \Big\}.
									\]
									Note that (\ref{conserv_F_mass}), (\ref{evol_energy}), and (\ref{entropy_bound}) hold for $0\leq t \leq T_{1}$. Nota that $f$ and $h$ satisfy (\ref{E_eqtn}) and (\ref{equation_h}) with $h^{\ell+1}=h= h^{\ell}$. Then we have (\ref{Duhamel_once}) with $h^{\ell+1}=h= h^{\ell}$.

									We apply Duhamel formula (\ref{Duhamel_once}) \textit{three times}, for $0\leq t \leq T_{1}$, and decompose integrand $h$ as
									\begin{equation}\label{decom_h}\begin{split}
									h =&  \ h \mathbf{1}_{|F- \mu_{E}| \geq \bar{\delta} \mu_{E}} + w \frac{F-\mu_{E}}{\sqrt{\mu_{E}}} \mathbf{1}_{|F- \mu_{E}| \leq \bar{\delta} \mu_{E}}  ,  \\
									\end{split}
									\end{equation}
									for sufficiently small $0< \bar{\delta}\ll 1$, to get
									
									\begin{eqnarray*} 
										&& h (t,x,v) \\
										&=& E(v,t,0) h (0) + \int^{t}_{0} E(v,t,s) G(s)  \ \dd s\\
										&& + \int^{t}_{0} E(v,t,s) e^{-\Phi (X(s))} \int_{u} k_{w}(u,v) h(s,  X(s),u)   \dd u  \dd s  \\ 
										&=& E(v,t,0) h(0) + \int^{t}_{0} E(v,t,s) G(s)    \dd s\\
										&& + \int^{t}_{0} E(v,t,s) e^{-\Phi (X(s))} \int_{u} k_{w}(u,v) E(u,s,0) \{ h(0) + \int_0^{s} E(u,s,s^{\prime}) G(s^{\prime}) \ \dd s^{\prime} \}   \\
										&& + \int^{t}_{0}  E(v,t,s) e^{-\Phi(X(s))} \int_{u} k_{w}(u,v) 
										\int^{s}_{0} E(u,s,s^{\prime}) e^{-\Phi (X(s^{\prime}))}  \\
										&& \quad \quad \quad    \times  \int_{u^{\prime}}
										k_{w}(u^\prime,u) h(s ^{\prime} , 
										X(s^{\prime}) ,u^{\prime} )   \dd u^{\prime}  \dd s^{\prime}  \dd u  \dd s
									\end{eqnarray*}
									\begin{equation} \label{full expan}
									\begin{split}
									&=  E(v,t,0) h(0)  + \int^{t}_{0} E(v,t,s) G(s)  \ \dd s  \\
									&\quad + \int^{t}_{0} E(v,t,s) e^{-\Phi } \int_{u} k_{w}(u,v) E(u,s,0) h(0)   \dd u   \dd s  \\
									& \quad +  \int^{t}_{0} \int_0^s \int_u E(v,t,s) e^{-\Phi } k_{w}(u,v) E(u,s,0) E(u,s,s^{\prime}) G(s^{\prime})  \dd u \dd s^{\prime} \dd s   \\
									& \quad + \int^{t}_{0}  E(v,t,s) e^{-\Phi (X(s))} \int_{u} k_{w}(u,v) \\
									&\quad\quad\quad\quad\quad\quad\quad \times
									\int^{s}_{0} E(u,s,s^{\prime}) e^{-\Phi (X(s^{\prime}))} \int_{u^{\prime}}
									k_{w}(u^\prime,u) E(u^{\prime},s^\prime,0) h(0)  \\
									& \quad +  \int^{t}_{0}  E(v,t,s) e^{-\Phi (X(s))} \int_{u} k_{w}(u,v) \\
									&\quad\quad\quad\quad\quad\quad\quad \times
									\int^{s}_{0} E(u,s,s^{\prime}) e^{-\Phi (X(s^{\prime}))} \int_{u^{\prime}}
									k_{w}(u^\prime,u) \int_0^{s^{\prime}} E(u^{\prime},s^{\prime},s^{\prime\prime}) G(s^{\prime\prime})   \\
									& \quad + \int^{t}_{0}  E(v,t,s) e^{-\Phi (X(s))} \int_{u} k_{w}(u,v) 
									\int^{s}_{0} E(u,s,s^{\prime}) e^{-\Phi (X(s^{\prime}))} \int_{u^{\prime}}
									k_{w}(u^\prime,u)   \\
									&  \ \ \quad \times \int_0^{s^\prime} E(u^{\prime},s^\prime,s^{\prime\prime}) e^{-\Phi (X(s^{\prime\prime}))} \int_{u^{\prime\prime}} k_{w}(u^{\prime\prime},u^\prime) h( s^{\prime\prime},  X( s^{\prime\prime} ) , u^{\prime\prime} ) \ \mathbf{1}_{|F- \mu_{E}| \leq \bar{\delta} \mu_{E}}   \\
									& \quad + \underline{ \int^{t}_{0}  E(v,t,s) e^{-\Phi (X(s))} \int_{u} k_{w}(u,v) 
										\int^{s}_{0} E(u,s,s^{\prime}) e^{-\Phi (X(s^{\prime}))} \int_{u^{\prime}}
										k_{w}(u^\prime,u) }  \\
									&  \ \ \quad \underline{  \times \int_0^{s^\prime} E(u^{\prime},s^\prime,s^{\prime\prime}) e^{-\Phi (X(s^{\prime\prime}))} \int_{u^{\prime\prime}} k_{w}(u^{\prime\prime},u^\prime) h( s^{\prime\prime},  X( s^{\prime\prime} ) , u^{\prime\prime} ) \ \mathbf{1}_{|F- \mu_{E}| \geq \bar{\delta} \mu_{E}} 
									}_{(*)} ,   
									\end{split}	 
									\end{equation}	
									where we abbreviated notations,
									\begin{equation*}
									\begin{split}
									X(s) &:= X(s;t,x,v) ,\quad X(s^{\prime}) := X^{\prime}(s^{\prime};s,X(s;t,x,v),u)  \\
									X(s^{\prime\prime})  &:= X(s^{\prime\prime};s^{\prime}, X^{\prime}(s^{\prime};s,X(s;t,x,v),u), u^{\prime}) ,
									\end{split}
									\end{equation*}
									\noindent and where we use similar definition as (\ref{E_G}).
									\begin{equation} \begin{split}
									E(v,t,s) &:= \exp \big\{- \int_{s}^{t} \nu_{E} (\tau, X(\tau;t,x,v), V(\tau;t,x,v)) \dd \tau \big\}
									\\&:=  \exp \big\{-\int_{s}^{t} \big[ e^{ -\Phi (X(\tau))}\nu(V(\tau))+\frac{1}{2} V(\tau )\cdot\nabla \phi(\tau,X(\tau))  \\
									&\quad \quad \quad \quad \quad \quad + \frac{1}{w}\nabla_x ( \phi(\tau,X(\tau)) + \Phi(X(\tau )) \cdot\nabla_v w(V(\tau )) \big]\dd \tau \big\} ,  \\
									G &:=   -  w \sqrt{\mu_E}V\cdot\nabla_x \phi + w e^{-\frac{\Phi }{2}} \Gamma(\frac{h}{w},\frac{h}{w}) .  \\
									\end{split}\end{equation}
									
									\noindent Under assumption of $\delta_{\phi} + \delta_{\phi}/ \lambda_{\phi} \ll 1$,
									\begin{equation} \label{exponent bound time dept}
									E(v,t,s) \leq e^{ -\frac{1}{2} e^{-\|\Phi\|_C}  \nu(v) (t-s) } := e^{ - \nu_{\Phi}(v) (t-s) } ,
									\end{equation}
									where we define ${\nu}_{\Phi}(v) := \frac{1}{2} e^{-\|\Phi\|_C}  \nu(v)$.  \\
									
									For (\ref{full expan}), every terms except $(*)$ are controlled by 
									\begin{equation} \label{non main}
									C_{\Phi, \lambda_{\phi}} ( \delta_{\phi} + \bar{\delta} + \|h(0)\|_\infty + \sup_{0\leq s\leq t} \|h(s)\|^{2}_{\infty} ) . 
									\end{equation}
									
									
									
									For $(*)$, we choose $m({N})$ so that 
									\begin{equation} \label{opeator k split}
									k_{w,m}(u,v) := \mathbf{1}_{\{ |u-v|\geq\frac{1}{m}, \ |u|\leq m \}} k_{w}(u,v) ,
									\end{equation}
									satisfies $\int_{\mathbb{R}^3} |k_{w,m}(u,v)-k_{w}(u,v)| \ \dd u \leq \frac{1}{{N}}$ for sufficiently large ${N}\geq 1$. Then, by splitting $k_w$,
									
									\begin{equation} \label{* decomp} \begin{split}
									(*) &\leq \underline{ \int_0^t \int_0^s \int_0^{s^{\prime}} e^{- {\nu}_{\Phi}(0)(t-s^{\prime\prime})} \int_u k_{w,m}(u,v) \int_{u^{\prime}} k_{w,m}(u^{\prime},u) }  \\
									& \quad\quad \underline{ \times \int_{ u^{\prime\prime} } k_{w,m}(u^{\prime\prime},u^{\prime}) h (s^{\prime\prime},X^{\prime\prime}(s^{\prime\prime}),u^{\prime\prime}) \mathbf{1}_{|F- \mu_{E}| \geq \bar{\delta} \mu_{E}} \ \dd u^{\prime\prime} \dd u^{\prime} \dd u \dd s^{\prime\prime} \dd s^{\prime} \dd s }_{(**)}  \\
									& \quad\quad + O_{\O}(\frac{1}{N}) \sup_{0\leq s \leq t} \|h (s)\|_\infty    \\
									\end{split} \end{equation}
									
									\noindent We analyze $(**)$. We use Theorem \ref{prop_full_rank}, then 
									\begin{equation}
									\begin{split}
									&\exists i_{s} \in \{ 1,2,\cdots, I_{\O,N}\} \quad\text{such that}\quad X(s) \in \mathcal{O}_{i_{s}},  \\
									&\exists j_{s,s^{\prime}} \in \{ 1,2,\cdots, I_{\O,N}\} \quad\text{such that}\quad X(s^{\prime}; s, X(s;t,x,v), u) \in \mathcal{O}_{j_{s,s^{\prime}}},
									\end{split}
									\end{equation}
									and then we can define following sets for fixed $n,\vec{n},i,k,m,\vec{m},j,k^{\prime}$, where Theorem \ref{prop_full_rank} does not work.
									\begin{equation}
									\begin{split}
									R_1 &:= \{ u \ \vert \ u \notin B( \vec{n}\delta ; 2\delta ) \cap \{\R^3\backslash \mathcal{V}_{i_{s}}(\hat{\mathbf{e}_{1}}, \hat{\mathbf{e}_{2}})\}  \} ,    \\
									R_2 &:= \{ s^{\prime} \ \vert \ |s-s^{\prime}| \leq \delta^{\gamma} \} , \\
									R_3 &:= \{ s^{\prime} \ \vert \ | s^{\prime} -  \psi_{1}^{n,\vec{n},i,k,m,\vec{m},j,k^{\prime}} ( n \delta, X(n \delta;t,x,v), \vec{n}\delta ) |\lesssim_{N} \delta^{\gamma} \|\psi_{1} \|_{C^{0,\gamma}} \} ,  \\
									R_4 &:= \{ u^{\prime} \ \vert \ u^{\prime} \notin B( \vec{m}\delta ; 2\delta ) \cap \{\R^3\backslash \mathcal{V}_{j_{s,s^{\prime}}} (\p_{|u|}X, \p_{\hat{u}_{1}}X )   \} \} ,   \\
									R_5 &:= \{ s^{\prime\prime} \ \vert \ |s^{\prime}-s^{\prime\prime}| \leq \delta^{\gamma} \} ,  \\
									R_6 &:= \{ s^{\prime\prime} \ \vert \  \min_{r=1,2} | s^{\prime\prime} -  \psi_{r}^{n,\vec{n},i,k,m,\vec{m},j,k^{\prime}} ( m \delta, X( m \delta; n \delta, X(n \delta; t,x,v ), \vec{m}\delta  ), \vec{n}\delta  ) | \\
									&\quad\quad\quad\quad \lesssim_{N} \delta^{\gamma} \min_{r=1,2} \|\psi_{r} \|_{C^{0,\gamma}} \}.
									\end{split}
									\end{equation}
									Therefore,
									\begin{equation} \label{**}
									\begin{split}	
									(**) &= \underline{ \sum_{n =0}^{[t/\delta]+1} \sum_{|\vec{n}| \leq N}  \sum_{m=0}^{[t/\delta]+1} \sum_{|\vec{m}| \leq N} \sum_{k}^{K_{i_{s}}} \sum_{k^{\prime}}^{K^{\prime}_{j_{s,s^{\prime}}}} \int^{(n+1) \delta}_{ (n-1) \delta} \ \int^{t^{k}-\delta^{\gamma}}_{t^{k+1}+\delta^{\gamma}} \int^{t^{k^{\prime}}-\delta^{\gamma}}_{t^{k^{\prime}+1}+\delta^{\gamma}} e^{-{\nu}_{\Phi}(0)(t-s^{\prime\prime})} } \\ 
									& \quad\underline{ \times \int_{ |u|\leq N, |u^{\prime}|\leq N, |u^{\prime\prime}|\leq N } \ | h( s^{\prime\prime},X(s^{\prime\prime}), u^{\prime\prime} ) | \ \mathbf{1}_{|F- \mu_{E}| \geq \bar{\delta} \mu_{E}} \ \mathbf{1}_{R_{1}^{c} \cap R_{2}^{c} \cap R_{3}^{c} \cap R_{4}^{c} \cap R_{5}^{c} \cap R_{6}^{c}} }_{\text{(MAIN)}} \\
									& \quad + B + R,  \\
									\end{split}	
									\end{equation}
									where $B$ term corresponds to where trajectory is near bouncing points and $R$ corresponds to where $(u, s^{\prime}, u^{\prime}, s^{\prime\prime})$ is in one of $R_1\sim R_6$. So we have the following small estimates for $B$ and $R$.
									
									\begin{equation} \label{BR}
									\begin{split}
									B &\leq \int_0^t \int_0^s \int_0^{s^{\prime}} e^{- {\nu}_{\Phi}(0)(t-s^{\prime\prime})} \int_{|u|\leq N} k_{w,m}(u,v) \int_{|u^{\prime}|\leq N} k_{w,m}(u^{\prime},u)   \\
									& \quad\quad \times \int_{|u^{\prime\prime}|\leq N} k_{w,m}(u^{\prime\prime},u^{\prime}) h (s^{\prime\prime},X^{\prime\prime}(s^{\prime\prime}),u^{\prime\prime}) \mathbf{1}_{|F- \mu_{E}| \geq \bar{\delta} \mu_{E}}   \\
									& \quad\quad \times \mathbf{1}_{|s^{\prime}-t^{k}|\leq \delta^{\gamma} \ \text{or} \ |s^{\prime\prime}-t^{k^{\prime}}|\leq \delta^{\gamma} }  \\
									&\leq C_{N}\delta^{\gamma} \sup_{0\leq s \leq t} \|h(s)\|_\infty,  \\
									R &\leq \int_0^t \int_0^s \int_0^{s^{\prime}} e^{- {\nu}_{\Phi}(0)(t-s^{\prime\prime})} \int_{|u|\leq N} k_{w,m}(u,v) \int_{|u^{\prime}|\leq N} k_{w,m}(u^{\prime},u)   \\
									& \quad\quad \times \int_{|u^{\prime\prime}|\leq N} k_{w,m}(u^{\prime\prime},u^{\prime}) h (s^{\prime\prime},X^{\prime\prime}(s^{\prime\prime}),u^{\prime\prime}) \mathbf{1}_{|F- \mu_{E}| \geq \bar{\delta} \mu_{E}}   \\
									& \quad\quad \times \mathbf{1}_{ R_1 \cup R_2 \cup R_3 \cup R_4 \cup R_5 \cup R_6 }  \\
									&\leq C_{N}\delta^{\gamma} \sup_{0\leq s \leq t} \|h(s)\|_\infty,  \\
									\end{split}
									\end{equation}
									
									For $\text{(MAIN)}$ in (\ref{**}), we are away from two sets $B$ and $R$. Under the condition of $(u, s^{\prime}, u^{\prime}, s^{\prime\prime}) \in R_{1}^{c} \ \cap \ R_{2}^{c} \ \cap R_{3}^{c} \ \cap R_{4}^{c} \ \cap R_{5}^{c} \ \cap R_{6}^{c}$, indices $n, \vec{n}, i_{s}, k, m, \vec{m}, j_{s,s^{\prime}}, k^{\prime}$ are determined so that
									\begin{eqnarray*}
										t \ &\in& \ [  (n-1)\delta,  (n+1)\delta ],   \\
										X(s;t,x,v) \ &\in& \ \mathcal{O}_{i_{s}} ,  \\
										X(s^{\prime};s,X(s;t,x,v),u) \ &\in& \ \mathcal{O}_{j_{s,s^{\prime}}} ,  \\
										u \ &\in& \ B(\vec{n}\delta;2\delta) \cap \R^3\backslash\mathcal{V}_{i_{s}}( \hat{\mathbf{e}}_{1}, \hat{\mathbf{e}}_{2} ) , \\
										u^{\prime} \ &\in& \ B(\vec{m}\delta;2\delta) \cap \R^3\backslash\mathcal{V}_{j_{s,s^{\prime}}}(\p_{|u|}X, \p_{\hat{u}_{1}}X )  .
									\end{eqnarray*}
									We can apply Theorem \ref{prop_full_rank} which gives local time-independent lower bound of 
									\[
									\Big| \det(\frac{\p(X(s^{\prime\prime}))}{\p(|u^{\prime}|, \zeta_1, \zeta_2)}) \Big| \ \geq \ \epsilon^{\prime}_{\delta } .
									\]
									Note that $\{ \zeta_{1}, \zeta_{2}\} \subset\{ |u|, \hat{u}_{1}, \hat{u}^{\prime}_{1},  \hat{u}^{\prime}_{2}  \}$ are chosen variables in Theorem \ref{prop_full_rank}   \ and $\{ \zeta_{3}, \zeta_{4}\} \subset\{ |u|, \hat{u}_{1}, \hat{u}^{\prime}_{1},  \hat{u}^{\prime}_{2}  \}$ are unchosen variables. Let us use $\mathcal{P}$ to denote projection of $B(\vec{n}\delta;2\delta) \cap \R^3\backslash\mathcal{V}_{i_{s}}( \hat{\mathbf{e}}_{1}, \hat{\mathbf{e}}_{2} )  \times
									B(\vec{m}\delta;2\delta) \cap \R^3\backslash\mathcal{V}_{j_{s,s^{\prime}}}(\p_{|u|}X, \p_{\hat{u}_{1}}X )$ into $\R^3$ which corresponds to $(|u^{\prime}|, \zeta_1, \zeta_2)$ components. If we choose sufficiently small $\delta$, there exist small $r_{\delta,n,\vec{n},i,k,m,\vec{m},j,k^{\prime}}$ such that there exist one-to-one map $\mathcal{M}$, 
									\begin{eqnarray*}
										\mathcal{M} &:& \mathcal{P} \Big( B(\vec{n}\delta;2\delta) \cap \R^3\backslash\mathcal{V}_{i_{s}}( \hat{\mathbf{e}}_{1}, \hat{\mathbf{e}}_{2} )  \times
										B(\vec{m}\delta;2\delta) \cap \R^3\backslash\mathcal{V}_{j_{s,s^{\prime}}}(\p_{|u|}X, \p_{\hat{u}_{1}}X ) \Big)  \\
										&& \mapsto
										B( X(s^{\prime\prime}; s^{\prime}, X(s^{\prime};s, X(s;t,x,v), u), u^{\prime}), r_{\delta,n,\vec{n},i,k,m,\vec{m},j,k^{\prime}} ).
									\end{eqnarray*} 
									So we perform change of variable for $\text{(MAIN)}$ in (\ref{**}) to obtain
									
									\begin{equation} \label{MAIN}
									\begin{split}	
									&\text{(MAIN)} \\
									&\leq  \sum_{n =0}^{[t/\delta]+1} \sum_{|\vec{n}| \leq N}  \sum_{m=0}^{[t/\delta]+1} \sum_{|\vec{m}| \leq N} \sum_{k}^{K_{i_{s}}} \sum_{k^{\prime}}^{K^{\prime}_{j_{s,s^{\prime}}}} \int^{(n+1) \delta}_{ (n-1) \delta} \ \int^{t^{k}-\delta^{\gamma}}_{t^{k+1}+\delta^{\gamma}} \int^{t^{k^{\prime}}-\delta^{\gamma}}_{t^{k^{\prime}+1}+\delta^{\gamma}} e^{-{\nu}_{\Phi}(0)(t-s^{\prime\prime})}  \\ 
									& \quad\quad \times \int_{u^{\prime\prime}} \dd u^{\prime\prime} \int_{\hat{u}_2, \zeta_3, \zeta_4} \ \mathbf{1}_{ |u|\leq N, |u^{\prime}|\leq N, |u^{\prime\prime}|\leq N } \ \dd \hat{u}_2 \dd \zeta_{3} \dd \zeta_{4} \\ 
									& \quad\quad \times \int_{ |u^{\prime}|, \zeta_1, \zeta_2 } \ \dd |u^{\prime}| \dd \zeta_{1} \dd \zeta_{2} \ 
									| h( s^{\prime\prime},X( s^{\prime\prime}), u^{\prime\prime} ) | \ \mathbf{1}_{|F- \mu_{E}| \geq \bar{\delta} \mu_{E}} \dd s \dd s^{\prime} \dd s^{\prime\prime}  
									\\
									&\leq  \sum_{n =0}^{[t/\delta]+1} \sum_{|\vec{n}| \leq N}  \sum_{m=0}^{[t/\delta]+1} \sum_{|\vec{m}| \leq N} \sum_{k}^{K_{i_{s}}} \sum_{k^{\prime}}^{K^{\prime}_{j_{s,s^{\prime}}}} \int^{(n+1) \delta}_{ (n-1) \delta} \ \int^{t^{k}-\delta^{\gamma}}_{t^{k+1}+\delta^{\gamma}} \int^{t^{k^{\prime}}-\delta^{\gamma}}_{t^{k^{\prime}+1}+\delta^{\gamma}} e^{-{\nu}_{\Phi}(0)(t-s^{\prime\prime})}  \\ 
									& \quad \times \int_{\hat{u}_2, \zeta_3, \zeta_4} \ \mathbf{1}_{ |u|\leq N, |u^{\prime}|\leq N, |u^{\prime\prime}|\leq N } \ \dd \hat{u}_2 \dd \zeta_{3} \dd \zeta_{4} \ \mathbf{1}_{|F- \mu_{E}| \geq \bar{\delta} \mu_{E}}  \\
									& \quad \times \int_{u^{\prime\prime}} \int_{ 	B( X(s^{\prime\prime}), r_{\delta,n,\vec{n},i,k,m,\vec{m},j,k^{\prime}} ) } \ 
									| h^{\ell+1}( s^{\prime\prime}, x, u^{\prime\prime} ) | \
									\frac{1}{ \epsilon^{\prime}_{\Omega, N, \| \Phi \|_{C^{2}}, \delta } } \dd x \dd u^{\prime\prime} \dd s \dd s^{\prime} \dd s^{\prime\prime}  
									\\
									&\leq \ C_{N,\delta,\Phi,\phi,\O} \int_{0}^{t} e^{-{\nu}_{\Phi}(0)(t-s^{\prime\prime})} \ \int_{\O}\int_{|u^{\prime\prime}|\leq N}  |h(s^{\prime\prime},x,u^{\prime\prime})| \ \mathbf{1}_{|F- \mu_{E}| \geq \bar{\delta} \mu_{E}} \  \dd u^{\prime\prime} \dd x \dd s^{\prime\prime}  \\
									&\leq C_{N,\delta,\Phi,\phi,\O} \sup_{0 \leq s^{\prime\prime} \leq t} \big\| h(s^{\prime\prime}) \mathbf{1}_{|F- \mu_{E}| \geq \bar{\delta} \mu_{E}} \big\|_{L^{1} (\O \times B_{N})}    \\
									&\leq C_{N,\delta,\Phi,\phi,\O} \Big\| \frac{w}{\sqrt{\mu}} \Big\|_{L^{\infty} (B_{N})} \sup_{0 \leq s^{\prime\prime} \leq t} \big\|   {(F(s^{\prime\prime}) -\mu_{E})} \mathbf{1}_{|F- \mu_{E}| \geq \bar{\delta} \mu_{E}} \big\|_{L^{1} (\O \times B_{N})} . \\
									\end{split}	
									\end{equation}
									
									\noindent From (\ref{est_via_entropy}) and (\ref{entropy_bound}), we can further bound it by
									\begin{equation}
									\begin{split}
									&\leq
									C_{N,\delta,\Phi,\phi,\O} \frac{1}{\bar{\delta}} \Big\| \frac{w}{\sqrt{\mu}} \Big\|_{L^{\infty} (B_{N})} \sup_{0 \leq s^{\prime\prime} \leq t } \Big\{
									\mathcal{H}(F(0)  )- \mathcal{H}(\mu_{E}) 
									\\
									&\quad\quad - \iint (F(s^{\prime\prime}) - \mu_{E}) + \iint  \Big(  \frac{|v|^{2}}{2} + \Phi(x) \Big) (F(s^{\prime\prime}) - \mu_{E}) 
									\Big\} .
									\end{split}
									\end{equation}
									
									\noindent Finally, utilizing (\ref{conserv_F_mass}), (\ref{normalize_M}) and (\ref{excess_energy}), we deduce that 
									\begin{equation}\label{main_final_E}
									\begin{split}
									\text{(MAIN)} &\leq C_{N,\delta,\Phi,\phi,\O} \frac{1}{\bar{\delta}} \Big\| \frac{w}{\sqrt{\mu}} \Big\|_{L^{\infty} (B_{N})} \\
									&\quad \times \Big( \mathcal{H}(F(0)) - \mathcal{H }(\mu_{E}) + \frac{\delta_{\phi}}{\lambda _{\phi}} \{1+ \sup_{0 \leq s^{\prime\prime } \leq t}\| wf (s^{\prime\prime } ) \|_{\infty}\} \Big).
									\end{split}
									\end{equation}

									\noindent For $\delta_{0}, \delta_{\phi},\frac{\delta_{\phi}}{\lambda _{\phi}} \ll 1 ,$ we collect (\ref{non main}), (\ref{* decomp}), (\ref{**}), (\ref{BR}), and (\ref{main_final_E}) to get,  \\
									\begin{equation}\label{time_uniform} 
									\sup_{0 \leq t \leq T_{1}} \| w f(t) \|_{\infty} \lesssim 
									\| wf _{0} \|_{\infty}  + \delta_{\phi} + \bar{\delta} +
									\mathcal{H}(0)- \mathcal{H}(\mu_{E}) 
									+ \frac{\delta_{\phi}}{\lambda_{\phi}} + \sup_{0 \leq t \leq T_{1}} \| w f(t) \|^2_{\infty}.
									\end{equation}
									By choosing small data we deduce $\sup_{0 \leq t \leq T_{1}} \| w f(t) \|_{\infty} < 2 (\delta_{0}+ C\delta_{\phi}) \ll 1$ from (\ref{local_est}). By continuity of $ \| w f(t) \|_{\infty} $ in Theorem \ref{local_existence} and a uniform bound, we conclude 
									\[T_{1}= \infty,
									\] and this proves the global-in-time existence.

								\end{proof}

								\section{A time-independent potential} 
								First we derive $L^2$-coercivity for the homogeneous linear Boltzmann of (\ref{E_eqtn})
								\begin{equation}\notag 
								\p_{t}f + v\cdot \nabla_{x} f  - \nabla_{x} \Phi( x) \cdot \nabla_{v} f  + e^{- \Phi  } L f 
								= 0,
								\end{equation}
								with specular reflection boundary condition on the boundary $\p\O$. From (\ref{normalize_M}),
								\begin{equation} 
								\iint_{\O\times\R^{3}} f(t)\sqrt{\mu_{E}}  = \iint_{\O\times\R^{3}} f_{0}\sqrt{\mu_{E}}
								=0 , \label{mass_linear} 
								\end{equation}
								\begin{equation}\begin{split}  \label{energy_linear}
								\iint_{\O\times\R^{3}}
								f(t)\Big( \frac{|v|^{2}}{2} + \Phi  \Big)\sqrt{\mu_{E}} 
								=  \iint_{\O\times\R^{3}} f_{0}\Big( \frac{|v|^{2}}{2} + \Phi   \Big) \sqrt{\mu_{E}}  =0
								.
								\end{split} \end{equation}
								If the domain is axis-symmetric (\ref{axis-symmetric}) and $\Phi$ is degenerated (\ref{degenerate})	then 
								\begin{equation}\label{angular_linear}
								\iint_{\O \times \R^{3}} f (t) \{ (x-x^{0} ) \times \varpi \} \cdot v \sqrt{\mu_{E}} =   \iint_{\O \times \R^{3}} f _{0} \{ (x-x^{0} ) \times \varpi \} \cdot v \sqrt{\mu_{E}}=0 .
								\end{equation}


								We prove Proposition \ref{prop_coercivity} by the contradiction argument of the proof of Proposition 11 in \cite{Guo10}. 
								We first study the geometric lemma, which allows estimating near the boundary via the interior bound, and postpone the proof of the proposition. Define the distance function toward the boundary as 
								\begin{equation}\label{dist_bdry}
								\dist(x,\p\O) = \inf\{|x-y|: y \in\p\O  \},
								\end{equation}
								which is well-defined if $\dist(x,\p\O) \ll1$. In this case there exists a unique $x_{*}\in\p\O$ satisfying $|x^{*}-x|=\dist(x,\p\O)$. We also define  
								\begin{equation}\label{n_bdry}
								n(x)= n(x_{*}) ,
								\end{equation}
								for $x \in\O$ with $\dist(x,\p\O) \ll1$.

								\begin{lemma}\label{boundary_interior}Let $g$ be a (distributional) solution to 
									\begin{equation}\label{eqtn_g}
									\p_{t} g + v\cdot\nabla_{x} g + E \cdot \nabla_{v} g= G,
									\end{equation}
									where $E= E(t,x) \in C^{1,\gamma}$. Then, for a sufficiently small $\e>0$,
									\begin{equation} \label{int_ext}
									\begin{split}
									&
									\int^{1-\e}_{\e}\|
									\mathbf{1}_{\dist(x, \p\O)<\e^{4} } 
									\mathbf{1}_{|n(x) \cdot v| > \e} g(t)
									\|_{2}^{2} \dd t
									\\
									&
									\quad\quad\quad \lesssim 
									\int^{1}_{0}\| \mathbf{1}_{\dist(x,\p\O)> \e^{3}/2 } g(t)\|_{2}^{2} \dd t
									+ \int^{1}_{0} 
									\iint_{\O\times\R^{3}}|gG|.
									\end{split}
									\end{equation}
								\end{lemma}
								\noindent Note that this lemma is true even for a time-dependent external field case.
								\begin{proof}
									For $x \in \bar{\Omega}$ with $\dist(x,\p\O)< \e^{4}$, $n(x) \cdot v< - \e$, and $y \in\p\Omega$ with $|y- x_{*}| \ll1$,
									\begin{eqnarray*}
										|X(t+\e;t,x,v)- y| 
										&\geq& |(X(t+\e;t,x,v)-y) \cdot n(x_{*})| \\
										& =& |  (x- y) \cdot n(x_{*})
										+ v\cdot n(x_{*}) \e^{2}  +
										O(1)\| E \|_{\infty} \frac{\e^{4}}{2}
										|\\
										&\geq&
										\e^{3} - \e^{4} - O(1)\| E \|_{\infty} \frac{\e^{4}}{2}
										\ \geq  \  \e^{3}/2.
									\end{eqnarray*}
									Hence
									\begin{equation}\label{lower_in}
									\text{dist} (X(t+\e;t,x,v) ,\O)= \inf_{y \in\p\O,  |y-x_{*}| \ll 1}  |X(t+\e;t,x,v)- y| \ge \e^{3}/2.
									\end{equation}
									We can prove the exactly same lower bound of $|X(t-\e;t,x,v)- y|$ when $n(x) \cdot v>  \e$. Hence we conclude, for $x\in\bar{\Omega}$ with $\dist(x,\p\O)< \e^{4}$ and $n(x) \cdot v> \e$, 
									\begin{equation}\label{lower_out}
									\text{dist} (X(t-\e;t,x,v) ,\O)= \inf_{y \in\p\O,  |y-x_{*}| \ll 1}  |X(t-\e;t,x,v)- y| \ge \e^{3}/2.
									\end{equation}
									
									\noindent Moreover, it is well-known that $(x,v) \mapsto(X(t+\e;t,x,v), V(t+\e;t,x,v))$ for \\
									$\dist(x,\p\O) < \e^{4}, \ n(x) \cdot v<-\e$ and $(x,v) \mapsto(X(t-\e;t,x,v), V(t+\e;t,x,v))$ for $\dist(x,\p\O) < \e^{4}, \ n(x) \cdot v>\e$ are local diffeomorphism (since they never hit the boundary) and satisfy
									\begin{equation}\label{unit_Jac}
									\text{Jac} \left(\frac{\p (X(t\pm\e;t,x,v), V(t\pm\e;t,x,v))}{\p(x,v)}\right)=1.
									\end{equation}
									Note that
									\begin{eqnarray*}
										\| X(t+ \e; t,\cdot, \cdot)\|_{C^{1,\gamma}}  \leq  \frac{\e^{2}}{2} \|E \|_{C^{1,\gamma}},\ \
										\|  V(t+ \e; t,\cdot, \cdot)\|_{C^{1,\gamma}}  \leq   \e   \| E \|_{C^{1,\gamma}}.
									\end{eqnarray*}
									By expansions, we conclude that there exist sufficiently small $\delta>0$ and $\e_{0}>0$ such that for all $0<\e< \e_{0}$, $(X(t+ \e; t,\cdot,\cdot), V(t+ \e; t,\cdot,\cdot))$ is one-to-one in $\{ (x,v) \in \bar{\O} \times \R^{3}: \dist(x,\p\O) < \e^{4}, \ n(x) \cdot v < -\e, \ |x-x^{0}| + |v-v^{0}| < \delta   \}$ and $(X(t- \e; t,\cdot,\cdot), V(t- \e; t,\cdot,\cdot))$ is so in $\{ (x,v) \in \bar{\O} \times \R^{3}: \dist(x,\p\O) < \e^{4}, \ n(x) \cdot v>  \e, \ |x-x^{0}| + |v-v^{0}| < \delta   \}$. 
									
									On the other hand, if $X( {t}+ \e ;  {t}, \tilde{x}, \tilde{v}) = X( {t}+ \e ;  {t},  {x},  {v})$ and $V( {t}+ \e ;  {t}, \tilde{x}, \tilde{v}) = V( {t}+ \e ;  {t},  {x},  {v})$, then $|v-\tilde{v}| \leq \| E \|_{\infty} \e$ and $|x-\tilde{x}| \leq 2\| E \|_{\infty} \e^{2}$. We deduce the same conclusion if $X( {t}- \e ;  {t}, \tilde{x}, \tilde{v}) = X( {t}- \e ;  {t},  {x},  {v})$ and $V( {t}- \e ;  {t}, \tilde{x}, \tilde{v}) = V( {t}- \e ;  {t},  {x},  {v})$. Hence for a sufficiently small $\e$, such $(x,v)$ and $(\tilde{x}, \tilde{v})$ are close as $|(x,v) - (\tilde{x},\tilde{v})|< \delta$. From the local one-to-one in the previous sentence we conclude $(x,v)=(\tilde{x},\tilde{v})$.
									
									Now we are ready to prove (\ref{int_ext}). Note that $$\frac{d}{ds} |g(s,X(s;t,x,v),V(s;t,x,v))|^{2} = 2 (gG)(s,X(s;t,x,v),V(s;t,x,v)).$$
									For $(x, v)$ with $\dist(x,\p\O)< \e^{4}$ and $n(x) \cdot v< - \e$, taking integration $s \in [t,t+\e]$ along the trajectory,
									\begin{eqnarray*}
										|g(t,x,v)|^{2} =  |g(t+\e,X(t+\e ), V(t+\e ) ) |^{2}
										- 2  \int^{t+\e}_{t}  2 (gG)(s,X(s ),V(s )) \dd s.
									\end{eqnarray*}
									From (\ref{lower_in}), $\dist(X(t+\e ), \p\O )\geq \e^{3}/2$. Using (\ref{unit_Jac}) and the one-to-one property of $(x,v) \mapsto (X(s),V(s))$ for any fixed $|s| \leq \e$, we take an integration over $\dist(x,\p\O)< \e^{4}$ and $n(x) \cdot v< - \e$ and conclude that 
									\begin{equation}\label{intg_+}
									\|
									\mathbf{1}_{\dist(x, \p\O)<\e^{4} } 
									\mathbf{1}_{ n(x) \cdot v  < -\e} g(t )
									\|_{2}^{2}  
									=
									\| \mathbf{1}_{\dist(x,\p\O)> \e^{3}/2 } g(t+\e)\|_{2}^{2} 
									+ \int^{t+\e}_{t} 
									\iint_{\O\times\R^{3}}|g(s)G(s)|.
									\end{equation}
									
									For the other case, $\dist(x,\p\O)< \e^{4}$ and $n(x) \cdot v> \e$, we repeat the same argument with changing $\e$ to $-\e$ and conclude that
									\begin{equation}\label{intg_-}
									\|
									\mathbf{1}_{\dist(x, \p\O)<\e^{4} } 
									\mathbf{1}_{ n(x) \cdot v  > \e} g(t )
									\|_{2}^{2}  
									=
									\| \mathbf{1}_{\dist(x,\p\O)> \e^{3}/2 } g(t-\e)\|_{2}^{2} 
									+ \int^{t}_{t-\e} 
									\iint_{\O\times\R^{3}}|g(s)G(s)|.
									\end{equation}
									Finally by $\int^{1-\e}_{\e} (\ref{intg_+}) \dd t$ and $\int^{1-\e}_{\e} (\ref{intg_-}) \dd t$, we conclude (\ref{int_ext}).
								\end{proof}

								\begin{proof}[\textbf{Proof of Proposition \ref{prop_coercivity}}] First, it is easy to check that equation (\ref{linearized_eqtn}) is translation invariant in time, i.e. $\tilde{f}(t,x,v) := f(t+c,x,v)$ also solves same equation for any $c$. Note that this is not true for time-dependent potential case anymore, unless the potential is periodic in time. Therefore it suffices to prove coercivity for finite time interval $t \in [0,1]$ and so we claim (\ref{coercive}) for $N=0$. 
									
									\noindent\textit{Step 1.} Assume that Proposition \ref{prop_coercivity} is wrong. 
									This means for any $m \gg1$ there exists a solution $f^{m}$ to (\ref{linearized_eqtn}) satisfying the specular reflection BC which solves

									\begin{equation}\label{eqtn_fm}
									\p_{t} {f}^{m} + v\cdot \nabla_{x }  {f} ^{m} - \nabla_x \Phi \cdot \nabla_v  {f}^{m} + 
									e^{- \Phi}L {f}^{m} =0, \ \ \text{for} \ t \in [0,1]
									\end{equation}
									and satisfies
									\begin{equation}\label{contra_coercivity_1}
									\int^{ 1}_{ 0} \| \mathbf{P} f^{m } (t) \|_{2}^{2} \dd t
									\geq m \int^{ 1}_{0} \| (\mathbf{I} - \mathbf{P}) f^{m } (t) \|_{\nu}^{2} \dd t.
									\end{equation}
									\noindent We define normalized form of $f^{m}$ by
									\begin{equation} \label{normal Zm}
									Z^{m} (t,x,v) : = \frac{ f^{m}(t,x,v)}{ \sqrt{\int^{1}_{0} \| \mathbf{P} f^{m} (t) \| _{2}^{2} \dd t  }} .
									\end{equation}
									Then $Z^{m}$ solves 
									\begin{equation}\label{eqtn_Zm}
									\p_{t} Z ^{m} + v\cdot \nabla_{x } Z^{m}- \nabla_x \Phi \cdot \nabla_v Z^{m} 
									+ e^{- \Phi}L Z^{m} =0,
									\end{equation}
									\noindent and 
									\begin{equation}\label{Zm_BC}
									Z^{m}(t,x,v) = Z^{m}(t,x,R_{x}v)\  \ \ \ x\in\p\O  ,
									\end{equation}
									and
									\begin{equation} \label{contra_coercivity_Z}
									\frac{1}{m}
									\geq  \int^{ 1}_{0} \| (\mathbf{I} - \mathbf{P}) Z ^{m}(t) \|_{\nu}^{2} \dd t.
									\end{equation}

									\vspace{4pt}
									
									\noindent\textit{Step 2. }We claim that 
									\begin{equation}\label{claim_Zm}
									\sup_{m}\sup_{0 \leq t \leq 1}
									\| Z^{m} (t) \|_{2}^{2} < \infty.
									\end{equation}
									From (\ref{eqtn_Zm}), for $0 \leq t \leq 1$,
									\begin{equation}\label{energy_Zm}
									\|Z^{m}(t)\|_{2}^{2}
									+ \int^{t}_{0} e^{- \Phi}\big( L Z^{m}, Z^{m} \big)
									=
									\|Z^{m}(0)\|_{2}^{2}
									%
									.
									\end{equation}
									From the non-negativity of $L$
									,
									\begin{equation} \label{sup_Zm}
									\sup_{0 \leq t \leq 1}\| Z^{m}(t) \|_{2}^{2} 
									\leq  \| Z^{m}(0) \|_{2}^{2}
									.
									\end{equation}

									On the other hand, by integration $\int^{1}_{0}(\ref{energy_Zm}) \dd t$ and  utilizing (\ref{contra_coercivity_Z}) and (\ref{normal Zm}),
									\begin{equation} \label{bound_Zm}
									\| Z^{m}(0) \|_{2}^{2} 
									\lesssim_{\Phi
									}
									\int^{1}_{0} \| Z^{m} \|_{2}^{2} 
									+ \int^{1}_{0}  \| (\mathbf{I} - \mathbf{P}) Z^{m} \|_{\nu}^{2}
									\lesssim_{\Phi
									} 
									1+ \frac{1}{m}.
									\end{equation} 
									Therefore, we prove the claim (\ref{claim_Zm}) from (\ref{sup_Zm}) and (\ref{bound_Zm}).

									\vspace{4pt}
									
									\noindent\textit{Step 3.} Therefore, the sequence $\{Z^{m}\}_{m\gg 1 }$ is uniformly bounded in \\
									$\sup_{0 \leq t \leq 1} \| g(t) \|_{\nu}^{2} \dd t$. By the weak compactness of $L^{2}$-space, there exists weak limit $Z$ such that 
									\begin{equation}\label{limit_Z}
									Z^{m}\rightharpoonup Z  \ \ \text{in} \ \ 
									L^{\infty} ([0,1]; L^{2}_{\nu}(\O\times\R^{3})) \cap 
									L^{2} ([0,1] ; L^{2}_{\nu}( \O\times\R^{3})).
									\end{equation}
									
									\noindent Therefore, in the sense of distributions, $Z$ solves
									\begin{equation}\label{eqtn_Z}
									\p_{t} Z  + v\cdot \nabla_{x } Z  - \nabla_x   \Phi  \cdot \nabla_v Z =0.
									\end{equation}

									Now we consider the limit of the linear conservation laws. Note that, taking a weak limit $Z^{m} \rightharpoonup Z $ in $L^{\infty}_{t} L^{2}_{x,v}$ of (\ref{limit_Z}) and using (\ref{mass_linear}), (\ref{energy_linear}), and (\ref{normal Zm}), we deduce linear conservation laws, for almost every $t \in [0,1]$,
									\begin{equation}\label{coserv_Z} 
									\iint_{\O\times\R^{3}} Z(t) \sqrt{\mu_{E}} = 0, \ \
									\iint_{\O\times\R^{3}} Z(t) \big( \frac{|v|^{2}}{2} + \Phi
									\big) \sqrt{\mu_{E}}=0. 
									\end{equation}
									In the case that both (\ref{axis-symmetric}) and (\ref{degenerate}) hold, from (\ref{angular_linear}), 
									\begin{equation}\label{angular_Z} \iint_{\O \times \R^{3}} \{ (x-x^{0} ) \times \varpi \} \cdot v Z  (t)\sqrt{\mu_{E}}=0.\end{equation}
									
									%

									On the other hand, since 
									\[
									\mathbf{P}Z ^{m} \rightharpoonup \mathbf{P}Z \quad\text{and}\quad (\mathbf{I-P})Z ^{m} \rightarrow 0 \quad\text{in}\quad \int_0^1 \|\cdot\|_{\nu}^2 \dd t, 
									\]
									we know that weak limit $Z$ has only hydrodynamic part, \textit{i.e}
									\begin{equation} \label{limit_Z}
									Z(t,x,v) = \{a(t,x) + v\cdot b(x,v) +  |v|^{2}c(t,x)\}\sqrt{\mu_E},
									\end{equation}
									and 
									\begin{equation}\begin{split}\label{bound_Z}
									\int^{1}_{0} \| Z\|_{\nu}^{2}\dd t \leq  \liminf_{m\rightarrow \infty}  \int^{1}_{0} \| Z^{m}\|_{\nu}^{2}\dd t \leq 1+ \frac{1}{m}\rightarrow 1.
									\end{split}
									\end{equation}
									\vspace{4pt}
									
									\noindent\textit{Step 4. Interior compactness. } Let $\chi_{\e}: \bar{\Omega} \rightarrow [0,1]$ be a smooth function such that $\chi_{\e}(x) =1$ if $\dist(x,\p\O)> 2\e^{4} $ and $\chi_{\e}(x) =0$ if $\dist(x,\p\O)< \e^{4}$. From (\ref{eqtn_Zm}),
									\begin{eqnarray*}
										[\p_{t} + v\cdot \nabla_{x}  ](\chi_{\e}  Z^{m})
										= 
										\nabla_{x} \Phi \cdot \nabla_{v}(\chi_{\e}  Z^{m})
										+ v\cdot \nabla_{x} \chi_{\e} Z^{m}
										- e^{-\Phi} L (\chi_{\e}Z^{m})
									\end{eqnarray*} 
									From the standard Average lemma in \cite{gl}, $\chi_{\e} Z^{m}$ is compact i.e. 
									\begin{equation}\label{compact_int}
									\chi_{\e} Z^{m} \rightarrow \chi_{\e} Z \ \ \text{strongly in } 
									L^{2}([0,1]; L^{2}_{\nu}(\O\times\R^{3})).
									\end{equation}
									\vspace{4pt}
									
									\noindent\textit{Step 5. Near-boundary compactness. } First we claim that 
									\begin{equation}\label{bdry_m}
									\begin{split}
									&\int^{1-\e}_{\e}	\| \big(Z^{m} (t,x,v)- Z(t,x,v) \big) 
									\mathbf{1}_{
										\dist(x,\p\O) < {\e^{4}} 
									}
									\mathbf{1}_{|n(x) \cdot v|> \e} \|_{2}^{2}\\
									&
									\lesssim \int^{1}_{0}
									\| \big(Z^{m} (t,x,v)- Z(t,x,v) \big)\mathbf{1}_{\dist(x,\p\O) > \frac{\e^{3}}{2} } \|_{2}^{2}+ O(\frac{1}{\sqrt{m}})
									\end{split}
									\end{equation}
									
									\noindent We are looking up the equation of $Z^{m}-Z$. 
									By subtracting (\ref{eqtn_Zm}) from (\ref{eqtn_Z}),
									\begin{equation}\label{eqtn_Zm_Z}\begin{split}
									[\p_{t}  + v\cdot \nabla_{x }  - \nabla_x \Phi \cdot \nabla_v] (Z^{m}
									-Z) + e^{- \Phi}L Z^{m} = 0. 
									\end{split}\end{equation} 
									
									\noindent Now we apply Lemma \ref{boundary_interior} to (\ref{eqtn_Zm_Z}) by equating $g$ and $G$ with $Z^{m}
									-Z$ and the RHS of (\ref{eqtn_Zm_Z}) respectively. Then
									\begin{eqnarray*}
										&& \int^{1-\e}_{\e}\|
										\mathbf{1}_{\dist(x, \p\O)<\e^{4} } 
										\mathbf{1}_{|n(x) \cdot v| > \e} 
										(Z^{m}-Z)
										(t)
										\|_{2}^{2} \dd t
										\\
										&& \lesssim  
										\int^{1}_{0}\| \mathbf{1}_{\dist(x,\p\O)> \e^{3}/2 } (Z^{m}-Z)(t)\|_{2}^{2} \dd t
										+
										\int^{1}_{0} 
										\iint_{\O\times\R^{3}} |Z^{m}-Z|\langle v\rangle
										|(\mathbf{I}- \mathbf{P})Z^{m}|
										.
									\end{eqnarray*}
									Using the H\"older's inequality, we bound the last line of the above estimate by
									\begin{eqnarray*}
										\sqrt{m}\int^{1}_{0} \| (\mathbf{I} - \mathbf{P}) Z^{m} \|_{\nu}^{2} 
										+
										\frac{1}{ \sqrt{m}}
										\int^{1}_{0} \| Z^{m} \|_{\nu}^{2} + \| Z\|_{\nu}^{2}
										.
									\end{eqnarray*}
									By (\ref{claim_Zm}) and (\ref{contra_coercivity_Z}), we conclude (\ref{bdry_m}).

									On the other hand, from (\ref{contra_coercivity_Z}), (\ref{limit_Z}), and (\ref{claim_Zm}),
									\begin{equation}
									\label{grazing_small}
									\begin{split}
									&  \int^{1-\e}_{\e} \| ( Z^{m}- Z) \mathbf{1}_{|n(x) \cdot v| \leq \e } \|_{2}^{2} \\
									&\leq \   \int_{\e}^{1-\e} \| (\mathbf{I} - \mathbf{P}) Z^{m} \|_{\nu}^{2} + 
									O(\e) \int^{1-\e}_{\e} \| \mathbf{P} Z^{m} \|_{2}^{2} + \| \mathbf{P} Z  \|_{2}^{2}\\
									&\leq \   \frac{1}{m} + 	O(\e)
									\end{split}\end{equation}
									
									\vspace{4pt}
									
									\noindent\textit{Step 6. Summary. }
									For given $\e>0$, we can choose $m\gg_{\e}1$ such that 
									\begin{eqnarray*}
										&&\int^{1}_{0} \iint_{\O \times \R^{3}}  | 
										Z^{m} - Z |^{2}\\
										&\leq&
										\int^{1}_{1-\e} \iint_{\O \times\R^{3}}+ \int^{\e}_{0} \iint_{\O \times\R^{3}} + 
										\int^{1-\e}_{\e} \iint_{\O_{\e} \times \R^{3}}\\
										&&+ \int^{1-\e}_{\e} \iint_{
											\substack{\O \backslash \O_{\e} \times \R^{3}\\
												\cap \ 
												\{ |n(x) \cdot v|<\e  \ \text{or} \ |v|\geq \e^{-1}\}
											}}
											+ \int^{1-\e}_{\e} \iint_{
												\substack{\O \backslash \O_{\e} \times \R^{3}\\
													\cap \ 
													\{ |n(x) \cdot v|\geq\e  \ \text{and} \ |v|\leq \e^{-1}\}
												}}\\
												&<&C \e ,
											\end{eqnarray*} 
											where we have used (\ref{claim_Zm}), (\ref{compact_int}), (\ref{bdry_m}), and (\ref{grazing_small}). Therefore, we conclude that $Z^{m} \rightarrow Z$ strongly in $L^{2}([0,1] \times\Omega \times\R^{3})$ and hence 
											\begin{equation}\label{Z=1}
											\int^{1}_{0} \| Z\|_{2}^{2}=1.
											\end{equation}
											
											\vspace{4pt}
											
											
											
											\noindent\textit{Step 7.} We consider the boundary condition of $Z$. Fix a small constant $\delta>0$. In order to control $Z$ in $\{ (x,v) \in \gamma_{\pm}: |n(x) \cdot v| <\delta \}$ we use smooth functions $\phi_{\pm}^{\delta} :   \bar{\O} \times \R^{3} \rightarrow [0,1]$ where $\phi_{\pm}^{\delta} \equiv 1$ on $\{ (x,v) \in \gamma_{\pm}: |n(x) \cdot v| <\delta \}$ and $\phi_{\pm}^{\delta}\equiv 0$ on $\{ (x,v) \in \gamma_{\pm}: |n(x) \cdot v| >2\delta \}$ respectively. 
											
											From the weak formulation, we have $(\p_{t} + v\cdot\nabla_{x} - \nabla_{x} \Phi \cdot \nabla_{v}) |Z|^{2}=0$. Testing it with $\phi^{\delta}_{\pm}$, we obtain
											\begin{eqnarray*} \int^{1}_{0}
												\int_{\gamma} |Z |^{2} \phi^{\delta}_{\pm}  (n \cdot v) 
												&=&
												- \iint _{\O\times \R^{3}}  \phi^{\delta}_{\pm} |Z(1)|^{2} + \iint_{\O\times \R^{3}}  \phi^{\delta}_{\pm}|Z(0)|^{2}
												\\
												&&							+\int^{1}_{0}
												\iint_{\O\times \R^{3}} 
												- v\cdot \nabla_{x} \phi^{\delta}_{\pm} |Z |^{2}+ \nabla_{x} \Phi \cdot \nabla_{v} \phi^{\delta}_{\pm}   |Z |^{2}
												.
											\end{eqnarray*}
											From (\ref{limit_Z}) and (\ref{bound_Z}), we deduce that $Z \in L^{2} ( \{ (x,v) \in\gamma_{\pm} : |n(x) \cdot v|< \delta\})$, and $a,b,c \in L^{2}([0,1] \times \p\O)$ such that 
											\begin{equation}\label{abc_boundary}\begin{split}
											\int^{1 }_{0} \int_{\p\O} |a |^{2} +|b |^{2} +|c |^{2} 
											\lesssim \int^{1}_{0} \| Z\|_{\nu}^{2} 
											+ \sup_{0 \leq t \leq 1} \| Z(t)\|_{2}^{2}.
											\end{split}\end{equation}

											Now we claim that 
											\begin{equation}\label{Z_BC}
											Z(t,x,v) = Z(t,x,R_{x}v) \ \ \ \text{almost every }  \ [\delta, 1-\delta] \times \gamma_{-}.
											\end{equation}
											Let $\phi: \bar{\O} \times \R^{3}\rightarrow \R$ be a smooth bounded function with strong decay in $v$. Moreover, we assume that this test function is even function in $\phi(n(x) \cdot v)$ at the boundary. Testing (\ref{eqtn_Z}) with such function $\phi$, we have 
											\begin{equation}
											\begin{split}
											\int^{1}_{0} \int_{\gamma} Z \phi(n(x) \cdot v) 
											&= 
											- \iint_{\O\times\R^{3}} (Z(1)  - Z(0))\phi \\
											&\quad + 
											\int^{1}_{0} \iint_{\O\times\R^{3}} Z (- v\cdot \nabla_{x} \phi \quad + \nabla_{x} \Phi \cdot \nabla_{v} \phi )   
											.
											\label{Z_weak}
											\end{split}
											\end{equation}
											On the other hand, employing the same test function, from (\ref{eqtn_Zm}) and (\ref{Zm_BC}), we conclude that
											\begin{eqnarray*}
												0
												&=&
												- \iint_{\O\times\R^{3}} (Z^{m}(1) - Z^{m}(0)) \phi
												\\
												&&\quad +
												\int^{1}_{0} \iint_{\O\times\R^{3}} Z^{m} (- v\cdot \nabla_{x} \phi +\nabla_{x} \Phi \cdot \nabla_{v} \phi )  + 
												\int^{1}_{0}  \iint_{\O\times\R^{3}} 
												e^{  \Phi}  L Z^{m}\phi 
												.
											\end{eqnarray*} 
											By passing to the limit $m \rightarrow \infty$, from (\ref{limit_Z}) and (\ref{contra_coercivity_Z}), we realize that RHS of (\ref{Z_weak}) equals zero. Therefore, we conclude that 
											\begin{equation}
											\int^{1}_{0} \int_{\gamma} Z \phi(n(x) \cdot v) 
											=0.
											\end{equation}
											for any smooth function $\phi$ which is even in $n(x) \cdot v$ at the boundary. This proves (\ref{Z_BC}).
											
											Finally combining (\ref{Z_BC}), (\ref{limit_Z}), and (\ref{abc_boundary}), we prove (\ref{specular_Z}).

											\vspace{4pt}
											
											\noindent\textit{Step 8.} We claim (\ref{Z=0}). We consider the system of $a,b,$ and $c$ which is obtained by plugging (\ref{limit_Z}) in (\ref{eqtn_Z}). From \cite{K2}, in the sense of distributions, they solve (\ref{macro_eq}).
											
											The first equation of (\ref{macro_eq}) implies that $c$ is only a function of $t$, \textit{i.e.} $c=c(t)$. From the first three equations of (\ref{macro_eq}) we can get 
											\begin{equation}\label{b_form}
											b(t,x) = - \p_{t} c(t) x + \varpi(t) \times x + m(t).
											\end{equation}	
											The proof of (\ref{b_form}) is based on direct computations. (See Lemma 12 in \cite{Guo10} for the details) 
											
											From the second equation of (\ref{macro_eq}), we obtain $\nabla_{x} \cdot b = - 3c^{\prime} (t)$. By the divergence theorem and (\ref{specular_Z}), 
											\[
											-3 c^{\prime} (t) |\O| = \int_{\p\O} b\cdot {n} = 0 .
											\]
											Therefore, $c^{\prime}(t) = 0,  c(t) = c_0$, and $b = \varpi(t) \times x + m(t)$. We conclude 
											\begin{equation}\label{c_0}
											c(t,x)= c_{0} .
											\end{equation}

											We split into two cases $\varpi=0$ and $\varpi \neq 0$.
											
											\bigskip
											
											\textit{Case of $\varpi=0$.}	If $\varpi=0$, then $b(t)=m(t)$. From (\ref{specular_Z}) we deduce that
											\begin{equation}
											b(t)\equiv m(t) \equiv 0.\label{b_vanishing}
											\end{equation}
											Then from the last equation of (\ref{macro_eq}), $a=a(x)$. From the fourth equation of (\ref{macro_eq}), for some constant $C$, we obtain that 
											\begin{equation}\label{a_Phi}
											a(t,x)=   2 c_{0}  \Phi (x)+ C.
											\end{equation}
											Plugging (\ref{c_0}) and (\ref{a_Phi}) into the conservation laws (\ref{coserv_Z}), 
											\begin{equation}\notag
											\iint ( 2 c_{0} \Phi(x) + C + c_{0}|v|^{2}) \mu_{E} =0
											= \iint ( 2 c_{0} \Phi(x) + C + c_{0}|v|^{2}) (\frac{|v|^{2}}{2} + \Phi (x)) \mu_{E}.
											\end{equation}
											From the direct computations, we deduce $c_{0}=0=C$ and hence (\ref{Z=0}).


											
											\bigskip
											
											\textit{Case of $\varpi\neq0$.}
											From (\ref{specular_Z}), at the boundary,	 
											\begin{eqnarray*}
												b(t,x)\cdot {n}(x)  =   \big( \varpi(t)\times x + m(t) \big)\cdot {n}(x)=0.
											\end{eqnarray*}
											Since $m(t)$ is fixed vector for given $t$, we decompose $m(t)$ into the parallel and orthogonal components to $\varpi(t)$ as
											\[
											m(t) = \alpha(t) \varpi(t) - \varpi(t)\times x_0(t).
											\] 
											Then 
											\begin{eqnarray}
											b(t,x)\cdot {n}(x) &=& \big( \varpi(t)\times x + m(t) \big)\cdot {n}(x)\notag  \\
											&=& \big( \varpi(t)\times (x-x_0(t)) \big) \cdot {n}(x) + \alpha(t)\varpi(t)\cdot {n}(x) = 0,\quad\forall x\in\p\O. \label{o_a=0}  
											\end{eqnarray}
											
											\noindent Choose $t$ with $\varpi(t) \neq 0$. We can pick $x^{\prime}\in\p\O$ such that $\varpi(t)\parallel  {n}(x^{\prime})$. Then the first term of the RHS in (\ref{o_a=0}) is zero. Hence we deduce, from (\ref{b_form}) and (\ref{c_0}), that
											\begin{equation}\label{alpha=0}
											\alpha(t)=0 \ \ \text{and} \ \ b(t,x)= \varpi(t) \times \big(x-x^{0}(t)\big).
											\end{equation}
											This yields
											\begin{equation}\label{axis_0}
											\big( \varpi(t)\times (x-x_0(t)) \big) \cdot {n}(x)=0,\quad\forall x\in\p\O .
											\end{equation}
											The equality (\ref{axis_0}) implies that $\O$ is axis-symmetric with the origin $x_0(t)$ and the axis $\varpi(t)$. From (\ref{angular_Z}) and (\ref{alpha=0}),
											\begin{eqnarray*}
												0
												&=& \iint_{\O}|\varpi\times(x-x_0(t))\cdot v|^2 \mu e^{-\Phi} \dd x \dd v  .
											\end{eqnarray*}
											Therefore, we conclude that $\varpi(t)\equiv 0$ for all $t$. This proves $b(t,x) \equiv 0$. Then we follow the argument of \textit{Case of} $\varpi=0$ and deduce (\ref{Z=0}).
											%
											
											\vspace{4pt}
											
											\noindent\textit{Step 9. }Finally we deduce a contradiction from (\ref{Z=1}) and (\ref{Z=0}). Hence we prove the theorem.\end{proof}

										Once such a coercivity is proven, we can directly deduce an exponential decay.  
										\begin{corollary}\label{decay_U}Assume the same conditions in Proposition \ref{prop_coercivity}. 
											Then there exists $\lambda>0$ such that a solution of (\ref{linearized_eqtn}) satisfies
											\begin{equation}\label{U_decay}
											\sup_{0 \leq t}e^{\lambda t} 	\|f(t)\|_{2}^{2} \lesssim \| f_{0} \|_{2}^{2}.
											\end{equation}
											
										\end{corollary}
										\begin{proof}[Proof of Corollary \ref{decay_U}] Assume that $0 \leq t \leq 1$. From the energy estimate of (\ref{linearized_eqtn}) in a time interval $[0,N]$, 
											\begin{equation}\label{energy_f} 
											\|f(N)\|_{2}^{2} + \int^{N}_{0} \iint_{\O\times\R^{3}} e^{- \Phi} f Lf 
											\leq \|f(0)\|_{2}^{2}  .
											\end{equation}
											
											\noindent From (\ref{linearized_eqtn}), for any $\lambda>0$
											\begin{equation}\label{eqtn_lamda}
											\big[\p_{t}  + v\cdot \nabla_{x}  - \nabla_{x} \Phi \cdot \nabla_{v}  \big] (e^{\lambda t} f)  + e^{- \Phi  } L (e^{ \lambda t}f )
											= \lambda e^{\lambda t} f.
											\end{equation}
											By the energy estimate,
											\begin{equation}\label{energy_f}  
											\| e^{\lambda t}f(N)\|_{2}^{2} + \underbrace{\int^{N}_{0} \iint_{\O\times\R^{3}} e^{- \Phi }e^{2\lambda s} f Lf } _{(I)}
											-  
											{
												\lambda \int^{N}_{0} \iint_{\O\times\R^{3}}
												|e^{\lambda s} f(s)|^{2}}
											\leq \|f(0)\|_{2}^{2} 
											. 
											\end{equation}
											
											\noindent Firstly we consider $(I)$ in (\ref{energy_f}). From (\ref{semi-positive}), the term $(I)$ in (\ref{energy_f}) is bounded below by 
											\begin{eqnarray*}
												%
												(I)\geq  
												\delta_{L} 
												\int^{N}_{0} \int_{\O} e^{- \Phi } \int_{\R^{3}} \langle v\rangle | e^{ \lambda s}(\mathbf{I} - \mathbf{P})f|^{2} \geq  \delta_{L} e^{-\| \Phi  \|_{\infty}} \int^{N}_{0} \| e^{ \lambda s} (\mathbf{I} - \mathbf{P}) f  \|_{\nu}^{2}.
											\end{eqnarray*}
											By time translation, we apply (\ref{coercive}) to obtain
											\begin{eqnarray*} 
												(I)
												&\geq&   \frac{\delta_{L} e^{-\| \Phi  \|_{\infty}}}{2} \int^{N}_{0} \| e^{ \lambda s}(\mathbf{I} - \mathbf{P}) f \|_{\nu}^{2}  + \frac{\delta_{L} e^{-\| \Phi \|_{\infty}}}{2 C} \int^{N}_{0} \|  e^{ \lambda s} \mathbf{P} f \|_{2}^{2} \\
												&\geq&  \frac{\delta_{L} e^{-\| \Phi  \|_{\infty}}}{2 C} \int^{N}_{0} \|  e^{ \lambda s}  f \|_{2}^{2}.
											\end{eqnarray*}
											
											%
											%
											%
											
											\noindent Therefore, we derive 
											\begin{equation}\label{energy_N0} 
											e^{\lambda N}\| f(N)\|_{2}^{2} 
											+ \Big(
											\frac{\delta_{L} e^{- \| \Phi \|_{\infty}}}{2C} - \lambda
											\Big)
											\int^{N}_{0} \| e^{ \lambda s} f \|_{2}^{2}  
											\leq \| f (0)\|_{2}^{2}. 
											\end{equation}
											
											On the other hand, from the energy estimate of (\ref{linearized_eqtn}) in a time interval $[N,t]$, using (\ref{semi-positive}), we have 
											\begin{equation}\label{energy_tN}
											\| f (t) \|_{2}^{2}  \leq \| f(N) \|_{2}^{2}.
											\end{equation}
											
											Finally choosing $\lambda \ll 1$, from (\ref{energy_N0}) and (\ref{energy_tN}), we conclude that 
											\begin{equation}
											e^{\lambda t } \| f(t) \|_{2}^{2} = e^{\lambda (t-N)} e^{\lambda N }\| f(N) \|_{2}^{2}
											\leq 2 \| f(0) \|_{2}^{2},
											\end{equation}
											and prove (\ref{U_decay}). \end{proof}

										\begin{proof}[\textbf{Proof of Theorem \ref{theorem_decay}}]  
											We sketch the proof of the nonlinear $L^\infty$ decay. Note that we have shown local existence result in (\ref{local_existence}) and global stability theorem \ref{theorem_time}, so we perform exponential decaying a-priori estimate for nonlinear problem to finish proof. 
											
											Note that for small $\|\Phi\|_{C^1} = \delta_{\Phi} \ll 1$, we have
											\begin{eqnarray*}
												e^{-\Phi} \nu(v) + \frac{1}{w}\nabla_y\Phi\cdot\nabla_v w  \geq \frac{1}{2} e^{-\delta_{\Phi}} \nu (v) .  \\
											\end{eqnarray*}
											This inequality implies,
											\begin{equation} \label{exponent bound}
											\begin{split}
											e^{-\int_{s}^{t}e^{ -\Phi(X)}\nu(V)\dd\tau - \int_s^t \frac{1}{w}\nabla_y\Phi\cdot\nabla_v w  } &\leq e^{ -\frac{1}{2} e^{-\delta_{\Phi}}  \nu(v) (t-s) }  := e^{ -\frac{1}{2} \nu_{\Phi}(v) (t-s) },
											\end{split}
											\end{equation}
											where we defined ${\nu}_{\Phi}(v) := e^{-\delta_{\Phi}}  \nu(v)$. Then, similar as proof of Theorem \ref{theorem_time},
											
											\begin{eqnarray*} 
												&& h (t,x,v) \\
												&=&  { E(v,t,T) h(T) }  + { \int^{t}_{T } E(v,t,s) e^{-\Phi} \int_{u} k_{w}(u,v) E(u,s,T) h(T) \dd u \dd s }  \\
												&& + \int^{t}_{T }  E(v,t,s) e^{-\Phi(X(s))} \int_{u} k_{w}(u,v) 
												\int^{s}_{T} E(u,s,s^{\prime}) e^{-\Phi(X(s^{\prime}))} 	\\
												&&\quad \times	\int_{u^{\prime}}
												k_{w}(u^\prime,u) E(u^{\prime},s^\prime,T) h (T) \dd u^{\prime} \dd s ^{\prime} \dd u \dd s     \\
												&& + \underline{ \int^{t}_{T }  E(v,t,s) e^{-\Phi(X(s))} \int_{u} k_{w}(u,v) 
													\int^{s}_{T} E(u,s,s^{\prime}) e^{-\Phi(X(s^{\prime}))} }	\\
												&& \underline{ \times \int_{u^{\prime}}
													k_{w}(u^\prime,u) \int_T^{s^\prime} E(u^{\prime},s^\prime,s^{\prime\prime}) e^{-\Phi(X(s^{\prime\prime}))}  }   \\
												&& \underline{ \times \int_{u^{\prime\prime}} k_{w}(u^{\prime\prime},u^\prime) h( s^{\prime\prime},  X( s^{\prime\prime} ; s^\prime, X( s^\prime ; s , X(s;t,x,v) , u ), u^\prime) , u^{\prime\prime} ) \dd u^{\prime\prime} \dd s ^{\prime\prime} \dd u^{\prime} \dd s ^{\prime} \dd u \dd s   }_{(IV)},  \\
											\end{eqnarray*}	
											where we defined,
											\begin{eqnarray*}
												E(v,t,s) &:=& e^{-\int_{s}^{t}e^{ -\Phi(X(s;t,x,v))}\nu(V(s;t,x,v))\dd\tau  - \int_s^t \frac{1}{w}\nabla_x\Phi(X(s;t,x,v))\cdot\nabla_v w(V(s;t,x,v))} .
											\end{eqnarray*}
											Except $(IV)$, the rest of terms are clearly bounded by 
											\begin{equation} \label{(I)}
											e^{- \frac{1}{2} {\nu}_{\Phi}(0) (t-T) } \| h(T) \|_\infty. 
											\end{equation}
											
											\noindent Estimate for $(IV)$ is gained by change of variable similar as (\ref{MAIN}) in proof of theorem \ref{theorem_time}. Using definition (\ref{opeator k split}) and performing change of variable,
											\begin{equation} \label{pre (IV)}
											\begin{split}
											(IV) &\lesssim \ C_{N,\O,\Phi,\b} \int_{T}^{t} \int_{X^{\prime\prime}}\int_{u^{\prime\prime}}  h (s^{\prime\prime},X^{\prime\prime}(s^{\prime\prime}),u^{\prime\prime}) \dd u^{\prime\prime} \dd X^{\prime\prime} \dd s 
											+ C_{N,\O,\Phi} \delta^{\gamma} \sup_{s\in[T,t]} \|h (s)\|_\infty  \\
											&\lesssim \ C_{N,\O,\Phi,\b} \int_{T}^{t} \|f (s)\|_{L^2_{x,v}} \dd s 
											+  C_{N,\O,\Phi} \delta^{\gamma} \sup_{s\in[T,t]} \|h (s)\|_\infty .  \\
											\end{split}	
											\end{equation}
											Hence,
											\begin{equation} \label{h infty est}
											\begin{split}
											\sup_{s\in[T,t]} \|h(s)\|_\infty & \lesssim_{N,\O,\Phi,\b} \ e^{-\frac{1}{2} {\nu}_{\Phi}(0) (t-T) } \|h (T)\|_\infty +  \int_{T}^{t} \|f (s)\|_2 \dd s  .
											\end{split}
											\end{equation}
											
											We assume that $m \leq t < m+1$ and define $\lambda^{*}:=\min\{ \frac{\nu_{\Phi}(0)}{2}, \lambda\}$, where $\lambda$ is some constant from Corollary \ref{decay_U}. We use (\ref{h infty est}) repeatedly for each time step, $[k,k+1), \ k\in\mathbb{N}$ and Corollary \ref{decay_U} to perform $L^{2}-L^{\infty}$ bootstrap, {i.e.}
											\begin{equation} \label{h infty decay}
											\begin{split}
											\|h(t)\|_{\infty} 
											&\lesssim_{N,\O,\Phi,\b} e^{-m \frac{\nu_{\Phi}(0)}{2}} \|h(0)\|_{\infty} + \sum_{k=0}^{m-1} e^{-k \frac{\nu_{\Phi}(0)}{2}} \int_{m-1-k}^{m-k} \|f(s)\| \dd s  \\
											&\lesssim_{N,\O,\Phi,\b} e^{-m \frac{\nu_{\Phi}(0)}{2}} \|h(0)\|_{\infty} + \sum_{k=0}^{m-1} e^{-k \frac{\nu_{\Phi}(0)}{2}} \int_{m-1-k}^{m-k} e^{-\lambda(m-1-k)} \|f(0)\| \dd s  \\
											&\leq C_{N,\O,\Phi,\b} e^{-\frac{\lambda^{*}}{2} t } \|h(0)\|_{\infty} . \\
											\end{split}
											\end{equation} 
											
											For nonlinear problem from Duhamel principle,
											\begin{equation} \label{Duhamel}
											\begin{split}
											h  &:= U(t)h_0 + \int_{0}^{t} U(t-s) w e^{-\frac{\Phi}{2}} \Gamma(\frac{h }{w},\frac{h }{w}) (s)  \dd s ,  \\
											\|h (t) \|_{\infty} &\lesssim_{N,\O,\Phi,\b} e^{-\frac{\lambda^{*}}{2} t } \|h(0)\|_{\infty} + \Big\| \int_{0}^{t}  U(t-s) w e^{-\frac{\Phi}{2}} \Gamma\big( \frac{h }{w}, \frac{h }{w} \big)(s) \dd s \Big\|_{\infty} , 
											\end{split}
											\end{equation}
											where $U(t)$ is linear solver for linearized Boltzmann equation. Inspired by \cite{Guo10}, we use Duhamel's principle again, \textit{i.e},
											\begin{equation} \label{double duhamel}
											U(t-s) = G(t-s) + \int_s^t G(t-s_1) K_{w}U(s_1-s) \dd s_1,
											\end{equation}
											where $G(t)$ is linear solver for the system
											\begin{equation} \label{step 2 eq}
											\begin{split}
											&\p_t h + v\cdot\nabla_x h - \nabla_x \Phi\cdot\nabla_v h + \frac{h}{w} \nabla\Phi\cdot\nabla_v w + e^{-\Phi} \nu h = 0 \\  &\quad \text{and}\quad |G(t)h_0| \leq e^{-\frac{1}{2}\nu_{\Phi}(v)t} |h_0|.
											\end{split}
											\end{equation}
											
											\noindent For the last term in (\ref{Duhamel}),
											\begin{equation} \label{nu increase}
											\begin{split}
											& \Big\| \int_{0}^{t}  U(t-s) w e^{-\frac{\Phi}{2}} \Gamma\big( \frac{h }{w}, \frac{h }{w} \big)(s) \dd s \Big\|_{\infty}   \\
											&\leq  \Big\| \int_{0}^{t}  G(t-s) w  \Gamma\big( \frac{h }{w}, \frac{h }{w} \big)(s) \dd s \Big\|_{\infty}  \\
											&\quad +   \Big\| \int_{0}^{t}  \int_{s}^{t} G(t-s_1) K_{w}  U(s_1-s) w \Gamma\big( \frac{h }{w}, \frac{h }{w} \big)(s) \dd s_1 \dd s  \Big\|_{\infty}   \\
											&\leq 
											C_{\Phi} e^{-\frac{\lambda^{*}}{2}t} \Big(\sup_{0\leq s\leq \infty} e^{\frac{\lambda^{*}}{2}s} \|h (s)\|_{\infty} \Big)^{2}.
											\end{split}
											\end{equation}	
											
											\noindent Therefore, for sufficiently small $\|h_0\|_{\infty} \ll 1$, we have a uniform bound
											\begin{equation} \label{uniform bound} 
											\sup_{0\leq t\leq \infty} e^{\frac{ \lambda^{*} }{2} t } \|h(t)\|_\infty \ll 1.
											\end{equation} 
											
											\noindent From this small uniform bound, we get global decay and uniqueness. Positivity was already proved in Theorem \ref{local_existence}.\end{proof}

        





\ack The authors thank Yan Guo and Sergey Bolotin for stimulating discussions. They also thank Nikolai Chernov that he answered kindly to many questions about the billiard theory. The author thanks a referee for the valuable comments. Their research is supported in part by NSF-DMS 1501031 and the University of Wisconsin-Madison Graduate School with funding from the Wisconsin Alumni Research Foundation. C. Kim thanks KAIST Center for Mathematical Challenges for the kind hospitality. 


\frenchspacing
\bibliographystyle{plain}

\end{document}